\documentclass[11pt]{article}
\usepackage{mathpazo }
\usepackage[margin=1in]{geometry}
\usepackage{yk, dsfont}
\usepackage{mathrsfs}
\usepackage{titletoc}
\usepackage[dvipsnames]{xcolor}
\usepackage{hyperref}
\usepackage{graphicx}
\usepackage{enumitem}
\usepackage[numbers,sort&compress]{natbib}
\usepackage{float}
\usepackage{subcaption}
\usepackage{url}
\usepackage{algorithm}
\usepackage{algpseudocode}
\algrenewcommand\algorithmicrequire{\textbf{Input:}}
\algrenewcommand\algorithmicensure{\textbf{Output:}}
\newcommand{\indep}{\perp\mkern-9.5mu\perp}

\newcommand{\R}{\mathrm{R}}

\newcommand{\rk}{s}
\newcommand{\sx}{\sigma_{0}}
\newcommand{\rout}{ \mathcal R^{\mathrm{pred}}(f;c)}
\newcommand{\rest}{ \mathcal R^{\mathrm{est}}(f;c)}

\usepackage{xcolor}
\usepackage{hyperref}
\usepackage{authblk}

\usepackage[T1]{fontenc}
\usepackage{forest}
\usepackage{xcolor}

\definecolor{thmcolor}{RGB}{173, 216, 230}   
\definecolor{lemcolor}{RGB}{144, 238, 144}   
\definecolor{propcolor}{RGB}{255, 218, 120}  
\definecolor{mainlemcolor}{RGB}{255, 223, 0}
\definecolor{emailcolor}{rgb}{0.0, 0.0, 0.0}

\title{Self-Distillation is Optimal Among Spectral Shrinkage Estimators}

\author[1]{Radu Lecoiu}
\author[2,$*$]{Debarghya Mukherjee}
\author[1,$*$,$\dagger$]{Pragya Sur}

\affil[1]{\small Department of Statistics, Harvard University}
\affil[2]{\small Department of Mathematics \& Statistics, Boston University}

\makeatletter
\renewcommand\AB@affilsepx{\\ \protect\Affilfont}

\makeatother

\begin{document}
\maketitle

\renewcommand{\thefootnote}{}
\footnotetext{$^*$co-senior authors.}
\footnotetext{$^\dagger$Emails: 
\href{mailto:rlecoiu@college.harvard.edu}{\textcolor{emailcolor}{rlecoiu@college.harvard.edu}}; 
\href{mailto:mdeb@bu.edu}{\textcolor{emailcolor}{mdeb@bu.edu}}; 
\href{mailto:pragya@fas.harvard.edu}{\textcolor{emailcolor}{pragya@fas.harvard.edu}}.}

\begin{abstract}
Self-distillation has emerged as a promising technique for improving model performance in modern machine learning systems.  We develop the statistical foundations of self-distillation in high-dimensional linear regression by introducing and analyzing a broad class of estimators, namely spectral shrinkage estimators. We consider both spiked and general covariance matrices.
For the former, 
we establish that when the covariance has $\rk$ distinct spikes, $\rk$-step self-distillation achieves optimal prediction among spectral shrinkage estimators, outperforming well-known estimators in statistics and machine learning. Moreover, we show that $\rk$ steps are necessary for optimality: any $(\rk-k)$-step distilled estimator is strictly suboptimal for $1 \leq k \leq s$. 
For general covariance matrices, we show that self-distillation approaches the optimal limiting prediction risk among continuous spectral shrinkage estimators as the number of distillation steps increases. We further develop a generalized cross-validation procedure for selecting the distillation parameters from data and establish an asymptotic oracle guarantee for the resulting estimator. We also study a federated approach where multiple data centers share spectral shrinkage estimators and a common server seeks to aggregate them to achieve optimal performance. In this case, we find that the best local rule again takes the form of self-distillation, though it differs from the optimal rule when data are hosted centrally on a single server. Together, our results explain why self-distillation improves predictive performance and provide a broader statistical framework that connects it to classical shrinkage-based methods.
\end{abstract}

\startcontents

\printcontents{}{1}{\setcounter{tocdepth}{2}}

\section{Introduction}\label{sec:intro}
Distillation is a central training paradigm in modern AI. It allows to transfer knowledge from large, complex models to smaller ones that can be deployed at scale.  Originally proposed in \cite{bucilua2006model,hinton2015distilling,ba2014deep}, it has been successfully applied across diverse settings,  from theorem proving to real-time disaster damage assessment \cite{lin2025goedel,jeon2026edgev}. Despite these practical successes, the statistical foundations explaining when and why distillation works remain seriously underdeveloped.  In this paper, we address this substantial gap by developing a rigorous theory for self-distillation in high-dimensional regression.

We consider a supervised learning problem where responses follow a linear model with high-dimensional covariates. 
Self-distillation refers to the subclass of distillation strategies where the model we distill from shares the same structure as the model we distill to. To illustrate in simple statistical terms, suppose we use the training data to fit an initial shrinkage estimator, such as ridge regression. Self-distillation trains successive ridge estimators, encouraging their predictions to align with those of the preceding estimator. Although this seems a counter-intuitive strategy that re-uses data multiple times, self-distillation has demonstrated strong empirical performance \cite{furlanello2018born, zhang2021self, li2017learningnoisylabelsdistillation}. 
In this paper, we develop a rigorous statistical theory explaining why self-distillation works and when it outperforms the first-stage estimator. In particular, we establish that self-distillation achieves optimal prediction error within a broad class of shrinkage-based estimators. To the best of our knowledge, our work provides the first such optimality result for self-distillation under high-dimensional statistical models.

We further expand our analysis to  decentralized learning, where data is distributed across multiple clients and communication with a central server is limited due to privacy and budgetary constraints. Such settings are commonly studied under the  paradigms of distributed learning \cite{dean2012large,zhang2012communication,zinkevich2010parallelized,deterministic_calculus} and federated learning \cite{konevcny2016federated,mcmahan2017communication,smith2017federated}, whose goal is to leverage information across clients without directly sharing raw data.

We summarize our main contributions below. 

\begin{enumerate}
\item We introduce a class of spectral shrinkage-based estimators (Definition \ref{def:shrinkage_class}) and establish that, when the covariate distribution is not isotropic, self-distillation achieves the optimal out-of-sample risk within this class.  First, for spiked covariance models \cite{johnstone2001distribution}  with $s$ spikes, we show that $s$-step self-distillation is optimal (Theorem \ref{thm:sd_optimal} (a)).  Interestingly, this advantage disappears in the isotropic setting, where ridge regression is itself optimal  (see Lemma \ref{thm: main corollary} and subsequent discussion).
\item We further show that, when the feature covariance matrix has $s$ distinct spikes, $k$-step self-distillation is sub-optimal  for any  $0 \leq k < s$ (Theorem~\ref{thm:sd_optimal}(b)). \stepcounter{footnote}
\hypertarget{sd-zero-text}{}
\hyperlink{sd-zero-note}{\textsuperscript{\arabic{footnote}}}
{\renewcommand{\thefootnote}{}
 \footnotetext{\hypertarget{sd-zero-note}{}
   \hyperlink{sd-zero-text}{\textsuperscript{\arabic{footnote}}}\,
   Here, 0-step self-distillation refers to the initial shrinkage estimator.}}
   
\item The shrinkage-based estimators in Definition \ref{def:shrinkage_class} allow us to compare self-distillation with well-known procedures such as ridge regression, minimum norm interpolation, and principal components regression. We establish that self-distillation strictly outperforms these alternatives under spiked covariance models (Theorem \ref{thm: f_ast dominates}). In fact, with an appropriately chosen number of steps, self-distillation is almost uniquely optimal (we clarify the notion of uniqueness in Theorem \ref{thm: main}).
\item In Section \ref{sec:general-covariance}, we consider  covariance matrices with arbitrary covariance structure and impose minimal assumptions on their eigenvalue and eigenvectors. We prove that, as the number of distillation steps grows, self-distillation achieves the minimum prediction error among estimators in the spectral shrinkage class (Theorem \ref{thm:general-covariance}).
\item Thus far, our optimality theorems show that self-distillation is optimal for certain choices of distillation weights and other tuning parameters. In Section \ref{sec:gcv}, we develop a generalized cross validation approach to estimate these tuning parameters (Theorem \ref{th:growing}). This yields a data-driven strategy for selecting the best self-distillation procedure and thereby, the  optimal shrinkage function within the spectral shrinkage class. For the general covariance case, we also determine how many distillation steps are needed in practice to achieve prediction error arbitrarily close to the minimum (Theorem \ref{thm:gcv}).

\item In Section~\ref{sec:theory_fed}, we extend our analysis to a decentralized setting where each client computes a local spectral shrinkage estimator and sends it to a central server, which aggregates the received estimators linearly. We characterize the limiting prediction risk of the aggregated estimator and derive the optimal local shrinkage rules and aggregation weights. The optimal local rule is again realized by self-distillation; however, it differs from the optimal rule for a single client (the focus of the first part of the paper), reflecting the new bias--variance tradeoff induced by aggregation across clients (Theorem \ref{thm: product_of_shrinkage}).

A key ingredient in our analysis which may be of independent interest, is a new random matrix result for quadratic forms involving products of functions of distinct sample covariance matrices (Theorem \ref{lemma:non-free limit}).

\end{enumerate}

Throughout, we operate within the proportional asymptotics regime, where the number of features and samples grows at a comparable rate. This regime has become increasingly popular in high-dimensional statistics and  machine learning theory, owing to its ability to capture new phenomena observed in high dimensions,   explain empirical behavior seen in modern machine learning systems, and produce new inference and prediction methods with impressive practical performance
\cite{amp_first_2009, LASSO_AMP_2012, AMP_2013_javanmard_montanari, el_karoui_2013, Zdeborova2016, ElKaroui2018, SurChenCandes2019, SurCandes_2019, hastie2022surprises, Fan2022, AMPSurvey_2022, Yue2023, CelentanoMontanariWei2023, MontanariSen_2024, ChenLiuMukherjee2025, Kuanhao2025, saenz2025characterizingfinitedimensionalposteriormarginals, DEPOPE2026, li2026optimal, LiSur2026}.
The framework has also yielded  powerful tools for studying the behavior of high-dimensional random matrices, which would be indispensable for our work. Before proceeding, we situate our work within the existing literature on self-distillation.

\subsection{Related Work}
We study a setting where the covariates and responses are sampled i.i.d.~from a population. In the form of self-distillation we consider, at each stage, an estimator is learned by training on the original data while borrowing information from the estimator obtained in the previous stage. This is encapsulated by the weighted loss function \eqref{eq:SD-recursion}: at stage $t$, $\xi_t$ denotes the weight on the component that learns from the previous stage estimator, while $(1-\xi_t)$ denotes the weight on the component that learns directly from the training data. 

Prior work most closely related to ours include \cite{dang2026optimalunconstrainedselfdistillationridge,cui2026asymptotic}. The former studies the optimal choice for $\xi_t$ while constraining estimators at each stage to use the same tuning parameter as the previous stage. This significantly limits the class of estimators self-distillation can capture---indeed, we demonstrate later that the optimal self-distillation strategy uses different tuning parameters for different stages (Theorem~\ref{thm:sd_optimal}(a)). On the other hand, \cite{cui2026asymptotic} studies a two-step empirical risk minimization procedure that generalizes self-distillation and derives a precise high-dimensional characterization of the second-stage test error for linear models under Gaussian-mixture features. But, in its present form, their analysis does not extend to the multi-step self-distillation setting we consider here.

Several other lines of work study problems adjacent to, but not directly related to, our setting. For instance,  \cite{mobahi2020selfdistillationamplifiesregularizationhilbert,wu2026selftraininghelpshurtsdenoising,moniri2025mechanismsweaktostronggeneralizationtheoretical,ildiz2025highdimensionalanalysisknowledgedistillation,ain2024scalinglawslearningreal,allentowards} study pure distillation (where $\xi_t=1$).
A parallel literature considers the case of fixed design and low dimensions \cite{das2023understandingselfdistillationpresencelabel,pareek2024understandinggainsrepeatedselfdistillation, phuong2021understandingknowledgedistillation}. Among these, \cite{ildiz2025highdimensionalanalysisknowledgedistillation,das2023understandingselfdistillationpresencelabel} restrict their analyses to a single distillation step, unlike ours. A separate line of work considers regenerating synthetic labels at each stage from the model fitted in the previous stage \cite{takanami2025effect,garg2026preventingmodelcollapseoverparametrization,javanmard2025self}, studying either the impact of synthetic data on the long-run generalization error or the optimal rule for aggregating synthetic and original labels.

On a different note, \cite{dong2019distillation} studies connections between knowledge distillation and early stopping, while \cite{saglietti2020solvablemodelinheritingregularization} provides a replica analysis of knowledge distillation in Gaussian mixture models. 
Finally, \cite{boixadsera2024theorymodeldistillation} formulates distillation as a PAC-style imitation problem between two hypothesis classes; this work is closer in spirit to model compression \cite{bucilua2006model} than to distillation in the sense of \cite{hinton2015distilling}.  A related body of literature investigates semi-supervised learning and self-training \cite{azar2024semi,lee_dong,Amini_2025,carmon2022unlabeleddataimprovesadversarial,pmlr-v119-kumar20c,chen2020selftrainingavoidsusingspurious,oymak2020statistical,wei2020theoretical,zhang2022unlabeled}. As the papers surveyed in this and the preceding paragraph address problems that differ substantially  from those studied here, we refrain from discussing these references further.

\subsection{Paper Organization}
The rest of the paper is organized as follows: 
Section~\ref{sec: setup + shrinkage framework} introduces our problem setup and other preliminaries, including formal definitions of self-distillation and the spectral shrinkage class.
Section~\ref{sec:theory} presents our main results with supporting experiments in Section~\ref{secL simulations}.
Section~\ref{sec: noverty + outline}  presents the main novelties in our proofs. 
Section~\ref{sec:discussion} concludes the paper with a discussion of future directions. 
All proofs 
are presented in the supplementary document.

\section{Preliminaries}\label{sec: setup + shrinkage framework}
\subsection{Problem Setup}
We observe a sequence of problem instances $\{y_i,\bx_i,\varepsilon_i, 1 \leq i \leq n, \bbeta_0(n) \}_{n \geq 1}$, where  $y_i \in \mathbb{R}, \bx_i \in \mathbb{R}^{p(n)}$ are i.i.d.~samples following the linear model
\begin{equation}
y_i = \bx_i^\top \bbeta_0(n) + \varepsilon_i, 
\qquad i = 1,\dots,n, 
\label{eq:model}
\end{equation}
and $\bbeta_0(n) \in \mathbb{R}^{p(n)}$ is the unknown regression vector. Throughout, we work under the proportional asymptotics regime where $p(n)/n \to c \in (0,\infty)$. Henceforth, we suppress the dependence on $n$ when clear from the context. Let $\bX \in \mathbb{R}^{n \times p}$ denote the design matrix with rows $\bx_i^\top$, and $\by \in \mathbb{R}^n$ the response vector. Denote $\widehat{\boldsymbol{\Sigma}} = (\bX^\top \bX)/n$ to be the sample covariance matrix. We assume the following structure for the covariates and the errors.

\begin{assumption}[Spiked Covariance Model]
\label{assm:dgp}
The covariates take the form 
\(
\bX = \bZ \boldsymbol{\Sigma}^{1/2},
\)
where $\bZ \in \mathbb{R}^{n \times p}$ has i.i.d.\ entries satisfying
\[
\mathbb{E}[Z_{ij}] = 0, \quad \mathbb{E}[Z_{ij}^2] = 1, \quad \mathbb{E}\bigl[|Z_{ij}|^{8+\eta}\bigr] < \infty
\]
for some $\eta > 0$. The regression errors $\varepsilon_1,\dots,\varepsilon_n$ are i.i.d.\ and satisfy
\[
\mathbb{E}[\varepsilon_i] = 0, \quad \mathbb{E}[\varepsilon_i^2] = \sigma_\beps^2, \quad 
\mathbb{E}\bigl[|\varepsilon_i|^{4+\eta}\bigr] < \infty
\]
for some $\eta > 0$. Moreover, $(\varepsilon_i)_{i=1}^n$ is independent of $(\bx_i)_{i=1}^n$.
The population covariance matrix $\bSigma$ follows an $\rk$-spiked covariance model for some fixed $\rk \in \mathbb{N} \cup \{0\}$, i.e.,
\[
\boldsymbol{\Sigma}
= \sx^2 \mathbf{I}_p + \sum_{j=1}^{\rk} \delta_j \bv_j \bv_j^\top,
\]
where $\sx^2 > 0$,  and $\{\bv_j\}_{j=1}^{\rk}$ is an orthonormal set of vectors in $\mathbb{R}^p$. The case $\rk = 0$ corresponds to the isotropic design $\boldsymbol{\Sigma} = \sx^2 \mathbf{I}_p$.
\end{assumption}

Spiked covariance models have been extensively studied in high-dimensional statistics since the seminal work of \cite{johnstone2001distribution} and have become a standard framework for analyzing principal component methods and related spectral procedures. Although the model represents a special case of a general anisotropic covariance matrix, it is already rich enough to capture the key phenomena governing self-distillation in anisotropic settings. In particular, as we show below, even a finite-spike perturbation away from isotropy yields qualitatively different conclusions regarding the superiority of self-distillation over classical procedures (e.g., ridge regression and principal components regression (PCR)) compared to the isotropic setting.

Under Assumption~\ref{assm:dgp}, the eigenvalues of 
$\widehat{\bSigma}$ exhibit a well-understood limiting structure. In 
the isotropic case $\bSigma = \sx^2\bI_p$, the 
Marchenko--Pastur law \cite{MarchenkoPastur1967} states that 
the empirical spectral distribution of $\widehat{\bSigma}$ converges 
almost surely to a limiting measure supported on the compact interval 
$\mathcal{S}_c = [\sx^2(1-\sqrt{c})^2,\, \sx^2(1+\sqrt{c})^2]$, 
known as the bulk, and additionally has a point mass at $0$ when $c > 1$. In the spiked model, \cite{baik2004} shows that sufficiently strong spikes generate additional outlier eigenvalues that escape the bulk, converging to deterministic locations outside $\mathcal{S}_c$. We collect this structure in the following definition.

\begin{definition}[Limiting spectral support]
\label{def:spectral-support}
Denote the bulk support of 
the Marchenko--Pastur distribution by
\[
\mathcal{S}_c 
= \bigl[\sx^2(1-\sqrt{c})^2,\;\sx^2(1+\sqrt{c})^2\bigr].
\]
Now, for each spike $\delta_j$ satisfying $\delta_j > \sx^2\sqrt{c}$, define an outlier atom $x_j^\star$ as: 
$$
x_j^\star := \frac{(\delta_j+\sx^2)(\delta_j+c\sx^2)}{\delta_j} \,.
$$
We collect these outlier atoms, 
together with $0$ when $c > 1$, in the set $\mathcal{A}_c$, and write $\mathcal{S}_c^+ := \mathcal{S}_c \cup \mathcal{A}_c$. 
\end{definition}
It is an immediate consequence of the  Baik-Ben Arous-P\'ech\'e (BBP) phase 
transition \cite{baik2004phasetransitionlargesteigenvalue} that the limiting spectral distribution of $\widehat{\bSigma}$ is supported on $\mathcal{S}_c^+$, when $\bSigma$ follows Assumption \ref{assm:dgp}.  
We work with the following  assumptions related to the spike strengths of the covariance matrix ${\boldsymbol{{\boldsymbol{\Sigma}}}}$ and the alignment between ${\boldsymbol{{\boldsymbol{\Sigma}}}}$ and $\bbeta_0$: 
\begin{assumption}
\label{assm:main}
Suppose that ${\boldsymbol{{\boldsymbol{\Sigma}}}}$ is of the above spiked form and that the following conditions hold:
\begin{enumerate}
    \item \emph{(Distinct spikes.)} The spike strengths $\delta_1,\dots,\delta_{\rk}$ are positive and distinct.
    Moreover,
    $$
    \delta_i \delta_j \neq c\sx^4,
    \qquad \text{for all } 1 \le i,j \le \rk.
    $$
\item \emph{(Signal Decomposition)} Writing 
$\bbeta_0 = \sum_{j=1}^s (\bbeta_0^\top \bv_j)\bv_j + \widetilde{\bbeta}_0$,
where $\widetilde{\bbeta}_0$ lies in the orthogonal complement of 
$\mathrm{span}\{\bv_1,\dots,\bv_\rk\}$, assume that the limits $ \lim_{n,p\to\infty}\|\widetilde{\bbeta}_0\|_2,$ and 
\[
  \alpha_j := \lim_{n,p\to\infty} \bbeta_0^\top \bv_j
  \quad (1 \le j \le \rk)
\]
all exist and are nonzero, and that
\[
  r := \lim_{n,p\to\infty}\|\bbeta_0\|_2 \in (0,\infty).
\] 
\end{enumerate}
\end{assumption}
Assumption \ref{assm:main} (1) rules out an edge case where the product of two spike strengths exactly equals $c\sigma_0^4$. 
This is a mild condition that, importantly, does \emph{not} require the spike strengths to lie either above or below the BBP threshold. Our theory allows all, some, or none of the spikes to exceed the BBP threshold. The assumption is imposed only to avoid boundary cases and to keep the presentation technically transparent. Our qualitative conclusion on the superiority of self-distillation remains valid even when this condition is violated, although the resulting risk characterization is less clean; see Section \ref{sec:discussion} for further discussion.
Assumption \ref{assm:main} (2) is a mild compatibility condition between $\bSigma$ and $\beta_0$. Specifically, it ensures that $\beta_0$ has a non-zero projection onto each spike direction, so that each spike carries information relevant for learning $\beta_0$, and that $\beta_0$ does not lie entirely within the span of the spike directions.

 In Section~\ref{sec:general-covariance}, we consider arbitrary covariance matrices, beyond the spiked covariance case. For this part, we will operate under the following assumption.

\begin{assumption}[General Covariance Model]
\label{assm:spectral-convergence}
We retain the distributional assumptions on the design and noise
from Assumption~\ref{assm:dgp}. There exist constants
$0<\ell\le L<\infty$ and $B<\infty$, independent of $n$, such that
 $\ell\bI_p\preceq\bSigma\preceq L\bI_p$, and $\|\bbeta_0\|_2\le B.$
Let $(\mathfrak{d}_j,\bv_j)_{j=1}^p$ be the eigenpairs of $\bSigma$,
with the eigenvectors chosen to form an orthonormal basis. Define
\begin{equation}\label{eq:general-covariance-measures}
 H_n:=\frac1p\sum_{j=1}^p\delta_{\mathfrak{d}_j},
 \qquad
 G_n:=\sum_{j=1}^p
       |\bv_j^\top\bbeta_0|^2\delta_{\mathfrak{d}_j}.
\end{equation}
We assume that $H_n\Rightarrow H$ and $G_n\Rightarrow G$ weakly,
where $H$ is a probability measure and $G$ is a finite measure on
$[\ell,L]$.
\end{assumption}
The measure $H_n$ assigns equal weight to each population eigenvalue,
whereas $G_n$ weights each eigenvalue by the squared signal component
in its eigendirection. We explicitly leave $G_n$ unnormalized, which allows $\|\beta_0\|_2 \to 0$.

\subsection{Self-Distillation}
Next, we define our self-distillation procedure of interest. 
Starting from a ridge estimator $\widehat{\bbeta}^{(0)}$ with penalty parameter $\lambda_0 > 0$, 

\begin{align}\label{eq:beta0}
\widehat{\bbeta}^{(0)} \equiv \widehat{\bbeta}^{(0)}_\mathrm{SD}
=
\arg\min_{\bbeta}
\left\{\frac{1}{n}\sum_{i=1}^n (y_i - \bx_i^\top \bbeta)^2 + 
\lambda_0 \|\bbeta\|_2^2
\right\},
\end{align}
the $k$-step self-distilled estimator is constructed recursively by calculating for $\lambda_t  > 0$, 
\begin{align}\label{eq:SD-recursion}
\widehat{\bbeta}^{(t)}_\mathrm{SD}
=
\arg\min_{\bbeta}
\left\{
(1-\xi_t)\frac{1}{n}\sum_{i=1}^n (y_i - \bx_i^\top \bbeta)^2
+
\xi_t \frac{1}{n}\sum_{i=1}^n \bigl(\bx_i^\top(\widehat{\bbeta}^{(t-1)}_\mathrm{SD}-\bbeta)\bigr)^2
+
\lambda_t \|\bbeta\|_2^2
\right\},
\end{align}
for $1 \le t \le k$, where the weights $\xi_t$ and penalties $\lambda_t$ are allowed to vary across stages. 
The weight $\xi_t$ trades off borrowing information from the estimator learned in the previous stage against fitting the data.   For $\lambda_t$, we extend the aforementioned definition to also allow for negative regularization. This is motivated by the recent literature on high-dimensional ridge regression, which has shown that the optimal ridge penalty $\lambda_0$ can be negative under overparametrization, depending on the specific structure of and interactions between the feature covariance and the signal geometry \cite{kobak2020optimalridgepenaltyrealworld,wu_and_xu,richards_negative_ridge,negative_ridge_arzela_ascoli,bartlett, negative_ridge_tibshirani}. 
To this end, computing the gradient of the objective in \eqref{eq:SD-recursion} and 
setting it to zero, the first-order condition for 
$\widehat{\bbeta}^{(t)}_\mathrm{SD}$ is given by
\begin{equation*}
    \bigl(\widehat{\bSigma} + \lambda_t \bI_p\bigr)\bbeta
    =
    (1-\xi_t)\frac{\bX^\top\by}{n}
    + \xi_t\,\widehat{\bSigma}\,\widehat{\bbeta}^{(t-1)}_\mathrm{SD}.
\end{equation*}
For 
$\lambda_t > 0$, the matrix $\widehat{\bSigma} + \lambda_t\bI_p$ is 
strictly positive definite, so the objective is strongly convex and has a unique minimum given by
\begin{equation}
\label{eq:SD-closedform}
    \widehat{\bbeta}^{(t)}_\mathrm{SD}
    =
    (\widehat{\bSigma} + \lambda_t\bI_p)^{-1}
    \left[
        (1-\xi_t)\frac{\bX^\top\by}{n}
        + \xi_t\,\widehat{\bSigma}\,\widehat{\bbeta}^{(t-1)}_\mathrm{SD}
    \right].
\end{equation}
For $\lambda_t < 0$, however, this no longer holds. 
We resolve this issue by replacing the inverse by the  Moore--Penrose pseudoinverse, setting
\begin{equation}
\label{eq:SD-MP}
    \widehat{\bbeta}^{(t)}_\mathrm{SD}
    :=
    (\widehat{\bSigma} + \lambda_t\bI_p)^{\dagger}\left[
        (1-\xi_t)\frac{\bX^\top\by}{n}
        + \xi_t\,\widehat{\bSigma}\,\widehat{\bbeta}^{(t-1)}_\mathrm{SD}
    \right],
    \quad \lambda_t \in \mathbb{R}.
\end{equation}
This definition is well-posed for all $\lambda_t \in \mathbb{R}$ and 
all $n$, and reduces exactly to 
\eqref{eq:SD-closedform}  when $\lambda_t > 0$. We adopt this definition throughout the paper. 
Note that zero-step self-distillation is simply ridge regression. 

\subsection{Introducing the Spectral Shrinkage Class}
We establish in Lemma~\ref{lemma: sd is in mathcal F} that after $t$ steps of self-distillation, the estimator $\widehat{\bbeta}^{(t)}_\mathrm{SD}$ admits a representation of the 
form
\begin{equation}
\label{eq:spectral-rep}
    \widehat{\bbeta}^{(t)}_\mathrm{SD}
    = f_t(\widehat{\bSigma})\frac{\bX^\top\by}{n},
\end{equation}
for a special rational function $f_t$, where for any function 
$f:[0,+\infty)\to\mathbb{R} \cup\{\infty\}$ and any symmetric matrix 
$\bA \in \mathbb{R}^{p \times p}$ with eigendecomposition $\bA = \bU\bD\bU^\top$ we define
\begin{equation}\label{eq: f(A)}
    f(\bA) := \bU f(\bD)\bU^\top,
    \quad \text{where} \quad 
    [f(\bD)]_{jj} = 
    \begin{cases} 
        f(d_j) & \text{if } |f(d_j)| < \infty, \\ 
        0      & \text{otherwise},
    \end{cases}
\end{equation}
so that whenever $|f(d_j)| = \infty$, the corresponding direction contributes nothing to the matrix $f(\bA)$. This is
consistent with the Moore--Penrose pseudoinverse convention. This representation shows that multi-step self-distillation belongs 
to a broader family of estimators obtained by applying a shrinkage 
function $f$ to the spectrum of $\widehat{\bSigma}$. 
Our primary goal is to develop a rigorous theory for self-distillation in high-dimensional regression and to characterize its performance relative to classical statistical estimators (e.g., ridge, PCR, etc.). The aforementioned observation, therefore, leads us to consider a general class of spectral shrinkage estimators.

\begin{definition}[Spectral shrinkage class]
\label{def:shrinkage_class}
Given a collection $\mathcal{F}$ of functions, define the \emph{spectral shrinkage class}
of estimators associated to $\mathcal{F}$ by
\begin{equation}
\label{eq:def_bf}
\mathcal{B}_{\mathcal{F}}
=
\left\{
\widehat{\bbeta}_f
=
f(\widehat{\bSigma})\,\frac{\bX^\top \by}{n}
:\,
f \in \mathcal{F}
\right\},
\end{equation}
where $f(\widehat{\bSigma})$ 
is defined as in~\eqref{eq: f(A)}, i.e.,
\begin{equation}\label{eq: definition of beta_f}
\widehat{\bbeta}_f 
    := \sum_{i\,:\,|f(d_i)| < \infty} 
       f(d_i)\,\bw_i\bw_i^\top 
       \frac{\bX^\top\by}{n},
    \quad \text{where} \quad
    \widehat{\bSigma} := \sum_{i=1}^p d_i\,\bw_i\bw_i^\top.
\end{equation}
\end{definition}

The breadth of the collection $\mathcal{B}_{\mathcal{F}}$ is determined by the richness of the underlying function class $\mathcal{F}$. 
In Lemma \ref{lemma: sd is in mathcal F}, we show that the self-distilled estimator belongs to $\cB_\cF$. 
We briefly discuss other commonly used estimators that fall within this class.
\begin{enumerate}
    \item \textbf{Ridge-regularized estimator.} For any $\lambda$, the ridge estimator
    $$
    \widehat{\bbeta}_\lambda
    =
    (\widehat{\boldsymbol{\Sigma}} + \lambda \mathbf{I}_p)^{\dagger}\frac{\bX^\top \by}{n}
    $$
    belongs to $\mathcal{B}_{\mathcal{F}}$ with shrinkage function $f(x) = (x+\lambda)^{-1}$. 
    \item \textbf{Gradient descent.} Gradient descent on the least-squares loss likewise yields an estimator in $\mathcal{B}_{\mathcal{F}}$. Starting from $\widehat{\bbeta}^{(0)} = 0$ and iterating with step size $\eta > 0$,
    $$
    \widehat{\bbeta}^{(t+1)}
    =
    \widehat{\bbeta}^{(t)}
    -
    \eta \left(\widehat{\boldsymbol{\Sigma}}\widehat{\bbeta}^{(t)}- \frac{\bX^\top \by}{n}\right),
    $$
    a direct calculation gives
    \begin{align}\label{eq: gradient descent shrinkage}
    \widehat{\bbeta}^{(T)}
    =
    f_T(\widehat{\boldsymbol{\Sigma}})\frac{\bX^\top \by}{n},
    \qquad
    f_T(x) = \eta \sum_{k=0}^{T-1} (1 - \eta x)^k,   
    \end{align}
    where $f_T$ is a polynomial of degree $T-1$.
    \item \textbf{Min-norm interpolator.} The min-norm interpolator
    \begin{align}\label{eq: min_norm shrinkage}
    \widehat{\bbeta}_{\mathrm{mn}} := \argmin_\bbeta \{\|\bbeta\|_2 : \bX\bbeta = \by\}
    \end{align}
 also belongs to $\mathcal{B}_F$, since it can be expressed as 
$\widehat{\boldsymbol{\Sigma}}^{\dagger} \mathbf{X}^{\top} \boldsymbol{y} / n$. 
Therefore, this corresponds to the shrinkage function $f(x) = x^{-1}$. 
    \item \textbf{Principal components regression.} Principal components regression retains the eigenvectors of $\widehat{\boldsymbol{\Sigma}}$ corresponding to eigenvalues that exceed a threshold $\tau > 0$, and then performs least squares on the resulting subspace. This yields
    \begin{align}\label{eq: pcr shrinkage}
    \widehat{\bbeta}_{\mathrm{PCR}}(\tau)
    =
    f(\widehat{\boldsymbol{\Sigma}})\frac{\bX^\top \by}{n},
    \qquad
    f(x) = x^{-1}\mathbf{1}(x\ge \tau).
   \end{align}   
\end{enumerate}

 More broadly, spectral shrinkage has a long history in statistics. A general treatment of spectral regularization in supervised learning appears in \cite{bauer2007regularization,logerfo2008spectral}. Shrinking or thresholding of singular values has been widely used in the unsupervised learning literature \cite{candes2013unbiased,gavishdonoho2017singular}. Following the first arxiv version of our paper, the spectral shrinkage class has also appeared in \cite{kim2026monotone,lu2026krylov}. However, neither connects self-distillation to spectral shrinkage or investigates whether suitably tuned self-distillation can approximate the optimal shrinkage function.

Having defined the general spectral shrinkage class (Definition \ref{def:shrinkage_class}), we turn to specifying  functions $f$ that yield well-behaved estimators. 
Since $f$ takes the eigenvalues of $\widehat{\bSigma}$ as its argument, 
a natural requirement is that $f$ behaves well wherever the spectrum of $\widehat{\bSigma}$ 
concentrates, namely on $\mathcal{S}_c^+$. To ensure this, we impose the following assumption on $\mathcal{F}$. Throughout the paper, for each $\eta > 0$ we define the $\eta$-neighborhood
of a point $x \geq 0$ as $[0,\infty)\cap(x-\eta,\,x+\eta)$,
and the $\eta$-neighborhood of a bounded interval $[a, b] \subset [0,\infty)$
as $[0,\infty)\cap(a-\eta,\,b+\eta)$.

\begin{assumption}
\label{assm:F}
Recall $\mathcal{S}_c$ and $\mathcal{A}_c$ from Definition~\ref{def:spectral-support}. The class $\mathcal{F}$ consists of all functions 
$f : [0,\infty) \to \mathbb{R} \cup \{\infty\}$ satisfying the following conditions:
\begin{enumerate}
    \item there exists $\eta > 0$, independent of $n$, such
          that $f$ is continuous on an $\eta$-neighborhood of
        $\mathcal{S}_c$. 
      \item there exists $\eta > 0$, independent of $n$, such 
           that $f$ is continuous on an $\eta$-neighborhood of 
           every point in $\mathcal{A}_{c}$.
\end{enumerate}
\end{assumption}

Assumption~\ref{assm:F} guarantees that $f$ is continuous on a fixed $\eta$-neighborhood 
of every point in $\mathcal{S}_c^+$. By \cite{baik2004}, almost surely for all sufficiently 
large $n$, every eigenvalue of $\widehat{\bSigma}$ falls within such a neighborhood. 
In particular, $f$ is finite and continuous at every eigenvalue of $\widehat{\bSigma}$, 
so the decomposition of $\widehat\bbeta_f$ in~\eqref{eq: definition of beta_f} 
contains all eigenvector directions of $\widehat{\bSigma}$.

Furthermore, Assumption \ref{assm:F} covers the  common estimator examples we provided for spectral shrinkage estimators, whether directly in finite samples or asymptotically. For instance,  Ridge regression, and more generally, the 
self-distillation shrinkage function $f_t$ corresponding to 
$\widehat{\bbeta}^{(t)}_\mathrm{SD}$ in \eqref{eq:SD-MP}, belong to $\mathcal{F}$ as long as  
$-\lambda_0, -\lambda_1, \dots, -\lambda_t 
\in \mathbb{R} \setminus \mathcal{S}_c^+$. 
Since the analyst chooses these tuning parameters, we can easily ensure this condition is true. Additionally, in Theorem \ref{thm:sd_optimal}, we show that the optimal self-distillation strategy for spiked covariance matrices involves ridge parameters $\lambda_0,\hdots,\lambda_t$ that satisfy this condition.

The gradient descent shrinkage function $f_T$ from 
\eqref{eq: gradient descent shrinkage} is a polynomial, hence 
continuous everywhere, and therefore belongs to $\mathcal{F}$. For the min-norm interpolator \eqref{eq: min_norm shrinkage}, the 
shrinkage function has a pole at $0$. Since $0 \in \mathcal{S}_c^+$ 
whenever $c > 1$, this estimator does not belong to $\mathcal{F}$ in this 
regime. Similarly, for PCR \eqref{eq: pcr shrinkage}, the threshold 
$\tau$ is in general random, depending on the eigenvalues of the 
sample covariance matrix, and hence cannot belong to the 
deterministic class $\mathcal{F}$. However, asymptotically, both 
estimators behave like estimators in $\mathcal{B}_{\mathcal{F}}$ 
corresponding to smooth approximations of their respective shrinkage 
functions. We make this precise in Lemma~\ref{lemma: min-norm is worse} 
for the min-norm interpolator, and in 
Lemma~\ref{pcr_with_all_spiked} and 
Lemma~\ref{lemma: proportional regime PCR} for PCR.

\subsection{Out-of-Sample Prediction Risk}
The main message of our paper is that suitably tuned self-distillation achieves uniquely 
optimal prediction performance within the spectral shrinkage class, thereby outperforming 
the well-known estimators described above. To quantify this, we study the prediction risk 
of spectral shrinkage estimators, defined as
\[
\mathbb{E}_{\bx_{\mathrm{new}}, y_{\mathrm{new}}}\left[
    \|y_{\mathrm{new}} - \bx_{\mathrm{new}}^\top \widehat{\bbeta}_f\|_2^2 
    \mid \bX, \by
\right],
\]
where $(\bx_{\mathrm{new}}, y_{\mathrm{new}})$ is a fresh sample independent of 
$(\bX, \by)$, drawn from the same distribution as in Assumption~\ref{assm:dgp}. The prediction risk decomposes as
\[
\mathbb{E}_{\bx_{\mathrm{new}}, y_{\mathrm{new}}}\left[
    \|y_{\mathrm{new}} - \bx_{\mathrm{new}}^\top\widehat{\bbeta}_f\|_2^2 
    \mid \bX, \by
\right]
= \|\widehat{\bbeta}_f - \bbeta_0\|_{\bSigma}^2 + \sigma_{\beps}^2,
\]
where $\|v\|_{\bSigma}^2 = v^\top\bSigma\, v$. Since the noise term $\sigma_{\beps}^2$ 
is independent of the estimator, the quantity of interest is the signal component 
$\|\widehat{\bbeta}_f - \bbeta_0\|_{\bSigma}^2$. We therefore define, for each 
$f \in \mathcal{F}$,
\begin{equation}
\label{eq:limiting_pred_risk}
\mathcal{R}^{\mathrm{pred}}(f;\, c)
=
\lim_{\substack{n,p\to\infty\\ p/n\to c}}
\|\widehat{\bbeta}_f - \bbeta_0\|_{\bSigma}^2
\qquad \text{almost surely}.
\end{equation}
  The existence of this limit is established in Proposition~\ref{prop: general formula}. 
A major contribution of our paper is to characterize the optimal shrinkage rule
\begin{equation}
\label{eq:optimal_pred_shrinkage}
f^{\rm pred}_* = \arg\min_{f \in \mathcal{F}} \mathcal{R}^{\mathrm{pred}}(f; c),
\end{equation}
by analyzing the interplay between the spectral geometry of $\boldsymbol{\Sigma}$ 
and the structure of  $\bbeta_0$.

\subsection{Extension to the Decentralized Setup}
Thus far, we have presented our framework in a centralized setting, where all samples are 
stored on a single machine, and the estimator is constructed from the full dataset. As 
discussed in Section~\ref{sec:intro}, we extend our analysis to a decentralized 
setting with $K > 1$ clients and a central server: each client constructs a local 
spectral shrinkage estimator from its own data, transmits it to the central server, 
and the server then aggregates these local estimators. Since this extension requires 
additional notation to reflect the multi-client setup, we defer these details and the 
corresponding results to Section~\ref{sec:theory_fed}.

\section{Main Results}
\label{sec:theory}
\subsection{Optimal Spectral Shrinkage in Spiked Covariance Matrices}
In this section, we characterize the optimal spectral shrinkage rule $f_*^{\mathrm{pred}}$ over the class $\cF$ (Assumption \ref{assm:F}), and show that self-distillation emerges naturally as the optimal procedure in our setting. To this end, we first characterize the structure of $f_*^{\mathrm{pred}}$.

\begin{theorem}
\label{thm: main}
Under Assumptions \ref{assm:dgp} and \ref{assm:main}, the following hold:
\begin{enumerate}
    \item[(a)] $f_*^{\rm pred}$ exists and is unique on $\mathcal{S}_c^+\setminus\{0\}$, as introduced in Definition~\ref{def:spectral-support}.
    \item[(b)] $f_*^{\rm pred}$ is a rational function taking the form 
\begin{equation}
\label{def:f_pred_closed_form}
f_\ast^{\mathrm{pred}}(x)
=
 \frac{b_0^{(1)}\nu(x) + \sum_{j=1}^\rk b_{j}^{(1)} \nu_{-j}(x)}{\sx^2\,r^2 x\left(\omega_0 \nu(x) + \sum_{j=1}^\rk \omega_j \nu_{-j}(x)\right) + c\sx^2\, \sigma_{\boldsymbol{\varepsilon}}^2 \nu(x)},
\qquad x \in \mathcal{S}_c^+.
\end{equation}

Here $\nu(x)$ and $\nu_{-j}(x)$ are polynomial functions that depend on the Radon-Nikodym derivatives of the Marchenko-Pastur law with respect to certain spiked probability measures introduced in Lemma \ref{lemma: F_mp << F_delta} and Lemma \ref{lemma: def mu_j mu_0}. The coefficients $b_0^{(1)}, b_1^{(1)}, \dots, b_\rk^{(1)}$ are 
defined in~\eqref{eq: linear system with b0_simple}, where the first coefficient 
admits the  closed form $b_0^{(1)} = \sx^2 r^2\omega_0$. The coefficients $\omega_0,\omega_1,\hdots\omega_s$ are defined in \eqref{eq:omega_js}. 

\item[(c)] The numerator of $f_*^\mathrm{pred}$ has degree $\rk$, while the denominator has degree $\rk + 1$.
\end{enumerate}
\end{theorem}
\noindent
The proof is deferred to Appendix~\ref{app: B}. Theorem~\ref{thm: main} shows that the optimal shrinkage rule within $\cF$ has an explicit formula and depends on the spike strengths $\delta_1,\dots,\delta_\rk$, the signal alignments $\alpha_1,\dots,\alpha_\rk$, and the aspect ratio $c$. 
As is clear from Theorem~\ref{thm: main}(c), the complexity of $f_\ast^{\mathrm{pred}}$ grows with the spectral complexity of ${\boldsymbol{{\boldsymbol{\Sigma}}}}$. In particular, if $\rk = 0$, that is, if $\cov(\bX)={\boldsymbol{{\boldsymbol{\Sigma}}}}=\sx^2 \bI_p$, then $f_*^{\rm pred}$ simplifies considerably as described in the following lemma: 
\begin{lemma}
\label{thm: main corollary}
If $\rk = 0$, i.e., $\boldsymbol{\Sigma} = \sx^2 \mathbf{I}_p$, then the optimal shrinkage rule $f_\ast^{\mathrm{pred}}$ is of the form
$$
f_\ast^{\mathrm{pred}}(x) = \frac{1}{x+\lambda_\ast},
\qquad
\lambda_* = \frac{c\sigma_\beps^2 }{r^2}.
$$
\end{lemma}
The proof is deferred to Appendix \ref{sec: lemma easy}. Thus, in the isotropic case, our result recovers the well-known 
optimal tuning formula for ridge regression. Prior work has 
studied the asymptotic risk of ridge regression and its optimal 
tuning under high-dimensional asymptotics 
\cite{dicker_bernoulli_2016, dobriban2018high, hastie2022surprises}, but these differ from our result in important respects. First, an optimality result of a similar spirit was proved in
\cite{dicker_bernoulli_2016}: under Gaussian isotropic designs 
and with risk measured unconditionally over both $\bX$ and $\beps$, \cite{dicker_bernoulli_2016} proves that optimally tuned ridge regression is asymptotically  minimax over the sphere $\{\bbeta_0 : \|\bbeta_0\|_2 = r\}$. This is the closest antecedent to our result, though the assumptions are substantially stronger. Subsequently, 
\cite{dobriban2018high} places a prior distribution on $\bbeta_0$ 
and, conditionally on $\bX$, characterizes the asymptotic risk of 
ridge regression and identifies the optimal $\lambda > 0$ within the ridge family. Finally, 
\cite{hastie2022surprises} works with a fixed $\bbeta_0$, 
conditionally on $\bX$, and similarly characterizes the asymptotic 
risk of ridge regression and identifies the optimal $\lambda > 0$ 
within the ridge family. Neither \cite{dobriban2018high} nor 
\cite{hastie2022surprises} address optimality over a broader class 
of estimators.

Our result differs from the aforementioned papers in two key aspects. 
First, we work conditionally on $\bX$ and $\by$ with a fixed $\bbeta_0$. Second, Lemma \ref{thm: main corollary} provides optimality of tuned ridge over the full spectral shrinkage class, and without requiring Gaussian assumptions as in \cite{dicker_bernoulli_2016}.

The situation, however, changes fundamentally in the anisotropic case, when the covariance matrix of $\bX$ has $\rk \ge 1$ spikes. As we will establish,  $\rk$-step self-distillation yields the optimal shrinkage rule $f_\ast^{\mathrm{pred}}$ in this regime. To this end, we first prove a structural lemma showing that multi-step self-distillation remains within the spectral shrinkage family $\cB_{\cF}$. In particular, given any $k \in \bbN$ and parameters
\[
\Theta_k := (\lambda_0, \dots, \lambda_k, \xi_1, \dots, \xi_k),
\]
the $k$-step self-distilled estimator can be written in the form $\widehat\bbeta_f$ for a suitable shrinkage function $f$ depending on $\Theta_k$, as follows: 
\begin{lemma}
\label{lemma: sd is in mathcal F}
Let $\widehat{\boldsymbol{{\boldsymbol{\Sigma}}}} = \bW {\boldsymbol{\Lambda}} \bW^\top$ be the eigendecomposition of the sample covariance matrix. Then, for any integer $k \geq 0$, the $k$-step self-distilled ridge estimator with parameters $(\lambda_0, \dots, \lambda_k, \xi_1, \dots, \xi_k)$ admits the representation
$$
\widehat\beta_{\mathrm{SD}}^{(k)}
=
\bW f_k({\boldsymbol{\Lambda}}) \bW^\top \frac{\bX^\top \by}{n},
$$
where
$$
f_k(x)
=
\sum_{j=0}^{k}
\left\{
(1-\xi_j)
\left(\prod_{t=j+1}^{k} \xi_t\right)
\left(\prod_{t=j+1}^{k} \frac{x}{x+\lambda_t}\right)
\right\}
\frac{1}{x+\lambda_j},
$$
with the convention that $\xi_0 = 0$ and that each reciprocal factor in this formula is interpreted separately as
a scalar Moore--Penrose inverse.
\end{lemma}
The proof is deferred to Appendix~\ref{sec: lemma easy}.
Having established that a multi-step self-distillation procedure produces an estimator in $\cB_\cF$, we next show that under a spiked covariance model, the optimal shrinkage rule $f_\ast^{\mathrm{pred}}$ can be realized exactly by finitely many self-distillation steps, in fact, the same number of steps as the number of spikes.
\begin{theorem}[Optimality of self-distillation]
\label{thm:sd_optimal}
Under Assumptions~\ref{assm:dgp}--\ref{assm:main}, suppose that ${\boldsymbol{{\boldsymbol{\Sigma}}}}$ has $\rk \ge 1$ non-zero spikes. Then:
\begin{enumerate}
    \item[(a)] there exists a choice of Ridge parameters $(\lambda_0,\dots,\lambda_\rk)$ and distillation weights $(\xi_1,\dots,\xi_\rk)$ such that the optimal shrinkage rule $f_*^{\rm pred}$ can be realized exactly by an $\rk$-step self-distilled estimator with parameter vector
    $$
    \Theta_\rk = (\lambda_0, \dots, \lambda_\rk, \xi_1, \dots, \xi_\rk)
    $$
Moreover, the Ridge parameters $(\lambda_0, \dots, \lambda_\rk)$ are strictly distinct, 
exactly $\rk$ of the $\rk+1$ parameters are negative, and 
$-\lambda_0, \dots, -\lambda_\rk \notin \cS_c^+$.

    \item[(b)] $\rk$ steps of self-distillation are, in general, necessary for optimality. Specifically, the optimal shrinkage rule $f_*^{\rm pred}$ cannot be realized by any 
 $k$-step self-distillation procedure for any $0\leq k < \rk$. \end{enumerate}
\end{theorem}
Theorem~\ref{thm:sd_optimal} shows that self-distillation is both sufficient and, in general, necessary for achieving the optimal spectral shrinkage rule $f_*^{\rm pred}$. Moreover, the optimal rule $f_*^{\rm pred}$ can be implemented through $\rk$ rounds of self-distillation where $\rk$ is the number of spikes in the underlying covariance matrix. On the other hand, any procedure with fewer than $\rk$ steps fails to attain the optimal risk.
As an immediate consequence, classical ridge regression, corresponding to zero distillation steps, is provably suboptimal as soon as $\rk \ge 1$. 
We defer the proof of part (a) to Appendix~\ref{app: D}, and the proof of part~(b) is deferred 
to Appendix~\ref{app: E}.

We highlight that the optimal self-distillation representation involves negative ridge parameters 
$\lambda_t < 0$. To the best of our knowledge, all existing theoretical analyses of self-distillation restrict attention to positive ridge parameters (c.f.~
\cite{pareek2024understandinggainsrepeatedselfdistillation, dang2026optimalunconstrainedselfdistillationridge} and references cited therein). Notably, our results hold without this restriction. In fact, 
Theorem~\ref{thm:sd_optimal}(a) shows that among the ridge 
parameters appearing in any optimal self-distillation 
representation, only one is positive. 
This corroborates and extends earlier observations in the pure ridge literature, which show that the optimal ridge hyperparameter may be negative \cite{kobak2020optimalridgepenaltyrealworld, wu_and_xu, richards_negative_ridge, negative_ridge_arzela_ascoli, bartlett,negative_ridge_tibshirani}.

The next theorem shows that multi-step self-distillation attains a strictly smaller limiting prediction risk than a number of popular alternatives:  
\begin{theorem}
\label{thm: f_ast dominates}
Suppose Assumption~\ref{assm:main} holds and $\rk \geq 1$. Then the estimator associated with $f_\ast^{\mathrm{pred}}$ has strictly smaller asymptotic out-of-sample prediction risk than each of the following procedures:
\begin{enumerate}
    \item[(a)] the min-norm interpolator (equivalently, ordinary least squares when $c<1$);
    \item[(b)] any self-distillation
procedure that uses a common ridge parameter $\lambda$
with $-\lambda\notin\mathcal S_c^+$ across all
iterations\stepcounter{footnote}
\hypertarget{sd-common-penalty-text}{}
\hyperlink{sd-common-penalty-note}{\textsuperscript{\arabic{footnote}}}
{\renewcommand{\thefootnote}{}
 \footnotetext{\hypertarget{sd-common-penalty-note}{}
   \hyperlink{sd-common-penalty-text}{\textsuperscript{\arabic{footnote}}}\,
   This is the approach used in
\cite{dang2026optimalunconstrainedselfdistillationridge}.}},
including ridge regression as the case of zero distillation steps;
    \item[(c)] gradient descent with any finite number of steps;
    \item[(d)] principal components regression with any fixed number of retained components;
    \item[(e)] principal components regression with a sequence $m_n \to \infty$ of retained components with $m_n/p \to \tau \in\bigl[0, \min(1,1/c)\bigr)$ as $n\to \infty$.
\end{enumerate}
\end{theorem}
The proofs of parts (a), (b), and (c) can be found in Appendix~\ref{app: F}, while the proofs of  (d) and (e) are deferred to Appendix~\ref{app: G}.

Theorem \ref{thm: f_ast dominates} emphasizes that  self-distillation outperforms not only  ridge regression but a broad range of well-known procedures. In particular, while classically one might expect PCR to perform well in the presence of a spiked covariance structure, since it explicitly exploits the dominant spectral directions, Theorem~\ref{thm: f_ast dominates} shows that the optimal self-distilled estimator strictly outperforms PCR. Hence, under the spiked covariance models, the optimal spectral transformation is neither a simple inverse as in ridge nor a thresholding rule as in PCR, but rather a more intricate rational function that can be obtained via self-distillation.

Thus far, we characterized the optimal shrinkage rule that minimizes the limiting prediction risk. We now show that analogous results hold for the limiting estimation risk, defined as: 
\begin{equation}
\label{eq:limiting_est_risk}
\mathcal{R}^{\mathrm{est}}(f; c)
=
\lim_{\substack{n,p\to\infty\\ p/n\to c}}
\|\widehat{\bbeta}_f - \bbeta_0\|_2^2, \qquad (\text{almost surely}) \,. 
\end{equation}
The corresponding optimal shrinkage function for estimation is defined as: 
\begin{equation}
\label{eq:optimal_est_shrinkage}
f_*^{\rm est}
=
\arg\min_{f \in \mathcal{F}} \mathcal{R}^{\mathrm{est}}(f; c).
\end{equation}
The following theorem establishes that the estimation-risk-optimal shrinkage function, $f_*^{\rm est}$, is also a rational function, analogous to $f_*^{\rm pred}$, and can be achieved via $s$-step self-distillation: 
\begin{theorem}[Optimality under estimation risk]
\label{thm: main_thm est}
Under Assumptions~\ref{assm:dgp}--\ref{assm:main}, the conclusions of Theorem~\ref{thm:sd_optimal} and Theorem~\ref{thm: f_ast dominates} continue to hold for the estimation-risk-optimal shrinkage function $f_*^{\rm est}$. In particular:
\begin{enumerate}
    \item[(a)] When $\cov(\bX)= \sx^2 \bI_p$, the optimally tuned ridge estimator minimizes the limiting estimation risk over the class $\cF$. In contrast, when ${\boldsymbol{{\boldsymbol{\Sigma}}}}$ follows a spiked covariance model with $\rk \ge 1$ spikes, $f_\ast^{\mathrm{est}}$ is a rational function
    \begin{equation}
f_\ast^{\mathrm{est}}(x)
=
\frac{r^2 \left(\omega_0 \nu(x) + \sum_{i=1}^\rk \omega_i \nu_{-i}(x)\right)}
{r^2 x\left(\omega_0 \nu(x) + \sum_{i=1}^\rk \omega_i \nu_{-i}(x)\right)
+ c \sigma_\beps^2 \nu(x)},
\qquad x \in \mathcal{S}_c^+.
\end{equation}
where we use the same notations as in Theorem~\ref{thm: main}. The numerator of $f_*^\mathrm{est}$ has degree $\rk$, while the denominator has degree $\rk + 1$.
    \item[(b)] 
    $f_*^{\rm est}$ is realized by an $\rk$-step self-distilled estimator.  Further, $\rk$ steps of self-distillation are always necessary for optimality, i.e. $f_*^{\text{est}}$ cannot be realized by any $(\rk-k)$-step self-distillation procedure for any $1 \le k \le \rk$.
    
    \item[(c)] Moreover, $f_*^{\rm est}$ achieves strictly smaller asymptotic estimation risk than the estimators mentioned in (a)-(e) of Theorem \ref{thm: f_ast dominates}.
\end{enumerate}
\end{theorem}
The proof is deferred to Appendix~\ref{app: H}.

\begin{remark}[Relation between optimal self-distillation and weighted ridge regression]
As mentioned in Theorem~\ref{thm: main}(c) and Theorem~\ref{thm: main_thm est}(a), 
both $f_\ast^{\mathrm{pred}}$ and $f_\ast^{\mathrm{est}}$ are rational functions 
with numerator degree $\rk$ and denominator degree $\rk+1$. Notably, the two shrinkage functions share the same denominator. A key ingredient in the proof of 
Theorem~\ref{thm:sd_optimal} is Lemma~\ref{lemma: polynomials P Q}, which 
guarantees that this common denominator has exactly $\rk+1$ distinct real roots  $-\tilde{\lambda}_0,-\tilde{\lambda}_1 ,\dots ,-\tilde{\lambda}_{\rk}$. A partial fraction 
decomposition then yields
\begin{equation*}
    f_\ast^{\mathrm{pred}}(x) = \sum_{j=0}^{\rk} \frac{w_j^{\mathrm{pred}}}{x + \tilde{\lambda}_j}, 
    \qquad
    f_\ast^{\mathrm{est}}(x) = \sum_{j=0}^{\rk} \frac{w_j^{\mathrm{est}}}{x + \tilde{\lambda}_j},
\end{equation*}
for real coefficients $w_j^{\mathrm{pred}}, w_j^{\mathrm{est}} \in \mathbb{R}$. 
Since the shrinkage function of a Ridge estimator with a regularization parameter 
$\tilde{\lambda}_j$ takes the form  $x \mapsto 1/(x + \tilde{\lambda}_j)$, each term 
$w_j/(x + \tilde{\lambda}_j)$ corresponds, up to scaling, to such a shrinkage function. 
Thus, it follows that $\widehat{\bbeta}_{f_\ast^{\mathrm{pred}}}$ and 
$\widehat{\bbeta}_{f_\ast^{\mathrm{est}}}$ can each be expressed as a weighted 
combination of Ridge estimators, one for each root {$-\tilde{\lambda}_j$}.
\end{remark}

\subsection{Self-Distillation and its Optimality under General Covariance}
\label{sec:general-covariance}

In the preceding subsection, we showed that under a spiked
covariance model, the optimal spectral shrinkage function
can be realized exactly by finitely many self-distillation
steps. We now relax the spiked covariance assumption and
establish an asymptotic optimality result for general
anisotropic covariance matrices. In particular, we show that self-distillation with
any fixed positive common ridge penalty approaches the
optimal limiting prediction risk over continuous spectral shrinkage functions as the number of distillation steps increases. 

Let $f_{\mathrm{SD}}^{(k)}$ denote the shrinkage function from
Lemma~\ref{lemma: sd is in mathcal F} with common penalty
$\lambda_0=\cdots=\lambda_k=\lambda>0$
and real weights $\xi_1,\ldots,\xi_k$.
The following theorem establishes risk convergence for
continuous shrinkage functions and shows that finite-step
self-distillation can approximate their optimal limiting risk.

\begin{theorem}
\label{thm:general-covariance}
Under Assumption~\ref{assm:spectral-convergence}, the following hold. 
\begin{enumerate}[label=(\alph*)]
\item Given any $f \in \mathcal C([0,\infty);\mathbb R)$ there exists a deterministic constant $\mathcal R^{\mathrm{pred}}(f;c) \ge 0$ such that
$$
 \|\widehat{\bbeta}_f-\bbeta_0\|_{\bSigma}^2
 \xrightarrow[]{\text{a.s.}}
 \mathcal R^{\mathrm{pred}}(f;c)\in[0,\infty).
$$
Furthermore, there exists a constant $C > 0$, such that for any $M>L(1+\sqrt c)^2$,
$$
\left|
 \sqrt{\mathcal R^{\mathrm{pred}}(f;c)}
 -
 \sqrt{\mathcal R^{\mathrm{pred}}(g;c)}
 \right|
 \le C\|f-g\|_{\infty,[0,M]},
 \qquad
 \forall \ f,g\in\mathcal C([0,\infty);\mathbb R).
$$
\item
For every fixed $\lambda>0$,
\begin{equation}\label{eq:approximation}
 \lim_{k\to\infty}
 \inf_{\xi_1,\ldots,\xi_k\in\mathbb R}
 \mathcal R^{\mathrm{pred}}(f_{\mathrm{SD}}^{(k)};c)
 = \inf_{f\in\mathcal C([0,\infty);\mathbb R)}
 \mathcal R^{\mathrm{pred}}(f;c).
\end{equation}
\end{enumerate}
\end{theorem}
All supporting lemmas and proofs for
Theorem~\ref{thm:general-covariance} are deferred to Appendix~\ref{app: I}.

Theorem~\ref{thm:general-covariance} shows that a fixed
positive penalty, such as $\lambda=1$, suffices. In particular, given any $\epsilon > 0$, there exits $k  \in \bbN$ and $\xi_1,\ldots,\xi_k \in \reals$ such that
\[
 \mathcal R^{\mathrm{pred}}(f_{\mathrm{SD}}^{(k)};c)
 \le
 \inf_{f\in\mathcal C([0,\infty);\mathbb R)}
 \mathcal R^{\mathrm{pred}}(f;c)+\epsilon.
\]
The required number of steps and weights may depend on the covariance,
the signal, the noise level, and $\epsilon$. The equation~\eqref{eq:approximation} indicates that the limiting risk of $f_{\rm SD}^{(k)}$ converges to the minimum limiting risk over all possible continuous shrinkage functions. In the next section, we will propose a method to construct an optimal shrinkage function based on available data through a generalized cross-validation criterion.

\subsection{Estimating the Optimal Self-Distillation Procedure}
\label{sec:gcv}
The preceding section shows that, for any fixed ridge penalty
$\lambda>0$ and any tolerance $\eps>0$, finitely many
self-distillation steps suffice to attain a limiting prediction risk
within $\eps$ of the optimum over continuous spectral shrinkage functions. In this subsection, we propose a data-driven method for selecting these common ridge penalty and the distillation weights. 
Towards that end, we propose a generalized cross-validation (GCV) criterion that estimates the out-of-sample prediction risk from the training data. Then we select the penalty and weights that minimize this estimate to produce an estimator.

To construct an estimate for the out-of-sample prediction risk, we first show that for $k$-step self-distillation, when the ridge penalty parameter is fixed at level $\lambda$, while the distillation weights $(\xi_1, \dots, \xi_k)$ are allowed to be different, then the expression for the corresponding spectral shrinkage function simplifies to the following form: 
\begin{equation}
\label{eq:def_f_a_lambda}
 f_{a,\lambda}(x)
 =\sum_{j=0}^k a_j\frac{x^j}{(x+\lambda)^{j+1}},
 \qquad
 \ba\in\mathbb R^{k+1},\quad \mathbf{1}^\top \ba=1.
\end{equation}
where $\{a_0, \cdots, a_k\}$ are functions of $(\xi_1, \dots, \xi_k)$. Consequently, we can write the corresponding spectral shrinkage estimator as: 
\begin{equation}
\label{eq:ridge_basis}
\widehat{\bbeta}_{f_{a,\lambda}}=\sum_{j=0}^k a_j \bb_j(\lambda), \qquad \bb_j(\lambda)
 =\widehat{\bSigma}^{\,j}
  (\widehat{\bSigma}+\lambda\mathbf I_p)^{-(j+1)}
  \frac{\bX^\top\by}{n} \,.
\end{equation}
Since $\mathbf{1}^\top\ba=1$, Equation~\eqref{eq:ridge_basis} yields: 
$$
 \widehat{\bbeta}_{f_{a,\lambda}}-\bbeta_0
 =\sum_{j=0}^k a_j\bigl(\bb_j(\lambda)-\bbeta_0\bigr).
$$
It is worth noting that if $a_1 = \cdots = a_k = 0$ and $a_0 = 1$ is $ \widehat{\bbeta}_{f_{a,\lambda}}$ is the ridge estimator with penalty parameter $\lambda$. 
Expanding the  squared $\bSigma$-norm, we can therefore write the conditional prediction risk as
\begin{align*}
 \mathcal R(f_{a,\lambda}) &= \sigma_\varepsilon^2+
   \bigl\|\widehat{\bbeta}_{f_{a,\lambda}}-\bbeta_0\bigr\|_{\bSigma}^2 =\sigma_\varepsilon^2+
   \sum_{i,j=0}^k a_i a_j
   \bigl(\bb_i(\lambda)-\bbeta_0\bigr)^\top
   \bSigma\bigl(\bb_j(\lambda)-\bbeta_0\bigr) =\sum_{i,j=0}^k a_i a_j \bT_{ij}(\lambda),
\end{align*}
where 
\[
 \bT_{ij}(\lambda)
 :=\sigma_\varepsilon^2+
   \bigl(\bb_i(\lambda)-\bbeta_0\bigr)^\top
   \bSigma\bigl(\bb_j(\lambda)-\bbeta_0\bigr).
\]
Thus, for each $(a, \lambda)$, estimating the entries $\bT_{ij}(\lambda)$ yields an estimator for $\mathcal R(f_{a,\lambda})$. Next, we present an estimation procedure for $\bT_{ij}(\lambda)$.  
To that end, define
\[
 \bW_\lambda=\left(\frac{\bX\bX^\top}{n}+\lambda\bI_n\right)^{-1},
 \qquad \text{ and } \quad 
 \widehat v_\lambda=\frac{1}{n}\tr(\bW_\lambda).
\]
In \cite{golub1979gcv, negative_ridge_arzela_ascoli}, it is proved that for ridge regression with penalty $\lambda$, $\|\bW_\lambda\by/\widehat v_\lambda\|_2^2/n$ consistently estimates the out-of-sample prediction risk in the proportional asymptotics regime uniformly over $\lambda$ over any compact interval of $[0, \infty)$. 
To extend this idea for general $\widehat{\bbeta}_{f_{a,\lambda}}$, let us first define a differential operator $D_j^\lambda$ as: 
\[D_j^\lambda h(\lambda) :=\frac{1}{j!}\partial_\lambda^j\{\lambda^j h(\lambda)\}, \qquad \forall \ j \in \bbN \cup \{0\} \,.
\] 
It is immediate that $D_j^\lambda1=1$ and 
\begin{align}
\label{eq:higher_order_derivs}
D_j^\lambda\left(\frac1{x+\lambda}\right)
 =\sum_{r=0}^j\binom jr
                 \frac{(-\lambda)^r}{(x+\lambda)^{r+1}}
 =\frac1{x+\lambda}
                   \left(1-\frac\lambda{x+\lambda}\right)^j
 =\frac{x^j}{(x+\lambda)^{j+1}}.
\end{align}
Henceforth, with slight abuse of notation, we use the notation $D_j^\lambda \widehat{\bbeta}_{f_{a,\lambda}}$ to denote $\widehat{\bbeta}_{D_j^\lambda f_{a,\lambda}}$. Then, from the above identity, we have $D_j^\lambda \bb_0(\lambda) = \bb_j(\lambda)$. Consequently, we have: 
$$
\bT_{ij}(\lambda)
 :=\sigma_\varepsilon^2+
   \bigl(D_i^\lambda \bb_0(\lambda)-\bbeta_0\bigr)^\top
   \bSigma\bigl(D_j^\lambda \bb_0(\lambda)-\bbeta_0\bigr)
$$
which motivates us to use the following estimate of $T_{ij}$: 
\begin{equation}\label{eq:matrices}
 \widehat \bT_{ij}(\lambda)
 :=\frac1n \boldsymbol{\rho}_i(\lambda)^\top\boldsymbol{\rho}_j(\lambda), \qquad \text{where } \ \boldsymbol{\rho}_j(\lambda)
 :=D_j^\lambda\!\left(\frac{\bW_\lambda\by}{\widehat v_\lambda}\right).
\end{equation}
Our next theorem bounds the estimation error of $\{T_{ij}\}$ uniformly over $(i, j)$ and $\lambda$, so that the bound 
remains valid even when $\lambda$ is chosen from the data: 
\begin{theorem}
\label{thm:gcv}
Fix any compact interval $I\subset (0,\infty)$. Then, for every deterministic sequence of integers $K_n\ge0$,
\begin{equation}
\label{eq:entrybound}
 \sup_{\lambda\in I}\max_{0\le i,j\le K_n}|\widehat \bT_{ij}(\lambda)-\bT_{ij}(\lambda)| = O_{\text{p}}\!\left(
 \frac{(K_n+1)\log^2(K_n+2)}{\sqrt n}
 \right).
\end{equation}
\end{theorem}
Now, based on $\{\widehat \bT_{ij}\}$, given any $(a, \lambda)$, the corresponding estimator for the limiting risk of $\widehat{\bbeta}_{f_{a, \lambda}}$ is $\widehat{\mathcal R}(a, \lambda) = \sum_{i, j} a_i a_j \widehat{\bT}_{ij}$. 
Then we minimize this risk with respect to $(\ba, \lambda)$. 
The search for $\ba$ will be done over sets of the following form $$\mathcal{B}(m, A) := \{\ba\in\mathbb R^{m+1}: \mathbf 1^\top\ba=1, \|\ba\|_1\le A\}$$

In particular, given any sequence of non-empty sets $\Lambda_n \subseteq I$ (with $\cap_n \Lambda_n \neq \varnothing$)  and suitably chosen sequences $\{A_n\} \uparrow \infty$ and $\{K_n\} \uparrow \infty$ (See Equation \eqref{eq:general-schedule} for a condition on their growth rate), we choose $(\hat a_n, \hat \lambda_n)$ such that 
\begin{equation}
\label{eq:fit}
 \widehat{\mathcal R}(\widehat{\ba}_n,\widehat\lambda_n)
 \le \min_{\substack{\lambda\in\Lambda_n\\ \ba\in\mathcal{B}(K_n,A_n)}} \widehat{\mathcal R}(\ba,\lambda) + \delta_n,
\end{equation}
where $\delta_n \downarrow 0$ denotes the optimization error. 
Our final estimator is
$\widehat{\bbeta}_{\widehat f_n}$, where
$\widehat f_n=f_{\widehat{\ba}_n,\widehat\lambda_n}$.  The following theorem shows that, with probability going to one, the prediction risk of the GCV-selected estimator is at most the best oracle prediction risk plus any prescribed tolerance.

\begin{theorem}
\label{th:growing}
Suppose the assumptions of Theorem~\ref{thm:general-covariance} hold.
Let $K_n\ge0$ be deterministic integers and $A_n\ge1$ deterministic
constants such that $K_n\to\infty$, $A_n\to\infty$, and
\begin{equation}\label{eq:general-schedule}
 \frac{A_n^2(K_n+1)\log^2(K_n+2)}{\sqrt n}\to 0.
\end{equation}
 Then, the estimator selected by~\eqref{eq:fit} satisfies
\begin{equation}\label{eq:optimal}
 \left[
 \mathcal R(\widehat f_n)
 -
 \left\{\sigma_\varepsilon^2
 +\inf_{f\in\mathcal C([0,\infty);\mathbb R)}
       \mathcal R^{\mathrm{pred}}(f;c)\right\}
 \right]_+ \xrightarrow{\text{p}}0, \qquad \text{where  }\,\,\, [u]_+=\max\{u,0\}.
\end{equation}
\end{theorem}
One choice of $(K_n, A_n)$ satisfying \eqref{eq:general-schedule} is $A_n=C\log(n+e)$ and $K_n=\lceil C\log(n+e)\rceil$ for any fixed $C>0$. 
\begin{remark}
Note that, by Theorem \ref{thm:general-covariance} (b), one may take $\Lambda_n = \{\lambda_0\}$ for some $\lambda_0 \in I$. In that case, we only need to obtain $\widehat \ba$ via Equation \eqref{eq:fit}, i.e. 
$$
\widehat{\mathcal R}(\widehat \ba_n, \lambda_0) \le \min_{\ba \in \mathcal B(K_n, A_n)} \widehat{\mathcal R}(\ba, \lambda_0). 
$$
Then the final estimator $\widehat{\bbeta}_{f_{\hat \ba_n, \lambda_0}}$ also satisfies Equation \eqref{eq:optimal}. 

Thus, a single convex quadratic program suffices for the
guarantee in~\eqref{eq:optimal}.
\end{remark}

\subsection{Theory for Decentralized Learning}
\label{sec:theory_fed}

In this section, we extend our results to the decentralized setup. To keep things simple, we study this case only for spiked covariance models. We consider a setting with $K$ clients where for $1 \le \ell \le K$, the $\ell^{th}$ client observes $n_\ell$ samples generated from the linear model
$$
y_{i\ell} = \bx_{i\ell}^\top \bbeta_0 + \varepsilon_{i\ell}, \quad \bx_{i\ell} = {\boldsymbol{{\boldsymbol{\Sigma}}}}^{1/2} \bz_{i\ell}, \quad 1\leq i \leq n_\ell,$$
where $\bz_{i\ell} \in \mathbb{R}^p$ has i.i.d.\ coordinates with mean zero and variance one, and the errors satisfy the same assumptions as in Assumption \ref{assm:dgp}. Thus, all clients share the same population covariance ${\boldsymbol{{\boldsymbol{\Sigma}}}}$ and regression vector $\bbeta_0$, but each client only has access to its own local sample. As in the previous sections, we work under the proportional regime, where $p/n_\ell \to c_\ell$ for $1 \le \ell \le K$.

To estimate $\bbeta_0$ locally, each client applies a spectral shrinkage rule $f_\ell \in \mathcal{F}$ to its local sample covariance matrix and obtains the local shrinkage-based estimator: 
\[
\widehat \bbeta_\ell
= f_\ell(\widehat{\boldsymbol{\Sigma}}_\ell)\frac{\bX_\ell^\top \by_\ell}{n_\ell},
\qquad
\widehat{\boldsymbol{\Sigma}}_\ell = \frac{\bX_\ell^\top \bX_\ell}{n_\ell}.
\]
The clients then transmit $\{\widehat \bbeta_\ell\}_{\ell=1}^K$ to a central server, which aggregates them linearly as follows to obtain the final estimator: 
\begin{equation}
\label{eq:hat_beta_agg}
\widehat\bbeta_{\mathrm{agg}}
=
\sum_{\ell=1}^K \rho_\ell \widehat\bbeta_\ell,
\qquad
\rho_\ell \in \mathbb{R}, \quad 1 \le \ell \le K.
\end{equation}

The full procedure is summarized in Algorithm \ref{alg:federated-learning}. 
 It is immediate that the prediction risk of $\widehat\bbeta_{\mathrm{agg}}$ depends on $2K$ parameters: the local shrinkage functions $\{f_j\}_{1 \le j \le K}$ and the weights $\{\boldsymbol{\rho}_j\}_{1 \le j \le K}$. For simplicity, we will work with $c_\ell \equiv c.$ We define the limiting prediction risk of $\widehat\bbeta_{\mathrm{agg}}$ as below: 
\begin{align}
\label{eq: limit of beta_agg}
\mathcal{R}_K^{\mathrm{pred}}(f_1, \rho_1, \dots, f_K, \rho_K; c)
&:=
\lim_{p \to \infty} \|\widehat\bbeta_{\mathrm{agg}} - \bbeta_0\|_{\boldsymbol{{\boldsymbol{\Sigma}}}}^2
\qquad (\text{almost surely}).
\end{align}
The existence of the limit is established in Lemma~\ref{prop: general formula for agg} in the Appendix. Similar to the previous section, we define the optimal local shrinkage rules and weights as the one that minimizes the limiting prediction risk: 
\begin{equation}
\label{eq: opt agg}
(f_{\ast, 1}^{\mathrm{pred}}, \dots, f_{\ast, K}^{\mathrm{pred}},\, \rho_{\ast, 1}, \dots, \rho_{\ast, K})
:=
\underset{\substack{f_1,\dots,f_K \in \mathcal{F}\\ \rho_1,\dots,\rho_K}}{\argmin}
\mathcal{R}_K^{\mathrm{pred}}(f_1, \rho_1, \dots, f_K, \rho_K; c).
\end{equation}

\begin{algorithm}[t]
\caption{Federated learning with local spectral shrinkage}
\label{alg:federated-learning}
\begin{algorithmic}[1]
\Require Local datasets \(\{(\bX_\ell,\by_\ell)\}_{\ell=1}^K\), shrinkage rules \(\{f_\ell\}_{\ell=1}^K\), aggregation weights \(\{\rho_{\ell}\}_{\ell=1}^K\)
\Ensure Aggregated estimator \(\widehat\bbeta_{\mathrm{agg}}\)

\For{each client \(\ell\in[K]\)}
    \State Compute the covariance
    \(
    \widehat{\boldsymbol{\Sigma}}_\ell=\bX_\ell^\top \bX_\ell/{n_\ell}
    \)
    \State Form the local shrinkage estimator
    \[
    \widehat\bbeta_\ell
    =
    f_\ell(\widehat{\boldsymbol{\Sigma}}_\ell)\frac{\bX_\ell^\top \by_\ell}{n_\ell}
    \]
    and send \(\widehat\bbeta_\ell\) to the central server
\EndFor

\State Central server aggregates:
$\widehat\bbeta_{\mathrm{agg}}
=
\sum_{\ell=1}^K \rho_{\ell} \widehat\bbeta_\ell
$

\State \Return \(\widehat\bbeta_{\mathrm{agg}}\)
\end{algorithmic}
\end{algorithm}

Analogous to Theorem~\ref{thm: main}, we seek to characterize 
$\{(f_{*,j}^{\mathrm{pred}}, \rho_{*,j})\}_{1 \le j \le K}$. 
Recall the coefficient $b_0^{(1)}$ defined in Theorem~\ref{thm: main} (b). In case of a single client $K=1$, 
it admits the closed form $b_0^{(1)} = 
\sx^2 r^2\omega_0$. We will require analogs of this quantity for the $K$-client case.  Here, $b_0^{(K)}$ no longer 
admits a closed form but can be defined via the linear system (Equation~\eqref{eq: linear system with b0} of the Appendix). Throughout this 
section, we work under the following assumption.

\begin{assumption}
\label{assm: b0^K neq 0}
Let $b_0^{(K)} \neq 0$, where $b_0^{(K)}$ is as defined in Equation~\eqref{eq: linear system with b0} of the Appendix. \end{assumption}
Appendix~\ref{sec: assm b0} discusses extensively when Assumption~\ref{assm: b0^K neq 0} holds in our setting.  We now state our main result on optimality under this federated setting. 

\begin{theorem}
\label{thm: product_of_shrinkage}
Let $K \ge 1$, and suppose Assumptions~\ref{assm:main}, \ref{assm: b0^K neq 0} hold. 
Then the solution to \eqref{eq: opt agg} satisfies the following:
\begin{enumerate}
\item[(a)] The optimal aggregation 
    weights are all equal, i.e.\
    \begin{equation}\label{eq:aggweights}
    \rho_{\ast, 1} = \cdots = \rho_{\ast, K} = \frac{b^{(K)}_0}{\sx^2 r^2\omega_0}.
    \end{equation}
    \item[(b)] All clients use the same 
    optimal shrinkage function, i.e.\
    \[
    f_{\ast, 1}^{\mathrm{pred}} = \cdots = f_{\ast, K}^{\mathrm{pred}} = \frac{\sx^2 r^2\omega_0f^\ast_K}{b^{(K)}_0}
    \]
    where 
     \begin{equation}\label{eq:feddistopt}
    f_K^\ast(x) =  \frac{ b_{0}^{(K)}\nu(x) + \sum_{j=1}^\rk b_{j}^{(K)} \nu_{-j}(x)}{\sx^2\,r^2 x\left(\omega_0 \nu(x) + \sum_{j=1}^\rk \omega_j \nu_{-j}(x)\right) + c \sx^2\,\sigma_\beps^2 \nu(x)}.
    \end{equation}
    The coefficients $b^{(K)}_0, b^{(K)}_1, \dots, b^{(K)}_\rk$ are defined in Equation~\eqref{eq: linear system with b0} of the Appendix and the coefficients $\omega_0,\omega_1,\hdots\omega_s$ are defined in \eqref{eq:omega_js}.
    \item[(c)]  The common shrinkage 
    functions $f_{\ast, 1}^{\mathrm{pred}}, \dots,  f_{\ast, \rk}^{\mathrm{pred}}$ can be realized by an $\rk$-step 
    self-distillation procedure.
    
\end{enumerate}
\end{theorem}
The proof is deferred to Appendix~\ref{app: K}.
Theorem~\ref{thm: product_of_shrinkage} shows that the federated problem retains the same fundamental feature as the centralized one: the optimal local rule is again a rational shrinkage function, and can be implemented by finitely many steps of self-distillation. 
 In the special case $K=1$, we recover the optimal shrinkage rule from Theorem \ref{thm: main}. In the isotropic case $\bSigma = \bI_p$, we recover the optimal aggregation of ridge estimators studied in the distributed ridge  setup \cite{deterministic_calculus}. 
 
 Moreover, Theorem \ref{thm: product_of_shrinkage} reveals a genuinely new phenomenon created by decentralization in the anisotropic case with $K>1$ clients:  the optimal local shrinkage rule $f_K^\ast$ for each client differs from the single-client optimum $f_\ast^{\mathrm{pred}}$ obtained in Theorem \ref{thm: main}. This distinction is noteworthy from a methodological perspective. In a centralized environment (single-client), the client chooses a shrinkage rule to optimize their own bias-variance tradeoff based on their sample covariance matrix. In the federated setting, by contrast, the local estimators are computed from independent sample covariance matrices and are then combined at the server. The interaction between local shrinkage and cross-client aggregation, therefore, creates a new optimization problem, and the final solution differs from that of the single-client problem. 
 
We further illustrate this phenomenon using Figure~\ref{fig:b_entries_vs_K}, which plots $\{b_j^{(K)} / b^{(K)}_0\}$ where $0 \le j \le s$ for varying $K$,  with $s = 2$. Note that the optimal shrinkage function in both the single client ($K=1$) and multi-client cases ($K > 1$), \eqref{def:f_pred_closed_form} and \eqref{eq:feddistopt}, depend on these coefficients. Thus, they vary significantly in the number of clients, as demonstrated by the strict monotonicity of these coefficients as a function of $K$. 

The key technical ingredient underlying Theorem \ref{thm: product_of_shrinkage} is the following result, which extends our spectral shrinkage machinery from one sample covariance matrix to products involving two independent sample covariance matrices. For the following result, we extend Definition~\ref{def:shrinkage_class} 
as follows. Let $\mathcal{F}_{c_\ell}$ and $\mathcal{F}_{c_k}$ 
denote the function class introduced in Assumption~\ref{assm:F}, 
corresponding to limiting ratios $c_\ell$ and $c_k$ respectively.
\begin{theorem}
\label{lemma:non-free limit} 
Let $\widehat{\boldsymbol{\Sigma}}_\ell$ and $\widehat{\boldsymbol{\Sigma}}_k$ denote the sample covariance matrices for the $\ell^{th}$ and $k^{th}$ client respectively, with $\ell \neq k$. Suppose $p/n_\ell \to c_\ell$ and $p/n_k \to c_k$ where $c_\ell, c_k \in (0, \infty)$. 
Then, for any $\phi \in\mathcal{F}_{c_\ell}$ and $\psi \in \mathcal{F}_{c_k}$, we have: 
\begin{align*}
\frac{\bbeta_0^\top \phi(\widehat{\boldsymbol{\Sigma}}_\ell)\psi(\widehat{\boldsymbol{\Sigma}}_k)\bbeta_0}
{\|\bbeta_0\|_2^2}
& \xrightarrow{\mathrm{a.s.}}
\omega_0 
\left(\int \phi \,\mathrm{d}F_{\mathrm{MP},\, c_\ell}\right)
\left(\int \psi \,\mathrm{d}F_{\mathrm{MP},\,c_k}\right) \\
& \qquad \qquad \qquad +
\sum_{j=1}^{\rk}\omega_j
\left(\int \phi \,\mathrm{d}F_{\delta_j,\,c_\ell}\right)
\left(\int \psi \,\mathrm{d}F_{\delta_j,\,c_k}\right),
\end{align*}
where $F_{\mathrm{MP},\,c}$ and $F_{\delta_j, \,c}$ denote the Marchenko-Pastur law and the spiked probability measures introduced in Lemma~\ref{lemma:f_delta} with limiting aspect ratio $c$. 
\end{theorem}
The proof is deferred to Appendix~\ref{app: L}.
Taken together, Theorems~\ref{thm: product_of_shrinkage} and~\ref{lemma:non-free limit} show that the role of self-distillation extends beyond the single-client setting. Even in this simplified federated model, the optimal distributed strategy is obtained by combining local self-distilled estimators through an explicit aggregation rule. We view this as a first step toward a more general statistical theory of federated learning with distillation, allowing heterogeneous client distributions, iterative communication, or more general covariance profiles.

\begin{figure}[tbp]
    \centering
    \includegraphics[width=0.5\linewidth]{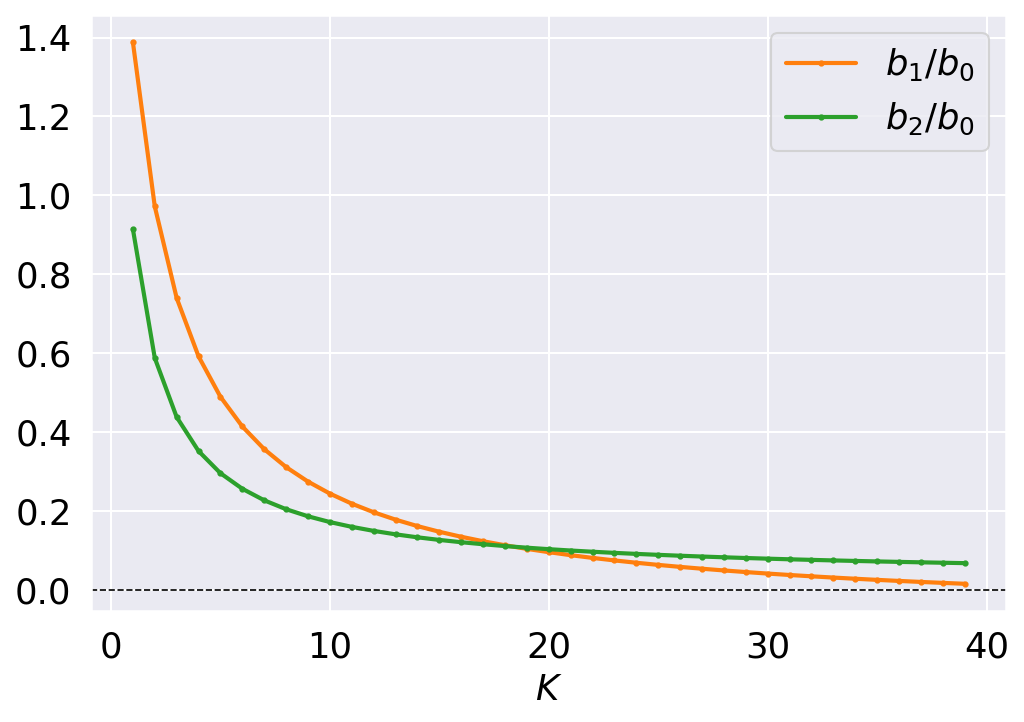}
    \caption{Entries of the vector $\left(b^{(K)}_1/ b^{(K)}_0, b^{(K)}_2 / b^{(K)}_0\right)$ 
    as a function of the number of local servers $K \in \{1, \ldots, 40\}$, 
    for a two-spike covariance with $\delta_1 = 2$, $\delta_2 = 3$, 
    $\alpha_1 = 3$, $\alpha_2 = 2.5$, and parameters $c = 3$, $r = 5$, 
    $\sigma_\varepsilon= 2, \sx = 1$. Note the variation in $K$, demonstrating the difference in the optimal local rule \eqref{eq:feddistopt} and aggregation weights \eqref{eq:aggweights} as a function of the number of clients.}
    \label{fig:b_entries_vs_K}
\end{figure}

\section{Simulation Studies}\label{secL simulations}

\subsection{Isotropic Design}
We first demonstrate our results for the case of isotropic design, where Lemma~\ref{thm: main corollary} shows that an optimally tuned ridge is optimal in our spectral shrinkage class. We generate synthetic data as follows. We sample $\bX \in \mathbb{R}^{n \times p}$ whose rows are i.i.d.\ draws from $\mathcal{N}(\mathbf{0}, \mathbf{I}_p)$, and draw $\boldsymbol{\varepsilon} \sim \mathcal{N}(\mathbf{0}, \sigma_\varepsilon^2\mathbf{I}_n)$ independently of $\bX$. Both $\bX$ and $\boldsymbol{\varepsilon}$ are drawn once and held fixed throughout. The response is then $\by = \bX\bbeta_0 + \boldsymbol{\varepsilon}$, where $\bbeta_0 \in \mathbb{R}^p$ is the true parameter vector, obtained by drawing a vector uniformly at random from the unit sphere in $\mathbb{R}^p$ and rescaling to have norm $r = \|\bbeta_0\|_2$. We subsequently hold the signal fixed. We compare the performance of Ridge regression for $\lambda > 0$ and one step of self-distillation. For self-distillation, for each fixed $\lambda$, we perform a grid search on $\xi_1, \lambda_1$ (see equation~\ref{eq:SD-recursion}). We also analyze principal components regression: we compare four cases where the number of retained components is 
$m_n\in\{1, 50, 100, 200\}$ to highlight both the finite and proportional number of components 
as proved in Theorem~\ref{thm: f_ast dominates}(c) and (d).
The prediction risk for each model is estimated using $200$ i.i.d.~test samples drawn from the same distribution. 
\begin{figure}[t]
    \centering
    \includegraphics[width=1\linewidth]{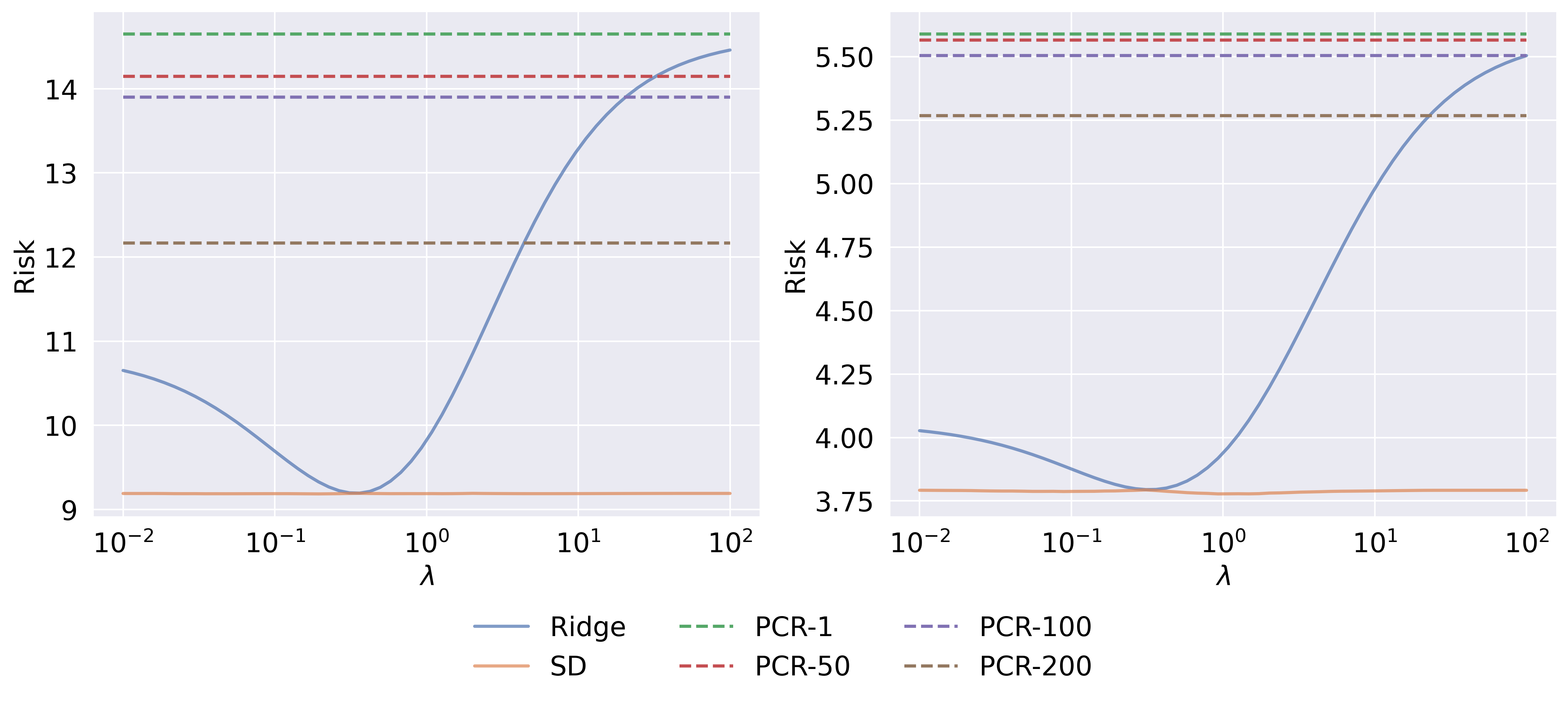}
     \caption{Sub-optimality of self-distillation and principal components regression in the isotropic setting. Left graph:  $n=1500,p=600$,  $\|\bbeta_0\|=3$, ${\sigma_\beps}=2.5, \sx = 1$.  Right graph: $n=1000,p=2000$, $\|\bbeta_0\|=2$, ${\sigma_\beps}=1, \sx = 1$.}
    \label{fig:isotropic_comparison_1}
\end{figure}

As proved in \cite{dang2026optimalunconstrainedselfdistillationridge}, for each $\lambda$ that is not a critical point of the asymptotic Ridge curve, self-distillation with $\lambda_1 = \lambda$ achieves strictly lower prediction risk than the original ridge estimator. We note that this is still the case without the restriction that $\lambda_1 = \lambda$, as clearly illustrated in Figure~\ref{fig:isotropic_comparison_1}: the orange curve (self-distillation) and the blue curve (ridge) meet only at the local minimum of Ridge, with strict improvement everywhere else. Moreover, in both 
the underparametrized regime (left panel) and the overparametrized regime (right panel), 
no estimator achieves lower risk than the optimally tuned ridge. PCR falls well short of 
this benchmark, while self-distillation at $\lambda \neq \lambda^\ast$ does not 
improve upon ridge at $\lambda^\ast$. Nevertheless, from a practical standpoint, 
self-distillation with suitable $\xi_1$ and $\lambda_1$ is appealing precisely because it uniformly 
dominates the ridge risk curve across all values of the tuning parameter; In particular, this implies that self-distillation can be advantageous in finite samples even in the isotropic case, where the population-optimal tuning parameter $\lambda^*$ is unknown and must be estimated, thereby incurring additional estimation error.

\subsection{Anisotropic Design}
\subsubsection{Closed-Form Limits} We next conduct experiments in the one-spike model ${\boldsymbol{\Sigma}} = \mathbf{I}_p + \delta \bv \bv^\top$ 
and the two-spike model ${\boldsymbol{\Sigma}} = \mathbf{I}_p + \delta_1 \bv_1 \bv_1^\top + \delta_2\bv_2 \bv_2^\top$. 
Lemma~\ref{prop: general formula} allows us to compute the limiting prediction 
risk for any $\widehat\bbeta_f$ with $f \in \mathcal{F}$. We demonstrate the results 
of Theorem~\ref{thm:sd_optimal} by numerically computing the limiting prediction risk 
of optimally tuned ridge regression and optimally tuned self-distillation $\widehat\bbeta_{f}$ for $f = f^\mathrm{pred}_\ast$. 
The results are displayed in Figure~\ref{fig:anisotropic_comparison}.
In panels~(a) and~(b), for small values of $\delta$, the two curves nearly 
coincide, which is consistent with our theory: as $\delta \to 0$ the covariance becomes isotropic, and an optimally tuned ridge is known to be optimal in that 
regime. As $\delta$ increases, we notice the advantages of self-distillation. On the other hand, in panels~(c) and~(d), even as 
$\delta_2 \to 0$, the two curves differ significantly. This is expected: as 
$\delta_2 \to 0$, the model reduces to a one-spike model with the fixed spike 
$\delta_1$, for which an optimally tuned ridge is already known to be strictly 
suboptimal. 

Regarding the improvements, the maximum percentage gain of 
self-distillation over optimally tuned ridge is $44.55\%$ in panel~(a),
$3.78\%$ in panel~(b), $29.53\%$ in panel~(c) and $12.31\%$ in panel~(d). 
The improvements are considerably more pronounced in the overparametrized  regime, suggesting that the benefits of self-distillation are amplified when 
$p \gg n$.

\begin{figure}[htbp]
\captionsetup[subfigure]{margin={0.8cm, 0cm}}
    \centering
    \begin{subfigure}[t]{0.49\linewidth}
        \centering
        \includegraphics[width=1\linewidth]{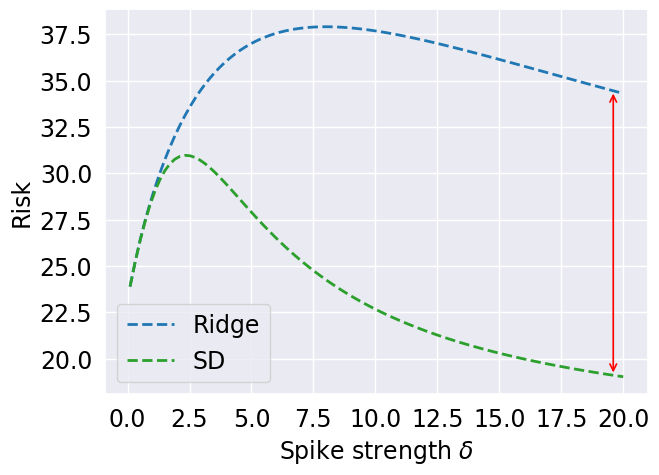}
        \caption{}
        \label{fig:aniso_plot5}
    \end{subfigure}
    \hfill
    \begin{subfigure}[t]{0.49\linewidth}
        \centering
        \includegraphics[width=1\linewidth]{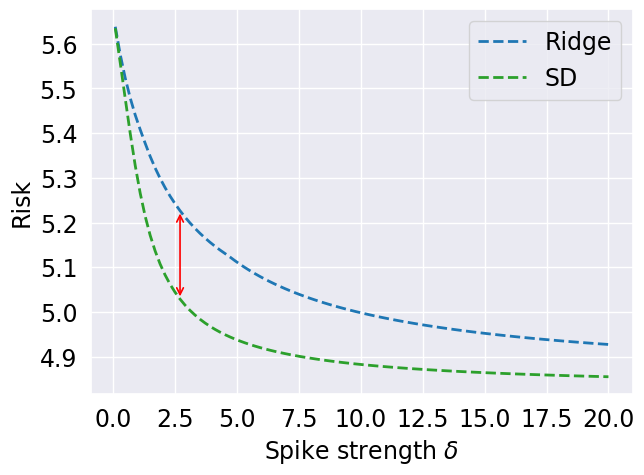}
        \caption{}
        \label{fig:aniso_plot6}
    \end{subfigure}
\vspace{0.5em}

    \begin{subfigure}[t]{0.49\linewidth}
        \centering
        \includegraphics[width=1\linewidth]{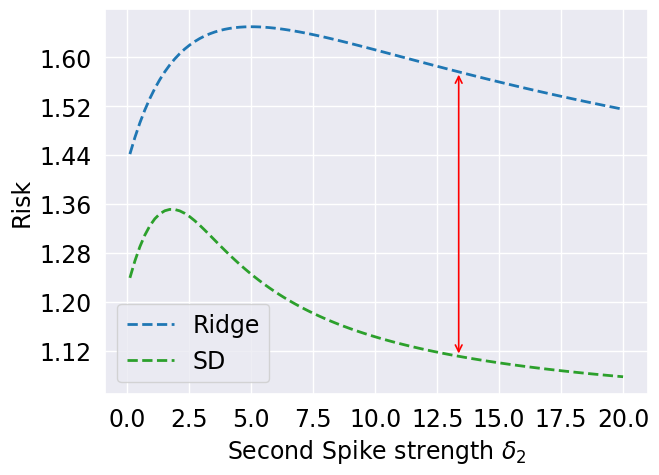}
                 \caption{}
        \label{fig:aniso_plot7}
    \end{subfigure}
    \hfill
    \begin{subfigure}[t]{0.49\linewidth}
        \centering
        \includegraphics[width=1\linewidth]{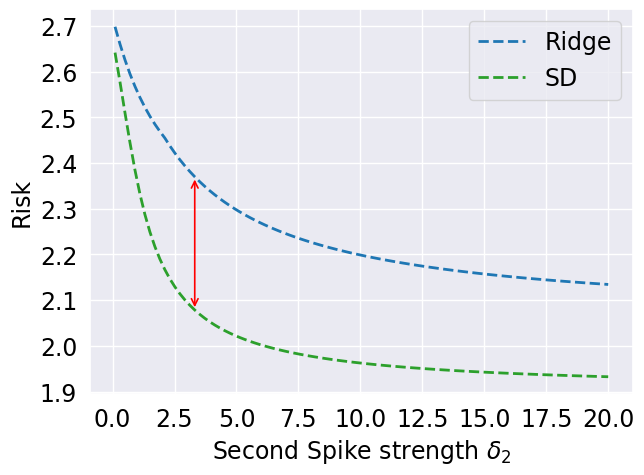}
        \caption{}
        \label{fig:aniso_plot8}
    \end{subfigure}

    \caption{Limiting prediction risk of optimally tuned 
    self-distillation versus optimally tuned ridge regression, as a 
    function of the spike strength $\delta$. The red arrow marks the 
    value of $\delta$ where the largest percentage improvement of 
    self-distillation over ridge is attained. Panels~(a) and~(b) 
    correspond to the one-spike model; panels~(c) and~(d) correspond 
    to the two-spike model with the first spike $\delta_1$ fixed and 
    the second spike $\delta_2$ varying along the horizontal axis. 
    Panels~(a) and~(c) are in the overparametrized regime ($c>1$); 
    panels~(b) and~(d) are in the underparametrized regime ($c<1$). 
    The parameter settings are as follows:
    \textbf{(a)}~$ c = 10$, $\|\bbeta_0\|_2=5$, $\bbeta_0^\top\bv=3$, 
    $\sigma_{\beps}=3$, $\sx=1$;\quad
    \textbf{(b)}~$ c = 0.5$, $\|\bbeta_0\|_2=5$, $\bbeta_0^\top\bv=3$, 
    $\sigma_{\beps}=3$, $\sx=1$;\quad
    \textbf{(c)}~$c = 7$, $\|\bbeta_0\|_2=1$, $\bbeta_0^\top\bv_1=0.7$, 
    $\bbeta_0^\top\bv_2=0.5$, $\sigma_{\beps}=1$, $\sx=1$, $\delta_1 = 3$;\quad
    \textbf{(d)}~$c = 0.6
    $, $\|\bbeta_0\|_2=3$, $\bbeta_0^\top\bv_1=1.5$, 
    $\bbeta_0^\top\bv_2=2$, $\sigma_{\beps}=2$, $\sx=1$, $\delta_1 = 1 $.}
    \label{fig:anisotropic_comparison}
\end{figure}

\subsubsection{Optimal Self-Distillation Parameters}
Using the construction from Theorem~\ref{thm:sd_optimal}, we identify the optimal 
parameters $(\lambda_0^\ast, \lambda_1^\ast, \xi^\ast)$  in the one-spike model ${\boldsymbol{\Sigma}} =\bI_p + \delta \bv \bv^\top$ such that $f_{\ast}^\mathrm{pred}$ 
is the shrinkage function corresponding to a self-distillation estimator. We plot these parameters 
as functions of the spike strength $\delta$ over a grid of values in $(0, 20)$.

\begin{figure}[htbp]
    \centering
    \captionsetup[subfigure]{margin={1.0cm, 0cm}, justification=centering}
    \begin{subfigure}[t]{0.49\linewidth}
        \centering
        \includegraphics[width=1\linewidth]{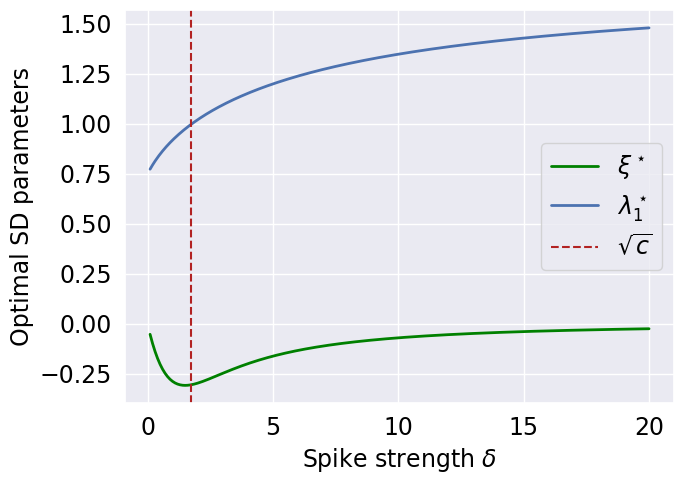}
        \caption{}
        \label{fig:lambda1-xi}
    \end{subfigure}
    \hfill
    \begin{subfigure}[t]{0.49\linewidth}
        \centering
        \includegraphics[width=1\linewidth]{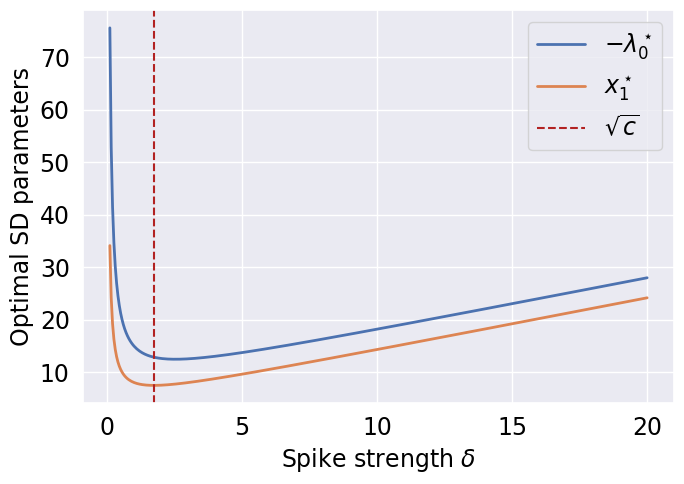}
        \caption{}
        \label{fig:lambda0-atom}
    \end{subfigure}

    \caption{Optimal self-distillation parameters as functions of 
    $\delta$. Problem parameters: $\|\bbeta_0\|_2 = 8$, 
    $\bbeta_0^\top\bv = 6$, $c = 3$, $\sigma_{\beps} = 4$, 
    $\sx = 1$.
    \textbf{(a)}~Optimal teacher Ridge $\lambda_1^\ast$ and 
    weight $\xi^\ast$.
    \textbf{(b)}~Negative of optimal student Ridge $-\lambda_0^\ast$ 
    and outlier $x_1^\star$}
    \label{fig:optimal-params}
\end{figure}

We note from Figure~\ref{fig:lambda1-xi} that $\xi^\ast$ starts near zero when 
$\delta \approx 0$, which is expected: setting the weight parameter to zero 
reduces our estimator to a ridge estimator tuned at the second step, which is 
known to be optimal under isotropic designs. As $\delta$ increases, $\xi^\ast$ 
becomes negative and remains so before returning asymptotically to zero as 
$\delta$ grows further. Thus, the optimal weight for self-distillation in this case is actually negative.

Turning to Figure~\ref{fig:lambda0-atom}, we observe that $\lambda_0^\ast$ is 
always negative, and its absolute value grows at the same rate as the outlier 
$x_1^\star$ from Definition~\ref{def:spectral-support}. 
We know that the largest eigenvalue of the sample covariance converges to $x_1^\star$ if and only if $\delta > \sqrt{c}$, which is precisely the BBP phase transition threshold 
\cite{baik2004}. The value $\delta = \sqrt{c}$ is marked on 
the horizontal axis of Figure~\ref{fig:lambda0-atom} for reference.
Interestingly, 
the growth rates of $-\lambda_0^\ast$ and $x_1^\star$ coincide even 
for values of $\delta$ below this BBP phase transition threshold, though the outlier cannot be estimated from $\widehat{{\boldsymbol{\Sigma}}}$ in that regime.

\subsubsection{Synthetic Data Experiments}

Figure~\ref{fig:synth_data_one_spike} reports the empirical 
prediction risk of Ridge, optimal self-distillation, and different 
PCR estimators. The design matrix 
$\bX \in \mathbb{R}^{n \times p}$ has i.i.d.\ rows drawn from 
$\mathcal{N}(\mathbf{0}, {\boldsymbol{\Sigma}})$, where 
${\boldsymbol{\Sigma}} = \mathbf{I}_p + \delta \bv \bv^\top$ is the 
one-spike covariance model with spike direction $\bv$ drawn 
uniformly on the unit sphere. The response is generated as 
$\by = \bX\bbeta_0 + \beps$, where 
$\beps \sim \mathcal{N}(\mathbf{0}, \sigma_\beps^2\mathbf{I}_n)$ 
independently of $\bX$. We let the spike strength $\delta$ vary 
over a grid in $(0.1, 10)$. The signal $\bbeta_0$ is fixed 
throughout and constructed to satisfy $\|\bbeta_0\|_2 = 2$ and 
$\bbeta_0^\top\bv = 1.7$. For each $\delta$, we keep $\bX, \beps$ fixed and estimate the prediction risk for each model using $500$ i.i.d.~test samples drawn from the same distribution. Finally, for each $\delta$, we report results averaged over $40$ independent runs of our full experimental pipeline.

For Ridge, the reported risk is the 
minimum over a grid of regularization parameters. 
For SD, we use the oracle optimal parameters $(\lambda_0^\ast, \lambda_1^\ast, 
\xi_1^\ast)$ from Theorem~\ref{thm: main}. For PCR, we report three choices of 
retained components $m \in \{1, 400, 700\}$. When $m = n = 700$, PCR 
retains all singular vectors and reduces to the minimum $\ell_2$-norm 
interpolator, since $\bX \in 
\mathbb{R}^{n \times p}$ with $p > n$ has at most $n$ nonzero singular values.

Note the following observations. First, as the spike strength 
increases, both $m = 1$ and  
$m = 400$ eventually outperform the tuned ridge. Notably, for 
$\delta \gtrsim 9$, tuned ridge is worse than PCR with a single component. 
Second, the min-norm interpolator ($m = 700$) performs poorly and does not improve much across the entire range of spike strengths. Finally, optimal self-distillation, with parameters from Theorem~\ref{thm: main}, uniformly dominates other estimators across all values of $\delta$. 

\begin{figure}[htbp]
    \centering
    \includegraphics[scale=0.5,keepaspectratio]{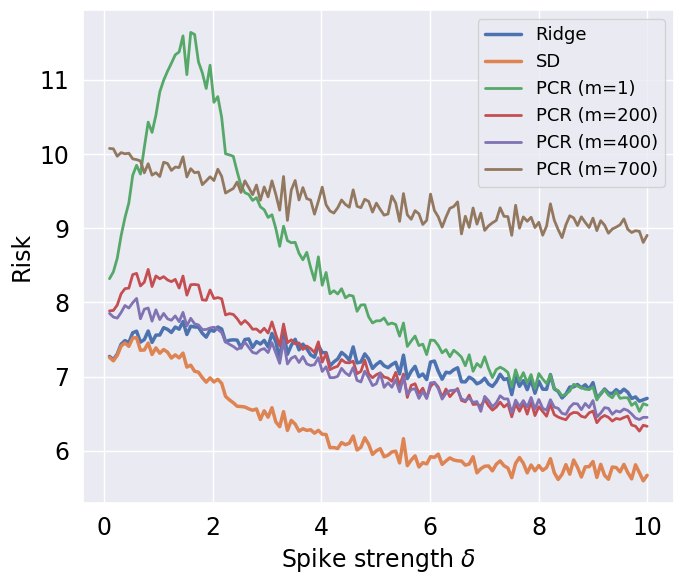}
    \caption{Prediction risk of Ridge, optimal self-distillation (SD), and PCR  as a function of the spike strength $\delta \in (0.1, 10)$, 
    estimated via simulation. We set $n = 700$, $p = 1400$, 
    $\|\bbeta_0\| = 2$, $\bbeta_0^\top \bv = 1.7$, ${\sigma_\beps} = 2$ and $\sx = 1$. 
    }
    \label{fig:synth_data_one_spike}
\end{figure}

While our theoretical results are established for spectral shrinkage estimators and do not cover the Lasso, it is natural to ask whether the general phenomenon that self-distillation strictly improves over the initial estimator continues to hold beyond this class. To investigate this, we consider the following one-step Lasso self-distillation procedure. Given a Lasso teacher $\widehat{\bbeta}_{\mathrm{Lasso}}(\lambda)$ with penalty parameter $\lambda > 0$, we define the one-step Lasso-SD estimator by replacing the $\ell_2$-squared loss with the $\ell_1$ loss in equation~\eqref{eq:SD-recursion} of Ridge-SD. The design matrix $\bX \in \mathbb{R}^{n \times p}$ has i.i.d.\ rows drawn from 
$\mathcal{N}(\mathbf{0}, {\boldsymbol{\Sigma}})$, where 
${\boldsymbol{\Sigma}} = \mathbf{I}_p + \delta \bv \bv^\top$ is a
one-spike covariance model with spike direction $\bv$ drawn uniformly on the unit 
sphere. The response is generated as $\by = \bX\bbeta_0 + \beps$, where 
$\beps \sim \mathcal{N}(\mathbf{0}, \sigma_\beps^2 \mathbf{I}_n)$ independently 
of $\bX$. The signal $\bbeta_0$ is sparse with $30\%$ nonzero entries, fixed throughout, and constructed to satisfy $\|\bbeta_0\|_2 = 2$ and $\bbeta_0^\top \bv = 1.7$. For Ridge-SD and Lasso-SD, we report the minimum over a joint 
grid of $(\lambda_1, \xi_1)$ as a function of the initial regularization parameter $\lambda$. The prediction risk for each estimator is 
estimated on $200$ i.i.d.\ test samples drawn from the same distribution.

Figure~\ref{fig:placeholder} provides numerical evidence that 
self-distillation improves the Lasso in the one-spike case. 
Specifically, we observe that optimally tuned SD-Lasso consistently 
outperforms  Lasso across all values of $\lambda_0$, and 
moreover, the optimal teacher parameter $\lambda_0$ for 
SD-Lasso differs from the optimal parameter for the plain Lasso, 
indicating that self-distillation benefits from a deliberately 
mistuned teacher. These observations suggest 
that the optimality of self-distillation is a robust phenomenon 
that extends beyond the spectral shrinkage settings covered 
by our theory.

\begin{figure}[htbp]
    \centering
    \includegraphics[scale=0.5,keepaspectratio]{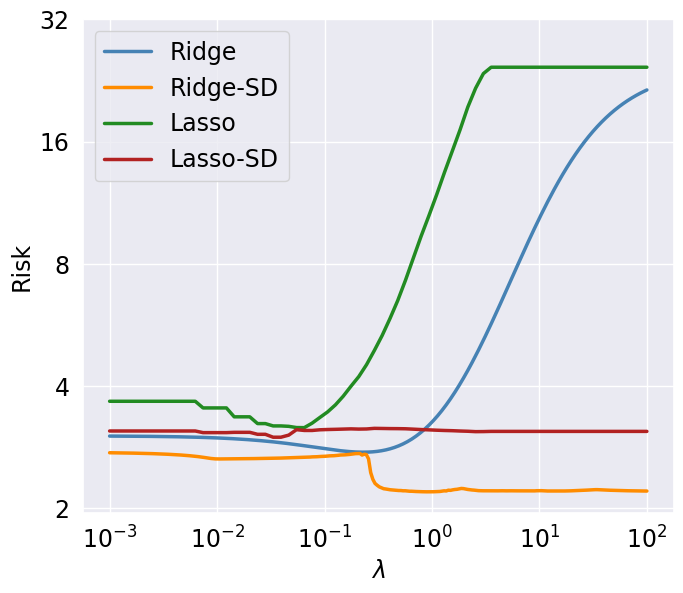}
    \caption{Prediction risk of Ridge, one-step Ridge self-distillation, Lasso, and one-step Lasso self-distillation as a function of $\lambda$. Parameters:  $n = 400$, $p = 800$, 
    $\|\bbeta_0\| = 2$, $\bbeta_0^\top \bv = 1.7$, $\delta = 7$, ${\sigma_\beps} = 1$ and $\sx = 1$.}
    \label{fig:placeholder}
\end{figure}

\section{Proof Outline}\label{sec: noverty + outline}
\subsection{Proof Outline for Theorem \ref{thm: main}}
\label{sec:outline_main}
In this section, we outline the proofs of our main theorems, emphasizing the technical novelties. For brevity, we focus on the argument for the limiting prediction error (Theorem \ref{thm: main}). The proof for the estimation error is analogous and deferred to Appendix~\ref{app: H}. 
\\\\
\noindent
{\bf Step I: }Recall that, for any $f \in \mathcal{F}$, the estimator $\widehat{\bbeta}_f$ is defined by
\[
\widehat \bbeta_f = f(\widehat {\boldsymbol{\Sigma}}) \frac{\mathbf{X}^\top \mathbf{y}}{n}
=
\bW f({\boldsymbol{\Lambda}}) \bW^\top \frac{\mathbf{X}^\top \mathbf{y}}{n},
\qquad
\text{where }
\widehat{{\boldsymbol{\Sigma}}} = \bW {\boldsymbol{\Lambda}} \bW^\top.
\]
The conditional prediction error of $\widehat \bbeta_f$, admits the following decomposition:
\[
\|\widehat{\bbeta}_f - \bbeta_0\|_{\boldsymbol{\Sigma}}^2
=
\mathcal{B}_{f}(\bX, \by) + \mathcal{V}_{f}(\bX, \by) + \mathcal{E}_{f}(\bX, \by),
\]
where
\allowdisplaybreaks
\begin{align*}
\mathcal{B}_{f}(\bX,\by)
&:=
\sx^2\cdot\bbeta_0^\top
\bW \bigl(\mathbf{I}_p - {\boldsymbol{\Lambda}} f({\boldsymbol{\Lambda}})\bigr)^2
\bW^\top \bbeta_0
+
\sum_{j=1}^{\rk} \delta_j
\left(
\bbeta_0^\top
\bW \bigl(\mathbf{I}_p - {\boldsymbol{\Lambda}} f({\boldsymbol{\Lambda}})\bigr)
\bW^\top \bv_j
\right)^2, \\
\mathcal{V}_{f}(\bX, \by)
&:=
\frac{\sx^2\,\sigma_\beps^2}{n}
\operatorname{tr}\left(
\bW {\boldsymbol{\Lambda}} f^2({\boldsymbol{\Lambda}}) \bW^\top
\right)
+
\sum_{j=1}^{\rk} \delta_j
\frac{\sigma_\beps^2}{n}
\bv_j^\top \bW {\boldsymbol{\Lambda}} f^2({\boldsymbol{\Lambda}}) \bW^\top \bv_j, \\
\mathcal{E}_{f
}(\bX, \by)&:= - 2\,\bbeta_0^\top\bW\bigl(\mathbf{I}_p - \boldsymbol{\Lambda} f(\boldsymbol{\Lambda})\bigr)
\bW^\top\boldsymbol{\Sigma}\bW f(\boldsymbol{\Lambda})\bW^\top\frac{\bX^\top\boldsymbol{\varepsilon}}{n} + \left[\frac{\beps^\top \bX f(\widehat{\boldsymbol{\Sigma}}) {\boldsymbol{\Sigma}} f(\widehat{\boldsymbol{\Sigma}}) \bX^\top \beps}{n^2} -\mathcal{V}_{f}(\bX, \by) \right].
\end{align*}
The first step is to derive the almost-sure limits of $\mathcal{B}_{f}(\bX, \by), \mathcal{V}_{f}(\bX, \by)$ and $\mathcal{E}_{f}(\bX, \by)$ as $n, p \uparrow \infty$ and $p/n \to c$. To the best of our knowledge, this is the first closed-form expression for the limit of the conditional prediction risk of $\widehat{\bbeta}_f$ for a general function $f$ when ${\boldsymbol{\Sigma}}$ is a spiked covariance matrix. In Lemma~\ref{prop: general formula}, we show that
\begin{align*}
\mathcal{B}_{f}(\bX, \by)
&\xrightarrow[\ ]{\mathrm{a.s.}} \cB_f^* := 
r^2\sx^2
\int (1 - x f(x))^2 \, \mathrm{d}F_\alpha(x)
+
\sum_{j=1}^{\rk}
\delta_j \alpha_j^2
\left(
\int (1 - x f(x)) \, \mathrm{d}F_{\delta_j}(x)
\right)^2, \\
\mathcal{V}_{f}(\bX, \by)
&\xrightarrow[\ ]{\mathrm{a.s.}} \cV_f^* := 
 c\sx^2\, \sigma_\beps^2 
\int x f^2(x)\, \mathrm{d}F_{\mathrm{MP}}(x), \\
\mathcal{E}_{f}(\bX, \by)&\xrightarrow[\ ]{\mathrm{a.s.}} 0,
\end{align*}
with $(r, \alpha_j)$ defined as in Assumption \ref{assm:main}. Here $F_{\alpha}$ is a mixture distribution defined as: 
\begin{equation}
\label{eq:omega_js}
F_\alpha = \omega_0 F_{\mathrm{MP}} + \sum_{j=1}^{\rk} \omega_j F_{\delta_j}, \qquad \omega_0 = 1- \frac{\sum_{j=1}^{\rk}\alpha_j^2}{r^2}, \qquad \omega_j = \frac{\alpha_j^2}{r^2} \,,
\end{equation}
where $F_\delta$ is the distribution whose Stieltjes transform $m_\delta(z)$ is given by
\[
m_\delta(z)
=
\frac{\sx^2m(z)}{\sx^2+\delta+\delta z\,m(z)},
\qquad
z \in \mathbb{C}^+ \,.
\]
Here $m(z)$ is the Stieltjes transform of the Marchenko--Pastur law. In particular, $m_\delta(z)$ denotes the Stieltjes transform of the limiting spectral distribution of the sample covariance matrix when the observations are generated according to the rank-one spiked model, i.e., ${\boldsymbol{\Sigma}} = \sx^2\mathbf{I}_p + \delta \bv \bv^\top$. 

Upon establishing the preceding limit, we derive \(f_*^{\rm pred}\) by minimizing the limiting risk with respect to \(f\). As is evident from the expression above, the limiting prediction risk depends on \(f\) in a nonlinear manner, and therefore minimizing it with respect to \(f\) is non-trivial. Our key novelty lies in  introducing the Radon--Nikodym derivatives \(\nu_j\) of \(F_{\mathrm{MP}}\) with respect to the spiked probability measures \(F_{\delta_j}\):
\[
\nu_j(x):=\frac{\mathrm d F_{\mathrm{MP}}(x)}{\mathrm d F_{\delta_j}(x)}
=
\frac{(\delta_j+c\sx^2)(\delta_j+\sx^2) - \delta_j x}{c\sx^2\,(\delta_j + \sx^2\,)}.
\]
As a consequence, we have
\[
\frac{\mathrm d F_\alpha(x)}{\mathrm d F_{\rm MP}(x)}
=
\frac{\omega_0\,\mathrm d F_{\mathrm{MP}}+\sum_{j=1}^{\rk}\omega_j\,\mathrm d F_{\delta_j}}{\mathrm d F_{\rm MP}(x)}
=
\omega_0+\sum_{j=1}^{\rk}\omega_j\frac{\mathrm d F_{\delta_j}(x)}{\mathrm d F_{\rm MP}(x)}
=
\omega_0+\sum_{j=1}^{\rk}\frac{\omega_j}{\nu_j(x)},
\]
which is well defined whenever \(\delta_j\neq -1\). A crucial observation is that \(\nu_j(x)\) is linear in \(x\); this structural property plays a central role in our proof. Using \(\nu_j(x)\), we can further derive the Radon--Nikodym derivatives of \(\mathrm d F_{\delta_j}(x)\) and \(\mathrm d F_{\rm MP}(x)\) with respect to their mixture \(\mathrm d F_\alpha\), denoted by \(\mu_j(x)\) and \(\mu_0(x)\), respectively, in Appendix~\ref{section: RN derivs}.
This construction allows us to rewrite the limiting prediction risk in terms of integrals with respect to the common measure \(F_\alpha\), through the use of \(\mu_j(x)\) and \(\mu_0(x)\). Building on this representation, and after further algebraic manipulations, we are able to express the limiting risk of \(\widehat\bbeta_f\) as a quadratic functional of \(f\):
\begin{align}\label{eq: formula for pred risk in proof outline section}
\cB_f^*+\cV_f^*
=
\sx^2\,r^2\, + \sum_{j=1}^\rk \delta_j\alpha_j^2
+ \langle f,f\rangle_w
- 2\,\langle f,\,g\rangle_w
+ \sum_{j=1}^\rk\delta_j\alpha_j^2\,\langle f, h_j\rangle_w^2,
\end{align}
where \(w(x)\), $g(x)$ and \(\{h_j(x)\}_{1\le j\le \rk}\) are fixed functions depending only on \(\{\mu_j(x)\}_{1\le j\le \rk}\) and \(\mu_0(x)\), and $\langle, \rangle_w$ denotes the weighted $\ell_2$ inner product with the weight function being $x\mapsto xw(x)$.  Minimizing this quadratic form yields the closed-form expression of the minimizer \(f_*^{\rm pred}\), as defined in \eqref{def:f_pred_closed_form}. The details are deferred to Appendix~\ref{app: B}. Finally, in Lemma~\ref{lemma: f_ast is unique} we prove that \(f_*^{\rm pred}\) is not only an optimizer of the limiting risk over the function class \(\cF\), but in fact unique on $\mathcal{S}_c^+\setminus\{0\}$. 
\\\\
{\bf Step II: (Optimality of self-distillation)} Our next contribution is to show that \(f_*^{\rm pred}\) can be attained by an \(\rk\)-step self-distillation procedure. Moreover, we show that there exists a choice of \(\{(\delta_j,\bv_j)\}_{1\le j\le \rk}\) for which \(\rk\) steps of self-distillation are genuinely necessary to attain \(f_*^{\rm pred}\), in the sense that no procedure with fewer than \(\rk\) steps can recover \(\widehat\bbeta_{f_*^{\rm pred}}\), regardless of the choice of tuning parameters. 
To this end, we first show in Lemma~\ref{lemma: polynomials P Q} that \(f_*^{\rm pred}\) is a rational function whose numerator has degree \(\rk\) and whose denominator has degree \(\rk+1\). We then use the general Lemma~\ref{lemma: lagrange interpolation} to prove that \(\rk\)-step self-distillation yields the shrinkage-based estimator \(\widehat\bbeta_{f_*^{\rm pred}}\) for a suitable choice of regularization and distillation weight parameters.

\subsection{Proof Outline for General Covariance and GCV}
\noindent
\textbf{General covariance (Theorem~\ref{thm:general-covariance}).}
We prove the result in two steps. First, we show that a single
prescribed penalty $\lambda>0$ suffices to approximate any continuous
spectral shrinkage function on $[0,M]$. Indeed, setting $u=x/(x+\lambda)$, we can write
a $k$-step SD spectral shrinkage function as
\[
 f_{\ba,\lambda}(x)=\frac{P(u)}{x+\lambda},
 \qquad
 P(u)=\sum_{j=0}^k a_j u^j,
 \qquad P(1)=\sum_{j=0}^k a_j=1.
\]
By definition, $x\in[0,M]$ implies
$u\in[0,M/(M+\lambda)]$, and hence bounded away from $1$. 
Thus, given a continuous target $f$, we can impose $P(1)=1$ by
writing $P(u)=1+(1-u)Q(u)$ and choosing $Q$ to approximate
\[
 \frac{(x+\lambda)f(x)-1}{1-u} = \frac{\left(\frac{\lambda u}{1-u}+\lambda\right)f\left(\frac{\lambda u}{1-u}\right)-1}{1-u} = \frac{\frac{\lambda}{1 - u}f\left(\frac{\lambda u}{1-u}\right)-1}{1-u},
\]
where the first equality follows from the fact $u = x/(x + \lambda) \iff x = \lambda u/(1 - u)$. 
The right-hand-side of the above display is a continuous function of $u$ on $[0,M/(M+\lambda)]$, so the Weierstrass theorem applies. Lemma~\ref{lem:density} provides the details on how to choose the approximating coefficients so that they are realized by finitely many distillation weights.

In the second step, we move from the approximation of spectral shrinkage functions to the approximation of
limiting prediction risks. Recall the notation of spectral shrinkage function $f_{\ba, \lambda}$ from Equation \eqref{eq:def_f_a_lambda}. 
Using the two-resolvent equivalent of Stark-D\"ornemann~\cite{starkDornemann2026}, we show in Lemma~\ref{lem:derivatives} that for some kernel $\mathcal K$ (defined in~\eqref{eq:kernel}), we have
\[
 \mathcal R^{\mathrm{pred}}(f_{\ba,\lambda};c)
 =
 \sum_{i,j=0}^k a_i a_j
 \left.D_i^\lambda D_j^\tau
       \mathcal K(-\lambda,-\tau)\right|_{\tau=\lambda}.
\]
for $(a_1, \dots, a_k) \in \bbQ^k$ with $k \in \bbN$, 
where $\mathcal R^{\mathrm{pred}}(f_{\ba,\lambda};c)$ is the limiting prediction-risk of $\widehat{\bbeta}_{f_{\ba, \lambda}}$. 
This gives deterministic risk limits on a countable dense family.
The remaining step is to show that, almost surely, the square root
of the empirical prediction risk is eventually Lipschitz in
$\|f-g\|_{\infty,[0,M]}$, with a common constant for all spectral shrinkage functions.
We can therefore extend the limits to every continuous spectral shrinkage function on
the same probability-one event.  Fix any such continuous function $f$ and a ridge penalty $\lambda>0$.
By Lemma~\ref{lem:density}, for every integer $m\geq 1$, we can choose a finite number of self-distillation steps $k_m$ and corresponding distillation
weights such that the resulting spectral shrinkage function $f_m$
satisfies
\[
\|f_m-f\|_{\infty,[0,M]} \leq \frac{1}{m}.
\]
Each of these self-distillation procedures uses the same penalty
$\lambda$ at every step; only the number of steps and the distillation
weights may vary with $m$. Thus, $f_m$ converges uniformly to $f$
on $[0,M]$. The continuity bound therefore implies that
$\mathcal R^{\mathrm{pred}}(f_m;c)\to
\mathcal R^{\mathrm{pred}}(f;c)$.
Since $f$ was arbitrary and every SD spectral shrinkage function
is itself continuous, taking the infimum over $f$ establishes
\eqref{eq:approximation}.
The details are provided in Appendix~\ref{app: I}.

\medskip
\noindent
\textbf{Tuning by GCV (Theorems~\ref{thm:gcv} and~\ref{th:growing}).}
For GCV, we work with the finite-sample conditional prediction risk
$\mathcal R(f)=\sigma_\varepsilon^2+
\|\widehat\beta_f-\beta_0\|_\Sigma^2$.
For each fixed continuous $f$, Theorem \ref{thm:general-covariance}(a) gives
$\mathcal R(f)\to
\sigma_\varepsilon^2+\mathcal R^{\mathrm{pred}}(f;c)$
almost surely. To prove Theorem~\ref{thm:gcv}, we consider the entrywise errors
\[
 E_{ij}(\lambda)
 :=\widehat \bT_{ij}(\lambda) - \bT_{ij}(\lambda).
\]
A leave-one-out argument and standard concentration inequalities controls the above for complex penalties $\lambda$ with real part, see Lemmas~\ref{lem:gcvapp-loo} and~\ref{lem:gcvapp-trace}. In the latter, we also introduce an event $\mathcal B_n$ satisfying $\mathbb P(\mathcal B_n^c)\le Cn^{-2}$. On this event, applying Cauchy's integral formula to the generating function
in~\eqref{eq:error-generating}, with contour radius
$(k+1)/(k+2)$, controls these coefficients up to depth $k$
\[\sup_{\lambda\in I}
 \mathbb E\left[ \mathbf1_{\mathcal B_n}
 \max_{0\le i,j\le k}|E_{ij}(\lambda)|
 \right] \le \frac{C\log^2(k+2)}{\sqrt n} \,.\]
See Lemma~\ref{lem:gcvapp-transfer} for details. To move the supremum over $\lambda$ inside the expectation,
we use the identity
\[
 \frac{d}{d\lambda}D_i^\lambda h(\lambda)
 =\frac{i+1}{\lambda}
   \bigl(D_{i+1}^\lambda h(\lambda)-D_i^\lambda h(\lambda)\bigr),
\]
which follows from Leibniz's rule. Applying this identity in
each argument of the two-penalty error and differentiating
along the diagonal yields
\[
 E_{ij}'(\lambda)
 =\frac{i+1}{\lambda}
      \{E_{i+1,j}(\lambda)-E_{ij}(\lambda)\}
  +\frac{j+1}{\lambda}
      \{E_{i,j+1}(\lambda)-E_{ij}(\lambda)\}.
\]
The fundamental theorem of calculus therefore gives
\begin{align*}
 \sup_{\lambda\in I}\max_{0\le i,j\le k}|E_{ij}(\lambda)|
 &\le \max_{0\le i,j\le k}|E_{ij}(\lambda_{\min})| +\frac{4(k+1)}{\lambda_{\min}}
 \int_I\max_{0\le i,j\le k+1}|E_{ij}(\theta)|\,d\theta.
\end{align*}
Multiplying by $\mathbf1_{\mathcal B_n}$, taking expectations,
and applying the fixed-penalty bound yields
\[
 \mathbb E\!\left[
 \mathbf1_{\mathcal B_n}\sup_{\lambda\in I}
 \max_{0\le i,j\le k}|E_{ij}(\lambda)|
 \right]
 \le \frac{C(k+1)\log^2(k+2)}{\sqrt n}.
\]
Markov's inequality then gives the uniform bound
in Theorem \ref{thm:gcv}. The details are provided in Appendix \ref{app: J}.

To establish Theorem \ref{th:growing}, we first use this bound to control
the GCV criterion uniformly over the search class. The quadratic
representations of $\widehat R(\ba,\lambda)$ and
$\mathcal R(f_{\ba,\lambda})$, together with the constraint
$\|\ba\|_1\leq A_n$, give
\[
\begin{aligned}
e_n
&:=
\sup_{\substack{\lambda\in I\\ \ba\in\mathcal B(K_n,A_n)}}
\left| \widehat R(\ba,\lambda)-\mathcal R(f_{\ba,\lambda})
\right| \leq A_n^2 \sup_{\lambda\in I}
\max_{0\leq i,j\leq K_n} \left| \widehat{\bT}_{ij}(\lambda)-\bT_{ij}(\lambda)\right|= o_{\text{p}}(1),
\end{aligned}
\]
where the last step follows from Theorem \ref{thm:gcv} and the growth
condition Equation \eqref{eq:general-schedule}.
Next, we compare the selected estimator with a fixed SD estimator
whose limiting risk is arbitrarily close to the optimum. Write
\[
\mathcal{R}_\ast:=
\inf_{f\in C([0,\infty);\mathbb R)}
\mathcal R^{\mathrm{pred}}(f;c),
\]
and fix $\lambda_0\in\bigcap_{n\geq1}\Lambda_n$.
For every $\epsilon>0$,
Theorem~\ref{thm:general-covariance}(b) provides a fixed
finite-depth SD spectral shrinkage function
$f_\epsilon=f_{\ba_\epsilon,\lambda_0}$ satisfying
\[
\mathcal R^{\mathrm{pred}}(f_\epsilon;c)
\leq \mathcal{R}_\ast+\epsilon.
\]
Since its depth and coefficient norm are finite, the conditions
$K_n,A_n\to\infty$ ensure that $f_\epsilon$ belongs to the search
class for all sufficiently large $n$, after padding
$\ba_\epsilon$ with zeros.
The approximate minimization property (Equation \eqref{eq:fit}) and the uniform
error bound therefore imply
\[
\begin{aligned}
\mathcal R(\widehat f_n)
&\leq
\widehat R(\widehat a_n,\widehat\lambda_n)+e_n \leq
\widehat R(\ba_\epsilon,\lambda_0)+e_n+\delta_n \leq
\mathcal R(f_\epsilon)+2e_n+\delta_n.
\end{aligned}
\]
Because $f_\epsilon$ is fixed,
Theorem~\ref{thm:general-covariance}(a) gives
\[
\mathcal R(f_\epsilon)
\xrightarrow{\mathrm{a.s.}}
\sigma_\varepsilon^2+
\mathcal R^{\mathrm{pred}}(f_\epsilon;c).
\]
Together with $e_n=o_{\text{p}}(1)$ and $\delta_n\to0$,
this yields
\[
\mathcal R(\widehat f_n)
\leq
\sigma_\varepsilon^2+\mathcal{R}_\ast+\epsilon+o_{\text{p}}(1).
\]
Since $\epsilon>0$ is arbitrary, we conclude that
\[
\left[
\mathcal R(\widehat f_n)
-\bigl(\sigma_\varepsilon^2+\mathcal{R}_\ast\bigr)
\right]_+
\xrightarrow{\text{p}}0,
\]
which proves Theorem \ref{th:growing}.

\subsection{Proof Outline of Theorem \ref{thm: product_of_shrinkage}}\label{proof outline federated}
In this subsection, we present a roadmap for the proof of Theorem \ref{thm: product_of_shrinkage}. The main challenge, and hence the primary technical novelty, lies in analyzing the limiting aggregated estimator $\widehat\bbeta_{\mathrm{agg}} = \sum_{\ell = 1}^K \rho_{\ell} \widehat\bbeta_\ell$, where \(\widehat\bbeta_\ell \equiv \widehat\bbeta_{f_\ell}\) is a shrinkage-based estimator constructed locally on server \(\ell\) using the shrinkage function \(f_\ell \in \mathcal F\).

\medskip
\noindent 
{\bf Step I: }As before, we begin with the decomposition of \(\widehat\bbeta_{\mathrm{agg}}\):

\begin{align*}
\|\widehat\bbeta_{\mathrm{agg}} - \bbeta_0\|_{\boldsymbol{\Sigma}}^2
&=\mathbb{E}\left[
\|\widehat\bbeta_{\mathrm{agg}} - \bbeta_0\|_{\boldsymbol{\Sigma}}^2
\;\middle|\; \mathbf{X}_{1:K}
\right] + \mathcal{E}^{\mathrm{agg}}
\end{align*}
where $\mathcal{E}^{\mathrm{agg}} \xrightarrow[]{\text{a.s.}} 0$ and 
\begin{align*}
\mathbb{E}\left[
\|\widehat\bbeta_{\mathrm{agg}} - \bbeta_0\|_{\boldsymbol{\Sigma}}^2
\;\middle|\; \mathbf{X}_{1:K}
\right]
&=
\frac{\sx^2}{K^2}\bbeta_0^\top \left(\sum_{\ell=1}^K g_\ell(\widehat{\boldsymbol{\Sigma}}_{\ell})\right)^2 \bbeta_0
+
\sum_{j=1}^{\rk}\frac{\delta_j}{K^2}\,
\left(\bbeta_0^\top \sum_{\ell=1}^K g_\ell(\widehat{\boldsymbol{\Sigma}}_{\ell}) \bv_j\right)^2 \\
&\hspace{15em}
+ \sum_{\ell = 1}^{K} \frac{\sigma_\beps^2}{K^2 n_\ell}
\operatorname{tr}\left({\boldsymbol{\Sigma}} \widehat{\boldsymbol{\Sigma}}_{\ell} \tilde f_\ell(\widehat{\boldsymbol{\Sigma}}_{\ell})^2\right) \\
&= \mathcal{B}_{0}^{\mathrm{agg}} + \sum_{j=1}^{\rk} \mathcal{B}_{j}^{\mathrm{agg}} + \sum_{\ell=1}^K \mathcal{V}_{\ell}^{\mathrm{agg}},
\end{align*}
where $g_\ell(x) = K \rho_{\ell} x f_\ell(x) - 1$. Therefore, in order to derive the limit of the prediction error of \(\widehat\bbeta_{\mathrm{agg}}\), it suffices to derive the limits of \(\mathcal{B}_{0}^{\mathrm{agg}}\), \(\mathcal{B}_{j}^{\mathrm{agg}}\), and \(\mathcal{V}_{\ell}^{\mathrm{agg}}\).

\medskip

\noindent
{\bf Step II:}
The error term $\mathcal{E}^\mathrm{agg}$ is negligible by the same concentration arguments used in the previous section. The derivations of the limits of \(\mathcal{B}_{j}^{\mathrm{agg}}\) and \(\mathcal{V}_{\ell}^{\mathrm{agg}}\) are similar to those in the proof of Proposition~\ref{prop: general formula}. However, an additional technical difficulty, and consequently the main novelty of our argument, arises in deriving the limit of \(\mathcal{B}_{0}^{\mathrm{agg}}\). Indeed, from its definition, \(\mathcal{B}_{0}^{\mathrm{agg}}\) contains cross terms of the form $\bbeta_0^\top g_\ell(\widehat{\boldsymbol{\Sigma}}_{\ell}) g_{\ell'}(\widehat{\boldsymbol{\Sigma}}_{\ell'}) \bbeta_0$ for $\ell \neq \ell' \in [K]$. 
As a result, novel techniques are required to establish their limit. This is precisely the role of Theorem~\ref{lemma:non-free limit}. The final step is analogous to Step~II of Section \ref{sec:outline_main}. Namely, we minimize the limiting prediction risk with respect to
$\{(f_1,\rho_1),\dots,(f_K,\rho_K)\}$. This optimization problem can be thought of as a multivariate generalization of the optimization problem elaborated in Step II of Section \ref{sec:outline_main} and consequently is significantly more involved to establish.

\section{Discussion}\label{sec:discussion}
We establish for spiked covariance matrices that self-distillation achieves optimal prediction error among spectral shrinkage estimators. Notably, for $s$ spikes, $s$ steps of self-distillation are necessary and sufficient, suggesting that each distillation step captures or corrects for a distinct spike direction in the covariance. 
The main theorems are stated under mild nondegeneracy assumptions (Assumption \ref{assm:main}), but several boundary regimes, when the assumption is violated, are also worth understanding. For example, when $\omega_0=0$, then $\beta_0$  belongs to the span of the spike directions $\{\bv_1, \cdots, \bv_{\rk}\}$. In this case, or if two spike strengths $(\delta_i, \delta_j)$ satisfy $\delta_i \delta_j = c\sigma_0^4$, self-distillation still yields the optimal shrinkage, but the number of distillation steps required can be strictly smaller than $\rk$ and/or we possibly need to start from the $0$ estimator instead of a ridge regularized estimator. Similarly, if multiple spikes have the same magnitude, then the optimal number of steps would be the number of distinct spike strengths. Thus, the relevant complexity is not always the raw number of spikes, but rather an effective number of spectral degrees of freedom. Our arguments extend to these settings, but the bookkeeping becomes substantially more involved without introducing new intuition, thus we refrain from presenting these here.

Several outstanding questions remain following our paper.  
A natural direction is to determine whether the optimal shrinkage rule admits a Bayesian interpretation (e.g., whether it is Bayes-optimal under a suitable prior) and can be recovered through an iterated Bayesian updating procedure. While our theory establishes optimality within the spectral shrinkage class, it remains important to clarify what makes the self-distillation parametrization statistically or algorithmically preferable to directly fitting a general rational shrinkage function. Finally, extending our framework beyond the current $\ell_2$-regularized linear model to more general settings --- such as $\ell_1$- or $\ell_p$-regularized models, more general convex penalized linear or generalized linear models, random-feature models, or simple neural networks --- would help determine how universal the optimality of the self-distillation phenomenon is. Our simulations for $\ell_1$-penalized regression already suggest that self-distillation can improve upon the corresponding standard estimator, but developing a complete theoretical understanding in this setting remains an interesting open problem for future work.

\section{Acknowledgements}
The authors would like to thank Qiyang Han, Lucas Janson, Tracy Ke, Art Owen, Giovanni Parmigiani and Souhardya Sengupta for helpful discussions. The authors would also like to thank Emmanuel Cand\`{e}s, Art Owen, Neil Shephard and the anonymous referees for constructive comments that significantly improved an earlier version of the manuscript. P.S.~gratefully acknowledges support from the National Science Foundation (DMS CAREER 2440824) and the Office of Naval Research (N00014-26-1-2144). D.M. gratefully acknowledges support from the National Science Foundation (NSF DMS 2515787).

\bibliographystyle{unsrtnat}
\bibliography{Refs}

\newpage
\begin{center}
{\Large \textbf{Appendix}}
\end{center}
We include a proof dependency chart for convenience. Section~\ref{app: A} 
sets up the preliminaries needed throughout the rest of the appendix.  Proofs of the main results are presented in Section~\ref{app: B} and onward.  We do not include Theorem~\ref{thm: main_thm est} in the chart, as its proof 
follows by similar methods to that of Theorem~\ref{thm: main}. 
We refer the reader to Appendix~\ref{app: H} for details.

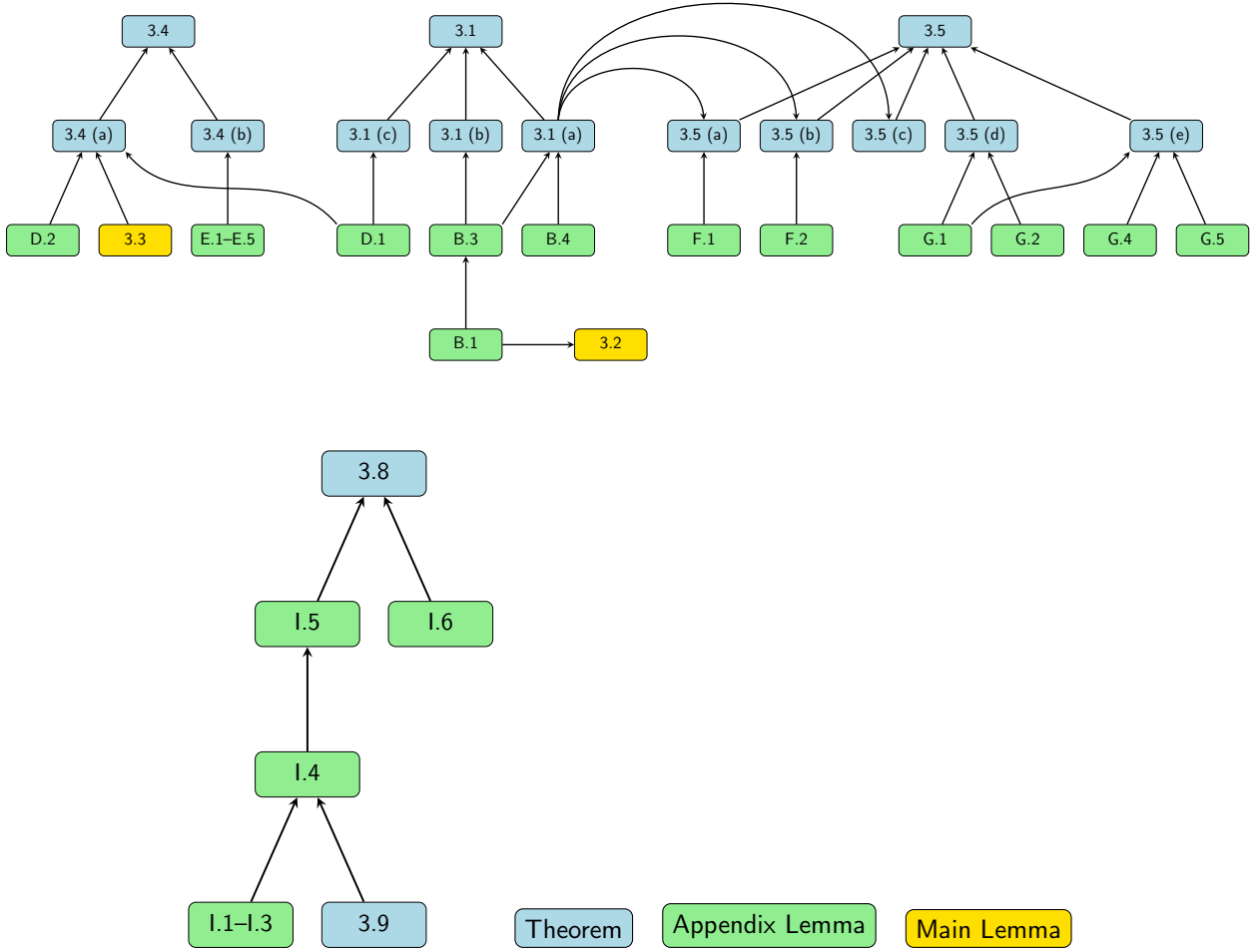
\begin{figure}[H]
\centering
\providecolor{thmcolor}{RGB}{240,206,248}
\providecolor{lemcolor}{RGB}{175,249,180}
\providecolor{mainlemcolor}{RGB}{211,240,248}

\resizebox{\textwidth}{!}{%
\begin{forest}
  for tree={
    draw,
    rounded corners   = 3pt,
    align             = center,
    font              = \small\sffamily,
    minimum width     = 1.5cm,
    minimum height    = 0.65cm,
    inner sep         = 3pt,
    l sep             = 1.5cm,
    s sep             = 0.4cm,
    edge              = {stealth-, thick},
    tier/.wrap pgfmath arg={tier #1}{level()},
  }
  [, phantom, s sep = 1.5cm
    [{\ref{thm:sd_optimal}}, fill=thmcolor
      [{\ref{thm:sd_optimal} (a)}, fill=thmcolor, name=thm34a
        [{\ref{lemma: lagrange interpolation}}, fill=lemcolor]
        [{\ref{lemma: sd is in mathcal F}}, fill=mainlemcolor]
      ]
      [{\ref{thm:sd_optimal} (b)}, fill=thmcolor
        [{\ref{lemma: coprime} - \ref{lem:necessity-cauchy}}, fill=lemcolor]
      ]
    ]
    [{\ref{thm: main}}, fill=thmcolor
      [{\ref{thm: main} (c)}, fill=thmcolor
        [{\ref{lemma: polynomials P Q}}, fill=lemcolor, name=lemC1]
      ]
      [{\ref{thm: main} (b)}, fill=thmcolor
        [{\ref{lemma: define A and solve f_ast}}, fill=lemcolor,
           name=lemB3, calign=first
          [{\ref{prop: general formula}}, fill=lemcolor, name=lemB1]
          [{\ref{thm: main corollary}}, fill=mainlemcolor, name=lem32,
             edge={draw=none}, xshift=3cm]
        ]
      ]
      [{\ref{thm: main} (a)}, fill=thmcolor, name=thm31a
        [{\ref{lemma: f_ast is unique}}, fill=lemcolor]
      ]
    ]
    [{\ref{thm: f_ast dominates}}, fill=thmcolor
      [{\ref{thm: f_ast dominates} (a)}, fill=thmcolor, name=thm35a
        [{\ref{lemma: min-norm is worse}}, fill=lemcolor]
      ]
      [{\ref{thm: f_ast dominates} (b)}, fill=thmcolor, name=thm35b]
      [{\ref{thm: f_ast dominates} (c)}, fill=thmcolor, name=thm35c]
      [{\ref{thm: f_ast dominates} (d)}, fill=thmcolor
        [{\ref{lemma: concentration PCR}}, fill=lemcolor, name=lemF1]
        [{\ref{pcr_with_all_spiked}}, fill=lemcolor]
      ]
      [{\ref{thm: f_ast dominates} (e)}, fill=thmcolor, name=thm35e
        [{\ref{lemma: proportional regime PCR}}, fill=lemcolor]
        [{\ref{lemma: technical approx of PCR}}, fill=lemcolor]
      ]
    ]
  ]
  \draw[-stealth, thick] (lemC1.north west)
    to[out=130, in=310] (thm34a.south east);
  \draw[-stealth, thick] (lemB3.north east) -- ([xshift=-5pt]thm31a.south);
  \draw[-stealth, thick] (lemB1.east) -- (lem32.west);
  \draw[-stealth, thick] (thm31a.north)
    to[out=85, in=95, looseness=1.2] (thm35a.north);
  \draw[-stealth, thick] (thm31a.north)
    to[out=87, in=93, looseness=1.2] (thm35b.north);
  \draw[-stealth, thick] (thm31a.north)
    to[out=89, in=91, looseness=1.2] (thm35c.north);
  \draw[-stealth, thick] (lemF1.north east)
    to[out=50, in=230] (thm35e.south west);
\end{forest}}

\par\vspace{1.3em}
\noindent\makebox[\linewidth][c]{%
  \raisebox{-\height}{\resizebox{0.44\linewidth}{!}{%
\begin{forest}
  for tree={
    draw,
    rounded corners   = 3pt,
    align             = center,
    font              = \small\sffamily,
    minimum width     = 1.5cm,
    minimum height    = 0.65cm,
    inner sep         = 3pt,
    l sep             = 1.0cm,
    s sep             = 0.4cm,
    edge              = {stealth-, thick},
    tier/.wrap pgfmath arg={tier #1}{level()},
  }
  [, phantom, s sep=1.3cm
    [{\ref{thm:general-covariance}}, fill=thmcolor, name=newGC
      [{\ref{thm:general-covariance} (a)}, fill=thmcolor, name=newGCa
        [{\ref{lem:derivatives}}, fill=lemcolor
          [{\ref{lem:cross_risk}}, fill=lemcolor
            [{\ref{lem:double-kernel_identity}}, fill=lemcolor]
          ]
        ]
        [{\phantom{\ref{lem:density}--\ref{lem:extension}}},
           name=newDEslot, draw=none, edge={draw=none}]
      ]
      [{\ref{thm:general-covariance} (b)}, fill=thmcolor, name=newGCb]
    ]
    [{\ref{th:growing}}, fill=thmcolor, name=newFit
      [{\ref{thm:gcv}}, fill=thmcolor
        [{\ref{lem:gcvapp-trace}}, fill=lemcolor, name=newTrace
          [{\ref{lem:gcvapp-loo}}, fill=lemcolor
            [{\ref{lem:gcvapp-delete}}, fill=lemcolor, name=newDelete]
          ]
        ]
        [{\ref{lem:gcvapp-transfer}}, fill=lemcolor, name=newTransfer]
      ]
    ]
  ]
  \path (newGCa.south) -- (newGCb.south) coordinate[midway] (gcMid);
  \node[
    draw, rounded corners=3pt, fill=lemcolor, font=\small\sffamily,
    minimum width=1.5cm, minimum height=0.65cm, inner sep=3pt
  ] (newDE) at (gcMid |- newDEslot.center)
    {\ref{lem:density}--\ref{lem:extension}};
  \draw[-stealth, thick] (newDE.north) -- (newGCa.south);
  \draw[-stealth, thick] (newDE.north) -- (newGCb.south);
  \draw[-stealth, thick] (newGCa.east) -- (newGCb.west);
  \draw[-stealth, thick] (newGC.east) -- (newFit.west);
  \draw[-stealth, thick] (newTrace.east) -- (newTransfer.west);
  \draw[-stealth, thick] (newDelete.north west)
    to[out=132, in=228] (newTrace.south west);
\end{forest}}}%
  \hspace{0.06\linewidth}%
  \raisebox{-\height}{\resizebox{0.25\linewidth}{!}{%
\begin{forest}
  for tree={
    draw,
    rounded corners   = 3pt,
    align             = center,
    font              = \small\sffamily,
    minimum width     = 1.5cm,
    minimum height    = 0.65cm,
    inner sep         = 3pt,
    l sep             = 1.5cm,
    s sep             = 0.4cm,
    edge              = {stealth-, thick},
    tier/.wrap pgfmath arg={tier #1}{level()},
  }
  [{\ref{thm: product_of_shrinkage}}, fill=thmcolor
    [{\ref{prop: decoupled formula}}, fill=lemcolor
      [{\ref{prop: general formula for agg}}, fill=lemcolor
        [{\ref{prop: normalized estimator federated} - \ref{prop: bias decomposition federated}},
           fill=lemcolor]
        [{\ref{lemma:non-free limit}}, fill=thmcolor]
      ]
    ]
    [{\ref{all f' need to be the same}}, fill=lemcolor]
  ]
\end{forest}}}%
}%

\par\vspace{1.2em}
\newcommand{\legendbox}[2]{\tikz\node[
  draw, rounded corners=3pt, fill=#1,
  minimum width=1.4cm, minimum height=0.5cm,
  font=\small\sffamily]{#2};}
\legendbox{thmcolor}{Theorem}\quad
\legendbox{lemcolor}{Appendix Lemma}\quad
\legendbox{mainlemcolor}{Main Lemma}

\caption{Proof dependency. Arrows point from results to the results they prove.}
\label{fig:proof-tree-all}
\end{figure}

\newpage
\appendix

\section{Preliminaries}\label{app: A}
\subsection{Stieltjes Transforms}\label{section: stieltjes}
In this section, we collect the random matrix tools used in the proofs. For any finite signed measure \(\mu\) on \(\mathbb R_+\), define its Stieltjes transform by
\[
m_\mu(z):=\int \frac{1}{x-z}\,d\mu(x),\qquad z\in\mathbb C^+.
\]
For a random matrix $\bA \in \reals^{n \times p}$ with i.i.d. entries with mean $0$ and variance $\sigma_0^2$, the law of limting spectral distribution (as $p/n \to c$) of $(\bA^\top \bA)/n$ is the Marchenko-Pastur (MP) distribution \cite{MarchenkoPastur1967, silverstein1995analysis}. The Stieltjes transform of $F_{\rm MP}$, denoted by $m_{\rm MP}(z)$ extends analytically to \(\mathbb C\setminus\mathbb R_+\) and admits the representation: 
\begin{equation}\label{eq:square root rep of m(z)}
    m(z)=\frac{\sx^2(1-c)-z+\sqrt{(z-a)(z-b)}}{2cz\sx^2},
\end{equation}
where $a:=\sx^2(1-\sqrt c)^2$, $b:=\sx^2(1+\sqrt c)^2$, and the branch of the square root is chosen so that \(\Im\sqrt{(z-a)(z-b)}>0\) for \(\Im z>0\). Furthermore, we use the notation $\cS_c:=[a,b]$ to denote the bulk support of the MP law. For \(x\in(a,b)\), the boundary values \(m(x):=\lim_{\eta\downarrow 0}m(x+i\eta)\) recover the Marchenko-Pastur density through
\[
f_{\mathrm{MP}}(x)=\frac{1}{\pi}\Im m(x)
=\frac{1}{2\pi\sx^2}\frac{\sqrt{(b-x)(x-a)}}{ c x}
\]
and when \(c>1\), \(F_{\mathrm{MP}}\) has an additional atom at \(0\) of mass \(1-\frac{1}{c}\). 
We will also use the companion transform
\begin{equation}
\label{eq:comp_st_1}
\underline m(z):=-\frac{1-c}{z}+c\,m(z),
\end{equation}
which is the Stieltjes transform of the companion Marchenko-Pastur law, and arises as the limiting Stieltjes transform of the Gram matrix $(\bA\bA^\top)/n$. It is characterized as the unique solution in \(\mathbb C^+\) of
\[
z=-\frac{1}{\underline m(z)}+\frac{c\sx^2}{1+\sx^2\, \underline m(z)},
\]
or equivalently,
\begin{equation}\label{eq: fixed point of companion}
    z\sx^2\,\underline m(z)^2+\bigl(z-\sx^2(c-1)\bigr)\underline m(z)+1=0.
\end{equation}
The following identity will be used in our proofs: 
\begin{lemma}
\label{lemma: companion bash}
Let \(\underline m\) be the companion Stieltjes transform of the Marchenko-Pastur law with aspect ratio \(c\) and scale $\sx$. For \(x\in\mathrm{int}(\cS_c)\), define the boundary value
\[
\underline m(x):=\lim_{\eta\downarrow 0}\underline m(x+i\eta),
\]
so that \(\Im \underline m(x)>0\). Then $|\underline m(x)|^2 = 1/(\sx^2 x)$. 
\end{lemma}

\begin{proof}
Let \(u:=\underline m(x)\). Evaluating \eqref{eq: fixed point of companion} at \(z=x\) gives
\[
x\sx^2\,u^2+\bigl(x-\sx^2(c-1)\bigr)u+1 = 0.
\]
Since \(x\in(\sx^2(1-\sqrt c)^2,\sx^2(1+\sqrt c)^2)\), we have \(\Im \underline m(x) > 0\), so by equation~\eqref{eq:comp_st_1} we have \(\Im u > 0\).
We deduce that \(u\notin\mathbb R\), and hence \(\bar u\) is the second root of the same quadratic. By Vieta's formula, $u\bar u = 1/(\sx^2 x)$, which proves the claim.
\end{proof}

\subsection{Limiting Spectral Distribution of the Spiked Covariance} 
\label{section: spectral_dis_sigma}
Recall that (Assumption \ref{assm:dgp}) we assume that $\cov(X) = \bSigma$ is a spiked covariance matrix of $\rk$ many spikes: 
$$
{\boldsymbol{\Sigma}} = \sx^2\mathbf{I}_p + \sum_{j=1}^\rk \delta_j \bv_j \bv_j^\top \,.
$$
The following proposition establishes that the LSD of $\bSigma$, as $p \uparrow \infty$, and the LSD of the sample covariance matrix $\widehat \bSigma = \bX^\top \bX/n$ (as $n, p \uparrow \infty$ and $p/n \to c$) does not depend on the number and the direction of the spikes:  
\begin{proposition}\label{prop: esd MP}
Let ${\boldsymbol{\Sigma}} = \sx^2\mathbf{I}_p + \sum_{j=1}^\rk \delta_j \bv_j \bv_j^\top$ be a rank $\rk$ perturbation of the isotropic model. Let $\mathfrak d_1,\dots,\mathfrak d_p$ be the eigenvalues of ${\boldsymbol{\Sigma}}$ and $d_1, \dots, d_p$ be the eigenvalues of $\widehat \bSigma$. Define the empirical spectral distribution (ESD) $\widehat H_n(x)$ of $\bSigma$ and ESD $\widehat F_{\widehat{\boldsymbol{\Sigma}}}$ of $\widehat \bSigma$ as: 
\[
\widehat H_n(x):=\frac{1}{p}\sum_{i=1}^p \mathbf{1}\{\mathfrak d_i\le x\}, \qquad \widehat F_{\widehat{\boldsymbol{\Sigma}}} :=\frac{1}{p}\sum_{i=1}^p \mathbf{1}\{d_i\le x\}.
\]
Then $\widehat H_n \xrightarrow{d} \mathbf{1}\{\sx^2\leq x\}$, and $F_{\widehat {\boldsymbol{\Sigma}}} \xrightarrow[]{d} F_{\mathrm{MP}}$ almost surely in the weak topology. Equivalently, for every bounded continuous function $f$,
\[
\int f\,d\widehat H_n \ \xrightarrow{n \to \infty}\ f(\sx^2) \qquad \text{a.s.}
\]
\[
\int f\,d \widehat F_{\widehat{\boldsymbol{\Sigma}}} \ \xrightarrow{n \to \infty}\int f\,d F_{\mathrm{MP}} \qquad \text{a.s.}
\]
\end{proposition}

\begin{proof}
By definition, the eigenvalues of ${\boldsymbol{\Sigma}}$ are
\[
\mathfrak d_j = \sx^2+\delta_j,\quad j=1,\dots,\rk,
\qquad
\mathfrak d_{\rk+1}=\cdots=\mathfrak d_p=\sx^2.
\]
For any bounded continuous \(f\),
$$
\int f(x)\,d\widehat H_n(x)
= \frac1p\sum_{i=1}^p f(\mathfrak{d}_i)
= \frac1p\sum_{j=1}^{\rk} f(\sx^2+\delta_j) + \left(1-\frac{\rk}{p}\right)f(\sx^2).
$$
Thus equivalently, $$
\int f(x)\,d\widehat H_n(x) - f(\sx^2)
=  \frac1p\sum_{j=1}^{\rk} f(\sx^2+\delta_j) - \frac{\rk}{p}f(\sx^2).
$$
Since \(\rk/p\to 0\), and $f$ is bounded, then
\(\int f\,d\widehat H_n \to f(\sx^2)\) i.e., the ESD is asymptotically the same as that of \(\sx^2\mathbf{I}_p\).  In particular, we can use the Marchenko-Pastur theorem \cite{MarchenkoPastur1967, Silverstein1995} to conclude that 
$
\widehat F_{\widehat {\boldsymbol{\Sigma}}} \xrightarrow[]{d} F_{\mathrm{MP}}.
$
\end{proof}

\subsection{Deterministic Equivalents}
\label{section: det equivalents}
In this section, we recall the notion of 
\emph{deterministic equivalents}, where we approximate a random matrix by a deterministic one, which is asymptotically close to the random matrix when tested against a broad class of matrices (see Chapter 2 of \cite{couillet2022random} for a detailed discussion). This perspective is very common in modern random matrix theory and its applications to statistical problems; see, e.g., \cite{RubioMestre2011SpectralConvergence, knowles2017anisotropic, deterministic_calculus, patil2023baggingoverparameterizedlearningrisk, hastie2022surprises}. 

\begin{definition}[Deterministic equivalence]
\label{def:DE}
Let $\bA_n$ and $\bB_n$ be (possibly non-symmetric and random) matrix sequences of growing dimension. We write $\bA_n \asymp \bB_n$ if for every deterministic sequence of test matrices $\boldsymbol{\bTheta}_n$ with uniformly bounded trace norm $\|\boldsymbol{\bTheta}_n\|_\ast := \tr(\sqrt{\bTheta^\top\bTheta})$  ,
\[
\tr\big(\boldsymbol{\bTheta}_n(\bA_n-\bB_n)\big)\longrightarrow 0
\qquad\text{(almost surely, as }n\to\infty\text{)}.
\]
\end{definition}
\noindent
Under Assumption~\ref{assm:dgp}, we can use Theorem 1 in \cite{RubioMestre2011SpectralConvergence}, to infer that
\begin{equation}\label{eq:determinsitic_equivalent formula}
\big(\widehat{{\boldsymbol{\Sigma}}}-zI\big)^{-1} \asymp \big(a_n {\boldsymbol{\Sigma}} - z\mathbf{I}_p\big)^{-1},
\qquad \text{for } z \in \mathbb{C}\setminus \mathbb{R}_+,
\end{equation}
where $a_n := a_n(z)$ is the unique solution of the fixed-point equation:
\begin{equation}\label{eq:fixed_point a}
\frac{n}{p}\left(\frac{1}{a_n}-1\right)
= \frac{1}{p}\,\tr\left[{\boldsymbol{\Sigma}}\big(a_n{\boldsymbol{\Sigma}} - z\mathbf{I}_p\big)^{-1}\right].
\end{equation}
The following Lemma characterizes the limit of $a_n$ when ${\boldsymbol{\Sigma}}$ is an $\rk$-spiked model. This will be very useful when analyzing deterministic equivalents. 

\begin{lemma}\label{lemma: a(z) = z + 1/m(z)}
Consider ${\boldsymbol{\Sigma}} = \sx^2\mathbf{I}_p + \sum_{j=1}^\rk \delta_j \bv_j \bv_j^\top$ to be an $\rk$-spiked covariance. For all $z \in \mathbb{C} \setminus\mathbb{R}_+$ let $a_n(z)$ be the unique solution of equation~\ref{eq:fixed_point a}. Then, there is a finite limit $a(z):= \lim_{n} a_n(z)$ and $a(z)$ satisfies $$a(z) =  \frac{z}{\sx^2} + \frac{1}{\sx^2\,m(z)} = -z\underline m(z).$$
\end{lemma}
Finally, the following property will be useful in the computation of the deterministic equivalent of the spiked covariance matrix: 
\begin{proposition}
\label{lem:resolvent-spiked}
Let ${\boldsymbol{\Sigma}} = \sx^2\mathbf{I}_p + \sum_{j=1}^\rk \delta_j \bv_j \bv_j^\top$ be an $\rk$-spiked covariance.
\[
(a_n{\boldsymbol{\Sigma}} - z\mathbf{I}_p)^{-1}
=
\frac{1}{\sx^2\,a_n - z}\left[\mathbf{I}_p - \sum_{j=1}^{\rk} \frac{\delta_j\,a_n}{\sx^2\,a_n - z + \delta_j\,a_n}\, \bv_j \bv_j^\top\right].
\]
\end{proposition}

\subsubsection{Proof of Lemma~\ref{lemma: a(z) = z + 1/m(z)}}

\begin{proof}
In the following, we write $x_n\approx y_n$ if $x_n-y_n\to 0$ almost surely. Using the deterministic equivalent relation~\ref{eq:determinsitic_equivalent formula} for $\boldsymbol{\bTheta}_n = \mathbf{I}_p/p$ we deduce that $$\frac{1}{p}\tr(\big(\widehat{{\boldsymbol{\Sigma}}}-z\mathbf{I}_p\big)^{-1}) \approx \frac{1}{p}\tr\big(a_n {\boldsymbol{\Sigma}} - z\mathbf{I}_p\big)^{-1}$$
    Note that the left hand side is $$m_{\widehat H_n}(z) \xrightarrow[]{a.s.} m(z)$$ since $\widehat H_{n}\xrightarrow[]{d} F_{\mathrm{MP}}$ as proved in Proposition~\ref{prop: esd MP}.
    Thus we deduce that $$m(z) = \lim_{n}\frac{1}{p}\tr\big(a_n {\boldsymbol{\Sigma}} - z\mathbf{I}_p\big)^{-1} = \lim_{n} \left[\sum_{j=1}^\rk \frac{1}{p} \cdot \frac{1}{a_n(\delta_j+\sx^2) - z} + \frac{p-\rk}{p}\cdot \frac{1}{\sx^2\,a_n - z}\right]$$
Suppose for the sake of contradiction that $\limsup_{n\to\infty}|a_n|=+\infty$ and choose a subsequence along which
$|a_n|\to\infty$. Then each summand in the right hand side converges to 0
hence the entire right-hand side converges to $0$ along that subsequence.
This contradicts the equality since $m(z)$ is finite and nonzero for $z\in\mathbb C\setminus\mathbb R_+$. Therefore $\{a_n(z)\}_n$ must be bounded. By Bolzano-Weierstrass, any subsequence has a further subsequence $a_{n_k}\to a$.
Passing to the limit in the equality below, the first $\rk$ terms are negligible and thus
\[
m(z)=\frac{1}{\sx^2\, a-z}.
\]
Consequently, $a$ must equal the unique solution of $m(z)=1/(\sx^2\,a-z)$, namely
\[
a(z)=\frac{z}{\sx^2}+\frac{1}{\sx^2\,m(z)}.
\]
Hence, every convergent subsequence of $\{a_n(z)\}$ has the same limit $a(z)$, and therefore we can conclude that
$a_n(z)\to a(z)$. 
The second equality follows from the relation between the Stieltjes transform and the companion Stieltjes transform (Equation \eqref{eq:comp_st_1}): 
$$
{z\underline{m}(z) + 1} = z\left(\underline{m}(z) + \frac{1}{z}\right) = z \left(cm(z) + \frac{c}{z}\right) = zcm(z) + c = 1 - \frac{z}{\sx^2} - \frac{1}{\sx^2\,m(z)} \,,
$$
which immediately implies: 
$$
-z \underline{m}(z) = \frac{z}{\sx^2} + \frac{1}{\sx^2\,m(z)} = a(z) \,.
$$
This completes the proof. 
\end{proof}

\subsubsection{Proof of Proposition~\ref{lem:resolvent-spiked} }
\begin{proof}
Since $\{\bv_1,\dots,\bv_{\rk}\}$ is orthonormal, we can extend this set by $p - \rk$ new vectors $\{\mathfrak u_1, \dots, \mathfrak u_{p-\rk}\}$ to an orthonormal basis of $\R^p$. In particular, we have 
\begin{align}
\label{eq: I_p = sum of spikes + rest}
    \mathbf{I}_p = \sum_{j=1}^\rk \bv_j \bv_j^\top + \sum_{i=1}^{p-\rk} \mathfrak{u}_i \mathfrak u_i^\top
\end{align}
We deduce that ${\boldsymbol{\Sigma}}$ is diagonal in this basis 
\begin{align*}
{\boldsymbol{\Sigma}} = \sx^2\mathbf{I}_p + \sum_{j=1}^\rk \delta_j \bv_j \bv_j^\top & = \left(\sx^2\mathbf{I}_p - \sx^2\sum_{j=1}^\rk \bv_j \bv_j^\top\right) + \sum_{j=1}^\rk (\delta_j + \sx^2)\bv_j \bv_j^\top \\
& =  \sx^2\sum_{i=1}^{p-\rk} \mathfrak{u}_i \mathfrak u_i^\top  + \sum_{j=1}^\rk (\delta_j + \sx^2)\bv_j \bv_j^\top \,,
\end{align*}
and thus \begin{align*}(a_n{\boldsymbol{\Sigma}} - z\mathbf{I}_p)^{-1} &= \sum_{i=1}^{p - \rk} \frac{1}{\sx^2\,a_n - z}\mathfrak{u}_i \mathfrak u_i^\top + \sum_{j=1}^{\rk}\frac{1}{a_n(\delta_j + \sx^2) - z}\bv_j \bv_j^\top \\
&=\frac{1}{\sx^2\,a_n-z}\mathbf{I}_p
+\sum_{j=1}^{\rk}\left(\frac{1}{a_n(\sx^2+\delta_j)-z}-\frac{1}{\sx^2\,a_n-z}\right)\bv_j \bv_j^\top\\
&=\frac{1}{\sx^2\,a_n-z}\mathbf{I}_p
-\sum_{j=1}^{\rk}\frac{\delta_j\,a_n}{(\sx^2\,a_n-z)\bigl(\sx^2\,a_n-z+\delta_j\,a_n\bigr)}\,\bv_j \bv_j^\top\\
&=\frac{1}{\sx^2\,a_n-z}\left[\mathbf{I}_p-\sum_{j=1}^{\rk}\frac{\delta_j\,a_n}{\sx^2\,a_n-z+\delta_j\,a_n}\,\bv_j \bv_j^\top\right]
\end{align*}
This completes the proof. 
\end{proof}

\subsection{Limiting Probability Measures for One-Spike Covariance Matrix}
\label{F_delta}
In this section, we consider a simple one-spike covariance matrix, as this will lay the foundation for analyzing general $\rk$-spike covariance matrices.
We will establish properties of the limiting spectral distribution for one non-trivial empirical measure. First, let us introduce the notations.
Consider the one spike covariance
\[{\boldsymbol{\Sigma}}'=\sx^2\mathbf{I}_p+\delta\,\mathfrak v_1\mathfrak v_1^\top,
\]
where $\mathfrak v_1 \in \mathbb{R}^p$ is a unit vector. Its eigenvalues are $\mathfrak d_1'=\sx^2+\delta$ and
$\mathfrak d_2'=\cdots=\mathfrak d_p'=\sx^2$ and we can extend $\mathfrak v_1$ to an orthonormal basis $\{\mathfrak v_1,\dots,\mathfrak v_p\}$ of $\R^p$.
Next, draw $n$ i.i.d.\ vectors $\bx_1',\dots, \bx_n'$ with covariance ${\boldsymbol{\Sigma}}'$, and let
\[
\widehat{\boldsymbol{\Sigma}}'=\frac{{\bX'}^\top\bX'}{n} = \sum_{j=1}^p d_j' \, \bw'_j \bw_j'^\top
\]
be the sample covariance matrix and its spectral decomposition, where $\{d_1', \dots, d_p'\}$ and $\{\bw'_1, \dots, \bw'_p\}$ are its eigenvalues and its corresponding eigenvectors. In this setup, we prove the following lemma.

\begin{lemma}
\label{lemma:f_delta}
Let $\bSigma'=\sx^2\mathbf{I}_p + \delta\,\mathfrak v_1\mathfrak v_1^{\top}$ with $\delta > 0$ be a one-spike covariance, where $\mathfrak v_1 \in \mathbb{R}^p$ is a unit vector. Using the notation introduced above, define the
population ESD of ${\boldsymbol{\Sigma}}'$ by
\[
\widehat H_n'(x)
:=
\frac{1}{p}\sum_{i=1}^p \mathbf 1\{\mathfrak d_i' \le x\}, \qquad \widehat F_{\widehat{\boldsymbol{\Sigma}}'} = \frac{1}{p}\sum_{i=1}^p \mathbf{1}\{d_i'
\le x\}
\]
and the two weighted empirical distribution functions
\[
\widehat G_n'(x)
:=
\sum_{i=1}^p \bigl\langle \mathfrak v_1,\mathfrak v_i \bigr\rangle^2\,
\mathbf 1\{\mathfrak d_i' \le x\},
\qquad
\widehat B_n'(x)
:=
\sum_{i=1}^p \bigl\langle \mathfrak v_1,\bw'_i \bigr\rangle^2\,
\mathbf 1\{d_i' \le x\}.
\]
Then the following statements hold.

\begin{enumerate}
\item[(a)]
$\widehat H_n' \xrightarrow[]{d} H' :=\mathbf 1\{\sx^2\le x\},$ $\widehat F_{\widehat{\boldsymbol{\Sigma}}'} \xrightarrow[]{d} F_{\mathrm{MP}}$ and $
\widehat G_n' \xrightarrow[]{d} G' = \mathbf 1\{\sx^2+\delta \le x\}.
$
\item[(b)]
\(
\widehat B_n' \xrightarrow[]{d} F_\delta,
\)
where $F_\delta$ is a probability measure with support contained in
\[
\operatorname{supp}(F_\delta)\subseteq \mathcal{S}_c\cup\{0\}\cup\{x_\delta^\star\},
\qquad
x_\delta^\star=\frac{(\delta+\sx^2)(\delta+c\sx^2)}{\delta},
\]
and where $\mathcal{S}_c=\bigl[\sx^2(1-\sqrt c)^2,\sx^2(1+\sqrt c)^2\bigr]$ denotes the Marchenko-Pastur bulk.

\item[(c)]The Stieltjes transform of $F_{\delta}$ is $m_\delta$ which satisfies 
\begin{equation}\label{eq: m delta definition}
m_{\delta}(z)
=\frac{\sx^2m(z)}{\sx^2+\delta+\delta z\,m(z)} \qquad z\in\mathbb{C}^+.
\end{equation}
\item[(d)]
\begin{itemize}
    \item[(i)]
The measure $F_\delta$ is absolutely continuous on $\mathrm{int}(\mathcal{S}_c)$ with density
\[
f_\delta(x)
= \frac{c\sx^2(\delta + \sx^2)\,f_{\mathrm{MP}}(x)}
{c\sx^4 + (c+1)\delta\sx^2 + \delta^2 - \delta x},
\qquad x \in \mathrm{int}(\mathcal{S}_c).
\]
\item[(ii)] If $c>1$, then $F_\delta$ has an atom at 0 with mass
\[
F_\delta(\{0\})
=
\frac{\sx^2(c-1)}{c\sx^2+\delta}.
\]
\item[(iii)] If $\delta>\sx^2\sqrt{c}$, then $F_\delta$  has an atom at $x_\delta^\star$ with mass
\begin{equation}\label{atom at x star}
F_\delta\bigl(\{x_\delta^\star\}\bigr)=
\frac{\delta^2-c\sx^4}{\delta(\delta+c\sx^2)}.
\end{equation}
In particular, if $c \leq 1$ and $\delta \leq \sx^2\sqrt{c}$ then $F_{\delta}$ has no atoms. 
\end{itemize}
\end{enumerate}
\end{lemma}
\begin{proof}[Proof of Lemma~\ref{lemma:f_delta}(a)]
Note that the 1 spike covariance structure is a special case of the $\rk$-spike covariance. The first two limits follow now by
Proposition~\ref{prop: esd MP}. Next, recall that $\{\mathfrak v_1,\dots,\mathfrak v_p\}$ is an orthonormal basis, so $\langle \mathfrak v_1, \mathfrak \bv_j\rangle = 0$ for $ j\neq 1$ so
\[
\widehat G_n'(x)
= \mathbf 1\{\sx^2 + \delta \le x\},
\]
so the sequence does not depend on $n$. The limit now follows.
\end{proof}
The proofs of (b), (c), and (d) are more involved. We group (b) and (c) together since they follow by adapting certain arguments from Lemma 1 and Lemma 3 in \cite{Green_2025}.

\subsubsection{Proof of Theorem~\ref{lemma:f_delta}(b)  and (c)}\begin{proof}
Using Lemma~1 in \cite{Green_2025} for $\widehat B_n'$ and $\bbeta^\star = \mathfrak v_1$
we deduce that for any $z\in \mathbb{C}^+$ we have $$\lim_{n\to \infty} m_{\widehat B_n' }(z) = -\frac{1}{z}\int \frac{1}{1+\tau\,\underline{m}(z)}\,dG'(\tau)$$
Note that \begin{align}
-\frac{1}{z}\int \frac{1}{1+\tau\,\underline{m}(z)}\,dG'(\tau) &= -\frac{1}{z}\,\frac{1}{1+(1+\delta)\underline{m}(z)} =\frac{1}{-z\underline{m}(z)(1+\delta)-z} \notag \\
&=\frac{1}{\sx^2\,a(z)-z+\delta\,a(z)} := m_{\delta}(z) \label{eq: m_delta in terms of a(z)}
\end{align}
where the last equality is true by Lemma~\ref{lemma: a(z) = z + 1/m(z)}. Using Lemma~3 in
\cite{Green_2025}, $m_\delta$ is the Stieltjes transform of a probability measure $F_\delta$ whose support satisfies
\[
\mathrm{supp}(F_\delta)\subseteq \mathcal{S}_c\cup\{0\}\cup\{x_\delta^\star\},
\qquad
x_\delta^\star=\frac{(\delta+\sx^2)(\delta+c\sx^2)}{\delta}\]
\end{proof}

\subsubsection{Proof of Theorem~\ref{lemma:f_delta}(d)(i)}
\begin{proof}
\label{density of F_delta}
 For all $E\in \mathcal{S}_c $ and $z = E + i\eta$ we have by inversion of the Stieltjes transform $$ f_\delta(E) = \frac{1}{\pi}\lim_{\eta \to 0^+} \Im(m_\delta(E + i \eta)) = \frac{1}{\pi}\lim_{\eta \to 0^+} \Im \frac{\sx^2m(E + i \eta)}{\sx^2 + \delta + \delta z m(E + i \eta)}$$
Note that $\Im(x/y) = \Im(x \overline{y}/|y|^2)$:
\begin{align*}
f_\delta(E) &= \frac{\sx^2}{\pi}\lim_{\eta \to 0^+}
\frac{\Im\bigl(m(z)(\delta + \sx^2) + \delta \bar{z} |m(z)|^2\bigr)}
{|\sx^2+\delta + \delta z m(z)|^2} \\
&= \frac{\sx^2}{\pi}\lim_{\eta \to 0^+}
\frac{(\delta + \sx^2)\Im(m(z)) - \eta \delta |m(z)|^2}
{|\sx^2+\delta + \delta z m(z)|^2} \\
&= \frac{\sx^2}{\pi}
\frac{(\delta + \sx^2)\,\Im(m(E + i0))}
{|\sx^2 + \delta + \delta E\,m(E + i0)|^2}
\end{align*}
Recall from Equation~\eqref{eq:square root rep of m(z)} that
\[
m(z) = \frac{\sx^2(1 - c) - z + \sqrt{(z-a)(z-b)}}{2cz\sx^2},
\]
where we choose the branch of the square root with a positive imaginary part; thus
\[
m(E+i0)
= \frac{\sx^2(1-c)-E}{2cE\sx^2}
+ i\,\frac{\sqrt{(b-E)(E-a)}}{2cE\sx^2}.
\]
We conclude that for $E$ in the bulk,
\begin{align*}
f_{\delta}(E)
&= \frac{\sx^2}{\pi}
\cdot\frac{(\sx^2+\delta)\,
\dfrac{\sqrt{(b-E)(E-a)}}{2cE\sx^2}}
{\left(\sx^2+\delta
+\dfrac{\delta\bigl(\sx^2(1-c)-E\bigr)}{2c\sx^2}
\right)^{2}
+\left(\dfrac{\delta}{2c\sx^2}\right)^{2}(b-E)(E-a)}
\\[6pt]
&= \frac{c\sx^2(\delta + \sx^2)\,f_{\mathrm{MP}}(E)}
{(\sx^2+\delta)(c\sx^2+\delta) - \delta E}
= \frac{c\sx^2(\delta + \sx^2)\,f_{\mathrm{MP}}(E)}
{c\sx^4 + (1+c)\delta\sx^2 + \delta^2 - \delta E}.
\end{align*}
\end{proof}
\subsubsection{Proof of Theorem~\ref{lemma:f_delta}(d)(ii)}\label{atom at 0 of F delta}
\begin{proof}
We compute the mass of the atoms. For $c < 1$ the MP law has no atom at 0, hence by inversion $$\lim_{\eta \to 0^+} - i\eta \, m(i\eta)= 0$$ and thus the same holds for $m_{\delta}$ using the definition of $m_{\delta}$ \eqref{eq: m delta definition}.

\noindent For $c > 1$ the MP law has an atom at 0, and we prove that $F_{\delta}$ have the same property. $$F_{\delta}(\{0\}) = \lim_{\eta \to 0^+} - i\eta \, m_\delta(i\eta) = \lim_{\eta \to 0^+} - i\eta \, \frac{\sx^2m(i\eta)}{\sx^2 + \delta + \delta\,i\eta\,m(i\eta)}$$
Note that by the same token $F_{\text{MP}}(\{0\}) = 1 - \frac{1}{c} = \lim_{\eta \to 0^+} - i\eta \, m(i\eta) $ thus $$F_{\delta}(\{0\}) = \frac{1 - \frac{1}{c}}{\sx^2 + \delta - \delta (1 - \frac{1}{c})} = \frac{\sx^2(c-1)}{c\sx^2+\delta}.$$
\end{proof}
\subsubsection{Proof of Theorem~\ref{lemma:f_delta}(d)(iii)}\label{atom at xstar of F_delta}
\begin{proof}
 As stated in Lemma 3 of \cite{Green_2025}, the mass of the atom of $F_{\delta}$ at $x_j^\star$ is nonzero if and only if the eigenvalue $\delta + \sx^2$ is above the BBP phase transition $\tau_{\mathrm{BBP}}$ \cite{baik2004phasetransitionlargesteigenvalue}. It is well known that $\tau_{\mathrm{BBP}}$ is defined as the fixed-point solution to the equation
$$\left(1 + c \int \frac{\tau'}{\tau_{\mathrm{BBP}}-\tau'}\,dH(\tau') \right)\tau_{\mathrm{BBP}} = \sx^2(1 + \sqrt{c})^2.$$
See for example Theorem 1.1 in \cite{baik2004phasetransitionlargesteigenvalue} or Equation 19 in \cite{Green_2025}. Recall from Theorem~\ref{lemma:f_delta}(a) that $H'(x)= \mathbf{1}\{\sx^2\leq x\}$, so the fixed point equation can be simplified to
\[
\tau_{\mathrm{BBP}}= \sx^2(1+\sqrt c).
\]
Using Lemma 3 in \cite{Green_2025} we can conclude that that if $\delta + 1 > \tau_{\mathrm{BBP}}$, or equivalently $\delta > \sqrt{c}$, the mass of the atom is
\begin{equation*}
F_\delta(\{x_j^\star\})
=
\frac{
1-c \displaystyle\int \left(\frac{\tau'}{\delta + 1-\tau'}\right)^{2}\, \mathrm{d}H'(\tau')
}{
1+c \displaystyle\int \frac{\tau'}{\delta + \sx^2-\tau'}\, \mathrm{d}H'(\tau')
}.
\end{equation*}
Since $H'(x)= \mathbf{1}\{\sx^2\leq x\}$. 
\[
\int \frac{\tau'}{\delta + \sx^2-\tau'}\,dH(\tau')
=\frac{\sx^2}{\delta},
\qquad
\int \left(\frac{\tau'}{\delta + \sx^2-\tau'}\right)^2 dH(\tau')
=\frac{\sx^4}{\delta^2},
\]
and substituting into the display above yields
\begin{align*}
F_\delta(\{x_j^\star\})
&=
\frac{1-\frac{c\sx^4}{\delta^2}}{1+\frac{c\sx^2}{\delta}}
=
\frac{\delta^2-c\sx^4}{\delta(\delta+c\sx^2)}.
\end{align*}
On the other hand, if $\delta^2\le c\sx^2$ (equivalently,  $\delta + \sx^2\le\tau_{\mathrm{BBP}}$), there is no outlier atom. This ends the proof.
\end{proof}

\subsection{Auxiliary Empirical Measures and Weak Convergence via Stieltjes Transforms}
\label{section: limits of empirical measures}
In this section, we introduce three empirical measures that will be useful for deriving limits of general $f$-estimators. We assume throughout this section that ${\boldsymbol{\Sigma}} = \sx^2\mathbf{I}_p + \sum_{j=1}^\rk \delta_j \bv_j \bv_j^\top$ is an $\rk$-spike covariance. Consider the spectral decomposition of the sample covariance
$$
\widehat{\boldsymbol{\Sigma}}=\sum_{i=1}^p d_i\, \bw_i \bw_i^\top,
$$
where $d_1,\dots,d_p$ are the eigenvalues and $\bw_1,\dots,\bw_p$ are the corresponding eigenvectors. Define
\begin{align}
\widehat{B}_n(x)
&:=\frac{1}{\|\bbeta_0\|_2^2}\sum_{i=1}^p 
\langle \bbeta_0,\bw_i\rangle^2\,\mathbf{1}\{d_i\le x\},
\label{eq: hat B_n}\\[6pt]
\widehat{C}_n^{\,j}(x)
&:= \frac{1}{\bbeta_0^\top \bv_j}\sum_{i=1}^p 
\langle \bbeta_0,\bw_i\rangle\,\langle \bv_j,\bw_i\rangle\,
\mathbf{1}\{d_i\le x\},
\label{eq: hat C_n^j}\\[6pt]
\widehat{D}_n^{\,j}(x)
&:=\sum_{i=1}^p \langle \bv_j,\bw_i\rangle^2\,
\mathbf{1}\{d_i\le x\},
\label{eq: hat D_n^j}
\end{align}
for each spike direction $\bv_j$, $j \in \{1,\dots,\rk\}$.
The following three lemmas allow us to represent these empirical measures in terms of specific spiked measures $F_{\delta}$ that we have introduced in Section~\ref{lemma:f_delta}.
\begin{lemma}
\label{lemma: f_alpha}
If ${\boldsymbol{\Sigma}}$ is an $\rk$-spike covariance and $\alpha_j \neq 0$ for all $j =1,2\dots, \rk$ then $$\widehat B_n \xrightarrow[]{d} F_\alpha := \omega_0 F_{\mathrm{MP}} + \sum_{j=1}^\rk \omega_jF_{\delta_j}$$
where $$\omega_0 = \frac{r^2 - \sum_{j=1}^{\rk}\alpha_j^2}{r^2} \qquad \omega_j = \frac{\alpha_j^2}{r^2}$$
are positive weights that sum to 1. 
\end{lemma}
The other two measures that depend on $j$ converge weakly to the same probability measure, as described in the two lemmas below.
\begin{lemma}\label{lemma: hat C nj}
  Fix $j \in \{1, 2\dots, \rk\}$. If ${\boldsymbol{\Sigma}}$ is an $\rk$-spike covariance and $ \alpha_j \neq 0$ then $\widehat C_n^j(x) \xrightarrow[]{} F_{\delta_j}$ vaguely.
\end{lemma}

\begin{lemma}\label{lemma: hat D nj}
   Fix $j \in \{1, 2\dots, \rk\}$. If ${\boldsymbol{\Sigma}}$ is an $\rk$-spike covariance then $\widehat D_n^j(x) \xrightarrow[]{d} F_{\delta_j}$.
   \end{lemma}

The proofs for Lemma~\ref{lemma: f_alpha} and Lemma~\ref{lemma: hat D nj} are standard and use the fact that for any sequence of random probability measures $\mu_n$ whenever $m_{\mu_n}(z)\to m(z)$ point-wise on $\mathbb C^+$ and $m$ is the Stieltjes transform of a probability measure $\mu$ we have $\mu_n \xrightarrow d\mu$ in weak topology. For Lemma~\ref{lemma: hat C nj}, we will use a more general result that vague convergence holds for signed measures that are uniformly bounded in total variation, using Theorem~B.9 in \cite{lifesaver}.

\subsubsection{Proof of Lemma~\ref{lemma: f_alpha}}\label{proof of quadratic f_alpha}
Note that $\widehat B_n$ is a probability distribution since the total mass is 1 and all weights are non-negative. By definition, we have
\begin{align*}
m_{\widehat B_n}(z)
&=\int \frac{1}{x-z}\,d\widehat B_n(x)
 =\frac{1}{\|\bbeta_0\|_2^2}\sum_{i=1}^p \frac{\langle \bbeta_0,\bw_i\rangle^2}{d_i-z}
=\frac{\bbeta_0^\top(\widehat{\boldsymbol{\Sigma}}-z\mathbf{I}_p)^{-1}\bbeta_0}{\|\bbeta_0\|_2^2}.
\end{align*}
To justify the deterministic equivalent for this quadratic form, we apply the deterministic equivalent property with
$
\boldsymbol{\bTheta}_n=\bbeta_0\bbeta_0^\top
$, which is bounded in trace norm since $\|\boldsymbol{\bTheta}_n\|_1=\tr(\boldsymbol{\bTheta}_n)=\|\bbeta_0\|_2^2$.
As before, we write $x_n\approx y_n$ if $x_n-y_n\to 0$ almost surely. Using the deterministic equivalent \eqref{eq:determinsitic_equivalent formula} and Proposition~\ref{lem:resolvent-spiked}  we deduce that
\begin{align*}
\frac{\bbeta_0^\top(\widehat{\boldsymbol{\Sigma}}-z\mathbf{I}_p)^{-1}\bbeta_0}{\|\bbeta_0\|_2^2}
\;\approx\;
\frac{1}{\|\bbeta_0\|_2^2}\,
\bbeta_0^\top
\frac{1}{\sx^2\,a_n-z}
\left[
\mathbf{I}_p-\sum_{j=1}^{\rk}\frac{\delta_j\,a_n}{\sx^2\,a_n-z+\delta_j\,a_n}\,\bv_j \bv_j^\top
\right]\bbeta_0,
\end{align*}
where the spike eigenvectors satisfy $\|\bv_j\|_2=1$. Expanding the quadratic form gives
\begin{align*}
&\frac{1}{\|\bbeta_0\|_2^2}\,
\bbeta_0^\top
\frac{1}{\sx^2a_n-z}
\left[
\mathbf{I}_p-\sum_{j=1}^{\rk}\frac{\delta_j\,a_n}{\sx^2\,a_n-z+\delta_j\,a_n}\,\bv_j \bv_j^\top
\right]\bbeta_0 \\
&\qquad=
\frac{1}{\sx^2\,a_n-z}\cdot\frac{1}{\|\bbeta_0\|_2^2}
\left[
\|\bbeta_0\|_2^2
-\sum_{j=1}^{\rk}\frac{\delta_j\,a_n}{\sx^2\,a_n-z+\delta_j\,a_n}\,(\bbeta_0^\top \bv_j)^2
\right].
\end{align*}
Using
\[
\frac{1}{\sx^2\,a_n-z}\cdot\frac{\delta_j\,a_n}{\sx^2\,a_n-z+\delta_j\,a_n}
=
\frac{1}{\sx^2\,a_n-z}-\frac{1}{\sx^2\,a_n-z+\delta_j\,a_n},
\]
we obtain
\begin{align*}
m_{\widehat B_n}(z)
\;\approx\;
\left(1-\sum_{j=1}^{\rk}\frac{(\bbeta_0^\top \bv_j)^2}{\|\bbeta_0\|_2^2}\right)\frac{1}{\sx^2\,a_n-z}
+\sum_{j=1}^{\rk}\frac{(\bbeta_0^\top \bv_j)^2}{\|\bbeta_0\|_2^2}\cdot\frac{1}{\sx^2\,a_n-z+\delta_j\,a_n}.
\end{align*}
Since $a_n(z)\to a(z)$ pointwise, we get the pointwise limit
\begin{align*}
m_{\widehat B_n}(z)\ \xrightarrow{\mathrm{a.s.}}\
\left(1-\sum_{j=1}^{\rk}\frac{\alpha_j^2}{r^2}\right)m(z)
+\sum_{j=1}^{\rk}\frac{\alpha_j^2}{r^2}\,m_{\delta_j}(z).
\end{align*}
Since $\widehat{B}_n$ is a probability measure, we deduce that $\widehat B_n$ converges weakly to the (unique) probability measure with the above Stieltjes transform. Note that this is nothing else but a finite mixture of 
$F_{\mathrm{MP}}$, and the $F_\delta$'s: $$F_\alpha := \omega_0 F_{\mathrm{MP}} + \sum_{j=1}^\rk \omega_jF_{\delta_j},$$
where $$\omega_0 = 1-\sum_{j=1}^{\rk}\frac{\alpha_j^2}{r^2} \qquad \omega_j = \frac{\alpha_j^2}{r^2}$$
are weights that sum to 1. By definition $\omega_j > 0$ for all $j\geq 1$. To show $\omega_0 > 0$ recall equation~\eqref{eq: I_p = sum of spikes + rest}, $$\mathbf{I}_p = \sum_{j=1}^\rk \bv_j \bv_j^\top + \sum_{i=1}^{p-\rk} \mathfrak{u}_i \mathfrak u_i^\top,$$
and note that if we multiply both sides by $\bbeta_0^\top$ to the left and $\bbeta_0$ to the right we have $$\|\bbeta_0\|^2_2 = \sum_{j=1}^\rk (\bbeta_0^\top \bv_j)^2 + \|\widetilde\bbeta_0\|_2^2,$$
which proves that $\omega_0 > 0$ after sending $n\to \infty$. We conclude that $\widehat B_n \xrightarrow[]{d} F_{\alpha}$ and $F_{\alpha}$ is a probability measure.

\subsubsection{Proof of Lemma~\ref{lemma: hat C nj}}
Fix $1\leq j \leq \rk$. This measure has total mass $1$ but is not necessarily nonnegative, so convergence requires additional care.
We start with the following uniform bound on the total variation norm.

\begin{lemma}\label{lem:tv-bound}
Fix $j$ such that $\bbeta_0^\top \bv_j\neq 0$. Then
\[
\|\widehat C_n^{\,j}\|_{\mathrm{TV}}
\;\le\;
\frac{\|\bbeta_0\|_2}{|\bbeta_0^\top \bv_j|}.
\]
In particular, $\{\widehat C_n^{\,j}\}_n$ is uniformly bounded in total variation.
\end{lemma}

\begin{proof}
For any $x$,
\begin{align*}
\|\widehat C_n^{\,j}\|_{\mathrm{TV}}
&\le \frac{1}{|\bbeta_0^\top \bv_j|}\sum_{i=1}^p |\langle \bbeta_0,\bw_i\rangle|\,|\langle \bv_j,\bw_i\rangle|
\;\le\;
\frac{1}{|\bbeta_0^\top \bv_j|}\,
\sqrt{\sum_{i=1}^p \langle \bbeta_0,\bw_i\rangle^2}\,
\sqrt{\sum_{i=1}^p \langle \bv_j,\bw_i\rangle^2} \\
&=
\frac{1}{|\bbeta_0^\top \bv_j|}\,\|\bbeta_0\|_2\,\|\bv_j\|_2
=
\frac{\|\bbeta_0\|_2}{|\bbeta_0^\top \bv_j|},
\end{align*}
using Cauchy-Schwarz and the orthonormality of $\{\bw_i\}_{i=1}^p$.
\end{proof}
We return now to the main proof. As above, we apply the deterministic equivalent \eqref{eq:determinsitic_equivalent formula} property and Proposition~\ref{lem:resolvent-spiked} 
$$
\bbeta_0^\top(\widehat{\boldsymbol{\Sigma}}-z\mathbf{I}_p)^{-1}v
=\tr((\widehat{\boldsymbol{\Sigma}}-z\mathbf{I}_p)^{-1}\boldsymbol{\bTheta}_n)
$$
with $\boldsymbol{\bTheta}_n=\bv_j\bbeta_0^\top$, which is bounded in trace norm since $\|\bv_j\|_2 = 1$ and $\|\bbeta_0\|_2$ are bounded.
Using the equivalent from \ref{section: det equivalents},
\begin{align*}
m_{\widehat C_n^j}(z) = \frac{\bbeta_0^\top(\widehat{\boldsymbol{\Sigma}}-z\mathbf{I}_p)^{-1}\bv_j}{\bbeta_0^\top \bv_j}
\;\approx\;
\frac{1}{\bbeta_0^\top \bv_j}\,
\bbeta_0^\top
\frac{1}{a_n-z}
\left[
\mathbf{I}_p-\sum_{i=1}^{\rk}\frac{\delta_i\,a_n}{\sx^2\,a_n-z+\delta_i\,a_n}\,v_iv_i^\top
\right]\bv_j.
\end{align*}
Expanding using $\|\bv_j\|_2=1$ and $\bv_j \perp v_i$ for all $i \neq j$ yields
\[
m_{\widehat C_n^j}(z)\ \approx\ \frac{1}{\sx^2\,a_n-z+\delta\,a_n}
\quad\Longrightarrow\quad
m_{\widehat C_n^j}(z)\ \xrightarrow{\mathrm{a.s.}}\ \frac{1}{\sx^2\,a-z+\delta\,a} = m_\delta(z).
\]
Hence, by Theorem~B.9 in \cite{lifesaver}, $\widehat C_n^j$ converges vaguely to $F_{\delta_j}$, the (unique) probability measure with Stieltjes transform $m_{\delta_j}$.

\subsubsection{Proof of Lemma~\ref{lemma: hat D nj}}
Note that $\widehat D_n^{\,j}$ is a probability measure since
$\widehat D_n^{\,j}(\infty)=\sum_{i=1}^p \langle \bv_j,\bw_i\rangle^2=\|\bv_j\|_2^2=1$. Its Stieltjes transform is
\begin{align*}
m_{\widehat D_n^{\,j}}(z)
&:=\int \frac{1}{x-z}\,d\widehat D_n^{\,j}(x)
=\sum_{i=1}^p \frac{\langle \bv_j,\bw_i\rangle^2}{d_i-z}
= \bv_j^\top(\widehat{\boldsymbol{\Sigma}}-z\mathbf{I}_p)^{-1}\bv_j.
\end{align*}
Mutatis mutandis, we use  the deterministic equivalent property with
$
\boldsymbol{\bTheta}_n=\bv_j \bv_j^\top,
$
which is bounded in trace norm since
$\|\boldsymbol{\bTheta}_n\|_1=\tr(\boldsymbol{\bTheta}_n)=\|\bv_j\|_2^2=1$.
Using the deterministic equivalent \eqref{eq:determinsitic_equivalent formula} property and Proposition~\ref{lem:resolvent-spiked} 
\begin{align*}
\bv_j^\top(\widehat{\boldsymbol{\Sigma}}-z\mathbf{I}_p)^{-1}\bv_j
\;\approx\;
\bv_j^\top
\frac{1}{a_n-z}
\left[
\mathbf{I}_p-\sum_{\ell=1}^{\rk}\frac{\delta_\ell a_n}{a_n-z+\delta_\ell a_n}\,\bv_\ell \bv_\ell^\top
\right]\bv_j.
\end{align*}
Expanding the quadratic form and using orthonormality $\bv_\ell^\top \bv_j=\mathbf{1}\{\ell=j\}$ gives
\begin{align*}
\bv_j^\top
\frac{1}{\sx^2\,a_n-z}
\left[
\mathbf{I}_p-\sum_{\ell=1}^{\rk}\frac{\delta_\ell \, a_n}{\sx^2\,a_n-z+\delta_\ell\,a_n}\,\bv_\ell \bv_\ell^\top
\right]\bv_j
&=
\frac{1}{\sx^2\,a_n-z}\left[
1-\frac{\delta_j\,a_n}{\sx^2\,a_n-z+\delta_j\,a_n}
\right]\\
&=
\frac{1}{\sx^2\,a_n-z+\delta_j\,a_n}\\& = m_{\delta_j}(z).\end{align*}
Hence,
\(
m_{\widehat D_n^{\,j}}(z)\ \xrightarrow{\mathrm{a.s.}}\ m_{\delta_j}(z)
\)
and because $\widehat D_n^{\,j}$ is a probability measure, we can deduce weak convergence of the underlying measures $
\widehat D_n^{\,j}\ \xrightarrow{d}\ F_{\delta_j}$.

\subsection{Radon-Nikodym Derivatives} \label{section: RN derivs}
In this section, we relate all the (limiting) probability measures we defined in the previous sections: $F_{\mathrm{MP}}, F_{\delta}$ via their Radon-Nikodym derivative w.r.t. $F_\alpha$.  
For all $j \in \{1,2,\dots \rk\}$ let $$x^\star_j = \frac{(\delta_j + \sx^2\,)(\delta_j + c\sx^2\,) }{\delta_j}$$ be the possible atom of $F_{\delta_j}$ if $\delta_j^2 > c\sx^2\,$ (see equation \eqref{atom at x star}). 
Recall that $\mathcal{S}_c = [\sx^2\,(1 - \sqrt{c})^2, \sx^2\,(1 + \sqrt{c})^2]$ is the bulk of the $F_{\mathrm{MP}}$ distribution. By Lemma~\ref{lemma:f_delta}(d)(i) we can deduce that this is also the bulk of $F_{\delta_j}$ for all $j$. Since $F_{\alpha}$ is just a mixture of the previous distributions, it must have the same bulk. We start by proving a basic lemma that relates $F_\delta$ and $F_{\mathrm{MP}}$.

\begin{lemma}\label{lemma: F_mp << F_delta}
    Let $\delta > 0$. Then $ F_{\mathrm{MP}}\ll F_{\delta}$ and moreover $$\nu_{\delta}(x):= \frac{\mathrm dF _{\text{MP}}(x)}{\mathrm dF _{\delta}(x)} = \frac{(\delta+c\sx^2)(\delta+\sx^2) - \delta x}{c\sx^2\,(\delta + \sx^2\,)} = \frac{\delta(x_\delta^\star - x)}{c\sx^2\,(\delta + \sx^2\,)},$$
    for $x \in \mathcal{S}_c \cup \{0\} \cup \{x_\delta^\star\}$ is a Radon-Nikodym derivative of $ F_{\text{MP}}$ w.r.t. $F_{\delta}$. 
    Equivalently, we have \begin{equation}\label{eq: rn deriv eq}\int \phi \, \mathrm{d}F_\mathrm{MP} = \int \phi\, \nu_\delta \, \mathrm{d}F_{\delta}
    \end{equation}
    for all test functions $\phi$ for which both integrals are finite.
\end{lemma}

Next, we assume a minimal condition on the spike strengths to avoid possible degeneracy. All results proved in this paper generalize to these edge cases via minimal changes. 
\begin{lemma}\label{lemma: distinct x_jstar}
If Assumption~\ref{assm:main}(1) holds, then the points $\{x_j^\star\}_{j=1}^\rk$ are distinct and at positive distance to the right of the bulk $\mathcal{S}_c$.
\end{lemma}

Next, we introduce some notations. Recall from Lemma~\ref{lemma: F_mp << F_delta} that for all $\delta_j > 0$ we can define $$\nu_j(x): = \frac{\mathrm dF _{\text{MP}}(x)}{\mathrm dF _{\delta_j}(x)} =\frac{\delta_j(x_j^\star - x)}{c\sx^2(\delta_j + \sx^2\,)}.$$
Further, define \begin{align}\label{eq: def nu_j}
\nu_{-i}(x) := \prod_{j\neq i}\nu_j(x) \qquad \qquad \nu(x) := 
\prod_{j}\nu_j(x).
\end{align}
\begin{remark}\label{rmk: lagrange}
The polynomials $\nu_{-i}(x)$ act as Lagrange interpolation polynomials since $\nu_j(x_i^\star) = 0$ if and only if $i= j$. It will become clear why this is instrumental in Section~\ref{app: D}.    
\end{remark}
Finally, we combine Lemma~\ref{lemma: F_mp << F_delta} with the assumptions of Lemma~\ref{lemma: distinct x_jstar} to find the Radon-Nikodym derivatives of $F_{\mathrm{MP}}$ and the $F_{\delta_j}$'s w.r.t. $F_{\alpha}$.
\begin{lemma}\label{lemma: def mu_j mu_0}
    Let ${\boldsymbol{\Sigma}} = \sx^2\,\mathbf{I}_p + \sum_{j=1}^\rk \delta_j \bv_j \bv_j^\top$ and assume that the conditions from Lemma~\ref{lemma: distinct x_jstar} hold. Also, assume that $\omega_j \neq 0$ for all $j \in \{1, 2,\dots, \rk\}$. Then $F_{\mathrm{MP}}\ll {F_\alpha}$ and for all $j\in\{1,2,\dots, \rk\}$ we have $F_{{\delta_j}}\ll {F_\alpha}.$ Moreover, 
    $$\frac{\mathrm{d}F_{\delta_j}}{\mathrm{d}F_{\alpha}} := \mu_j(x) := \frac{\nu_{-j}(x)}{\omega_0 \nu(x) + \sum_{i=1}^\rk \omega_i \nu_{-i}(x)}$$ and 
$$\frac{\mathrm{d}F_{\mathrm{MP}}}{\mathrm{d}F_{\alpha}} := \mu_0(x) := \frac{\nu(x)}{\omega_0 \nu(x) + \sum_{i=1}^\rk \omega_i \nu_{-i}(x)},$$
for all $x \in \mathcal{S}_c^+$.   
\end{lemma}

\subsubsection{Proof of Lemma~\ref{lemma: F_mp << F_delta}}
\begin{proof}
Note that by equation~\eqref{eq: rn deriv eq}, it is enough to define the Radon-Nikodym derivative only on 
$\operatorname{supp}(F_\delta) \cup \operatorname{supp}(F_\mathrm{MP})$. By Lemma~\ref{lemma:f_delta}(b), this is $\operatorname{supp}(F_\delta)\subseteq \mathcal{S}_c\cup\{0\}\cup\{x_\delta^\star\}$. 
We want to show that
$\nu_\delta$ satisfies
\[
\int \phi \, \mathrm{d}F_{\mathrm{MP}}=\int \phi\,\nu_\delta\, \mathrm{d}F_\delta
\]
for all test functions $\phi$ for which both integrals are finite.
Write the two integrals by separating the absolutely continuous and atomic parts:
\[
\int \phi \, \mathrm{d}F_{\mathrm{MP}}
=
\int_{\mathcal S_c}\phi(x)\,f_{\mathrm{MP}}(x)\,dx
+\phi(0)\,F_{\mathrm{MP}}(\{0\}),
\]
and
\[
\int \phi\,\nu_\delta\, \mathrm{d}F_\delta
=
\int_{\mathcal S_c}\phi(x)\,\nu_\delta(x)\,f_\delta(x)\,dx
+\phi(0)\,\nu_\delta(0)\,F_\delta(\{0\})
+\phi(x_\delta^\star)\,\nu_\delta(x_\delta^\star)\,F_\delta(\{x_\delta^\star\}),
\]
On $\operatorname{int}(\mathcal S_c)$, Lemma~\ref{lemma:f_delta}(d)(i) gives
\[
f_\delta(x)
= \frac{c\sx^2(\delta + \sx^2)\,f_{\mathrm{MP}}(x)}
{c\sx^4 + (c+1)\delta\sx^2 + \delta^2 - \delta x},
\qquad x\in \operatorname{int}(\mathcal S_c),
\]
hence
\[
\frac{f_{\mathrm{MP}}(x)}{f_\delta(x)}
=
\frac{(\delta+c\sx^2)(\delta+\sx^2) - \delta x}{c\sx^2\,(\delta + \sx^2\,)}=:\nu_\delta(x),
\qquad x\in \operatorname{int}(\mathcal S_c).
\]
Therefore, for this choice,
\[
\int_{\mathcal S_c}\phi(x)\,\nu_\delta(x)\,f_\delta(x)\,dx
=
\int_{\mathcal S_c}\phi(x)\,f_{\mathrm{MP}}(x)\,dx.
\]
It remains to match the contributions from the atoms.  If $c\le 1$, then $F_\delta(\{0\})= F_\mathrm{MP}(\{0\})= 0$ (Lemma~\ref{lemma:f_delta}(d)(ii)), so this holds regardless of the value of $\nu_\delta(0)$. 
If $c>1$, then an application of Lemma~\ref{lemma:f_delta}(d)(ii) yields
$F_\delta(\{0\})=\sx^2(c-1)/(c\sx^2+\delta)>0$, and $F_{\mathrm{MP}}(\{0\}) = (c-1)/c$ so we need $$\frac{F_{\mathrm{MP}}(\{0\})}{F_\delta(\{0\})} = \frac{(c-1)/c}{\sx^2(c-1)/(c\sx^2+\delta)} = \frac{c\sx^2+\delta}{c\sx^2}$$
which matches $\nu_\delta(0)$ as defined in the statement.
Finally, $F_{\mathrm{MP}}$ has no atom at $x_\delta^\star$, so we must have
$\nu_\delta(x_\delta^\star) \,F_\delta(\{x_\delta^\star\})=0$. By the formula for $\nu_\delta$,
\[
\nu_\delta(x_\delta^\star) = 
\frac{\delta(x_\delta^\star-x_\delta^\star)}{c\sx^2(\delta+\sx^2)}=0,
\]
so this holds whether or not $F_\delta$ has an atom at $x_\delta^\star$.
Putting the continuous and atomic parts together, we obtain
\[
\int \phi \, \mathrm{d}F_{\mathrm{MP}}
=
\int \phi\,\nu_\delta\, \mathrm{d}F_\delta
\]
for all test functions $\phi$ for which both integrals are finite. Hence $F_{\mathrm{MP}}\ll F_\delta$ and
$\nu_\delta=\mathrm dF _{\mathrm{MP}}/\mathrm dF _\delta$ on $\operatorname{supp}(F_\delta)$.
\end{proof}

\subsubsection{Proof of Lemma~\ref{lemma: distinct x_jstar}}
\begin{proof}
The proof is purely algebraic. Define $\varphi:(0,\infty)\to(0,\infty)$ by
\[
\varphi(x):=\frac{(x+\sx^2)(x+c\sx^2)}{x}=x+\frac{c\sx^4}{x}+\sx^2+c\sx^2.
\]
Then
\[
\varphi'(x)=1-\frac{c\sx^4}{x^2}, 
\qquad
\varphi''(x)=\frac{2c\sx^4}{x^3}>0.
\]
Hence $\varphi'$ vanishes only at $x=\sqrt c\sx^2\,$, and since $\varphi''(\sqrt c\,\sx^2)>0$, the point $x=\sqrt c\,\sx^2\,$ is the unique global minimizer of $\varphi$ on $(0,\infty)$. In particular, $\varphi$ is strictly decreasing on $(0,\sqrt c\,\sx^2)$ and strictly increasing on $(\sqrt c\,\sx^2,\infty)$, and
\[
\min_{x>0}\varphi(x)=\varphi(\sqrt c\,\sx^2\,)
=\sx^2\,(1+\sqrt c)^2.
\]
Note that by definition $x_j^\star=\varphi(\delta_j)$, so we obtain
\(
x_j^\star \ge \sx^2\,(1+\sqrt c)^2,
\)
with strict inequality whenever $\delta_j\neq \sqrt c\sx^2$. By assumption, $\delta_j^2\neq c$ for all $j$, hence $\delta_j\neq \sqrt c$ and therefore
\[
x_j^\star > \sx^2\,(1+\sqrt c)^2,
\]
i.e.\ each $x_j^\star$ lies strictly to the right of the bulk $\mathcal S_c=[\sx^2\,(1-\sqrt c)^2,\sx^2\,(1+\sqrt c)^2]$.

Next we show that the points $\{x_j^\star\}_{j=1}^\rk$ are distinct. Suppose $\varphi(\delta_i)=\varphi(\delta_j)$. Then
\[
\delta_i+\frac{c\sx^4}{\delta_i}=\delta_j+\frac{c\sx^4}{\delta_j}
\iff
(\delta_i-\delta_j)\left(1-\frac{c\sx^4}{\delta_i\delta_j}\right)=0.
\]
Thus either $\delta_i=\delta_j$ or $\delta_i\delta_j=c\sx^4\,$. The assumptions $\delta_i\neq \delta_j$ for $i\neq j$ and $\delta_i\delta_j\neq c
\sx^4$ for all $i,j$ rule out both possibilities unless $i=j$. Hence the values $x_j^\star=\varphi(\delta_j)$ are pairwise distinct.
\end{proof}

\subsubsection{Proof of Lemma~\ref{lemma: def mu_j mu_0}}
\label{nu_j(x) > 0}
\begin{proof} 
The proof is very similar to that of Lemma~\ref{lemma: F_mp << F_delta}, i.e., it is enough to show that in the bulk $\mu_j$ is the ratio of the densities of $F_{\delta_j}$ and $F_\alpha$, and the ratio of the masses of the atoms also matches the Radon-Nikodym derivative. 
Recall from Lemma~\ref{lemma: distinct x_jstar} that all $x_j^\star$ are distinct and $x_j^\star > \sx^2\,(1 + \sqrt{c})^2$. In particular, for all $x \in \mathcal S_c$ we have $x_j^\star > \sx^2\,(1 + \sqrt{c})^2 \geq x$.
Furthermore, from Lemma~\ref{lemma: F_mp << F_delta} we know that 
$$\nu_j(x) = \frac{\delta(x_j^\star - x)}{c\sx^2\,(\delta + \sx^2\,)},$$
so $\nu_j(x) > 0$ for all $x\in \mathcal{S}_c$. As a consequence $\mu_j$ and $\mu_0$ are well-defined and positive for all $x\in \mathcal{S}_c$. 
Moreover, note that by definition $$\nu_j(0) = \frac{\delta x_j^\star}{c\sx^2\,(\delta+\sx^2\,)} > 0,$$ so $\mu_j$ and $\mu_0$ are also well-defined at $x = 0$.

Fix now $j \in \{1, 2, \dots \rk\}$. We prove that $\mu_j$ is a valid Radon-Nikodym derivative for $F_{\delta_j}$ w.r.t. $F_{\alpha}$. 
Recall that
\begin{equation}
\label{eq: F_alpha = sum}
    F_{\alpha} = \omega_0 F_{\mathrm{MP}} + \sum_{i=1}^\rk \omega_i F_{\delta_i}.
\end{equation}
We start with the density for $x \in \mathcal{S}_c$. By the formula above, \begin{equation}\label{eq: f_alpha = sum}
    f_\alpha(x) = \omega_0 f_{\mathrm{MP}}(x) + \sum_{i=1}^\rk \omega_i f_{\delta_i}(x).
\end{equation}
By Lemma~\ref{lemma: F_mp << F_delta}, we know that $f_{\mathrm{MP}}(x) = \nu_i(x)f_{\delta_i}(x)$ for all $i$ so in particular for all $i$ we have $$f_{\delta_i}(x) = \frac{f_{\mathrm{MP}}(x)}{\nu_i(x)}  = \frac{\nu_j(x)f_{\delta_j}(x)}{\nu_i(x)},$$ so plugging this back in equation~\eqref{eq: f_alpha = sum} gives $$f_\alpha(x) = \omega_0 \nu_j(x)f_{\delta_j}(x) + \sum_{i =1 }^\rk \omega_i \frac{\nu_j(x)f_{\delta_j}(x)}{\nu_i(x)}.$$
Rearranging this yields that
$$\frac{f_{\delta_j}(x)}{f_\alpha(x)} = \frac{1}{\omega_0 \nu_j(x) + \sum_{i=1}^\rk \omega_i \frac{\nu_j(x)}{\nu_i(x)}} =  \mu_j(x),$$
where the last equality follows by multiplying both the denominator and numerator by $\nu_{-j}(x)$. Thus, $\mu_j(x)$ is matches the ratio in the bulk $\mathcal{S}_c$.

Next, we check the ratio at the possible atom at $x = 0$. By equation~\eqref{eq: F_alpha = sum} we deduce that $$F_{\alpha}(\{0\}) = \omega_0 F_{\mathrm{MP}}(\{0\}) + \sum_{i=1}^\rk \omega_j F_{\delta_j}(\{0\}).$$ By Lemma~\ref{lemma: F_mp << F_delta}, we know that $F_{\mathrm{MP}}(\{0\}) = \nu_i(0)F_{\delta_i}(\{0\})$ for all $i$ so in particular for all $i$ we have $$F_{\delta_i}(\{0\}) = \frac{F_{\mathrm{MP}}(\{0\})}{\nu_i(0)}  = \frac{\nu_j(0)F_{\delta_j}(\{0\})}{\nu_i(0)},$$ so plugging this back in the equation above gives $$F_{\alpha}(\{0\}) = \omega_0 \nu_j(0)F_{\delta_j}(\{0\}) + \sum_{i =1 }^\rk \omega_i \frac{\nu_j(0)F_{\delta_i}(\{0\})}{\nu_i(0)}.$$
Rearranging this yields that $$\frac{F_{\delta_j}(\{0\})}{F_{\alpha}(\{0\})} = \frac{1}{\omega_0 \nu_j(0) + \sum_{i=1}^\rk \omega_i \frac{\nu_j(0)}{\nu_i(0)}} =  \mu_j(0),$$
where the last equality follows by multiplying both the denominator and numerator by $\nu_{-j}(0)$. Thus, $\mu_j(0)$ is matches the ratio in the bulk $\mathcal{S}_c$.

Next, we check the ratio at the other possible atoms of $F_{\alpha}$, i.e., $\{x_1^\star, \dots, x_\rk^\star\}$. For $i\neq j$, $F_{\delta_j}$ has no atom at $x_i^\star$ so we need to check that $\mu_j(x_i^\star) = 0$. This is true since $\nu_{-j}(x_i^\star) = 0$ if and only if $i \neq j$.

Finally, we need to check the atom at $x_j^\star$. Using equation~\eqref{eq: F_alpha = sum} gives $$F_\alpha(\{x_j^\star\}) = \omega_j F_{\delta_j}(\{x_j^\star\}).$$
If the mass is non-zero, then the ratio is also given by $\mu_j(x_j^\star)$
as $$\mu_j(x_j^\star) = \frac{\nu_{-j}(x_j^\star)}{\omega_0 \nu(x_j^\star) + \sum_{i=1}^\rk \omega_i \nu_{-i}(x_j^\star)} = \frac{\nu_{-j}(x_j^\star)}{\omega_j \nu_{-j}(x_j^\star)} = \frac{1}{\omega_j} = \frac{F_{\delta_j}(\{x_j^\star\})}{F_\alpha(\{x_j^\star\})},$$
where we used $\omega_j \neq 0$ and $\nu_{-j}(x_j^\star) \neq 0$. We conclude that $\mu_j$ is a valid Radon-Nikodym derivative for $F_{\delta_j}$ w.r.t. $F_{\alpha}$. Mutatis mutandis,  $\mu_0$ is a valid Radon-Nikodym derivative for $F_{\mathrm{MP}}$ w.r.t. $F_{\alpha}$. 
\end{proof}

\subsection{Convergence under Weak/Vague Topology}
In this section, we prove the following technical Lemma that allows us to compute limits of general test functions $\phi \in \mathcal{F}$ using only weak/vague convergence. 
\begin{lemma}\label{lem:int-conv}
Let ${\phi} \in \mathcal F$ be a test function. Then
\[
\int \phi \, \mathrm{d}\widehat B_n \ \xrightarrow[]{a.s.}\ \int \phi \, \mathrm{d}F_{\alpha}.
\]
Similar results hold for $F_{\widehat{\boldsymbol{\Sigma}}}$, $\widehat C_n^{\,j}$ and $\widehat D_n^{\,j}$.
\end{lemma}

\begin{proof}
Let $\rk^+$ be the number of spikes 
above the BBP threshold.  Using the results in~\cite{baik2004}, we infer that
\begin{align*}d_j \xrightarrow[]{a.s.} x_j^\star,
\qquad
d_{\rk^+ + 1}\xrightarrow[]{a.s.}\sx^2\,(1 + \sqrt{c})^2,
\qquad
d_{\min\{n, p\}} \xrightarrow[]{a.s.}
\sx^2\,(1 - \sqrt{c})^2.
\end{align*}
Since $\rk^+ \leq \rk$ is finite, there exists an event $\Omega_0$ with $\mathbb P(\Omega_0)=1$ on which all the above convergences hold simultaneously. Fix $\eta$ small enough so that $\phi$ is continuous on $S_\eta:=\{x:\mathrm{dist}(x,\mathcal{S}^{+}_c)\le \eta\}$. By the convergences above, there exists $N(\omega)$ such that for all $n\ge N(\omega)$,
\[
d_j(n,\omega)\in (x_j^\star-\eta,\ x_j^\star+\eta)\quad \text{for } j\le \rk^+,
\]
and
\[
d_{\rk^++1}(n,\omega)\le \sx^2\,(1+\sqrt c)^2+\eta,\qquad
d_p(n,\omega)\ge \sx^2\,(1-\sqrt c)^2-\eta \ \ \text{if } c<1,
\]
while if $c\ge 1$ then the smallest eigenvalues lie in $[0,\eta]$ and the nonzero eigenvalues lie above $\sx^2\,(1-\sqrt c)^2-\eta$.
In particular, for all $n\ge N(\omega)$,
\[
\{d_1(n,\omega),\dots,d_p(n,\omega)\}\ \subset\ S_\eta.
\]
Since the restriction of $\phi$ to $S_\eta$ is bounded and continuous, we can choose a compactly supported continuous function $g\in C_c(\mathbb R)$ such that
\[
\psi(x)=\phi(x)\quad \text{for all }x\in S_\eta.
\]
For $n\ge N(\omega)$ we then have, because $\widehat B_n$ is supported on the eigenvalues $\{d_i\}_{i=1}^p\subset S_\eta$,
\[
\int \phi\, \mathrm{d}\widehat B_n(\omega)=\int \psi\, \mathrm{d}\widehat B_n(\omega).
\]
Moreover, the support of $F_\alpha$ is contained in $\mathcal{S}^{+}_c$, hence also in $S_\eta$, so
\[
\int \phi\, \mathrm{d}F_\alpha=\int \psi\, \mathrm{d}F_\alpha.
\]
Finally, we already know that $\widehat B_n\xrightarrow{d}F_\alpha$ almost surely, on a set $\Omega_0'$ of probability 1. Since $\psi\in C_c(\mathbb R)$, weak/vague convergence implies
\[
\int \psi\, \mathrm{d}\widehat B_n(\omega) \ \xrightarrow[]{}\ \int \psi\, \mathrm{d}F_\alpha(\omega).
\]
Combining the equalities with $\psi$ yields
\[
\int \phi\, \mathrm{d}\widehat B_n \ \xrightarrow[]{a.s.}\ \int \phi\, \mathrm{d}F_\alpha
\]
on $\Omega_0\cap \Omega_0'$, which happens with probability 1, as claimed. The same argument applies to $F_{\widehat{\boldsymbol{\Sigma}}}$, $\widehat C_n^{\,j}$ and $\widehat D_n^{\,j}$ since they are measures supported on $\{d_i\}_{i=1}^p$.
\end{proof}

\subsection{Concentration Properties}
We collect here two standard concentration results used throughout the proofs. Both are adapted from~\cite{negative_ridge_arzela_ascoli}.

\begin{lemma}
\label{lem:linear_concentration}
\textup{(\cite{negative_ridge_arzela_ascoli}, Lemma~S.4.1)}
Suppose $\beps \in \mathbb{R}^n$ satisfies Assumption~\ref{assm:dgp}, and let 
$\{\ba_n\} \subset \mathbb{R}^n$ be a sequence of random vectors independent 
of $\beps$ with $\limsup_n \|\ba_n\|^2/n < \infty$ almost surely. Then, as $n\to\infty$
\[
\frac{\ba_n^\top \beps}{n} \xrightarrow{\text{a.s.}} 0.
\]
\end{lemma}

\begin{lemma}
\label{lem:quadratic_concentration}
\textup{(\cite{negative_ridge_arzela_ascoli}, Lemma~S.4.2)}
Suppose $\beps \in \mathbb{R}^n$ satisfies Assumption~\ref{assm:dgp}, and let 
$\{\bA_n\} \subset \mathbb{R}^{n \times n}$ be a sequence of random matrices 
independent of $\beps$ with $\limsup_n \|\bA_n\|_{\mathrm{op}} < \infty$. 
Then, as $n\to\infty$
\[
\frac{\beps^\top \bA_n \beps}{n} - \frac{\sigma_{\boldsymbol{\varepsilon}}^2 \tr(\bA_n)}{n}
\xrightarrow{a.s.} 0.
\]
\end{lemma}

\section{Characterization of $f^\mathrm{pred}_\ast$ - Proof of Theorem~\ref{thm: main}}\label{app: B}
\subsection{Proof of Theorem~\ref{thm: main}} 
Using the tools introduced in Section~\ref{app: A}, we can provide an exact formula for $\rout$ for all functions $f\in \mathcal{F}$.
\begin{lemma}\label{prop: general formula}
For any $f \in \mathcal{F}$ we have
\begin{align*}
\rout \, \, \overset{a.s.}{=} \, \cB_f^* +\cV_f^* ,
\end{align*}
where 
\begin{align*}
\cB_f^* &:= 
\sx^2\,r^2
\int (1 - x f(x))^2 \, \mathrm{d}F_\alpha(x)
+
\sum_{j=1}^{\rk}
\delta_j \alpha_j^2
\left(
\int (1 - x f(x)) \, \mathrm{d}F_{\delta_j}(x)
\right)^2, \\
\cV_f^* &:= 
c\sx^2\,\sigma_{\boldsymbol{\varepsilon}}^2 \int x f^2(x)\, \mathrm{d}F_{\mathrm{MP}}(x). \\
\end{align*}
\end{lemma}
The goal now is to optimize $\rout$ with respect to $f$. 
Optimizing $\rout$ in the anisotropic case requires careful analysis, as it is a nonlinear functional of $f$ and the integrals are with respect to different measures.  
The key idea is to rewrite this as an optimization problem in a specific Hilbert Space. Let us introduce the following notations 
\begin{align}
\label{eq: define w ,g , h_j}
w(x) & :=\begin{cases}
\sx^2 r^2x + c\sx^2\,\sigma_{\boldsymbol{\varepsilon}}^2\,\mu_0(x) & \text{for } x \in \mathcal{S}^{+}_c \\
    0, &  \text{otherwise}
\end{cases} \notag \\
g(x) & :=\begin{cases}\dfrac{\sx^2\,r^2 + \sum_{j=1}^\rk \delta_j \alpha_j^2 \,\mu_j(x)}{w(x)}, & \text{for } x \in \mathcal{S}^{+}_c \\
    0, & \text{otherwise}
    \end{cases} \notag \\
h_j(x) & :=\begin{cases}\dfrac{\mu_j(x)}{w(x)},& \text{for } x \in \mathcal{S}^{+}_c \\
    0, & \text{otherwise}
    \end{cases} \qquad\qquad \text{for } j\in \{0, 1, \dots, \rk\} \,,
\end{align}
where we recall the Radon-Nikodym derivatives $\mu_j$ from Lemma~\ref{lemma: def mu_j mu_0}. Note that $w(x) > 0$ for all $x\in \mathcal{S}_c^+$, because at the non-zero atoms $\mu_0(x_j^\star) = \omega_j^{-1} > 0$ and $\mu_0(0) > 0$. The bulk positivity is trivial. In the rest of the analysis, we use the notation $\langle \cdot, \cdot \rangle_w$ to denote the $xw(x)$-weighted $L_2(F_\alpha)$ inner product, i.e., 
\begin{equation}\label{eq: define inner}
    \langle \phi, \psi \rangle_w = \int \phi(x) \psi(x) xw(x) \ \mathrm dF _{\alpha}(x).
\end{equation}
Using these notations, the following two lemmas characterize the optimizer of $\rout$: 
\begin{lemma}\label{lemma: define A and solve f_ast}
{Recall the definition of $\cF$ from Assumption~\ref{assm:F}}. 
Define the linear operator $\mathcal{A} : \mathcal{F} \to \mathcal{F}$ by
$$\mathcal{A}f = f + \sum_{j=1}^\rk \delta_j \alpha_j^2\langle f, h_j\rangle_w\, h_j.$$
Then there exist real coefficients $b^{(1)}_0, b^{(1)}_1, \dots, b_{\rk}^{(1)}$, 
with $b_0^{(1)} = \sx^2 r^2 \omega_0$, such that 
\begin{align}\label{eq: rep of f*}
f^\ast(x) = \sum_{j=0}^\rk b^{(1)}_j h_j(x) \qquad x\in\mathcal S_c^+.\end{align}
satisfies $[\mathcal{A}f^\ast](x) = g(x)$ for all $x \in \mathcal{S}_c^+$. 
Moreover, $f^\ast$ minimizes $\rout$ over all $f \in \mathcal{F}$. 
\end{lemma}
\begin{lemma}\label{lemma: f_ast is unique}
The minimizer $f_\ast^{\mathrm{pred}}$ over $\mathcal{F}$ is uniquely determined 
on $\mathcal{S}_c^+\setminus \{0\}$.
\end{lemma}
With these lemmas in hand, we can now prove Theorem~\ref{thm: main}. 
Substituting the definitions of $\mu_0, \{\mu_j\}_j$ into 
equation~\eqref{eq: rep of f*}, we obtain
$$
f_\ast^{\mathrm{pred}}(x) = \sum_{j=0}^\rk b^{(1)}_j h_j(x) 
= \frac{\sum_{j=0}^\rk b^{(1)}_j\mu_j(x)}{w(x)} 
= \frac{b_0^{(1)}\mu_0(x) + \sum_{j = 1}^\rk b^{(1)}_j\mu_j(x)}
{\sx^2 r^2x + c\sx^2\,\sigma_{\boldsymbol{\varepsilon}}^2\,\mu_0(x)}.
$$
Recall the explicit Radon-Nikodym derivatives $\mu_0$ and $\{\mu_j\}_j$ 
from Lemma~\ref{lemma: def mu_j mu_0}, which share the common denominator 
$\omega_0 \nu(x) + \sum_{i=1}^\rk \omega_i \nu_{-i}(x)$. Canceling this 
common denominator from the numerator and the denominator yields
$$
f_\ast^{\mathrm{pred}}(x)=
\frac{b_0^{(1)}\nu(x) + \sum_{j=1}^\rk b_{j}^{(1)} \nu_{-j}(x)}
{\sx^2 r^2x\!\left(\omega_0 \nu(x) + \sum_{j=1}^\rk \omega_j 
\nu_{-j}(x)\right) + c\sx^2\, \sigma_{\boldsymbol{\varepsilon}}^2 \nu(x)},
\qquad x \in \mathcal{S}_c^+,
$$
which is exactly the form stated in Theorem~\ref{thm: main}(b). The only thing left to check is that $f_\ast^{\mathrm{pred}}\in \mathcal{F}$. Note that $f_\ast^{\mathrm{pred}}$ has no poles in $\mathcal{S}_c$ because $\nu_j(x) > 0$ for all $x\in\mathcal{S}_c$ hence $\nu(x), \nu_{-j}(x) > 0$ for all $j$ (see \ref{nu_j(x) > 0}). On $\mathcal{A}_c$, 
evaluating at $x = x_j^\star$ and using Remark~\ref{rmk: lagrange} the denominator reduces to $\sx^2 r^2 x_j^\star \cdot \omega_j\,\nu_{-j}(x_j^\star) \neq 0$,
since $\omega_j \neq 0$ and $\nu_{-j}(x_j^\star) \neq 0$ by 
Remark~\ref{rmk: lagrange}. We deduce that the denominator of 
$f_\ast^{\mathrm{pred}}$ does not vanish on $\mathcal{S}_c^+$. Finally, since 
the Radon-Nikodym derivatives $\nu_j(x)$ are affine functions, both numerator 
and denominator are polynomials in $x$, hence continuous on a neighborhood of 
$\mathcal{S}_c^+$, confirming $f_\ast^{\mathrm{pred}} \in \mathcal{F}$. This 
completes the proof of Theorem~\ref{thm: main}(a) and~(b). For a proof of Theorem~\ref{thm: main}(c), we refer the reader to Lemma~\ref{lemma: polynomials P Q}, which proves a more general result.
\subsection{Proof of Lemma~\ref{prop: general formula}}\label{proof: general formula}

\begin{proof}
Using the model \(\by=\bX\bbeta_0+\boldsymbol{\varepsilon}\), it follows immediately from the definition of \(\widehat\bbeta_f\) that
\[
\widehat \bbeta_f
=
\bW f(\boldsymbol{\Lambda})\bW^\top \frac{\bX^\top \by}{n}
=
\bW f(\boldsymbol{\Lambda})\bW^\top \frac{\bX^\top \bX}{n}\bbeta_0
+
\bW f(\boldsymbol{\Lambda})\bW^\top\frac{\bX^\top \boldsymbol{\varepsilon}}{n},
\]
or equivalently \[
\widehat \bbeta_f
=
\bW f(\boldsymbol{\Lambda})\boldsymbol{\Lambda} \bW^\top\bbeta_0
+
\bW f(\boldsymbol{\Lambda})\bW^\top\frac{\bX^\top \boldsymbol{\varepsilon}}{n}.
\]
Since \(f(\boldsymbol{\Lambda})\) and \(\boldsymbol{\Lambda}\) are diagonal 
matrices, they commute, and therefore
\begin{align}\label{formula for beta_f}
\widehat{\bbeta}_f - \bbeta_0 
= -\bW\bigl(\mathbf{I}_p - \boldsymbol{\Lambda} f(\boldsymbol{\Lambda})\bigr)\bW^\top\bbeta_0 
+ \bW f(\boldsymbol{\Lambda})\bW^\top\frac{\bX^\top \boldsymbol{\varepsilon}}{n}.
\end{align}
Expanding the squared $\boldsymbol{\Sigma}$ norm as a quadratic gives
\begin{align*}
\|\widehat\bbeta_f - \bbeta_0\|_{\boldsymbol{\Sigma}}^2
&=\bbeta_0^\top \bW\bigl(\mathbf{I}_p - \boldsymbol{\Lambda} f(\boldsymbol{\Lambda})\bigr)
\bW^\top\boldsymbol{\Sigma}\bW\bigl(\mathbf{I}_p - \boldsymbol{\Lambda} f(\boldsymbol{\Lambda})\bigr)
\bW^\top\bbeta_0\\
&\quad - 2\,\bbeta_0^\top\bW\bigl(\mathbf{I}_p - \boldsymbol{\Lambda} f(\boldsymbol{\Lambda})\bigr)
\bW^\top\boldsymbol{\Sigma}\bW f(\boldsymbol{\Lambda})\bW^\top\frac{\bX^\top\boldsymbol{\varepsilon}}{n}\\
&\quad + \frac{\boldsymbol{\varepsilon}^\top\bX}{n}\bW f(\boldsymbol{\Lambda})\bW^\top
\boldsymbol{\Sigma}\bW f(\boldsymbol{\Lambda})\bW^\top\frac{\bX^\top\boldsymbol{\varepsilon}}{n}.
\end{align*}
Alternatively,
\[
\|\widehat{\bbeta}_f - \bbeta_0\|_{\boldsymbol{\Sigma}}^2
=
\mathcal{B}_{f}(\bX, \by) + \mathcal{V}_{f}(\bX, \by)+ \mathcal{E}_{f}(\bX, \by),
\]
where
\begin{align*}
\mathcal{B}_{f}(\bX,\by)
&:=
\bbeta_0^\top \bW\bigl(\mathbf{I}_p - \boldsymbol{\Lambda} f(\boldsymbol{\Lambda})\bigr)
\bW^\top\boldsymbol{\Sigma}\bW\bigl(\mathbf{I}_p - \boldsymbol{\Lambda} f(\boldsymbol{\Lambda})\bigr)
\bW^\top\bbeta_0, \notag  \\
\mathcal{V}_{f}(\bX, \by)
&:=
\frac{\sigma_\beps^2}{n^2}
\tr\left(\bX f(\widehat{\boldsymbol{\Sigma}}) {\boldsymbol{\Sigma}} f(\widehat{\boldsymbol{\Sigma}}) \bX^\top \right) = 
\frac{\sigma_\beps^2}{n}
\tr\left(\widehat{\boldsymbol{\Sigma}} f(\widehat{\boldsymbol{\Sigma}}) {\boldsymbol{\Sigma}} f(\widehat{\boldsymbol{\Sigma}}) \right), \notag \\
\mathcal{E}_{f}(\bX, \by)
&:= - 2\,\bbeta_0^\top\bW\bigl(\mathbf{I}_p - \boldsymbol{\Lambda} f(\boldsymbol{\Lambda})\bigr)
\bW^\top\boldsymbol{\Sigma}\bW f(\boldsymbol{\Lambda})\bW^\top\frac{\bX^\top\boldsymbol{\varepsilon}}{n} \\
& \qquad \qquad \qquad \qquad + \left[\frac{\beps^\top \bX f(\widehat{\boldsymbol{\Sigma}}) {\boldsymbol{\Sigma}} 
f(\widehat{\boldsymbol{\Sigma}}) \bX^\top \beps}{n^2} - \mathcal{V}_{f}(\bX, \by)\right].
\end{align*}
Using the spike decomposition
\(
{\boldsymbol{\Sigma}} = \sx^2\mathbf{I}_p + \sum_{j=1}^\rk \delta_j \bv_j \bv_j^\top,
\)
and cyclicity of trace, the bias and variance terms become
\begin{align}\label{eq: bias formula}
\mathcal{B}_{f}(\bX,\by) 
&= \sx^2\cdot\bbeta_0^\top \bW \bigl(\mathbf{I}_p - \boldsymbol{\Lambda} f(\boldsymbol{\Lambda})\bigr)^2 
\bW^\top \bbeta_0 +
\sum_{j=1}^\rk \delta_j
\left(
\bbeta_0^\top \bW \bigl(\mathbf{I}_p - \boldsymbol{\Lambda} f(\boldsymbol{\Lambda})\bigr) 
\bW^\top \bv_j
\right)^2,
\end{align}
\begin{align}\label{eq: variance formula}
    \mathcal{V}_{f}(\bX, \by)
&= \frac{\sx^2\,\sigma_{\beps}^2}{n}\tr\left( \bW \boldsymbol{\Lambda} 
f^2(\boldsymbol{\Lambda})\bW^\top\right) + \sum_{j=1}^\rk\delta_j  
\frac{\sigma_{\beps}^2}{n}\bv_j^\top \bW\boldsymbol{\Lambda} 
f^2(\boldsymbol{\Lambda})\bW^\top \bv_j.
\end{align}
We start by computing the limit of the bias term $\mathcal{B}_{f}(\bX, \by)$. Note that we can rewrite equation~\eqref{eq: bias formula} using the empirical 
measures~\eqref{eq: hat B_n} and~\eqref{eq: hat C_n^j} defined in 
Section~\ref{section: limits of empirical measures} as
\begin{align*}
    \mathcal{B}_{f}(\bX, \by) 
    & = \|\bbeta_0\|_2^2 \, \sx^2\,
      \int \left(1 - xf(x)\right)^2 \,\mathrm{d}\widehat{B}_n(x) \\
      & \qquad \qquad + \sum_{j=1}^\rk \delta_j (\bbeta_0^\top \bv_j)^2 
      \left(\int \left(1 - xf(x)\right) \mathrm{d}\widehat{C}_n^{\,j}(x)
      \right)^2
\end{align*}
for all  sufficiently large $n$, since Assumption~\ref{assm:F} guarantees 
that $|f(d_i)| < \infty$ for all eigenvalues $d_i$ of 
$\widehat{\bSigma}$ eventually almost surely.

Using Lemma~\ref{lemma: f_alpha} and Lemma~\ref{lemma: hat C nj} we know that $\widehat B_n \xrightarrow[]{d} F_{\alpha}$ and $\widehat C_n^j \xrightarrow{} F_{\delta_j}$ vaguely. Note that since we do not impose that the functions in $\mathcal{F}$ be compactly supported, extra care is needed to justify the convergence using the weak/vague topology, which is elaborated in Lemma~\ref{lem:int-conv}. Therefore, we deduce that
\begin{align}\label{limit of bias term}
\mathcal{B}_{f}(\bX, \by)  \xrightarrow[]{a.s.} \sx^2\,r^2\int (1 - xf(x))^2 \ \mathrm dF _\alpha(x) \ + \sum_{j=1}^\rk\delta_j \alpha_j^2\left(\int (1 - xf(x)) \ \mathrm dF _{\delta_j}(x)\right)^2.
\end{align}
Similarly, we can rewrite equation~\eqref{eq: variance formula} using the empirical measure~\eqref{eq: hat D_n^j} as
\begin{equation}\label{eq: variance formula empirical}
\mathcal{V}_{f}(\bX, \by) = \frac{\sx^2\,{\sigma_{\beps}^2}\, p }{n} 
\int xf^2(x) \, \mathrm{d}F_{\widehat{\boldsymbol{\Sigma}}}(x) 
+ \sum_{j=1}^\rk\delta_j \frac{{\sigma_{\beps}^2}}{n} 
\int xf^2(x)\, \mathrm{d}\widehat{D}_n^{\,j}(x)
\end{equation}
for all 
sufficiently large $n$, as in the bias term. Using Lemma~\ref{lemma:f_delta} and Lemma~\ref{lemma: hat D nj}, we know that $F_{\widehat{\boldsymbol{\Sigma}}} \xrightarrow[]{d} F_{\mathrm{MP}}$ and $\widehat D_n^j \xrightarrow{d} F_{\delta_j}$. 
Using Lemma~\ref{lem:int-conv} again we know that both integrals have finite limit as $n, p \to \infty$. Note that because of the $n$ in the denominator, all spike terms in the sum converge to 0, hence 
\begin{align}\label{limit of variance}
    \mathcal{V}_{f}(\bX, \by) \xrightarrow{\text{a.s.}} \sx^2\,{\sigma_{\beps}^2} c \int xf^2(x)\, \mathrm{d}F_{\mathrm{MP}}.
\end{align}
Finally, we prove that $$\mathcal{E}_{f}(\bX, \by)  \xrightarrow{\text{a.s.}} 0.$$
We show that each term in $\mathcal{E}_{f}(\bX, \by)$ converges to zero almost surely. Define
\[
\ba_n := 2\bX\bW f(\boldsymbol{\Lambda})\bW^\top\boldsymbol{\Sigma}
\bW\bigl(\mathbf{I}_p - \boldsymbol{\Lambda} f(\boldsymbol{\Lambda})\bigr)\bW^\top\bbeta_0 
\in \mathbb{R}^n,
\]
so that the cross term equals $-\ba_n^\top\beps/n$. Its squared Euclidean norm is
\begin{align*}
\|\ba_n\|_2^2 
&= 4\,\bbeta_0^\top\bW\bigl(\mathbf{I}_p - \boldsymbol{\Lambda} f(\boldsymbol{\Lambda})\bigr)
\bW^\top\boldsymbol{\Sigma}\bW f(\boldsymbol{\Lambda})
\underbrace{\bW^\top\bX^\top\bX\bW}_{=\, n\boldsymbol{\Lambda}}
f(\boldsymbol{\Lambda})\bW^\top\boldsymbol{\Sigma}
\bW \\
& \hspace{22em} \times \bigl(\mathbf{I}_p - \boldsymbol{\Lambda} f(\boldsymbol{\Lambda})\bigr)\bW^\top\bbeta_0\\
&= 4n\,\bbeta_0^\top\bW\bigl(\mathbf{I}_p - \boldsymbol{\Lambda} f(\boldsymbol{\Lambda})\bigr)
\bW^\top\boldsymbol{\Sigma}\bW\boldsymbol{\Lambda} f^2(\boldsymbol{\Lambda})
\bW^\top\boldsymbol{\Sigma}\bW\bigl(\mathbf{I}_p - \boldsymbol{\Lambda} f(\boldsymbol{\Lambda})\bigr)
\bW^\top\bbeta_0.
\end{align*}
Since
\[
\|\mathbf{I}_p - \boldsymbol{\Lambda} f(\boldsymbol{\Lambda})\|_{\mathrm{op}} 
= \max_{i\in [p]} |1 - d_i f(d_i)| \quad \text{and} \quad 
\|f(\boldsymbol{\Lambda})\|_{\mathrm{op}} = \max_{i\in [p]} |f(d_i)|,
\]
and since $f \in \mathcal{F}$ is continuous on a fixed neighborhood of the bulk $\mathcal{S}_c$ and 
in fixed neighborhoods around the outliers $x_j^\star$ and around 0
(see Assumption~\ref{assm:F}), it follows 
that almost surely, for all large $n$, both quantities are bounded by constants 
depending only on $f$. Since $\|\bW\|_{\mathrm{op}} = 1$ and 
$\|\boldsymbol{\Sigma}\|_{\mathrm{op}} = \sx^2\, + \max_{j\in[r]}\delta_j$, we 
conclude that, almost surely,
\[
\limsup_{n\to\infty} \frac{\|\ba_n\|_2^2}{n} < \infty,
\]
and hence Lemma~\ref{lem:linear_concentration} gives 
$\ba_n^\top\beps/n \xrightarrow{a.s.} 0$.
Similarly, define
\[
\bA_n := \frac{1}{n}\bX f(\widehat{\boldsymbol{\Sigma}}){\boldsymbol{\Sigma}}
f(\widehat{\boldsymbol{\Sigma}})\bX^\top \in \mathbb{R}^{n\times n},
\]
so that the bracketed term in $\mathcal{E}_{f,n}$ equals 
$\beps^\top \bA_n \beps / n - \sigma_\beps^2\tr(\bA_n)/n$,
since
\[
\frac{\sigma_\beps^2}{n}\tr(\bA_n) 
= \frac{\sigma_\beps^2}{n}
\tr\left(\widehat{\boldsymbol{\Sigma}}f(\widehat{\boldsymbol{\Sigma}})
{\boldsymbol{\Sigma}}f(\widehat{\boldsymbol{\Sigma}})\right) = \mathcal{V}_{f}(\bX).
\]
Its operator norm satisfies
\[
\|\bA_n\|_{\mathrm{op}} 
\leq \|\widehat{\boldsymbol{\Sigma}}\|_{\mathrm{op}}\,
\|f(\widehat{\boldsymbol{\Sigma}})\|_{\mathrm{op}}^2\,
\|{\boldsymbol{\Sigma}}\|_{\mathrm{op}},
\]
where we used $\|\bX\|_{\mathrm{op}}^2/n = \|\widehat{\boldsymbol{\Sigma}}\|_{\mathrm{op}}$.
By the same argument as above, $\limsup_n\|\bA_n\|_{\mathrm{op}}$ is bounded 
by a finite constant almost surely, and hence 
Lemma~\ref{lem:quadratic_concentration} gives
\[
\frac{\beps^\top \bX f(\widehat{\boldsymbol{\Sigma}}){\boldsymbol{\Sigma}}
f(\widehat{\boldsymbol{\Sigma}})\bX^\top\beps}{n^2} - \mathcal{V}_{f}(\bX, \by)
\xrightarrow{a.s.} 0.
\]
This completes the proof.
\end{proof}
\subsection{Proof of Lemma~\ref{lemma: define A and solve f_ast}}\label{HD is positive-definite}
    
\begin{proof}
First, we need to see how $\rout$ relates to operator $A$. Expanding everything in $\rout$ in the formula from Lemma~\ref{prop: general formula}we have:
\begin{align*}
\rout
&= \sx^2\,r^2 + \sum_{j=1}^\rk \delta_j \alpha_j^2
- 2 \sx^2\,r^2 \int x\,f(x)\, \mathrm{d}F_{\alpha}(x) \\
& + \sx^2\,r^2 \int x^2f(x)^2 \mathrm dF _{\alpha} - 2\sum_{j=1}^\rk \delta_j \alpha_j^2 \int x f(x) \, \mathrm{d}F_{\delta_j}\\
&+ \sum_{j=1}^\rk \delta_j\alpha_j^2 \left(\int x f(x) \, \mathrm{d}F_{\delta_j}\right)^2 + c\sx^2\,{\sigma_{\beps}^2} \int xf^2(x) \ \, \mathrm d F_{\rm MP}(x) \\
&= \sx^2\,r^2 + \sum_{j=1}^\rk \delta_j \alpha_j^2 + \int \Bigl(\sx^2\,r^2x^2 + c\sx^2\,\sigma_{\beps}^{2}\,x\,\mu_0(x)\Bigr)\,f(x)^2\,\mathrm d F_\alpha(x)\\
& \qquad
- 2 \int \Bigl(\sx^2\,r^2x + \sum_{j=1}^\rk \delta_j\,\alpha_j^2\, x \,\mu_j(x)\Bigr)\,f(x)\,\mathrm dF_\alpha(x) \\
& \qquad + \sum_{j=1}^\rk \delta_j\alpha_j^2 \left( \int x\, f(x)\mu_j(x)\, \mathrm{d}F_{\alpha} \right)^{2}.
\end{align*}
Using the notations we introduced in equation~\eqref{eq: define w ,g , h_j} we can rewrite the risk as
\begin{align}
\label{eq: risk f in inner form}
\rout & = \sx^2\,r^2 + \sum_{j=1}^\rk \delta_j\alpha_j^2 
+ \int x w(x)f(x)^2\,\mathrm dF _\alpha(x) \notag \\
& -  \int 2 x\, g(x)w(x) \, f(x)\,\mathrm dF _\alpha(x)
+ \sum_{j=1}^\rk\delta_j\alpha_j^2 \left( \int xw(x)h_j(x) f(x)\, \mathrm{d}F_{\alpha}(x) \right)^{2} \notag \\
&= \sx^2\,r^2 + \sum_{j=1}^\rk \delta_j\alpha_j^2
+ \langle f,f\rangle_w
- 2\,\langle f,\,g\rangle_w
+ \sum_{j=1}^\rk\delta_j\alpha_j^2\,\langle f, h_j\rangle_w^2,
\end{align}
where in the last line we use the $xw(x)$-weighted inner product, as defined in \eqref{eq: define inner}. 
Recall the definition of $A: \mathcal{F} \to \mathcal{F}$ 
$$\mathcal{A}f = f + \sum_{j=1}^\rk \delta_j \alpha_j^2\langle f, h_j\rangle_w\, h_j.$$
Thus, the function~\eqref{eq: risk f in inner form} we want to minimize is \begin{equation*}
    \|f\|_w^2 - 2\langle f,g\rangle_w + \sum_{j=1}^\rk \delta_j\alpha_j^2 \, \langle f, h_j\rangle_w^2 = \langle \mathcal{A}f, f\rangle_w - 2\langle f,g\rangle_w.
\end{equation*}
Note that $\mathcal{A}$ is self-adjoint with respect to $\langle\cdot, \cdot\rangle_w$ because $$\langle \phi, \mathcal{A}\psi \rangle_w = \langle \mathcal{A}\phi, \psi \rangle_w = \langle \phi, \psi \rangle_w + \sum_{j=1}^\rk \delta_j\alpha_j^2 \langle \phi, h_j \rangle_w\langle \psi, h_j\rangle_w,$$
and positive-semidefinite, as 
\begin{equation}\label{eq: Apsi psi}
\langle \mathcal{A}\psi , \psi\rangle_w = \|\psi\|_w^2 + \sum_{j=1}^\rk \delta_j\alpha_j^2 \langle \psi, h_j\rangle_w^2 \geq \|\psi\|_w^2.
\end{equation}
First, let us show $f^\ast$ exists.  Recall from Lemma~\ref{lemma: def mu_j mu_0} that
$$\omega_0 \mu_0(x) + \sum_{j=1}^\rk \omega_j \mu_j(x) = 1,$$ so from Proposition~\ref{eq: define w ,g , h_j} we can rewrite \begin{align}\label{eq: g is linear in h_j}
    g(x) &= \frac{\sx^2\,r^2 + \sum_{j=1}^\rk \delta_j \alpha_j^2 \mu_j(x)}{w(x)} = \frac{\sx^2\,r^2(\omega_0 \mu_0(x) + \sum_{j=1}^\rk \omega_j \mu_j(x)) + \sum_{j=1}^\rk \delta_j \alpha_j^2 \mu_j(x)}{w(x)} \notag \\
    &= \frac{\sx^2\,r^2\omega_0 \mu_0(x) + \sum_{j=1}^\rk (\delta_j + \sx^2) \alpha_j^2 \mu_j(x)}{w(x)}
    = \sx^2\,r^2\omega_0 h_0(x) + \sum_{j=1}^\rk (\delta_j + \sx^2)\alpha_j^2 h_j(x).
    \end{align}
We consider a candidate of the form $$f^\ast(x) = \sum_{j=0}^\rk b^{(1)}_j h_j(x).$$ We wish to find real coefficients $\{b^{(1)}_0, b^{(1)}_1,  \dots, b^{(1)}_\rk\}$ for which \[
\mathcal{A} f^\ast
=
\sum_{i=0}^{\rk} b^{(1)}_i h_i
+
\sum_{j=1}^{\rk}
\delta_j \alpha_j^2
\left(\sum_{i=0}^{\rk} b^{(1)}_i \langle h_i,h_j\rangle_w\right) h_j
\]
is exactly equal to $g(x)$.  Equivalently, let 
\(
\bH
\)
be the (0-indexed) $(\rk + 1) \times (\rk + 1)$ Gram matrix with $H_{ij} = \langle h_i, h_j\rangle_w$, and define
\[
\mathfrak D
:=
\mathrm{diag}\Bigl(0, \delta_1\alpha_1^2,\dots,
\delta_\rk \alpha_\rk^2\Bigr).
\]
Thus, by~\ref{eq: g is linear in h_j}, the equality $[\mathcal{A}f^\ast](x) = g(x)$ on $x\in \mathcal{S}_c^+$ is equivalent to solving 
\begin{align}\label{eq: linear system with b0_simple}
    \left(\mathbf I_{\rk+ 1} + \mathfrak D \bH\right) \begin{pmatrix}
    b^{(1)}_0\\ b^{(1)}_1 \\\vdots\\ b^{(1)}_{\rk} 
\end{pmatrix} = \begin{pmatrix}
    \sx^2\,r^2\omega_0\\ (\delta_1 + \sx^2)\alpha_1^2 \\\vdots\\ (\delta_\rk + \sx^2)\alpha_\rk^2 
\end{pmatrix}.
\end{align}
Since $\bH$ is a Gram matrix, it is positive-semidefinite, hence so is $\mathfrak D^{1/2}\bH\mathfrak D^{1/2}$. Moreover, the nonzero eigenvalues of $\mathfrak D\bH$ coincide with the
nonzero eigenvalues of $\mathfrak D^{1/2}\bH\mathfrak D^{1/2}$, because
\begin{align}\label{eq: HD is invertible}
\mathfrak D\bH
=
\mathfrak D^{1/2}\bigl(\mathfrak D^{1/2}\bH\bigr),
\qquad
\mathfrak D^{1/2}\bH\mathfrak D^{1/2}
=
\bigl(\mathfrak D^{1/2}\bH\bigr)\mathfrak D^{1/2},
\end{align}
Therefore, all eigenvalues of $\mathfrak D\bH$ are real and nonnegative, with possible additional eigenvalues equal to zero.
We conclude that $\mathbf{I}_{\rk+1} + {\mathfrak D}\bH$ is invertible, so we can find the coefficients $\{b_j^{(1)}\}_{j=0}^\rk$ that satisfy~\eqref{eq: linear system with b0_simple}. Since $\mathfrak{D}$ has 0 on the first row and column, it is straightforward to see that $b_0^{(1)} = \sx^2\,r^2\omega_0$.
\\\\
\noindent
We are now ready to prove that $f_\ast$ is the minimizer. The objective can be rewritten as \begin{equation}\label{eq: Risk out f up to prop}
    \langle \mathcal{A}f, f\rangle_w - 2\langle f,g\rangle_w = \langle \mathcal{A}f, f\rangle_w - 2\langle f, \mathcal{A}f^\ast\rangle_w = \left\langle  \mathcal{A}(f-f^\ast),  f-f^\ast \right\rangle_w  - \langle \mathcal{A}f^\ast, f^\ast\rangle_w
\end{equation}
where we used the fact that $\mathcal A$ is self-adjoint.
Since $f^\ast$ is fixed and $\mathcal{A}$ is positive-semidefinite, we conclude that the minimizer is $f_\ast^{\mathrm{pred}}= f^\ast$. The proof of Lemma~\ref{lemma: define A and solve f_ast} now follows.
\end{proof} 

\subsection{Proof of Lemma~\ref{lemma: f_ast is unique}}
To prove uniqueness, first of all observe that for any $\psi \in \cF$, $\|\psi\|_w = 0$ if and only if $\psi \equiv 0$ on $\mathcal{S}_c^+\setminus \{0\}$. This essentially follows from the definition of $\|\cdot\|_w$; note that, 
$$
\|\psi\|_w^2 = \int_{\mathcal{S}_c} \psi(x)^2 xw(x) f_\alpha(x) \, \mathrm{d}x  +  \sum_{j=1}^\rk \psi(x_j^\star)^2 w(x_j^\star) F_\alpha(\{x_j^\star\}) \,.
$$
Hence, $\|\psi\|_w = 0$ if and only if both the terms above are 0. Since $w(x) > 0$ for all $x \in \mathcal{S}_c^+$ and since $\psi$ is continuous on $\mathcal{S}_c$ we must have $\psi \equiv 0$ on $\mathcal{S}_c$. Finally, $\psi$ must also vanish at the atoms of $F_{\alpha}$, i.e., $\psi(x) = 0$ for all $x\in \mathcal{S}_c^+\setminus \{0\}$. 
Recall from equation~\eqref{eq: Apsi psi} that $\langle \mathcal{A}\psi, \psi\rangle_w \geq \|\psi\|_w^2$, so $\langle \mathcal{A}\psi, \psi\rangle_w = 0$ if and only if $\psi(x) = 0$ for all $x\in \mathcal{S}_c^+\setminus \{0\}$.
Next, recall from the proof of Lemma~\ref{lemma: define A and solve f_ast}, equation~\eqref{eq: Risk out f up to prop}, that 
$$
\rout = \left\langle  \mathcal{A}(f-f^\ast),  f-f^\ast \right\rangle_w  - \langle \mathcal{A}f^\ast, f
^\ast\rangle_w + C\,.
$$
where $C$ is some constant independent of $f$. 
Thus, any minimizer $\widetilde f$ of $\rout$ should satisfy $\langle  \mathcal{A}(\widetilde f-f^\ast),  \widetilde f-f^\ast\rangle_w  = 0$. Since $\langle \mathcal{A}\psi , \psi\rangle_w  \geq \|\psi\|_w^2$ we must have $\|\widetilde f - f^\ast\|_w^2 = 0$ and hence $\widetilde f= f^\ast$ for all points in $\mathcal{S}_c^+$.
\section{Proof of Lemma~\ref{thm: main corollary} and Lemma~\ref{lemma: sd is in mathcal F}}\label{app: C}
\subsection{Proof of Lemma~\ref{thm: main corollary}}\label{sec: lemma easy}
\begin{proof}
    Note that in the extreme case $\delta = 0$, the spiked measure $F_{\delta}$ is just the MP law $F_{\mathrm{MP}}$ (which essentially follows from the observation that $m_\delta(z) = m(z)$ at $\delta = 0$). 
Hence, when $\delta_1 = \dots = \delta_\rk = 0$, we have $F_{\delta_1} = \dots = F_{\delta_\rk} = F_{\mathrm{MP}}$ and consequently $F_{\alpha} = F_{\mathrm{MP}}$. Thus, in the isotropic case: 
\begin{align*}
\rout & \overset{a.s.}{=}  \sx^2\,r^2\int (1 - xf(x))^2 \ \mathrm dF _\mathrm{MP}(x) + c\,\sx^2\, {\sigma_{\beps}^2} \int xf^2(x) \ \mathrm dF _{\rm MP}(x)\\
&= \int \left[\sx^2\,r^2(1 - xf(x))^2 + c\sx^2\, {\sigma_{\beps}^2}  xf^2(x)\right] \ \mathrm dF _{\rm MP}\\
&= \sx^2\,r^2 + \sx^2\,\int x   \left[(r^2x + c\, {\sigma_{\beps}^2} )f(x)^2 - 2r^2 f(x) \right]\ \mathrm{d}F_{\rm MP}.
\end{align*}
Note that the term that depends on $f(x)$ is a convex quadratic uniquely minimized at $$f_\ast^{\mathrm{pred}}(x) = \frac{2r^2}{2(r^2x + c\sigma_{\beps}^2)} = \frac{1}{x + c\sigma_{\boldsymbol{\varepsilon}}^2/r^2} = \frac{1}{x + \lambda^*}.$$
Since this is a continuous function on the positive real line, $f_\ast^{\mathrm{pred}}\in \mathcal{F}$, and therefore is the minimizer of $\rout$. This is precisely the optimal tuning parameter $\lambda^*$ obtained in \cite{dicker_bernoulli_2016,dobriban2018high}.
\end{proof}

\subsection{Proof of Lemma~\ref{lemma: sd is in mathcal F}}\label{proof: sd is in mathcal F}
\begin{proof}
We prove this inductively in $t$. Note that for $t = 0$ the 
representation follows directly from the definition of the Ridge estimator with the Moore-Penrose 
pseudoinverse. Assume the equality holds for $t$, so that
\[
\widehat{\bbeta}^{(t)}_{\mathrm{SD}} 
= f_t(\widehat{\bSigma})\frac{\bX^\top\by}{n}.
\]
Let ($\lambda_{t+1}, \xi_{t+1}$) be the new hyperparameters and let 
$\widehat{f}:\mathbb{R}\to\mathbb{R}\cup\{\infty\}$ be given by
\[
\widehat{f}(x) = \frac{1}{x+\lambda_{t+1}},
\]
where we use the convention $1/0 := 0$ when $x = -\lambda_{t+1}$, 
consistent with \eqref{eq: definition of beta_f}. 
Using the 
closed-form expression \eqref{eq:SD-MP},
\begin{align*}
\widehat{\bbeta}^{(t+1)}
&= \left(\widehat{\bSigma} + \lambda_{t+1}\bI_p\right)^{\dagger}
   \frac{\bX^\top\!\left(
       (1-\xi_{t+1})\by + \xi_{t+1}\,\bX\widehat{\bbeta}^{(t)}
   \right)}{n} \\
&= (1-\xi_{t+1})
   \left(\widehat{\bSigma}+\lambda_{t+1}\bI_p\right)^{\dagger}
   \frac{\bX^\top\by}{n}
   +\xi_{t+1}
   \left(\widehat{\bSigma}+\lambda_{t+1}\bI_p\right)^{\dagger}
   \widehat{\bSigma}\,\widehat{\bbeta}^{(t)} \\
&= (1-\xi_{t+1})\,\widehat{\bbeta}_{\widehat{f}}
   \;+\;
   \xi_{t+1}\,\widehat{\bbeta}_{\widetilde{f}},
\end{align*}
where we define $\widetilde{f}$ by
\[
\widetilde{f}(x) := \frac{x\,f_t(x)}{x+\lambda_{t+1}},
\]
again with the convention $\widetilde{f}(-\lambda_{t+1}) = 0$. Recall from the induction hypothesis that
\begin{align}
\label{eq:part 1 of induction k-step}
f_t(x)
&= \frac{1-\xi_t}{x+\lambda_t}
+ \frac{\xi_t(1-\xi_{t-1})\,x}{(x+\lambda_t)(x+\lambda_{t-1})}
+ \cdots
+ \frac{\xi_t\xi_{t-1}\cdots\xi_2(1-\xi_1)\,x^{t-1}}
       {(x+\lambda_t)(x+\lambda_{t-1})\cdots(x+\lambda_1)} 
       \notag\\
&\quad
+ \frac{\xi_t\xi_{t-1}\cdots\xi_1\,x^t}
       {(x+\lambda_t)(x+\lambda_{t-1})\cdots(x+\lambda_0)}.
\end{align}
Therefore,
\begin{align}
\label{eq:part 2 of induction k-step}
& (1-\xi_{t+1})\widehat{f}(x)+\xi_{t+1}\widetilde{f}(x) \\
&= \frac{1-\xi_{t+1}}{x+\lambda_{t+1}}
  +\frac{\xi_{t+1}\,x\,f_t(x)}{x+\lambda_{t+1}} \notag\\
&= \frac{1-\xi_{t+1}}{x+\lambda_{t+1}}
+ \frac{\xi_{t+1}(1-\xi_t)\,x}
       {(x+\lambda_{t+1})(x+\lambda_t)}
+ \frac{\xi_{t+1}\xi_t(1-\xi_{t-1})\,x^{2}}
       {(x+\lambda_{t+1})(x+\lambda_t)(x+\lambda_{t-1})}
       \notag \\
&\quad + \cdots
+ \frac{\xi_{t+1}\xi_t\cdots\xi_2(1-\xi_1)\,x^{t}}
       {(x+\lambda_{t+1})(x+\lambda_t)\cdots(x+\lambda_1)}
+ \frac{\xi_{t+1}\xi_t\cdots\xi_1\,x^{t+1}}
       {(x+\lambda_{t+1})(x+\lambda_t)\cdots(x+\lambda_0)}
       \notag\\
&= f_{t+1}(x).
\end{align}
Combining \eqref{eq:part 1 of induction k-step} and 
\eqref{eq:part 2 of induction k-step}, the induction step is 
complete.
\end{proof}

\section{$\rk$-step Self Distillation - Proof of Theorem~\ref{thm:sd_optimal}(a)}\label{app: D}
\subsection{Proof of Theorem~\ref{thm:sd_optimal}(a)}
A direct implication of Lemma~\ref{lemma: sd is in mathcal F} is that any shrinkage function achieved by self-distillation is a rational function. Thus, the first step to prove that $f_\ast^{\mathrm{pred}}$ is achieved by self-distillation is to prove that it is a rational function, which is established in the following lemma: 

\begin{lemma}
\label{lemma: polynomials P Q}
There exist monic real-coefficient polynomials \(P\) and \(Q\) with \(\deg P=\rk+1\) and \(\deg Q=\rk\) such that
\[
f_\ast^{\mathrm{pred}}(x)=\frac{Q(x)}{P(x)}
\]
for all $x$. Moreover, \(P\) has \(\rk+1\) distinct nonzero real roots, exactly one of which is negative. Finally, none of the roots of \(P\) lies in \(\overline{\mathcal S_c}:= \mathcal{S}_c \cup \{0,x_1^\star, \dots, x_\rk^\star\}\).
\end{lemma}
\noindent
The first claim in the above Lemma is exactly the statement of Theorem~\ref{thm: main}(c). The next Lemma will be used to represent the ratio $f_\ast^{\mathrm{pred}} = Q/P$ in the form of Lemma~\ref{lemma: sd is in mathcal F}: 
\begin{lemma}\label{lemma: lagrange interpolation}
Let $d \geq 0$ be an integer and let $M$ be a set of $d+1$ non-zero distinct real numbers. Let $\mathfrak{g}$ be a real polynomial with $\deg \mathfrak{g} \leq d$, and let $\mathfrak{q}$ denote the coefficient of $x^d$ in $\mathfrak{g}$ (with $\mathfrak{q} = 0$ if $\deg \mathfrak{g} < d$).
Then, there exists an ordering $(\gamma_0, \dots, \gamma_d)$ of the elements of $M$ and real numbers $t_0, \dots, t_d$, depending on $\mathfrak{g}$, such that:
\begin{enumerate}
\item The following holds for all $x \in \mathbb{R}$:
\begin{equation}\label{coeff_match}
\mathfrak{g}(x) = t_0 x^d + \sum_{j=1}^d t_j\,  x^{d-j} \, \prod_{i=0}^{j-1}(x - \gamma_{i})
\end{equation}
\item For all $j \geq 0$:
\begin{equation}\label{coeff_non_zero}
t_0 + t_1 + \dots + t_j \neq \mathfrak{q} - 1
\end{equation}
\end{enumerate}
\end{lemma}
\begin{proof}[Proof of Theorem~\ref{thm:sd_optimal}(a)]
    From Lemma~\ref{lemma: polynomials P Q}  we know that $\deg P = \rk + 1$ and $P$ has $\rk + 1$ distinct non-zero real roots and $\deg Q = \rk$. Next, apply Lemma~\ref{lemma: lagrange interpolation} for $d = \rk$, $\mathfrak g = Q$ and $M$ to be the set of real roots of $P$ to exhibit an ordering $(\gamma_0, \dots, \gamma_\rk)$ of the roots of $P$ and $\rk+1$ real numbers $t_0, \dots, t_{\rk}$ such that
$$ Q(x)
= t_{\rk} (x - \gamma_{\rk-1})\cdots (x-\gamma_{0})
+ t_{\rk-1} x (x - \gamma_{\rk-2})\cdots (x-\gamma_{0})
+ \cdots
+ t_0 x^{\rk}$$
and equation~\ref{coeff_non_zero} holds, where $\mathfrak q$ is the coefficient of $x^\rk$ monomial in $Q$. Recall from Lemma~\ref{lemma: polynomials P Q} that  $Q$ is monic, hence $\mathfrak q = 1$.
Since the display above holds for all $x$, the coefficient of $x^\rk$ has to be the same, hence $1 = t_0 + t_1 + \dots + t_{\rk}$. 
Finally we choose 
\begin{equation}
\label{eq: roots are negative ridge}
    \lambda_i = - \gamma_i, \qquad \qquad \xi_i = \frac{t_0 + t_1 + \dots + t_{i-1}}{t_0 + t_1 + \dots + t_i},
\end{equation} 
which are well-defined due to equation~\eqref{coeff_non_zero}. Note that 
\begin{align*}
& \xi_\rk\dots\xi_{i+1}(1-\xi_i) \\
& = \frac{t_0 + t_1 + \dots + t_{\rk-1}}{t_0 + t_1 + \dots + t_{\rk}} \times \dots \times \frac{t_0 + t_1 + \dots + t_{i}}{t_0 + t_1 + \dots + t_{i+1}} \times \frac{t_i}{t_0 + t_1 + \dots + t_{i}} = t_i,
\end{align*}
since the terms in the product cancel alternatively, and the total sum is 1. Since $P$ is monic, then $P(x) = (x- \gamma_0)\dots (x-\gamma_\rk)$ hence 
\begin{align*}
f_\ast^{\mathrm{pred}}(x)
&= \frac{Q(x)}{(x-\gamma_0)\cdots(x-\gamma_\rk)} \\
&= \frac{
t_{\rk} (x-\gamma_{\rk-1})\cdots(x-\gamma_0)
+ t_{\rk-1}x (x-\gamma_{\rk-2})\cdots(x-\gamma_0)
+ \cdots
+ t_0 x^\rk
}{(x-\gamma_0)\cdots(x-\gamma_\rk)} \\
&= \frac{t_{\rk}}{x-\gamma_\rk}
+ \frac{t_{\rk-1}x}{(x-\gamma_\rk)(x-\gamma_{\rk-1})}
+ \cdots
+ \frac{t_0 x^\rk}{(x-\gamma_\rk)(x-\gamma_{\rk-1})\cdots(x-\gamma_0)} \\
&= \frac{1-\xi_\rk}{x+\lambda_\rk} + \frac{\xi_\rk(1-\xi_{\rk-1}) x}{(x+\lambda_\rk)(x+\lambda_{\rk-1})} + \dots  + \frac{\xi_\rk \dots \xi_1\, x^{\rk}}{(x+\lambda_\rk)(x+\lambda_{\rk-1})\dots (x+\lambda_0)}\\
&= f_\rk(x),
\end{align*}
as defined in Lemma~\ref{lemma: sd is in mathcal F}.  We conclude that the optimal shrinkage function $f_\ast^{\mathrm{pred}}$ is achieved by $\rk$ steps of self-distillation, initialized at a Ridge estimator $\widehat\bbeta^{(0)}$ with parameter $\lambda_0$. The proof that the Ridge parameters $(\lambda_0, \dots, \lambda_\rk)$ are strictly distinct, exactly $\rk$ of the $\rk+1$ Ridge parameters are negative and $-\lambda_0, \dots, -\lambda_\rk \notin \cS_c^+$ follows by Lemma~\ref{lemma: polynomials P Q}.
\end{proof}

\subsection{Proof of Lemma~\ref{lemma: polynomials P Q}}\label{proof of lemma r+1 roots}

\begin{proof}
Recall from Lemma~\ref{lemma: define A and solve f_ast} that
$$f_\ast^{\mathrm{pred}}(x)
=
\frac{\sx^2r^2\omega_0\, \nu(x) + \sum_{j=1}^\rk b_{j}^{(1)} \nu_{-j}(x)}{\sx^2 r^2x\left(\omega_0 \nu(x) + \sum_{j=1}^\rk \omega_j \nu_{-j}(x)\right) + c\sx^2\, \sigma_{\boldsymbol{\varepsilon}}^2 \nu(x)},
\qquad x \in \mathcal{S}_c^+.$$
Consider the denominator $$P^0(x) := {\sx^2 r^2x\left(\omega_0 \nu(x) + \sum_{j=1}^\rk \omega_j \nu_{-j}(x)\right) + c\sx^2\, \sigma_{\boldsymbol{\varepsilon}}^2 \nu(x)}$$ 
and the numerator $$Q^0(x) := \sx^2r^2\omega_0\, \nu(x) + \sum_{j=1}^\rk b_{j}^{(1)} \nu_{-j}(x). $$
Recall from Lemma~\ref{lemma: F_mp << F_delta} that $\nu_j(x)$ is a linear function $$\nu_j(x) = \frac{\delta_j(x_j^\star - x)}{c\sx^2\,(\delta_j + \sx^2\,)}.$$
Thus $\nu(x)$ is a degree $\rk$ polynomial and $\nu_{-j}(x)$ is a degree $\rk - 1$ polynomial for all $j$~\eqref{eq: def nu_j}.  
Hence $P^0$ is an $\rk+1$ degree polynomial with leading coefficient $$\mathfrak p := \sx^2\,r^2\omega_0 (-1)^\rk \prod_{j=1}^\rk \frac{\delta_j}{c\sx^2\,(\delta_j + \sx^2\,)} \neq 0.$$
Similarly, $Q^0$ is a degree $\rk$ polynomial with the same leading coefficient $\mathfrak p$. Thus if we define $P = P^0/\mathfrak p$ and $Q = Q^0/\mathfrak p$ then $$f_\ast^{\mathrm{pred}}= \frac{Q^0}{P^0} =\frac{Q}{P}$$
and $P,Q$ are both monic polynomials. 
Finally, we need to prove that $P$ or equivalently $P^0$ has $\rk + 1$ distinct, non-zero real roots that are not in $\overline{\mathcal{S}_c}$. Recall from the definition of $\nu_j$ that $\nu_j(x_j^\star) = 0$ and $\nu_j(x) < 0 \iff x > x_j^\star$. 
Recall from Lemma~\ref{lemma: distinct x_jstar} that $\{x_1^\star, \dots, x^\star_\rk\}$ are all distinct and larger than all numbers in $\mathcal{S}_c$. Assume without loss of generality that $x_1^\star < x_2^\star < \dots <x_\rk^\star$. Recall from Remark~\ref{rmk: lagrange} that $\nu_{-i}(x_j^\star) = 0$ for all $i\neq j$.  Thus, by definition of $P^0$ we have $$P^0(x_i^\star) = \sx^2\,r^2x_i^\star \omega_i \nu_{-i}(x_i^\star).$$ 
This means that $P^0(x_i^\star)$ has the same sign as $(-1)^{i-1}$ since $$\nu_1(x_i^\star), \dots, \nu_{i-1}(x_i^\star) < 0 < \nu_{i+1}(x_i^\star), \dots , \nu_\rk(x_i^\star).$$
Thus, the sequence $P^0(x_1^\star), P^0(x_2^\star), \dots, P^0(x_\rk^\star)$ alternates sign, so by the intermediate value theorem  there is a root $\gamma_i$ of $P^0$ such that $x_i^\star < \gamma_i < x_{i+1}^\star$. This gives us $\rk-1$ real roots and $\gamma_1,\dots, \gamma_{\rk-1} \notin \overline{\mathcal{S}_c}$.
Next, note that $\lim_{x\to \infty} P^0(x) = (-1)^\rk \infty$ and since $P^0(x_\rk^\star)$ has the sign of $(-1)^{\rk-1}$ there exists another root $\gamma_\rk > x_\rk^\star $, and thus $\gamma_\rk \notin \overline{\mathcal{S}_c}$. 
Since $P^0$ is a degree $\rk+1$ polynomial, and we found $\rk$ real roots, we conclude that $P^0$ has indeed $\rk+1$ real roots. Let the other root be $\gamma_0$ such that
$P^0(x) = \mathfrak p(x-\gamma_0)\dots (x-\gamma_\rk)$. 
By Vieta's formulae, $$P^0(0) = c\sx^2\,\sigma_{\boldsymbol{\varepsilon}}^2\nu(0) = \mathfrak p (-1)^{\rk+1} \prod_{j=0}^\rk \gamma_j = -\sx^2\,r^2\omega_0 \prod_{j=1}^\rk \frac{\delta_j}{c\sx^2\,(\delta_j + \sx^2\,)} \prod_{j=0}^\rk \gamma_j.$$
Since $\nu_j(0) > 0 $ for all $j$, we deduce that all terms in $P^0(0)$  are positive, hence $P^0(0) > 0$. Since $\gamma_1, \dots, \gamma_\rk$ are all positive, we deduce that $\gamma_0 < 0$ and hence $\gamma_0 \notin \overline{\mathcal{S}_c}$. The proof is now complete.
\end{proof}

\subsection{Proof of Lemma~\ref{lemma: lagrange interpolation}}
\begin{proof}
We choose the numbers $\gamma_k, t_k$ inductively over $k \leq  d$ such that  \begin{equation}\label{coeff_match_induction}\mathfrak g(\gamma_k) = t_{k}\gamma_k^{d-k}(\gamma_k - \gamma_{k-1})\dots (\gamma_k - \gamma_0) + t_{k-1}\gamma_k^{d-k+1}(\gamma_k - \gamma_{k-2})\dots (\gamma_k - \gamma_0) + \dots + t_0 \gamma_k^d
\end{equation}
and such that~\eqref{coeff_non_zero} is satisfied. 
The base case $k = -1$ (no pairs chosen) is trivial. Next, assume we have chosen $\gamma_0, \dots, \gamma_k$ and $t_0, \dots, t_k$ for some $-1 \leq k\leq d-1$ such that~\eqref{coeff_non_zero} and~\eqref{coeff_match_induction} are satisfied for all $0 \leq j \leq k$. 
Consider the following polynomial: 
\begin{align*}
R_k(x) & := \mathfrak g(x) - \Big[ t_0 x^d + t_1x^{d-1}(x - \gamma_0) + \dots +  t_k x^{d-k}( x- \gamma_0) \dots (x - \gamma_{k-1}) \Big] \\
& \hspace{8em}+ (t_0 + \dots + t_k - \mathfrak q + 1)x^{d - k- 1}(x- \gamma_0) \dots (x - \gamma_{k}).
\end{align*}
where for $k = -1$ the bracket term is empty. Note that the bracket term has degree at most $d$, and so does the last term. Since $\deg \mathfrak g \leq d$, we deduce that $\deg R_k \leq d$. 
Moreover, we immediately deduce that the coefficient of $x^d$ in $R_k$ is $$\mathfrak q - (t_0 + \dots + t_k) + (t_0 + \dots + t_k - \mathfrak q +1) = 1.$$
This means that $\deg R_k = d$ hence $R_k$ has at most $d$ real roots. Note that by the base case of induction, equation~\eqref{coeff_match_induction} yields that $R_k$ vanishes at $\gamma_0, \dots, \gamma_k$. Since $M$ has $d+1$ elements, we can choose $\gamma_{k+1} \neq 0$ that was not chosen before such that $R_k(\gamma_{k+1}) \neq 0$. 
Finally, let 
\begin{align*}
t_{k+1} & = \frac{\mathfrak g(\gamma_{k+1}) - \Big[ t_{k}\gamma_{k+1}^{d-k}\Pi_{j = 0}^{k-1}(\gamma_{k+1}- \gamma_j) + \dots + t_1 \gamma_{k+1}^{d-1}(\gamma_{k+1} - \gamma_0) + t_0 \gamma_{k+1}^d\Big]}{\gamma_{k+1}^{d-k-1}(\gamma_{k+1} - \gamma_k)\dots (\gamma_{k+1} - \gamma_0)},
\end{align*}
which is well-defined since the roots $\gamma_0, \dots, \gamma_{k+1}$ are distinct and non-zero. Multiplying both sides of this equality by the denominator yields equation~\eqref{coeff_match_induction} for $\gamma_{k+1}$. 
Note that 
\begin{align*}
R_k(\gamma_{k+1}) &= \mathfrak g(\gamma_{k+1}) \\
& - \Big[ t_0 \gamma_{k+1}^d + t_1\gamma_{k+1}^{d-1}(\gamma_{k+1} - \gamma_0) + \dots +  t_k \gamma_{k+1}^{d-k}( \gamma_{k+1}- \gamma_0) \dots (\gamma_{k+1} - \gamma_{k-1}) \Big] \\
& \hspace{7em}+ (t_0 + \dots + t_k - \mathfrak q + 1)\gamma_{k+1}^{d - k- 1}(\gamma_{k+1}- \gamma_0) \dots (\gamma_{k+1} - \gamma_{k}) \\
&= t_{k+1}\gamma_{k+1}^{d-k-1}(\gamma_{k+1} - \gamma_k)\dots (\gamma_{k+1} - \gamma_0) \\
& \qquad \qquad \qquad + (t_0 + \dots + t_k - \mathfrak q + 1)\gamma_{k+1}^{d - k- 1}(\gamma_{k+1}- \gamma_0) \dots (\gamma_{k+1} - \gamma_{k})\\
&=  (t_0 + \dots + t_k + t_{k+1} - \mathfrak q + 1)\gamma_{k+1}^{d - k- 1}(\gamma_{k+1}- \gamma_0) \dots (\gamma_{k+1} - \gamma_{k}).
\end{align*}
Since $R_k(\gamma_{k+1}) \neq 0$, we deduce that $t_0 + t_1 + \dots + t_{k+1} \neq \mathfrak q - 1$ and the induction step is complete.
We claim that this is a valid choice of coefficients. Consider the polynomial $$\mathfrak g(x) - \big[ t_d (x - \gamma_{d-1})\cdots (x-\gamma_{0})
+t_{d-1} x (x - \gamma_{d-2})\cdots (x-\gamma_{0})
+ \cdots
+ t_0 x^{d}\big],$$
which has degree at most $d$ since $\deg \mathfrak g \leq d$. By the construction, equation~\ref{coeff_match_induction} yields that this polynomial vanishes at $\gamma_0, \dots, \gamma_{d}$. We conclude that it must be the 0 polynomial, hence equation~\ref{coeff_match} holds. Finally, by construction, equation~\eqref{coeff_non_zero} is satisfied so the proof is complete.
\end{proof}

\section{Necessity of $\rk$ Distillation Steps - Proof of Theorem~\ref{thm:sd_optimal}(b)}
\label{app: E}
\subsection{Proof of Theorem~\ref{thm:sd_optimal}(b)}
We prove that $\rk$ distillation steps after the initial ridge estimator
are necessary for every parameter choice satisfying
Assumptions~\ref{assm:dgp} and~\ref{assm:main}.
Recall polynomials $P, Q$ from Lemma~\ref{lemma: polynomials P Q}.
First, we provide a purely algebraic condition on $P$ and $Q$ that
guarantees that $\rk$ is the smallest number of distillation steps at
which the optimum is achieved.

\begin{lemma}\label{lemma: coprime}
Recall $P, Q$ from Lemma~\ref{lemma: polynomials P Q}.
If $P$ and $Q$ are coprime (they do not share a root), then
$f_\ast^{\mathrm{pred}}$ can not be achieved by
$(\rk-k)$-step self-distillation procedure for any $1 \le k \le \rk$.
\end{lemma}
Next, we establish an identity relating the $xw(x)$-inner product of $f_\ast^{\mathrm{pred}}$ with the $h_j$'s.

\begin{lemma}\label{lem:necessity-optimality}
For $j=1,\ldots,\rk$, let 
\begin{align}\label{eq:A_j=int}
    A_j: = \langle f_\ast^{\mathrm{pred}}, h_j\rangle_w  = \int x f_\ast^{\mathrm{pred}}(x)\,dF_{\delta_j}(x).
\end{align} 
Then
\begin{equation}\label{eq:necessity-optimality}
x_j^\star f_\ast^{\mathrm{pred}}(x_j^\star)
=
1+\frac{\delta_j}{\sx^2}(1-A_j).
\end{equation}
\end{lemma}

The following identities relate the Stieltjes transforms of
$F_{\mathrm{MP}}$ and $F_{\delta_j}$.

\begin{lemma}\label{lem:necessity-stieltjes}
For $j=1,\ldots,\rk$ and $z\in\mathbb C\setminus
\bigl(
\operatorname{supp}F_{\alpha}
\cup\{x_j^\star\}
\bigr),$
we have
\begin{equation}\label{eq:necessity-stieltjes}
\frac{1}{x_j^\star-z}
+\frac{\delta_j}{\sx^2}m_{\delta_j}(z)
=
\left(1+\frac{\delta_j}{\sx^2}\right)
\frac{1+c\sx^2m(z)}{x_j^\star-z}.
\end{equation}
Moreover, for every $z\in\mathbb C\setminus
\operatorname{supp}F_{\mathrm{MP}}$, $1+c\sx^2m(z)\neq 0$.
\end{lemma}

Finally, we use the following algebraic fact.

\begin{lemma}\label{lem:necessity-cauchy}
Let $x_1^\star,\ldots,x_\rk^\star$ be pairwise distinct, and let
$\gamma_0,\ldots,\gamma_\rk$ be pairwise distinct, with
$x_j^\star\ne\gamma_k$ for all $j,k$.
If a nonzero vector $(\tau_0,\ldots,\tau_\rk)$ satisfies
\[
\sum_{k=0}^{\rk}
\frac{\tau_k}{x_j^\star-\gamma_k}=0,
\]
for all $j=1,\ldots,\rk$, then $\tau_k\ne0$ for every $k=0,\ldots,\rk$.
\end{lemma}

\begin{proof}[Proof of Theorem~\ref{thm:sd_optimal}(b)]
By Lemma~\ref{lemma: coprime}, it suffices to show that $P$ and $Q$
are coprime. Let $\gamma_0,\ldots,\gamma_\rk$ be the roots of $P$.
Lemma~\ref{lemma: polynomials P Q} shows that they are distinct,
nonzero, and outside $\overline{\mathcal S_c}$.
In particular, they lie outside the supports of $F_{\mathrm{MP}}$
and all $F_{\delta_j}$, and differ from every $x_j^\star$. Since $P,Q$ are monic of degrees $\rk+1,\rk$, respectively, we can decompose the ratio in
partial fractions as
\begin{equation}\label{eq:necessity-partial-fractions}
f_\ast^{\mathrm{pred}}(x)
=
\sum_{k=0}^{\rk}
\frac{Q(\gamma_k)}{P'(\gamma_k)}
\frac{1}{x-\gamma_k},
\end{equation}
Clearing denominators, and using the fact that $P$ and $Q$ are monic, we deduce that
$$\sum_{k=0}^{\rk}\frac{Q(\gamma_k)}{P'(\gamma_k)}=1.$$
Since $F_{\delta_j}$ is a probability measure, equation~\eqref{eq:A_j=int} gives
\begin{align*}
A_j
&=
1+\sum_{k=0}^{\rk}
\frac{Q(\gamma_k)}{P'(\gamma_k)}
\gamma_km_{\delta_j}(\gamma_k).
\end{align*}
Multiplying~\eqref{eq:necessity-partial-fractions} by $x$ and evaluating at $x = x_j^\star$, we have
\begin{align*}
x_j^\star f_\ast^{\mathrm{pred}}(x_j^\star) &= 1 + \sum_{k=0}^{\rk}
\frac{Q(\gamma_k)}{P'(\gamma_k)}
\frac{\gamma_k}{x_j^\star-\gamma_k}.
\end{align*}
Substituting into Lemma~\ref{lem:necessity-optimality} and applying
Lemma~\ref{lem:necessity-stieltjes}(a) yields
\[
\sum_{k=0}^{\rk}
\frac{
Q(\gamma_k)\gamma_k[1+c\sx^2m(\gamma_k)]
}{
P'(\gamma_k)(x_j^\star-\gamma_k)
}=0,
\]
thus if we let
\[
\tau_k
=
\frac{
Q(\gamma_k)\gamma_k[1+c\sx^2m(\gamma_k)]
}{
P'(\gamma_k)
},\]
we have
\[\sum_{k=0}^{\rk}\frac{\tau_k}{x_j^\star-\gamma_k}=0,
\]
for all $j=1,\ldots,\rk.
$
Note that the vector $(\tau_0,\ldots,\tau_\rk)$ is nonzero. Indeed, the residues in \eqref{eq:necessity-partial-fractions} sum to one,
and Lemma~\ref{lem:necessity-stieltjes}(b) gives $\gamma_k[1+c\sx^2m(\gamma_k)]\ne 0$ for every $k$.
Lemmas~\ref{lemma: polynomials P Q}
and~\ref{lemma: distinct x_jstar} guarantee that the conditions in Lemma~\ref{lem:necessity-cauchy} are satisfied, so we can conclude that every $\tau_k$ is nonzero, and hence $Q(\gamma_k)\ne0$. Thus $P,Q$ are coprime, and Lemma~\ref{lemma: coprime} proves necessity.
\end{proof}

\subsection{Proof of Lemma~\ref{lemma: coprime}}

\begin{proof}
Recall from Lemma~\ref{lemma: sd is in mathcal F} that the spectral
shrinkage function corresponding to a $t$-step SD is
\[
f_t(x)
=
\frac{1-\xi_t}{x+\lambda_t}
+
\frac{\xi_t(1-\xi_{t-1})x}
     {(x+\lambda_t)(x+\lambda_{t-1})}
+\dots+
\frac{\xi_t\dots\xi_1\,x^t}
     {(x+\lambda_t)(x+\lambda_{t-1})\dots(x+\lambda_0)}
\]
which, after raising to a common denominator, is the ratio of a
polynomial $U(x)$ and a polynomial $V(x)$ with
$\deg U\leq t$ and $\deg V=t+1$.
Recall from Lemma~\ref{lemma: polynomials P Q} that
$\deg P=\rk+1$ and $\deg Q=\rk$.
If $f_t=f_\ast^{\mathrm{pred}}$, then
\[
\frac{Q(x)}{P(x)}=\frac{U(x)}{V(x)},
\]
where we may restrict to a nonempty open interval in the bulk
avoiding $-\lambda_0,\ldots,-\lambda_t$.
hence $P(x)U(x)=Q(x)V(x)$
and since $(P,Q)=1$ then $P(x)\mid V(x)$ so
\[
t+1=\deg V(x)\geq\deg P(x)=\rk+1.
\]
This proves that one cannot achieve $f_\ast^{\mathrm{pred}}$
with less than $\rk$ rounds of self-distillation.
\end{proof}

\subsection{Proof of Lemma~\ref{lem:necessity-optimality}}

\begin{proof}
Recall $P^0,Q^0$ from
Section~\ref{proof of lemma r+1 roots}, where $Q/P=Q^0/P^0$.
By Remark~\ref{rmk: lagrange} and $r^2\omega_j=\alpha_j^2$,
\[
P^0(x_j^\star)
=
\sx^2\alpha_j^2x_j^\star\nu_{-j}(x_j^\star)\ne0,
\qquad
Q^0(x_j^\star)
=
b_j^{(1)}\nu_{-j}(x_j^\star).
\]
Expanding the numerator in the fixed-point representation from Lemma~\ref{lemma: define A and solve f_ast}, and using
$r^2\omega_j=\alpha_j^2$, gives
\[
b_j^{(1)}
=\sx^2r^2\omega_j+\delta_j\alpha_j^2(1-A_j)
=\alpha_j^2\bigl[\sx^2+\delta_j(1-A_j)\bigr],
\] therefore
\[
x_j^\star f_\ast^{\mathrm{pred}}(x_j^\star)
=
x_j^\star\frac{Q^0(x_j^\star)}{P^0(x_j^\star)}
=
\frac{b_j^{(1)}}{\sx^2\alpha_j^2}
=
1+\frac{\delta_j}{\sx^2}(1-A_j).
\]
and we are done.
\end{proof}

\subsection{Proof of Lemma~\ref{lem:necessity-stieltjes}}

\begin{proof}
Lemma~\ref{lemma: F_mp << F_delta} gives
$dF_{\mathrm{MP}}=\nu_j\,dF_{\delta_j}$.
Since $F_{\delta_j}$ is a probability measure,
\begin{align*}
m(z)
&=
\frac{\delta_j}{c\sx^2(\delta_j+\sx^2)}
\int\frac{x_j^\star-x}{x-z}\,dF_{\delta_j}(x) = \frac{\delta_j}{c\sx^2(\delta_j+\sx^2)}
\bigl[(x_j^\star-z)m_{\delta_j}(z)-1\bigr].
\end{align*}
Rearranging, we conclude that
\[
m_{\delta_j}(z)
=
\frac{
1+\dfrac{c\sx^2(\delta_j+\sx^2)}{\delta_j}m(z)
}{
x_j^\star-z
}.
\]
Multiplying by $\delta_j/\sx^2$ and adding
$(x_j^\star-z)^{-1}$ proves
\eqref{eq:necessity-stieltjes}. 

For the second part, the Marchenko-Pastur equation gives
\[
\bigl[1+c\sx^2m(z)\bigr]\bigl[1+zm(z)\bigr] = \sx^2 m(z).
\]
If $1+c\sx^2m(z)=0$, this identity implies $m(z)=0$,
a contradiction. Therefore $1+c\sx^2m(z)\ne0$ throughout
the stated domain.
\end{proof}

\subsection{Proof of Lemma~\ref{lem:necessity-cauchy}}

\begin{proof}
Consider the polynomial
\[
R(x)
=
\sum_{k=0}^{\rk}\tau_k
\prod_{\substack{0\le\ell\le\rk\\\ell\ne k}}
(x-\gamma_\ell).
\]
Choose $k$ with $\tau_k\ne0$. Since the $\gamma_\ell$ are distinct,
$R(\gamma_k)
=
\tau_k\prod_{\ell\ne k}(\gamma_k-\gamma_\ell)\ne0$, 
so $R$ is not identically zero. On the other hand, for all $j=1,\ldots,\rk$
\[
R(x_j^\star)
=
\left[
\prod_{k=0}^{\rk}(x_j^\star-\gamma_k)
\right]
\sum_{k=0}^{\rk}\frac{\tau_k}{x_j^\star-\gamma_k}
=0
\]
Since $\deg R\le\rk$ and these $\rk$ roots are distinct,
$R(x)=C\prod_{j=1}^{\rk}(x-x_j^\star)$ for some constant $C \neq 0$.
Evaluating at each $\gamma_k$ yields
\[
\tau_k\prod_{\ell\ne k}(\gamma_k-\gamma_\ell)
=
C\prod_{j=1}^{\rk}(\gamma_k-x_j^\star)\ne0.
\]
Therefore $\tau_k\ne0$ for every $k$.
\end{proof}

\section{$f_\ast^{\mathrm{pred}}$ versus other Estimators - Proofs of Theorems~\ref{thm: f_ast dominates}(a), (b) and (c)}\label{app: F}
As discussed in Section~\ref{sec:theory}, many of the common estimators used in practice correspond to or are well-approximated by simple shrinkage rules. In this section, we prove Theorem~\ref{thm: f_ast dominates} parts (a), (b) and (c) i.e., the shrinkage rule $f_\ast^{\mathrm{pred}}$ we computed in Appendix~\ref{app: B} achieves a strictly smaller loss than the min-norm interpolator, self-distillation shrinkage rules that use the same $\lambda$ for all rounds of distillation (in particular Ridge Regression) and finally, early stopped gradient descent.

\subsection{Notation}\label{notations for SVD}
We consider the SVD decomposition
\(
\bX = \bU \bar{\boldsymbol{S}} \bW^\top,
\)
where $\bU \in \mathbb{R}^{n\times p}$ has orthogonal columns, $\bW \in \mathbb{R}^{p\times p}$ is orthogonal, and
$\bar{\boldsymbol{S}} \in \mathbb{R}^{n\times p}$ is rectangular diagonal with singular values
$\bar s_1 \geq \dots \geq \bar s_{ \min\{n,p\}}>0$ on the main diagonal (and zeros elsewhere), where $\mathrm{rank}(\bX)= \min\{n,p\}$ a.s. Recall that the sample covariance matrix is
\[
\widehat{\boldsymbol{\Sigma}} = \frac{1}{n} \bX^\top \bX
= \frac{1}{n} \bW \bar{\boldsymbol{S}}^\top \bar{\boldsymbol{S}} \bW^\top
= \bW \boldsymbol{\Lambda} \bW^\top,
\]
where
\(
d_j = \bar s_j^2/n.
\) Since
\[
\bX^\top \by
= (\bU \bar{\boldsymbol{S}} \bW^\top)^\top \by
= \bW \bar{\boldsymbol{S}} \bU^\top \by
= \sum_{j=1}^p \bar s_j \bw_j (\bu_j^\top \by),
\]
for any function $f \in \mathcal{F}$ we have
\begin{equation}\label{eq: SVD general formula}    
f(\widehat{\boldsymbol{\Sigma}}) \frac{\bX^\top \by}{n}
= \bW f(\boldsymbol{\Lambda}) \bW^\top \frac{\bW \bar{\boldsymbol{S}} \bU^\top \by}{n}
= \sum_{j=1}^p \frac{f(d_j) \bar s_j}{n} \bw_j (\bu_j^\top \by).
\end{equation}

\subsection{Proof of Theorem~\ref{thm: f_ast dominates}(a)}\label{proof: min norm}
Consider the min-norm interpolator $$\widehat\bbeta_{\mathrm{mn}} =\argmin_\bbeta \{\|\bbeta\|_2 : \bX\bbeta = \by\}.$$ In particular, if $n > p$ then $\widehat\bbeta_{\mathrm{mn}} = \widehat\bbeta_{\mathrm{OLS}}$ a.s. Using the notations introduced in Appendix~\ref{notations for SVD}, it is well-known that $$\widehat\bbeta_{\mathrm{mn}} = \sum_{j=1}^{\min\{n,p\}} \frac{\bu_j^\top \by}{\bar s_j} \bw_j.$$
It is well-known \cite{hastie2022surprises} that the asymptotic risk of the min-norm interpolator is infinite for $c = 1$, so work with $c \neq 1$. The key claim that proves that $\widehat\bbeta_{\mathrm{mn}}$ has a larger prediction risk than $f_\ast^{\mathrm{pred}}$ is the following Lemma. 

\begin{lemma}\label{lemma: min-norm is worse}
    There exists $f_0 \in \mathcal{F}$ such that $f_0 \not\equiv f_\ast^{\mathrm{pred}}$ in $\mathcal{S}_c^+ \setminus\{0\}$ and $$ \lim_{n \to \infty} \|\widehat\bbeta_{\mathrm{mn}} - \bbeta_0\|_{\boldsymbol{\Sigma}}^2 \stackrel{a.s.}{=} \mathcal{R}^{\mathrm{pred}}(f_0; c).$$
\end{lemma}
The proof of Theorem~\ref{thm: f_ast dominates}(a) now follows from Lemma~\ref{lemma: f_ast is unique}.

\subsection{Proof of Theorem~\ref{thm: f_ast dominates}(b)}
The key step is to reduce the proof to the $0$-step estimator, which is not optimal. Suppose for contradiction that there exists an SD procedure that uses the same 
$\lambda$ at each step and achieves prediction risk no greater than that of 
$f_\ast^{\mathrm{pred}}$.

Recall from Lemma~\ref{lemma: sd is in mathcal F} that any SD procedure with $\lambda:=\lambda_0 = \lambda_1 = \dots = \lambda_t$ corresponds to a shrinkage function $f_t \in \mathcal{F}$ given by $$f_t(x) = \frac{1-\xi_t}{x+\lambda} + \frac{\xi_t(1-\xi_{t-1}) x}{(x+\lambda)^2} + \dots +  \frac{\xi_t \xi_{t-1} \dots \xi_2(1-\xi_1)\, x^{t-1}}{(x+\lambda)^{t}} + \frac{\xi_t \dots \xi_1\, x^{t}}{(x+\lambda)^{t+1}},$$
which after raising to a common denominator is $$f_t(x) = \frac{U(x)}{(x+\lambda)^{t+1}},$$
for some monic polynomial with real coefficients $U(x)$ with $\deg U =  t$ as $$(1-\xi_t) + \xi_t(1-\xi_{t-1}) + \dots + \xi_t \xi_{t-1} \dots \xi_2(1-\xi_1) + \xi_t \xi_{t-1} \dots \xi_2\xi_1 = 1.$$ Recall from Lemma~\ref{lemma: polynomials P Q} that $$f_\ast^{\mathrm{pred}}(x) = \frac{Q(x)}{P(x)},$$
where $P$ has $\rk + 1$ distinct real roots.  Since $f_t \in \mathcal{F}$ and $f_\ast^{\mathrm{pred}}$ achieves the smallest risk across all function in $\mathcal{F}$ then they must have equal asymptotic risk. Thus, by Lemma~\ref{lemma: f_ast is unique} $$\frac{Q(x)}{P(x)} = \frac{U(x)}{(x+\lambda)^{t+1}},$$
for all $x\in \mathcal{S}_c$. Equivalently, we have $Q(x)(x+\lambda)^{t+1} = U(x)P(x)$. Since this is a polynomial equation that holds for all $x\in \mathcal{S}_c$ which is infinite, then it must hold for all $x\in \mathbb{R}$. Since $P$ has $\rk + 1$ distinct roots, then $(x+\lambda)^2 \nmid P(x) $ so $(x+\lambda)^t \mid U(x)$. Since $U$ is monic, we deduce that we must have $U(x) = (x+\lambda)^{t}$ hence $$f_\ast^{\mathrm{pred}}(x) =\frac{Q(x)}{P(x)} = \frac{1}{x + \lambda},$$
which is a contradiction to { Theorem~\ref{thm:sd_optimal}(b) for $s = 0$}.

\subsection{Proof of Theorem~\ref{thm: f_ast dominates}(c)}
Consider gradient descent applied to the squared error loss
$$
\mathcal{L}_n(\bbeta) = \frac{1}{2n}\|\by - \bX \bbeta\|_2^2 .
$$
The gradient is
$$
\nabla \mathcal{L}_n(\bbeta) = -\frac{1}{n}\bX^\top(\by - \bX\bbeta) = \widehat{\boldsymbol{\Sigma}} \bbeta - \frac{1}{n}\bX^\top \by.
$$
Thus, gradient descent with step size $\eta > 0$ updates as
$$
\widehat \bbeta_{t+1} = \widehat \bbeta_t - \eta \nabla \mathcal{L}_n(\widehat \bbeta_t) = \widehat \bbeta_t + \eta \frac{1}{n}\bX^\top(\by - \bX \widehat \bbeta_t) = (\mathbf{I}_p - \eta \widehat {\boldsymbol{\Sigma}})\widehat \bbeta_t + \eta \frac{1}{n}\bX^\top \by .
$$
Suppose we initialize at $\widehat \bbeta_0 = 0$ and run gradient descent for $T$ steps. 
Unravelling the recursion yields: 
$$
\widehat \beta_T = (\mathbf{I}_p - \eta \widehat {\boldsymbol{\Sigma}})^T \widehat \bbeta_0 + \eta \sum_{k=0}^{T-1}(\mathbf{I}_p - \eta \widehat {\boldsymbol{\Sigma}})^k\frac{\bX^\top \by}{n}.
$$
Since $\widehat \bbeta_0 = 0$, this simplifies to
$$
\widehat \beta_T =\eta\left(\mathbf{I}_p+ (\mathbf{I}_p - \eta \widehat {\boldsymbol{\Sigma}})+ \cdots+ (\mathbf{I}_p - \eta \widehat {\boldsymbol{\Sigma}})^{T-1}\right)\frac{\bX^\top \by}{n}
 = \bW \, f_T(\boldsymbol{\Lambda})\, \bW^\top\frac{\bX^\top \by}{n},
$$
where the shrinkage function $f_T$ is: 
$$
f_T(x)=\eta \sum_{k=0}^{T-1} (1 - \eta x)^k=\eta \left(1 + (1 - \eta x) + \cdots + (1 - \eta x)^{T-1}\right),
$$
which is a polynomial in $x$ of degree $T-1$. In particular, $f_T$ is not equal to rational functions $f_\ast^{\mathrm{pred}}$ or $f_\ast^{\text{ast}}$. By Lemma~\ref{lemma: f_ast is unique}, we deduce that early-stopped gradient descent is suboptimal.
\subsection{Proof of Lemma~\ref{lemma: min-norm is worse}}
\begin{proof}
We distinguish between the two cases \(c>1\) and \(c<1\). Assume for now that \(c>1\) and without loss of generality that $p > n$. 
Let \(0<\eta<\sx^2\,(\sqrt c-1)^2\) be fixed and define \(f_0:\mathbb R_{\geq 0}\to\mathbb R\) by
\[f_0(x)=
\begin{cases}
0, & x\in [0,\eta/2),\\[4pt]
\dfrac{1}{x}, & x\in \big(\sx^2\,(1 - \sqrt c)^2-\eta/2,\infty\big),
\end{cases}
\]
and interpolate smoothly on \((\eta/2,\,\sx^2\,(\sqrt c-1)^2-\eta/2)\) so that \(f_0\) is smooth on \(\mathbb R_{\geq 0}\). Since \(f_*^{\rm pred}\) does not have an atom at \(0\), we have \(f_0\neq f_*^{\rm pred}\) on the support \(\mathcal S_c\).

\medskip
\noindent 
Next, we claim that there exists an event \(\Omega_0\) of probability one such that
\[
\widehat\bbeta_{\mathrm{mn}}(n,\omega)=\widehat\bbeta_{f_0}(n,\omega)
\]
for all \(\omega\in\Omega_0\) and all \(n\ge N(\omega)\). Recall from \cite{baik2004} that
\[
d_{n+1}=\cdots=d_p=0,
\qquad
d_n \xrightarrow[]{\mathrm{a.s.}} \sx^2\,(1 - \sqrt c)^2,
\]
where \(d_1\ge \cdots \ge d_p\) are the eigenvalues of \(\widehat{\boldsymbol{\Sigma}}\). Let \(\Omega_0\) be the event of probability one on which this convergence holds. Then for every \(\omega\in\Omega_0\), there exists \(N(\omega)\) such that
\[
d_n(n,\omega)>\sx^2\,(\sqrt c-1)^2-\eta/2
\]
for all \(n\ge N(\omega)\). In particular, for all such \(n\), every nonzero eigenvalue of \(\widehat{\boldsymbol{\Sigma}}\) lies in the region where \(f_0(x)=1/x\), while the zero eigenvalues lie in the region where \(f_0(x)=0\).
Therefore, using \eqref{eq: SVD general formula}, for all \(n\ge N(\omega)\),
\begin{align*}
f_0(\widehat{\boldsymbol{\Sigma}})\frac{\bX^\top \by}{n}
&= \sum_{j=1}^p \frac{f_0(d_j)\,\bar s_j}{n}\, \bw_j (\bu_j^\top \by) \\
&= \sum_{j=1}^n \frac{\bar s_j}{n d_j}\, \bw_j (\bu_j^\top \by)
   + \sum_{j=n+1}^p 0 \cdot \frac{\bar s_j}{n}\, \bw_j (\bu_j^\top \by) \\
&= \sum_{j=1}^n \frac{1}{\bar s_j}\, \bw_j (\bu_j^\top \by).
\end{align*}
Since \(d_j=\bar s_j^2/n\), the last expression is exactly the minimum-norm interpolator hence,
\[
\widehat\bbeta_{f_0}(n,\omega)=\widehat\bbeta_{\mathrm{mn}}(n,\omega)
\]
for all \(\omega\in\Omega_0\) and all \(n\ge N(\omega)\). It follows that
\[\lim_{n\to\infty}\|\widehat\bbeta_{f_0}-\bbeta_0\|_{\boldsymbol{\Sigma}}^2
\stackrel{a.s.}{=}
\lim_{n\to\infty} \|\widehat\bbeta_{\mathrm{mn}}-\bbeta_0\|_{\boldsymbol{\Sigma}}^2 
\]
which proves the claim. The proof for $c < 1$ follows mutatis mutandis using $d_p \xrightarrow[]{\mathrm{a.s.}} \sx^2\,(1 - \sqrt c)^2$. The details are omitted for brevity.
\end{proof}

\section{$f_\ast^{\mathrm{pred}}$ versus PCR - Proofs of 
Theorems~\ref{thm: f_ast dominates}(d) and (e)}\label{app: G}

\subsection{Notation and General form of PCR}
Throughout, we assume that
\[
m_n \leq \rank(\bX)=\min\{n,p\},
\]
since any additional components are redundant.  Following \cite{Green_2025} equation (25), define
\[
Q_c(x)=1-F_{\mathrm{MP}}(x),
\]
and for \(\tau'\in\bigl[0,\min\{1,1/c\}\bigr]\), let \(Q_c^{-1}(\tau')\in \mathcal S_c\) denote the unique value such that
\[
\int_{Q_c^{-1}(\tau')}^{{\sx^2(1+\sqrt{c})^2}} f_{\mathrm{MP}}(x)\,dx=\tau'.
\]
It is well-known that if $m/p\to \tau \in (0, \min\{1,1/c\})$ then we can invert the quantiles at the continuity points (see \cite{Green_2025}) to deduce that \begin{align}\label{eq: convergence of d_m}
d_m \xrightarrow[]{a.s.} Q_{c}^{-1}(\tau).\end{align}

 We start by introducing standard notations for the PCR estimator.
 Let $\bU_m \in \mathbb{R}^{n\times m}$ and $\bW_m \in \mathbb{R}^{p\times m}$ denote the
first $m$ columns of $\bU$ and $\bW$, respectively, and let
$\bar\bS_m = \mathrm{diag}(s_1,\dots,s_m)$.
The $m$-component PCR estimator is
\begin{equation}\label{eq: form for m PCR}
    \widehat\bbeta_{\mathrm{PCR},m}
= \bW_m \bar \bS_m^{-1} \bU_m^\top \by
= \sum_{j=1}^m \bw_j \frac{\bu_j^\top \by}{\bar s_j},
\end{equation}
where $\bu_j$ and $\bv_j$ are the $j$-th columns of $\bU$ and $\bW$. Another way to obtain the PCR estimator is to define $\mathcal{P}_m : \mathbb{R}^p \to \mathbb{R}^p $ be the projection onto the $m$ leading eigenspaces of $\widehat{\boldsymbol{\Sigma}}$,
$$\mathcal{P}_m = \bW_m\bW_m^\top = \sum_{i=1}^m \bw_i \bw_i^\top,$$
and likewise, let $\mathcal{P}_m ^\perp = \mathbf{I}_p - \mathcal{P}_m $ be the projection on the orthogonal space.
First, using standard concentration results, we prove the following Lemma which will be used for both Theorem~\ref{thm: f_ast dominates}(d) and (e).

\begin{lemma}\label{lemma: concentration PCR}
If $m/p \to \tau \in [0, \min\{1, 1/c\})$, then as $n \to \infty$
$$\|\widehat\bbeta_{\mathrm{PCR},m}-\bbeta_0\|_{\boldsymbol{\Sigma}}^2 - \mathbb{E}\left[\|\widehat\bbeta_{\mathrm{PCR},m}-\bbeta_0\|_{\boldsymbol{\Sigma}}^2 \mid \bX \right] \xrightarrow[]{\text{a.s.}} 0.$$
\end{lemma}

\subsection{Proof of Theorem~\ref{thm: f_ast dominates}(d)}
In this section, we prove that $f_\ast^{\mathrm{pred}}$ achieves smaller asymptotic than any PCR procedure that uses any fixed number of retained components $m_n = m$. The proof is based on one Lemma that shows that if we wanted to perform better than $f_\ast^{\mathrm{pred}}$ using PCR, we would need more components than the number of spikes that are above the BBP phase transition. Recall that $\rk^+ \leq \rk$ is the number of $\delta_j$'s that are above the BBP phase transition.

\begin{lemma}\label{pcr_with_all_spiked}
For all $m \leq \rk^+$, there exists $f_{\text{pcr}} \in \mathcal{F}$ such that $f_{\text{pcr}} \not\equiv f_\ast^{\mathrm{pred}}$ on $\mathcal{S}_c^+$ and \[
\lim_{n\to\infty}\|\widehat\bbeta_{\mathrm{PCR},m}-\bbeta_0\|_{\boldsymbol{\Sigma}}^2
=\mathcal{R}^{\mathrm{pred}}(f_{\mathrm{pcr}}; c).
\]
\end{lemma}
 
Using this characterization of the prediction risk, we are ready to prove 
Theorem~\ref{thm: f_ast dominates}(d).
\begin{proof}[Proof of Theorem~\ref{thm: f_ast dominates}(d)]
The key observation is that adding more than $\rk^+$ components is redundant in the limit. 
First, from Lemma~\ref{lemma: concentration PCR}, we know that
\begin{align}\label{eq: limit is just expected pcr}
    \lim_{n\to \infty}\|\widehat\bbeta_{\mathrm{PCR},m}-\bbeta_0\|_{\boldsymbol{\Sigma}}^2  
    \stackrel{\text{a.s.}}{=} 
    \lim_{n\to \infty} \mathbb{E}\left[\|\widehat\bbeta_{\mathrm{PCR},m}
    -\bbeta_0\|_{\boldsymbol{\Sigma}}^2 \mid \bX \right].
\end{align}
We can now compute the limit on the right-hand side of~\eqref{eq: limit is just 
expected pcr} using Theorem~3 in~\cite{Green_2025}. Indeed, for any $m \geq \rk^+$, 
the contribution of the outliers vanishes in equation~(59) of~\cite{Green_2025}, 
so the limit no longer depends on $m$. In particular,
\[
\lim_{n\to \infty}\mathbb{E}\left[\|\widehat\bbeta_{\mathrm{PCR},\rk^+}
-\bbeta_0\|_{\boldsymbol{\Sigma}}^2 \mid \bX \right] 
\stackrel{\text{a.s.}}{=} 
\lim_{n\to \infty} \mathbb{E}\left[\|\widehat\bbeta_{\mathrm{PCR},m}
-\bbeta_0\|_{\boldsymbol{\Sigma}}^2 \mid \bX \right],
\]
and combining this with Lemma~\ref{lemma: concentration PCR} applied to $\rk^+$ 
components and~\eqref{eq: limit is just expected pcr}, we deduce that
\[
\lim_{n\to \infty}\|\widehat\bbeta_{\mathrm{PCR},m}-\bbeta_0\|_{\boldsymbol{\Sigma}}^2  
\stackrel{\text{a.s.}}{=} 
\lim_{n\to \infty}\|\widehat\bbeta_{\mathrm{PCR},\rk^+}-\bbeta_0\|_{\boldsymbol{\Sigma}}^2.
\]
Combining this with Lemma~\ref{pcr_with_all_spiked} and 
Lemma~\ref{lemma: f_ast is unique}, we conclude that the asymptotic risk of 
PCR with finitely many components is strictly larger than that of 
$f_\ast^{\mathrm{pred}}$, thus completing the proof of 
Theorem~\ref{thm: f_ast dominates}(d).
\end{proof}
\subsection{Proof of Theorem~\ref{thm: f_ast dominates}(e)}
In this section, we prove that \(f_\ast^{\mathrm{pred}}\) is strictly better than PCR with \(m\) components for any sequence \(m=m_n {\to \infty}\) satisfying $m_n/p \to \tau \in\bigl[0, \min(1,1/c)\bigr)$ as $n\to \infty$. The first observation is that, in the proportional regime, it is enough to restrict attention to the case
\(
\tau\in\bigl(0,\min\{1,1/c\}\bigr).
\) 
Indeed, if \(\tau=0\), then equation (59) in \cite{Green_2025} implies that 
$$\lim_{n\to\infty}\mathbb{E}\left[\|\widehat\bbeta_{\mathrm{PCR},m}-\bbeta_0\|_{\boldsymbol{\Sigma}}^2 \mid \bX \right] = \lim_{n\to\infty}\mathbb{E}\left[\|\widehat\bbeta_{\mathrm{PCR},\rk^+}-\bbeta_0\|_{\boldsymbol{\Sigma}}^2 \mid \bX \right], $$
since in that case the contribution of the bulk vanishes in their formula. By equation~\eqref{eq: limit is just expected pcr}, we deduce that  $$\lim_{n\to\infty}\|\widehat\bbeta_{\mathrm{PCR},m}-\bbeta_0\|_{\boldsymbol{\Sigma}}^2 = \lim_{n\to\infty}\|\widehat\bbeta_{\mathrm{PCR},\rk^+}-\bbeta_0\|_{\boldsymbol{\Sigma}}^2$$
and by Lemma~\ref{pcr_with_all_spiked} applied for $\rk^+$ components, the right term is strictly larger than the prediction risk of \(f_\ast^{\mathrm{pred}}\). Thus, throughout this section, we will assume the following.
\begin{assumption}\label{assumption: prop regime PCR}
    $$\lim_{n\to\infty}\frac{m_n}{p}=\tau\in\bigl(0,\min\{1,1/c\}\bigr).$$
\end{assumption}
Since $\tau$ is assumed to be an interior value, it follows that
\(
Q_c^{-1}(\tau)\in \mathrm{int}(\mathcal S_c).
\) We continue with the following lemma.
\begin{lemma}\label{lemma: proportional regime PCR}
 Assume that \ref{assumption: prop regime PCR} holds. Then as $n\to\infty$,
\begin{align*}
\lim_{n \to \infty} \|\widehat\bbeta_{\mathrm{PCR},m} - \bbeta_0\|^2_{{\boldsymbol{\Sigma}}} & =
\sx^2 r^2\int_{0}^{Q_{c}^{-1}(\tau)} \mathrm dF_{\alpha}(x) + \sum_{j=1}^{\rk} \delta_j \alpha_j^2
  \left(\int_{0}^{Q_c^{-1}(\tau)} \mathrm dF_{\delta_j}(x)\right)^2
\\
& \qquad \qquad \qquad \qquad + c\sx^2\,\sigma_{\boldsymbol{\varepsilon}}^2\int_{Q_{c}^{-1}(\tau)}^{{\sx^2(1+\sqrt{c})^2}} \frac{f_{\mathrm{MP}}(x)}{x}\,dx.
\end{align*}
\end{lemma}
Next, recall from Lemma~\ref{prop: general formula} that for a general function $f\in \mathcal{F}$ the asymptotic risk of $\widehat\bbeta_f$ is
\begin{align*}
\rout & \stackrel{\text{a.s.}}{=} \sx^2\,r^2\int (1 - xf(x))^2 \ \mathrm dF _\alpha(x) \ + \sum_{j=1}^\rk\delta_j \alpha_j^2\left(\int (1 - xf(x)) \ \mathrm dF _{\delta_j}(x)\right)^2 \\
& \qquad \qquad \qquad + c\sx^2\,\sigma_{\boldsymbol{\varepsilon}}^2 \int xf^2(x) \ \mathrm dF _{\rm MP}(x)  \,.
\end{align*}
Using Lemma~\ref{lemma: proportional regime PCR}, we deduce that the limit of $m$-PCR is the same as that of $$f_\text{PCR}(x) = \begin{cases}
    0, \qquad \text{if } x < Q_{c}^{-1}(\tau), \\
    1/x, \qquad \text{otherwise}. 
    \end{cases}$$
Note that \(f_{\mathrm{PCR}}\) is not continuous on \(\mathcal{S}_c\), and therefore \(f_{\mathrm{PCR}} \notin \mathcal{F}\). Thus, compared to the proofs of Lemma~\ref{pcr_with_all_spiked} and Lemma~\ref{lemma: min-norm is worse}, extra work is required to conclude. The next technical Lemma proves that PCR is suboptimal to our $f_\ast^{\mathrm{pred}}$ procedure. 
\begin{lemma}\label{lemma: technical approx of PCR}
 Assume \ref{assumption: prop regime PCR} holds. Then, there exists a uniformly bounded sequence of functions \(\{f_n\}_n\), with \(f_n \in \mathcal{F}\), and a constant \(\eta > 0\) such that:
\begin{enumerate}
    \item \(\lim_{n\to\infty} f_n(x)=f_{\mathrm{PCR}}(x)\) for all \(x\in\mathbb{R}_{\geq 0}\),
    \item \(\mathcal{R}^\mathrm{pred}(f_n; c) \geq \mathcal{R}^\mathrm{pred}(f_\ast^{\mathrm{pred}}; c) + \eta\).
\end{enumerate}
\end{lemma}
Indeed, since the functions \(f_n\) are uniformly bounded and the supports of \(F_\alpha\), \(F_{\delta_j}\), and \(F_{\mathrm{MP}}\) are all bounded subsets of the nonnegative real line, we may apply the dominated convergence theorem using Lemma~\ref{prop: general formula} to deduce that
\[
\lim_{n\to \infty} \mathcal{R}^\mathrm{pred}(f_n)
=
\lim_{n \to \infty} \|\widehat\bbeta_{\mathrm{PCR},m} - \bbeta_0\|^2_{{\boldsymbol{\Sigma}}}.
\]
Taking the limit as $n \to \infty$ in property 2 in the above Lemma completes the proof of Theorem~\ref{thm: f_ast dominates}(e).

\subsection{Proof of Lemma~\ref{lemma: concentration PCR}}
\begin{proof}
   For simplicity of notation, define
\[
T_n:= \|\widehat\bbeta_{\mathrm{PCR},m}-\bbeta_0\|_{\boldsymbol{\Sigma}}^2.
\]
Using \(\by=\bX\bbeta_0+\beps\), we have
\[
\widehat{\bbeta}_{\mathrm{PCR},m}
=
\bW_m \bar{\bS}_m^{-1}\bU_m^\top \by
=
\mathcal{P}_m \bbeta_0
+
\bW_m \bar{\bS}_m^{-1}\bU_m^\top \beps,
\]
so
\[
\widehat{\bbeta}_{\mathrm{PCR},m}-\bbeta_0
=
-\mathcal{P}_m ^\perp\bbeta_0
+
\bW_m \bar{\bS}_m^{-1}\bU_m^\top \beps.
\]
Therefore,
\[
T_n
=
\left\|\mathcal{P}_m ^\perp\bbeta_0\right\|_{\boldsymbol{\Sigma}}^{2}
+
\left\|\bW_m \bar{\bS}_m^{-1}\bU_m^\top \beps\right\|_{\boldsymbol{\Sigma}}^{2}
-2\,
\left(\mathcal{P}_m ^\perp\bbeta_0\right)^\top
\boldsymbol{\Sigma}
\bW_m \bar{\bS}_m^{-1}\bU_m^\top \beps.
\]

Since \(\beps \indep \bX\), \(\mathbb E[\beps\mid \bX]=0\), and
\(\cov(\beps\mid \bX)=\sigma_{\beps}^2 \bI_n,
\) we obtain
\[
\begin{aligned}
& T_n-\mathbb E[T_n\mid \bX] \\
&=
-2\,
\bigl(\mathcal{P}_m ^\perp\bbeta_0\bigr)^\top
\boldsymbol{\Sigma}\,
\bW_m \bar{\bS}_m^{-1}\bU_m^\top \beps
\\
&\quad+
\beps^\top
\bU_m \bar{\bS}_m^{-1}\bW_m^\top
\boldsymbol{\Sigma}\,
\bW_m \bar{\bS}_m^{-1}\bU_m^\top
\beps -
\mathbb{E}\left[
\beps^\top
\bU_m \bar{\bS}_m^{-1}\bW_m^\top
\boldsymbol{\Sigma}\,
\bW_m \bar{\bS}_m^{-1}\bU_m^\top
\beps
\,\middle|\, \bX
\right]
\\
&=
-2\,
\bbeta_0^\top \mathcal{P}_m ^\perp
\boldsymbol{\Sigma}\,
\bW_m \bar{\bS}_m^{-1}\bU_m^\top \beps
\\
&\quad+
\beps^\top
\bU_m \bar{\bS}_m^{-1}\bW_m^\top
\boldsymbol{\Sigma}\,
\bW_m \bar{\bS}_m^{-1}\bU_m^\top
\beps
-
\sigma_{\beps}^2\,\tr \left(
\boldsymbol{\Sigma}\,\bW_m \bar{\bS}_m^{-2}\bW_m^\top
\right).
\end{aligned}
\]
Finally, as in Section~\ref{proof: general formula}, we use standard 
concentration results to show that both terms are negligible. Define
\[
\ba_n = 2n\,\bU_m\bar{\bS}_m^{-1}\bW_m^\top\boldsymbol{\Sigma}
\mathcal{P}_m ^\perp\bbeta_0 \in \mathbb{R}^n,
\]
so that the cross term equals $-\ba_n^\top\beps/n$. As in Section~\ref{proof: general formula}, $\|\bW_m\|_{\mathrm{op}} = \|\mathcal{P}_m^\perp\|_{\mathrm{op}} = 1$,
and $\|\boldsymbol{\Sigma}\|_{\mathrm{op}} = \sx^2\, + \max_{j\in[r]}\delta_j$
are all deterministically of order 1. Since $d_i = \bar{s}_i^2/n$, we have 
$$\|\bar{\bS}_m^{-1}\|_{\mathrm{op}} = s_m^{-1} = \frac{1}{\sqrt{n\,d_m}},$$
hence,
\[
\frac{\|\ba_n\|_2^2}{n} \leq 
\frac{4\|\boldsymbol{\Sigma}\|_{\mathrm{op}}^2\,\|\bbeta_0\|_2^2}{d_m}.\]
If $\tau = 0$, then there exists a small constant $\eta > 0$ such that $d_m \geq d_{\lfloor p\eta\rfloor}$ for all large enough $n$. Since the lower bound $d_{\lfloor p\eta\rfloor}$ converges to a point in $\text{int} \left(\mathcal{S}_c\right)$, $d_m$ is almost surely bounded away from zero for all large $n$. Similarly if $\tau > 0$, then $d_m \to Q^{-1}_{c}(\tau) \in \text{int}(\mathcal{S}_c) $ almost surely, so $d_m$ is almost surely bounded away from zero for all large $n$. In both cases, we can conclude that almost surely,
$$\limsup_{n\to\infty}\frac{\|\ba_n\|_2^2}{n} < \infty,$$
so we can apply Lemma~\ref{lem:linear_concentration} to deduce that 
$$-2\,
\bbeta_0^\top \mathcal{P}_m ^\perp
\boldsymbol{\Sigma}\,
\bW_m \bar{\bS}_m^{-1}\bU_m^\top \beps \xrightarrow[]{\text{a.s.}} 0.$$
Similarly, define
\[
\bA_n = n\,\bU_m\bar{\bS}_m^{-1}\bW_m^\top\boldsymbol{\Sigma}
\bW_m\bar{\bS}_m^{-1}\bU_m^\top \in \mathbb{R}^{n\times n},
\]
so that the quadratic term equals $\beps^\top\bA_n\beps/n - \sigma_\beps^2\tr(\bA_n)/n$, since
\[
\frac{\sigma_\beps^2}{n}\,\tr(\bA_n) 
= \sigma_\beps^2\, \tr\left(
\boldsymbol{\Sigma}\,\bW_m\bar{\bS}_m^{-2}\bW_m^\top
\right),
\]
where we used the cyclic property of the trace. Its operator norm satisfies
\[
\|\bA_n\|_{\mathrm{op}} \leq n\,\|\bar{\bS}_m^{-1}\|_{\mathrm{op}}^2\,
\|\boldsymbol{\Sigma}\|_{\mathrm{op}},\]
which yields $\limsup_{n\to \infty}\|\bA_n\|_{\mathrm{op}} < \infty$ by the same argument as above. Lemma~\ref{lem:quadratic_concentration} gives
\[
\beps^\top
\bU_m \bar{\bS}_m^{-1}\bW_m^\top
\boldsymbol{\Sigma}\,
\bW_m \bar{\bS}_m^{-1}\bU_m^\top
\beps
-
\sigma_{\beps}^2\,\tr \left(
\boldsymbol{\Sigma}\,\bW_m \bar{\bS}_m^{-2}\bW_m^\top
\right)  = \frac{\beps^\top\bA_n\beps}{n} - \frac{\sigma_\beps^2\tr(\bA_n)}{n} \xrightarrow{\text{a.s.}} 0,\]
thus completing the proof.
\end{proof}

\subsection{Proof of Lemma~\ref{pcr_with_all_spiked}}
\begin{proof}
We start with the case \(m = \rk^+=1\). Let \(x_1^\star\) be the outlier that corresponds to $\delta_1 > \sx^2\,(1+\sqrt{c})^2$ as in Lemma~\ref{lemma:f_delta}(b). Choose
\[
0<\eta<x_1^\star-\sx^2\,(1+\sqrt c)^2.
\]
As in the proof of Lemma~\ref{lemma: min-norm is worse}, we define \(f_{\mathrm{pcr}}\) by
\[
f_{\mathrm{pcr}}(x)=
\begin{cases}
\dfrac{1}{x}, & x \in (x_1^\star-\eta/2,\;x_1^\star+\eta/2),\\[4pt]
0, & x\in [0,x_1^\star-\eta)\cup(x_1^\star+\eta,\infty)
\end{cases},
\]
and interpolate smoothly on \((x_1^\star-\eta,\;x_1^\star-\eta/2)\) and \((x_1^\star+\eta/2,\;x_1^\star+\eta)\) so that \(f_{\mathrm{pcr}}\) is smooth on \(\mathbb R_{\geq 0}\). Note that \(f_{\mathrm{pcr}}\equiv 0\) on the bulk, whereas \(f_\ast\) is not identically zero there. 
We now claim that there exists an event \(\Omega\) of probability one such that
\[
\widehat\bbeta_{f_{\mathrm{pcr}}}(n,\omega)=\widehat\bbeta_{\mathrm{PCR},1}(n,\omega)
\]
for all \(\omega\in\Omega\) and all \(n\ge N(\omega)\). Indeed, by \cite{baik2004},
\[
d_1 \xrightarrow[]{\mathrm{a.s.}} x_1^\star,\qquad
d_2 \xrightarrow[]{\mathrm{a.s.}} \sx^2\,(1+\sqrt c)^2,\qquad
d_p \xrightarrow[]{\mathrm{a.s.}}
\begin{cases}
\sx^2\,(1-\sqrt c)^2, & \text{if } c<1,\\[4pt]
0, & \text{if } c\ge 1.
\end{cases}
\]
Hence there exists an event \(\Omega\) of probability one on which all of the above convergences hold simultaneously. For each \(\omega\in\Omega\), and all \(n\) large enough, we have
\[
d_1(n,\omega)\in (x_1^\star-\eta/2,\;x_1^\star+\eta/2),
\]
while
\[
d_2(n,\omega),\dots,d_p(n,\omega)\in [0,x_1^\star-\eta)\cup(x_1^\star+\eta,\infty).
\]
By construction of \(f_{\mathrm{pcr}}\), it follows that for all such \(n\),
\[
f_{\mathrm{pcr}}(d_1)=\frac{1}{d_1},
\qquad
f_{\mathrm{pcr}}(d_j)=0,\quad j\ge 2.
\]
Using \eqref{eq: SVD general formula}, we obtain
\begin{align*}
f_{\mathrm{pcr}}(\widehat{\boldsymbol{\Sigma}})\frac{\bX^\top \by}{n}
&= \sum_{j=1}^p \frac{f_{\mathrm{pcr}}(d_j)\,\bar s_j}{n}\, \bw_j (\bu_j^\top \by) = \frac{\bar s_1}{n d_1}\, \bw_1 (u_1^\top \by).
\end{align*}
Since \(d_1=\bar s_1^2/n\), this simplifies to
\[
f_{\mathrm{pcr}}(\widehat{\boldsymbol{\Sigma}})\frac{\bX^\top \by}{n}
= \frac{1}{\bar s_1}\, \bw_1 (u_1^\top \by),
\]
which is exactly the one-component PCR estimator~\eqref{eq: form for m PCR}. Therefore,
\[
\widehat\bbeta_{f_{\mathrm{pcr}}}(n,\omega)=\widehat\bbeta_{\mathrm{PCR},1}(n,\omega)
\]
for all \(\omega\in\Omega\) and all \(n\ge N(\omega)\). Consequently,
\[
\lim_{n\to\infty}\|\widehat\bbeta_{\mathrm{PCR},1}-\bbeta_0\|_{\boldsymbol{\Sigma}}^2
=
\lim_{n\to\infty}\|\widehat\bbeta_{f_{\mathrm{pcr}}}-\bbeta_0\|_{\boldsymbol{\Sigma}}^2,
\]
and the conclusion follows. 

The general case follows by the same argument. Since the outliers 
$x_j^\star$ are pairwise distinct by Lemma~\ref{lemma: distinct x_jstar}, one can choose sufficiently small disjoint neighborhoods 
around each of them on which the shrinkage function equals $1/x$, 
and then smoothly interpolate between these neighborhoods on the remainder of the domain. The 
extension is immediate and we omit the details.
\end{proof}

\subsection{Proof of Lemma~\ref{lemma: proportional regime PCR}}

\begin{proof}

Using equation (47) in \cite{Green_2025}, we have the following bias-variance decomposition:
\begin{equation}\label{eq: formula for PCR Risk}
    \mathbb{E}\Big[\|\widehat\bbeta_{\mathrm{PCR},m} - \bbeta_0\|^2_{{\boldsymbol{\Sigma}}} \mid \bX \Big] = \|{\boldsymbol{\Sigma}}^{1/2}\mathcal{P}_m^\perp  \bbeta_0\|^2 + \frac{\sigma_{\boldsymbol{\varepsilon}}^2}{n} \tr\left((\mathcal{P}_m \widehat{\boldsymbol{\Sigma}}\mathcal{P}_m )^{\dagger} {\boldsymbol{\Sigma}}\right).
\end{equation}
 We use the following decomposition  $${\boldsymbol{\Sigma}}^{1/2} = \sx\,\mathbf{I}_p + \sum_{j=1}^\rk \left(\sqrt{\delta_j + \sx^2\,} - \sx\,\right)\bv_j\bv_j^\top,$$
which is straightforward from the definition of ${\boldsymbol{\Sigma}}$. We deduce that 
\begin{align*}
\Big\|{\boldsymbol{\Sigma}}^{1/2}\mathcal{P}_m^\perp \bbeta_0\Big\|^2
&=
\Big\|\sx\,\mathcal{P}_m^\perp \bbeta_0
+\sum_{j=1}^{\rk}\bigl(\sqrt{\delta_j+\sx^2}-\sx\bigr)
\bv_j\bv_j^\top \mathcal{P}_m^\perp \bbeta_0\Big\|^2 \\
&=
\|\sx\,\mathcal{P}_m^\perp \bbeta_0\|^2
+\sum_{j=1}^{\rk}\left(\sqrt{\delta_j+\sx^2\,}-\sx\right)^2
\bigl(\bv_j^\top \mathcal{P}_m^\perp \bbeta_0\bigr)^2 \\
& \qquad \qquad \qquad \qquad \qquad +\,2\sum_{j=1}^{\rk}\sx\left(\sqrt{\delta_j+\sx^2}-\sx\right)
\bigl(\bv_j^\top \mathcal{P}_m^\perp \bbeta_0\bigr)^2 \\
&=
\sx^2\,\|\mathcal{P}_m^\perp \bbeta_0\|^2
+\sum_{j=1}^{\rk}\delta_j
\bigl(\bv_j^\top \mathcal{P}_m^\perp \bbeta_0\bigr)^2.
\end{align*}
Using the measure $\widehat B_n$ from Lemma~\ref{lemma: f_alpha}, we obtain
\begin{equation}\label{eq: b hat in PCR}
\frac{\|\mathcal P_m^\perp \bbeta_0\|^2}{\|\bbeta_0\|^2}
=
\int_{[0,d_{m+1}]} d\widehat B_n(x).    
\end{equation}
Recall that \(d_m\to Q_c^{-1}(\tau) \in \mathrm{int}(\mathcal{S}_c)\) almost surely. Moreover, \(\widehat B_n \xrightarrow{d} F_\alpha\) almost surely by Lemma~\ref{lemma: f_alpha}. Therefore, since $Q_c^{-1}(\tau)$ is not an atom of $F_{\alpha}$, a standard limiting argument gives
\[
\int_{[0,d_m)} d\widehat B_n(x)
\stackrel{\mathrm{a.s.}}{\longrightarrow}
\int_{0}^{Q_c^{-1}(\tau)} \mathrm dF _\alpha(x).
\]
By the same argument, using the signed measures \(\widehat C_n^j\) and their vague limits \(F_{\delta_j}\) from \ref{lemma: hat C nj}, we obtain
\[
\frac{\bv_j^\top\mathcal{P}_m^\perp\bbeta_0}{\bv_j^\top \bbeta_0}
=
\int_{[0, d_m)} d\widehat C_n^j(x)
\stackrel{\mathrm{a.s.}}{\longrightarrow}
\int_{0}^{Q_c^{-1}(\tau)} \mathrm dF _{\delta_j}(x).
\]
Therefore,
\[
(\bv_j^\top\mathcal{P}_m^\perp\bbeta_0)^2
\stackrel{\mathrm{a.s.}}{\longrightarrow}
\alpha_j^2 \left(\int_{0}^{Q_c^{-1}(\tau)} \mathrm dF _{\delta_j}(x)\right)^2.
\]
Finally, by Theorem 3 in \cite{Green_2025},
\[
\frac{\sigma_{\boldsymbol{\varepsilon}}^2}{n} \tr\left((\mathcal{P}_m \widehat{\boldsymbol{\Sigma}}\mathcal{P}_m )^{\dagger} {\boldsymbol{\Sigma}}\right)
\xrightarrow{\text{a.s.}}
c\sigma_{\boldsymbol{\varepsilon}}^2   \int_{Q_{c}^{-1}(\tau)}^b \frac{1}{x^2 |\underline{m}(x)|^2} f_{\mathrm{MP}}(x) \, \mathrm{d}x,
\]
where \(\underline m\) is the companion Stieltjes transform of the Marchenko-Pastur law (see equation~\eqref{eq:comp_st_1}). 
By Lemma~\ref{lemma: companion bash}, we have
\[
c\sigma_{\boldsymbol{\varepsilon}}^2   \int_{Q_{c}^{-1}(\tau)}^b \frac{1}{x^2 |\underline{m}(x)|^2} f_{\mathrm{MP}}(x) \, \mathrm{d}x
=
c\sx^2\,\sigma_{\boldsymbol{\varepsilon}}^2\int_{Q_{c}^{-1}(\tau)}^b \frac{f_{\mathrm{MP}}(x)}{x}\,dx.
\]
Combining the three limits with Lemma~\ref{lemma: concentration PCR} gives the result.
\end{proof}

\subsection{Proof of Lemma~\ref{lemma: technical approx of PCR}}
\begin{proof}Fix \(\sx^2\,(1-\sqrt{c})^2<\eta<Q_c^{-1}(\tau)\). Choose \(\mathfrak e_n\downarrow 0\) such that
$
Q_c^{-1}(\tau)-\mathfrak e_n>\eta$
for all \(n\).  Let \(\psi_n:\mathbb R_{\ge 0}\to[0,1]\) be a smooth function satisfying
\[
\psi_n(x)=
\begin{cases}
0, & x\le Q_c^{-1}(\tau)-\mathfrak e_n,\\[4pt]
1, & x\ge Q_c^{-1}(\tau)+\mathfrak e_n,
\end{cases}
\]
and define
\[
f_n(x) =\frac{\psi_n(x)}{x}.
\]
It is straightforward to check that \(f_n\to f_{\mathrm{PCR}}(x)\) pointwise on \(\mathbb R_{\ge 0}\). Moreover, \(\{f_n\}_n\) are smooth and uniformly bounded by \(1/\eta\), so Lemma~\ref{lemma: technical approx of PCR}(1) is proved. 

Next, we prove the second property. Recall from equation~\eqref{eq: Risk out f up to prop} that for any $f \in \mathcal{F}$ the limiting risk $\mathcal{R}^\mathrm{pred}(f ; c)$ is, up to an additive constant, equal to $$\left\langle   \mathcal{A}(f-f_\ast^{\mathrm{pred}}),  f-f_\ast^{\mathrm{pred}}\right\rangle_w  - \langle Af_\ast, f_\ast\rangle_w.$$
Using equation~\ref{eq: Apsi psi}, we can immediately deduce that $$\mathcal{R}^\mathrm{pred}(f_n; c) - \mathcal{R}^\mathrm{pred}(f_\ast^{\mathrm{pred}}; c) =\left\langle  \mathcal{A}(f_n-f_\ast^{\mathrm{pred}}),  f_n-f_\ast^{\mathrm{pred}}\right\rangle_w \geq \|f_n - f_\ast^{\mathrm{pred}}\|_w^2.$$
Recall from~\ref{eq: define w ,g , h_j} that $w(x) > 0$ for all $x\in \mathcal S_c^+$ and note that \(f_n(x)=0\) for all \(x\le \eta\). We deduce that
\begin{align*}
\|f_n-f_\ast^{\mathrm{pred}}\|_w^2
&=
\int x\,w(x)\bigl(f_n(x)-f_\ast^{\mathrm{pred}}(x)\bigr)^2\,\mathrm dF _\alpha(x) \\
&\ge
\int_{\sx^2\,(1-\sqrt c)^2}^{\eta}
x\,w(x)\bigl(f_n(x)-f_\ast^{\mathrm{pred}}(x)\bigr)^2\,\mathrm dF _\alpha(x) \\
&=
\int_{\sx^2\,(1-\sqrt c)^2}^{\eta}
x\,w(x)f_\ast^{\mathrm{pred}}(x)^2\,\mathrm dF _\alpha(x).
\end{align*}
Thus, we can define $\rho$ as the lower bound above, and note that $\rho > 0$ because $f_\ast^{\mathrm{pred}}$ cannot be constant 0 on $\mathcal{S}_c$ by Lemma~\ref{lemma: polynomials P Q}.
\end{proof}

\section{Estimation Risk - Proof of Theorem~\ref{thm: main_thm est}}\label{app: H}
\subsection{Proof of Theorem~\ref{thm: main_thm est}(a) and (b)}
In this section, we prove that the same ideas used for the Out-of-sample risk hold for estimation risk. We prove the first two steps of Theorem~\ref{thm: main_thm est} following the steps of proof of Theorem~\ref{thm: main} mutatis mutandis. First, we provide a general formula for $\rest$ that holds for all $f \in \mathcal{F}$.

\begin{lemma}\label{lemma: general formula for f_est} For any $f \in \mathcal{F}$ we have
\begin{align*}
\rest & \overset{a.s.}{=}  r^2\int (1 - xf(x))^2 \ \mathrm dF _\alpha(x)  + c{\sigma_{\beps}^2} \int xf^2(x) \ \mathrm dF _{\rm MP}(x).
\end{align*}
\end{lemma}
Note that the formula is very similar to that in Lemma~\ref{prop: general formula} for the prediction risk, the only difference being the absence of the squared integrals terms. This significantly simplifies the process of finding $f_\ast^{\mathrm{est}}$.  

\begin{proof}[Proof of Theorem~\ref{thm: main_thm est}(a) and (b)]
Recall $w(x), g(x)$ and $\mu_j(x)$ from \eqref{eq: define w ,g , h_j}. Expanding everything in the formula from Lemma~\ref{lemma: general formula for f_est}  we have:
\[
\begin{aligned}
\rest
&= r^2 
- 2 r^2 \int x\,f(x)\, \mathrm{d}F_{\alpha}(x) + r^2 \int x^2f(x)^2 \mathrm dF _{\alpha} +  c\,{\sigma_{\beps}^2} \int xf^2(x) \ \mathrm dF _{\rm MP}(x) \\
&= r^2 +  \int \Bigl(r^2x^2 + c\,\sigma_{\beps}^{2}\,x\,\mu_0(x)\Bigr)\,f(x)^2\,\mathrm dF _\alpha(x) - 
2 r^2 \int x\,f(x)\, \mathrm{d}F_{\alpha}(x)\\
&= r^2 +  \int \left[\left(r^2x^2 + c\,\sigma_{\beps}^{2}\,x\,\mu_0(x)\right)\,f(x)^2 - 2 r^2 x\,f(x)\right]\,\mathrm dF _\alpha(x)
\end{aligned},
\]
where we used the Radon-Nikodym derivative $\mu_0$ of $F_{\mathrm{MP}}$ w.r.t. $F_{\alpha}$ from Lemma~\ref{lemma: def mu_j mu_0}.
Note that for each $x$ the integrand is a quadratic in $f(x)$ which is uniquely minimized for $x \in \mathcal{S}_c^+ \setminus \{0\}$ at $$f_\ast^{\mathrm{est}}(x) = \frac{r^2}{r^2x + c\sigma_\beps^2\mu_0(x)} = \frac{r^2 \left(\omega_0 \nu(x) + \sum_{i=1}^\rk \omega_i \nu_{-i}(x)\right)}{r^2 x\left(\omega_0 \nu(x) + \sum_{i=1}^\rk \omega_i \nu_{-i}(x)\right) + c \sigma_{\boldsymbol{\varepsilon}}^2 \nu(x)}.$$
Uniqueness follows, since any function that achieves equality must agree 
with the minimum of the quadratic on a set of $F_{\alpha}$-measure $1$. 
Since the functions must be continuous, this proves part~(a).

Next, note that the denominator of $f_\ast^{\mathrm{est}}(x)$ is the same as the denominator of $f_\ast^{\mathrm{pred}}$ up to a factor of $\sx^2$ (see Theorem~\ref{thm: main}). Thus, using Lemma~\ref{lemma: polynomials P Q} we know that the denominator has $\rk + 1$ real roots, and a similar application of Lemma~\ref{lemma: lagrange interpolation} yields the first part of Theorem~\ref{thm: main_thm est}(b). We omit the details.

Finally, we prove that $\rk$ steps are always necessary to achieve $f_\ast^\mathrm{est}$. Note that any common root of the denominator and numerator in the display above must also be a root of $\nu(x)$. Recall from Lemma~\ref{lemma: F_mp << F_delta} that the only roots of $$\nu(x) = \nu_1(x) \cdot \nu_2(x)\cdots \nu_\rk(x)$$ are the outliers $x_1^\star, \dots, x_\rk^\star$. However, by Lemma~\ref{lemma: polynomials P Q}, we know that the denominator has exactly $\rk + 1$ real roots that lie outside $\overline{\mathcal{S}_c}$. In particular, $x_1^\star, \dots, x_\rk^\star$ can not be common roots, so by the same method as in Lemma~\ref{lemma: coprime} conclusion follows. 
\end{proof}

\subsection{Proof of Theorem~\ref{thm: main_thm est}(c)}
\begin{proof}
The proofs that $f_\ast^{\mathrm{est}}$ achieves smaller risk than the min-norm interpolator, early stopped gradient descent and PCR with a fixed number of components follow mutatis mutandis from their corresponding Theorems for prediction risk. We omit the details for brevity. 

Next, as proved in the previous section, the denominator $f_\ast^{\mathrm{est}}$ has $\rk + 1$ distinct, real roots and the rational function is irreducible. In particular, Ridge regression is suboptimal and so is any procedure that uses the same $\lambda$ across distillation steps.

Finally, the only case left is PCR with diverging number of retained components. As before, we work under Assumption~\ref{assumption: prop regime PCR}. Using the same concentration inequalities as in Lemma~\ref{lemma: concentration PCR}, we infer that
$$\lim_{n\to \infty}\|\widehat\bbeta_{\mathrm{PCR},m} - \bbeta_0\|^2_{2}  \,\stackrel{\text{a.s.}}{=} \,\lim_{n\to \infty}\mathbb{E}\Big[\|\widehat\bbeta_{\mathrm{PCR},m} - \bbeta_0\|^2_{2} \mid \bX \Big],$$
where we omit the details for the sake of brevity. Thus, it is enough to work with the conditional asymptotic risk. 
From equation (27) in \cite{Green_2025}, \begin{equation}\label{eq: formula for est PCR Risk}
    \mathbb{E}\Big[\|\widehat\bbeta_{\mathrm{PCR},m} - \bbeta_0\|^2_{2} \mid \bX \Big] = \|\mathcal{P}_m^\perp  \bbeta_0\|^2 + \frac{\sigma_{\boldsymbol{\varepsilon}}^2}{n} \tr\left((\mathcal{P}_m \widehat{\boldsymbol{\Sigma}}\mathcal{P}_m )^{\dagger}\right).
\end{equation}
As proved in Lemma~\ref{lemma: proportional regime PCR}, the limit of the bias term is $$\lim_{n\to\infty}\|\mathcal{P}_m^\perp  \bbeta_0\|_2^2 \, \stackrel{\text{a.s.}}{=} \, r^2  \int_{0}^{Q_{c}^{-1}(\tau)} \mathrm dF _{\alpha}(x),$$
and using Theorem 1 in ~\cite{Green_2025} 
$$\lim_{n\to\infty}\frac{\sigma_{\boldsymbol{\varepsilon}}^2}{n} \tr\left((\mathcal{P}_m \widehat{\boldsymbol{\Sigma}}\mathcal{P}_m )^{\dagger}\right) \stackrel{\text{a.s.}}{=} \,  c\sigma_{\boldsymbol{\varepsilon}}^2\int_{Q_{c}^{-1}(\tau)}^{{(\sx^2(1+\sqrt{c})^2}} \frac{f_{\mathrm{MP}}(x)}{x}\,\mathrm{d}x.$$
As in the prediction error case, the limit of $m$-PCR is the same as that of $$f_\text{PCR}(x) = \begin{cases}
    0, \qquad \text{if } x < Q_{c}^{-1}(\tau) \\
    1/x, \qquad \text{otherwise}. 
    \end{cases}$$
We can now finish using the same approximation as in Lemma~\ref{lemma: technical approx of PCR}. This completes the proof of Theorem~\ref{thm: main_thm est}(c).
\end{proof}

\subsection{Proof of Lemma~\ref{lemma: general formula for f_est}}
\begin{proof}
Recall from~\eqref{formula for beta_f} that: 
$$
\widehat \bbeta_f = \bW(\boldsymbol{\Lambda} f(\boldsymbol{\Lambda}))\bW^\top\bbeta_0 + \bW f(\boldsymbol{\Lambda})\bW^\top\frac{\bX^\top \boldsymbol{\varepsilon}}{n}.
$$
Therefore, following the same ideas as in Section~\ref{proof: general formula},
\begin{align*}
\|\widehat\bbeta_f - \bbeta_0\|_{2}^2
&=\bbeta_0^\top \bW\bigl(\mathbf{I}_p - \boldsymbol{\Lambda} f(\boldsymbol{\Lambda})\bigr)^2
\bW^\top\bbeta_0\\
&\quad - 2\,\bbeta_0^\top\bW\bigl(\mathbf{I}_p - \boldsymbol{\Lambda} f(\boldsymbol{\Lambda})\bigr)
f(\boldsymbol{\Lambda})\bW^\top\frac{\bX^\top\boldsymbol{\varepsilon}}{n}\\
&\quad + \frac{\boldsymbol{\varepsilon}^\top\bX}{n}\bW f^2(\boldsymbol{\Lambda})\bW^\top
\frac{\bX^\top\boldsymbol{\varepsilon}}{n},
\end{align*}
where we used $\bW^\top\bW = \mathbf{I}_p$. Alternatively,
\[
\|\widehat{\bbeta}_f - \bbeta_0\|_{2}^2
=
\mathcal{B}_{f}^{\mathrm{est}}(\bX, \by) + \mathcal{V}_{f}^{\mathrm{est}}(\bX, \by)
+ \mathcal{E}_{f,n}^{\mathrm{est}}(\bX, \by),
\]
where
\begin{align*}
\mathcal{B}_{f}^{\mathrm{est}}(\bX,\by)
&:=
\bbeta_0^\top \bW\bigl(\mathbf{I}_p - \boldsymbol{\Lambda} f(\boldsymbol{\Lambda})\bigr)^2
\bW^\top\bbeta_0, \\
\mathcal{V}_{f}^{\mathrm{est}}(\bX, \by)
&:=
\frac{\sigma_\beps^2}{n^2}
\tr\left(\bX f^2(\widehat{\boldsymbol{\Sigma}}) \bX^\top \right) = 
\frac{\sigma_\beps^2}{n}
\tr\left(\widehat{\boldsymbol{\Sigma}} f^2(\widehat{\boldsymbol{\Sigma}}) \right), \\
\mathcal{E}_{f}^{\mathrm{est}}(\bX, \by)
&:= - 2\,\bbeta_0^\top\bW\bigl(\mathbf{I}_p - \boldsymbol{\Lambda} f(\boldsymbol{\Lambda})\bigr)
f(\boldsymbol{\Lambda})\bW^\top\frac{\bX^\top\boldsymbol{\varepsilon}}{n} \\
& \qquad \qquad \qquad + \left[\frac{\beps^\top \bX f^2(\widehat{\boldsymbol{\Sigma}}) \bX^\top \beps}{n^2} 
- \mathcal{V}_{f}^{\mathrm{est}}(\bX, \by)\right].
\end{align*}
Note that for all 
sufficiently large $n$ we can rewrite the bias term using the empirical measures introduced in Section~\ref{section: limits of empirical measures} as
$$\mathcal{B}^\mathrm{est}_{f}(\bX, \by) := \|\bbeta_0\|_2^2 \int \left(1 - xf(x)\right)^2 \,\mathrm{d}\widehat B_n(x).$$
Using Lemma~\ref{lemma: f_alpha} and  Lemma~\ref{lem:int-conv}, we deduce that
\begin{align*}
\mathcal{B}^\mathrm{est}_f \xrightarrow[]{a.s.} r^2\int (1 - xf(x))^2 \ \mathrm dF _\alpha(x).
\end{align*}
Consider now the variance term and note that for all 
sufficiently large $n$ we can rewrite it as 
$$\mathcal{V}^\mathrm{est}_f = \frac{{\sigma_{\beps}^2}\, p }{n} \int xf^2(x) \, \mathrm{d}F_{\widehat{\boldsymbol{\Sigma}}}(x).$$
Using Lemma~\ref{lemma:f_delta} and Lemma~\ref{lemma: hat D nj}, we know that $F_{\widehat{\boldsymbol{\Sigma}}} \xrightarrow[]{d} F_{\mathrm{MP}}$ and $\widehat D_n^j \xrightarrow{d} F_{\delta_j}$. 
Using Lemma~\ref{lem:int-conv} and  Lemma~\ref{lem:int-conv} again, we know that $$\mathcal{V}^\mathrm{est}_f \xrightarrow{\text{a.s.}} {\sigma_{\beps}^2} c \int xf^2(x)\, \mathrm{d}F_{\mathrm{MP}}.$$
Finally, $\mathcal{E}_{f}^{\mathrm{est}}(\bX, \by) \xrightarrow{\text{a.s.}} 0$ 
by the same concentration argument used in Section~\ref{proof: general formula}. This concludes the proof.
\end{proof}

\section{General Covariance Limits - Proof of Theorem~\ref{thm:general-covariance}}\label{app: I}
\subsection{Finite Deterministic Transforms and Kernels}
Recall the population measures $H_n,G_n$ from~\eqref{eq:general-covariance-measures} and the aspect ratio $c_n=p/n$. For $z\in\mathbb C^+$, let $v_n(z)$ be the unique solution in $\mathbb C^+$ of
\begin{equation}\label{eq:fixedpoint}
 \frac1{v_n(z)}
 = -z+c_n\int\frac{x}{1+xv_n(z)}\,dH_n(x) = -z + \frac1n \sum_{j = 1}^p \frac{\mathfrak{d}_j}{1 + \mathfrak{d}_j\, v_n(z)}.
\end{equation}

This is the deterministic companion Stieltjes transform
associated with $(H_n,c_n)$
\cite{silverstein1995analysis}. It is the Stieltjes transform
of a probability measure on $[0,\infty)$ and extends
holomorphically to $\mathcal U$, where
\eqref{eq:fixedpoint} continues to hold. For $\lambda > 0$, with $z = -\lambda$ in Equation \eqref{eq:fixedpoint}, we define $v_n(-\lambda)$ as the unique positive solution of
\[1=\lambda v_n(-\lambda)
   +c_n\int\frac{xv_n(-\lambda)}{1+xv_n(-\lambda)}\,dH_n(x).
\]
Recall that under Assumption~\ref{assm:spectral-convergence}, $H_n \Rightarrow H$, $G_n\Rightarrow G$, and $c_n\to c$.
Let $v$ denote the deterministic companion Stieltjes
transform associated with $(H,c)$, defined by replacing
$(H_n,c_n)$ with $(H,c)$ in \eqref{eq:fixedpoint}.
It is well-known that 
$v_n(z)\to v(z)$ for all $z\in\mathcal{U}$ \cite{ hastie2022surprises, starkDornemann2026, wonder}.

For $z,w\in\mathcal U$, define
\begin{align}
 \Delta_n(z,w)
 &=
 1-c_nv_n(z)v_n(w)
 \int\frac{x^2\,dH_n(x)}
 {(1+xv_n(z))(1+xv_n(w))},
 \label{eq:Delta}\\
 \mathcal{K}_n(z,w)
 &=
 \frac{\displaystyle
 \sigma_{\varepsilon}^2+\int\frac{x\,dG_n(x)}
 {(1+xv_n(z))(1+xv_n(w))}}
 {\Delta_n(z,w)}-\sigma_{\varepsilon}^2.
 \label{eq:kernel}
\end{align}
Let $\Delta,\mathcal{K}$ denote the corresponding expressions with
$(c_n,H_n,G_n,v_n)$ replaced by $(c,H,G,v)$.
The following lemma establishes that these
quotients are well-defined throughout $\mathcal{U}\times\mathcal{U}$, and that convergence holds locally uniformly.

\begin{lemma}\label{lem:double-kernel_identity}
Under Assumptions~\ref{assm:dgp}
and~\ref{assm:spectral-convergence}, the following hold.
\begin{enumerate}[label=(\alph*),ref=(\alph*)]
\item\label{part:kernel-identities}
The functions $\Delta_n,\Delta$ do not vanish on
$\mathcal{U}\times\mathcal{U}$. For $z,w\in\mathcal U$,
\begin{equation}\label{eq:divided}
 \Delta_n(z,w)^{-1}
 =
 \begin{cases}
 \displaystyle
 \frac{v_n(z)-v_n(w)}
      {(z-w)v_n(z)v_n(w)}, & z\ne w,\\[10pt]
 \displaystyle
 \frac{v_n'(z)}{v_n(z)^2}, & z=w.
 \end{cases}
\end{equation}
The same identities hold for $v,\Delta$.

\item\label{part:kernel-analytic}
The transforms $v_n$ converge locally uniformly on
$\mathcal U$ to $v$, together with all their fixed-order
derivatives. The holomorphic kernels $\Delta_n,\mathcal{K}_n$
converge locally uniformly on $\mathcal{U}\times\mathcal{U}$ to $\Delta,\mathcal{K}$,
together with all their fixed-order mixed derivatives.
\end{enumerate}
\end{lemma}

\subsection{Self-distillation with a common penalty}

Fix $M>L(1+\sqrt c)^2$ and set $\mathcal F_M=C([0,M];\mathbb R).$
Whenever a spectral shrinkage function is defined only on $[0,M]$, we define its
estimator using the continuous extension that equals $f(M)$
for $x>M$.

Almost surely, the spectrum of
$\widehat\bSigma$ is contained in $[0,M]$ for all sufficiently
large $n$ \cite{silverstein1995analysis}. Consequently, any two extensions agreeing on
$[0,M]$ give the same estimator eventually. In particular,
the limiting risk depends only on the restriction of the
spectral shrinkage function to $[0,M]$.

\begin{lemma}\label{lem:density}
Fix $\lambda>0$.  Given any $f \in \mathcal{F}_M$ and $\eps > 0$, there exists $k \in \bbN$ and $(\xi_1, \cdots, \xi_k \in \bbQ)$ such that, $f_{\mathrm{SD}}^{(k)}(x)$, with these parameters, satisfies
\[
 \|f_{\mathrm{SD}}^{(k)}-f\|_{\infty,[0,M]}\le\eps.
\]
\end{lemma}

\begin{lemma}\label{lem:extension}
There is a deterministic constant $C<\infty$ with the
following properties.
\begin{enumerate}[label=(\alph*),ref=(\alph*)]

\item\label{part:extension-bound}
Almost surely, for every $f,g\in\mathcal F_M$,
\begin{equation}\label{eq:continuity}
 \limsup_{n\to\infty}
 \left|
 \|\widehat{\bbeta}_f-\bbeta_0\|_\bSigma
 -
 \|\widehat{\bbeta}_g-\bbeta_0\|_\bSigma
 \right|
 \le C\|f-g\|_{\infty,[0,M]}.
\end{equation}

\item\label{part:extension-limit}\label{part:extension-event}
Suppose a countable dense subset of $\mathcal F_M$ has
finite deterministic almost-sure risk limits. Their square
roots have a unique continuous extension to $\mathcal F_M$.
Denoting its square by $\mathcal R^{\mathrm{pred}}(f;c)$,
for every $f\in\mathcal F_M$,
\[
 \|\widehat{\bbeta}_f-\bbeta_0\|_\bSigma^2
\xrightarrow{\text{a.s.}}
 \mathcal R^{\mathrm{pred}}(f;c)\in[0,\infty).
\]
Moreover,
\begin{equation}\label{eq:lim-continuity}
 \left|
 \sqrt{\mathcal R^{\mathrm{pred}}(f;c)}
 -
 \sqrt{\mathcal R^{\mathrm{pred}}(g;c)}
 \right|
 \le C\|f-g\|_{\infty,[0,M]},
 \qquad f,g\in\mathcal F_M.
\end{equation}
\end{enumerate}

For $f\in C([0,\infty);\mathbb R)$, define
\[
 \mathcal R^{\mathrm{pred}}(f;c)
 :=
 \mathcal R^{\mathrm{pred}}(f|_{[0,M]};c).
\]
The same risk convergence then holds simultaneously for
all such $f$ on the same probability-one event.
\end{lemma}

\subsection{Proof Of Theorem~\ref{thm:general-covariance}(a)} We first identify the part of the ridge prediction risk that depends on
the signal. For $z,w\in\mathcal U$, Theorem~3.2 and equation~(3.3) of
Stark-D\"ornemann \cite{starkDornemann2026}, together with the
divided-difference identity \eqref{eq:divided}, give
\begin{equation}\label{eq:two-DE}
 (\widehat{\bSigma}-z\bI_p)^{-1}\bSigma
 (\widehat{\bSigma}-w\bI_p)^{-1}
 \asymp
 \frac{
   (\bI_p+v_n(z)\bSigma)^{-1}\bSigma
   (\bI_p+v_n(w)\bSigma)^{-1}
 }{zw\,\Delta_n(z,w)}.
\end{equation}
Testing \eqref{eq:two-DE} against $\bbeta_0\bbeta_0^\top$ gives the
deterministic equivalent for the signal term. The next lemma combines it with the terms involving the training noise. At $z = w = -\lambda$,
the resulting quantity is the excess prediction risk of ridge
conditional on $(\bX,\by)$. We keep $z$ and $w$ separate because the SD risk also involves cross-terms between different basis functions.

\begin{lemma}\label{lem:cross_risk}
 Define, for $z,w\in\mathcal U$,
\[
 \mathcal C_n(z,w)
 =
 \left((\widehat{\bSigma}-z\bI_p)^{-1}
       \frac{\bX^\top\by}{n}-\bbeta_0\right)^\top
 \bSigma
 \left((\widehat{\bSigma}-w\bI_p)^{-1}
       \frac{\bX^\top\by}{n}-\bbeta_0\right).
\]
Then, for every
nonnegative integers $r,s$,
$
 \partial_z^r\partial_w^s(\mathcal C_n-\mathcal K_n)
\xrightarrow{\text{a.s.}} 0,
 $ and $\partial_z^r\partial_w^s\mathcal C_n
\xrightarrow{\text{a.s.}} \partial_z^r\partial_w^s\mathcal K,$ 
locally uniformly on $\mathcal U\times\mathcal U$.
\end{lemma}

To pass from ridge to self-distillation, recall~\eqref{eq:ridge_basis}. Each such function can be obtained from the ridge spectral shrinkage function by differentiating with respect to its penalty. Recall that
\[
 D_j^\lambda\!\left(\frac{1}{x+\lambda}\right)
 =\frac{x^j}{(x+\lambda)^{j+1}},
 \qquad
 D_j^\lambda 1=1.
\]
Thus $D_j^\lambda$ replaces the ridge estimator by the $j$th SD basis
estimator and leaves the subtracted signal $\bbeta_0$ unchanged.
Applying this identity in each argument of
$\mathcal C_n(-\lambda,-\tau)$, and then setting $\tau=\lambda$,
gives the cross-term for basis functions $i$ and $j$.
Lemma~\ref{lem:derivatives} identifies its limit.

\begin{lemma}\label{lem:derivatives}
For $\lambda > 0$, let $\bQ_{\lambda} = (\widehat{\bSigma}+ \lambda \bI_p)^{-1}$. Then, almost surely, for every
$i,j\ge0$ and $\lambda>0$,
\begin{equation}\label{eq:basis-cross}
\begin{aligned}
 &\left(\widehat{\bSigma}^{\,i}\bQ_\lambda^{i+1}
        \frac{\bX^\top\by}{n}-\bbeta_0\right)^\top
  \bSigma
  \left(\widehat{\bSigma}^{\,j}\bQ_\lambda^{j+1}
        \frac{\bX^\top\by}{n}-\bbeta_0\right) -
 \left.D_i^\lambda D_j^\tau\mathcal K_n(-\lambda,-\tau)
 \right|_{\tau=\lambda}
 \xrightarrow{\text{a.s.}} 0.
\end{aligned}
\end{equation}
In particular, for every positive common penalty, every fixed number of steps
$k$, and all real weights,
\begin{equation}\label{eq:SD-risk}
 \mathcal R^{\mathrm{pred}}(f_{\mathrm{SD}}^{(k)};c)
 =
 \sum_{i,j=0}^k a_i a_j
 \left.D_i^\lambda D_j^\tau\mathcal K(-\lambda,-\tau)
 \right|_{\tau=\lambda}.
\end{equation}
\end{lemma}

Equation~\eqref{eq:SD-risk} gives a deterministic risk limit for each
fixed-depth SD spectral shrinkage function. For part~(a) of
Theorem~\ref{thm:general-covariance}, we use these limits on a countable
dense family and extend them by continuity. For part~(b), we use the
same density result at the prescribed penalty $\lambda$ to approximate the infimum of the limiting risk.

\begin{proof}[Proof of Theorem~\ref{thm:general-covariance}]
    For part (a), fix the common penalty $\lambda=1$. Lemma~\ref{lem:density}
gives the countable dense family
\[
 \mathcal D=\bigcup_{k\ge0}
       \{f_{\mathrm{SD}}^{(k)}:\lambda=1,\ (\xi_1,\ldots,\xi_k)\in\mathbb Q^k\}
       \subset\mathcal{C}([0, \infty); \mathbb{R}).
\]
Lemma~\ref{lem:derivatives} supplies a finite deterministic almost-sure risk limit for every member of this family. Lemma~\ref{lem:extension}\ref{part:extension-event}
therefore extends the limits uniquely to every $f\in\mathcal{C}([0, \infty); \mathbb{R})$ on one
probability-one event, which proves part (a).

For part (b), define
\[
 \mathcal{R}_*:=\inf_{f\in\mathcal{C}([0, \infty); \mathbb{R})}\mathcal{R}^{\mathrm{pred}}(f;c).
\]
The zero spectral shrinkage function has limiting risk $\int x\,dG(x)$, so $\mathcal R_\ast$ is finite. Fix a positive penalty $\lambda$ and $\epsilon>0$ and choose $f\in\mathcal{C}([0, \infty); \mathbb{R})$ with
\[
 \mathcal{R}^{\mathrm{pred}}(f;c)<\mathcal{R}_*+\epsilon/2.
\]
Using Lemma~\ref{lem:density} we can choose SD spectral shrinkage functions $f_m$, all using
this same penalty, such that $\|f_m-f\|_{\infty, [0, M]}\to0$.
Lemma~\ref{lem:extension}\ref{part:extension-limit} now gives
\[
 \left|\sqrt{\mathcal{R}^{\mathrm{pred}}(f_m;c)}-\sqrt{\mathcal{R}^{\mathrm{pred}}(f;c)}\right|
 \le C\|f_m-f\|_{\infty, [0, M]}\longrightarrow0.
\]
Consequently, one SD spectral shrinkage function has limiting risk below
$\mathcal{R}_*+\epsilon$. Setting $\xi_1 = 0$,  $(\xi_2, \dots, \xi_{k+1})$ can recover any $k$-step distillation so 
$$\inf_{\xi_1,\ldots,\xi_{k+1}\in\mathbb{R}}\mathcal{R}^{\mathrm{pred}}(f^{k+1}_{\mathrm{SD}};c) \leq \inf_{\xi_1,\ldots,\xi_k\in\mathbb{R}}\mathcal{R}^{\mathrm{pred}}(f_{\mathrm{SD}}^{(k)};c)$$ the same upper bound holds at every larger depth. Every SD spectral shrinkage function belongs to $\mathcal{C}([0, \infty); \mathbb{R})$, so the corresponding infima
are also bounded below by $\mathcal{R}_*$. Thus
\[
 \mathcal{R}_*\le\lim_{k\to\infty}\inf_{\xi_1,\ldots,\xi_k}\mathcal{R}^{\mathrm{pred}}(f_{\mathrm{SD}}^{(k)};c)
      \le \mathcal{R}_*+\epsilon.
\]
Letting $\epsilon\downarrow0$ proves part (b).
\end{proof}

\subsection{Proof of Lemma~\ref{lem:double-kernel_identity}\ref{part:kernel-identities}}
The Stieltjes property implies, for every $z\in\mathcal{U}$ and
$x\ge0$, $v_n(z)\ne0,$ and $1+xv_n(z)\ne0.$ Indeed, off the real axis $v_n(z)$ has nonzero imaginary part,
and on the negative spectral axis it is positive. The same holds for $v$.
Subtracting the fixed-point equations at two distinct spectral parameters gives
\[
 \Delta_n(z,w)
 \frac{v_n(z)-v_n(w)}
      {(z-w)v_n(z)v_n(w)}=1.
\]
This proves nonvanishing off the diagonal. Since $\Delta_n$ is holomorphic in each variable and $v_n$ is
holomorphic on $\mathcal{U}$, letting $w\to z$ in the identity
above gives $\Delta_n(z,z)v_n'(z)=v_n(z)^2.$
Because $v_n(z)\ne0$, it follows that $\Delta_n(z,z)\ne0$, and
\[
 \frac{1}{\Delta_n(z,z)}
 =\frac{v_n'(z)}{v_n(z)^2}.
\]
The same argument applies to $\Delta$ and $v$.

\subsection{Proof of Lemma~\ref{lem:double-kernel_identity}\ref{part:kernel-analytic}}

It is well-known that $v_n(z)\to v(z)$ for every $z\in\mathcal{U}$.
For every compact $B\subset\mathcal{U}$, the Stieltjes representation gives
\[
 \sup_n\sup_{z\in B}|v_n(z)|
 \le
 \frac1{\operatorname{dist}(B,[0,\infty))}
 <\infty.
\]
Thus $(v_n)$ is locally uniformly bounded. Vitali's
theorem~\ref{thm: vitali-porter}, applied using convergence on the
negative real axis, implies that $v_n\to v$ locally uniformly
on $\mathcal{U}$.

Fix a compact $A\subset\mathcal{U} \times \mathcal{U}$, and write
\[
 g_n(x;z,w)=
 \frac{x^2}{(1+xv_n(z))(1+xv_n(w))},
 \qquad
 g(x;z,w)=
 \frac{x^2}{(1+xv(z))(1+xv(w))}.
\]
Since the limiting denominators do not vanish, $g$ is continuous and consequently the family
$\{g(\,\cdot\,;z,w):(z,w)\in A\}$ is uniformly bounded and equicontinuous on $[\ell,L]$.
Using \cite[Theorem 3.1]{rangarao1962},
applied to $H_n\Rightarrow H$, therefore gives
\[
 \sup_{(z,w)\in A}
 \left|
 \int g(x;z,w)\,dH_n(x)
 -
 \int g(x;z,w)\,dH(x)
 \right|
\xrightarrow{\text{a.s.}}0.
\]
Moreover, local uniform convergence $v_n\to v$ and nonvanishing of the limiting denominators imply
\[
 \sup_{\substack{x\in[\ell,L]\\(z,w)\in A}}
 |g_n(x;z,w)-g(x;z,w)|\longrightarrow0.
\]
Thus
\begin{align*}
 &\sup_{(z,w)\in A}
 \left|
 \int g_n(x;z,w)\,dH_n(x)
 -
 \int g(x;z,w)\,dH(x)
 \right|\\
 &\quad\le
 \sup_{\substack{x\in[\ell,L]\\(z,w)\in A}}|g_n-g|
 +
 \sup_{(z,w)\in A}
 \left|\int g(x;z,w)\,d(H_n-H)(x)\right|
\xrightarrow{\text{a.s.}}0.
\end{align*}
Together with $c_n\to c$ and local uniform convergence of $v_n$,
this proves $\Delta_n\to\Delta$ locally uniformly. Next, replacing $x^2$ by $x$ and $H_n,H$ by $G_n,G$ gives uniform convergence of the integral in the numerator of $\mathcal{K}_n$.
Since $\Delta$ is nowhere zero, and $\Delta_n \to \Delta$ locally uniformly, division gives
$\mathcal{K}_n\to\mathcal{K}$ locally uniformly.

To conclude, note that both kernels are continuous and separately holomorphic on $\mathcal{U} \times \mathcal{U}$, hence jointly holomorphic by Osgood's lemma. Their locally uniform convergence, together with Cauchy's multivariate integral formula therefore implies that, for every fixed pair of nonnegative integers $r,s$,
\[
 \partial_z^r\partial_w^s\Delta_n
\xrightarrow{\text{a.s.}}
 \partial_z^r\partial_w^s\Delta,
 \qquad
 \partial_z^r\partial_w^s\mathcal{K}_n
\xrightarrow{\text{a.s.}}
 \partial_z^r\partial_w^s\mathcal{K}
\]
locally uniformly on $\mathcal{U} \times \mathcal{U}$.

\subsection{Proof of Lemma~\ref{lem:density}}

Fix $\lambda>0$, $f\in\mathcal{C}([0, \infty); \mathbb{R})$ and $\epsilon>0$. Set
\[
 u=\frac{x}{x+\lambda},
 \qquad
 b=\frac{M}{M+\lambda}<1,
\]
and define the continuous function
\[
 h(u)=
 \lambda\,
 \frac{f\!\left(\frac{\lambda u}{1-u}\right)-\frac{1-u}{\lambda}}
      {(1-u)^2},
 \qquad u\in[0,b].
\]
By the Weierstrass approximation theorem, choose polynomial $
 \tilde P$
such that $\|\tilde P-h\|_{\infty,[0,b]}
 <{\lambda\epsilon}/{2}$. Since the degree of $\tilde{P}$ is now fixed, we can further approximate each monomial coefficient with numbers in $\mathbb Q\setminus\{0\}$ and put $P(u)=\sum_{j=0}^{k-1}p_j u^j,$
with $p_j \in \mathbb Q\setminus\{0\}$ such that 
$ \|P-h\|_{\infty,[0,b]}<\lambda\epsilon.$ Let $ Q(u)=1+(1-u)P(u) = \sum_j a_j u^j $ and note that its coefficients satisfy $\sum_{j=0}^k a_j=q(1)=1,$ and $T_j:=\sum_{r=j}^k a_r=-p_{j-1}\ne0,$ for all $1\le j\le k$.  Thus, we can define the  weights by
 $\xi_{k-j+1}=T_j/T_{j-1}$, where the empty product is $T_0=1$. Therefore, for $j<k$,
\[
 (1-\xi_{k-j})\prod_{r=k-j+1}^k\xi_r
 =\left(1-\frac{T_{j+1}}{T_j}\right)T_j
 =T_j-T_{j+1}=a_j,
\]
and  finally $\prod_{r=1}^k\xi_r=T_k=a_k$. Thus we can realize the self-distillation spectral shrinkage function
\[
 f_\mathrm{SD}^k(x)
 =\frac{Q(x/(x+\lambda))}{x+\lambda}.
\]
By the definition of $h$,
\[
 f_\mathrm{SD}^k(x)-f(x)
 =\frac{(1-u)^2}{\lambda}\bigl(P(u)-h(u)\bigr),
\]
Consequently,
\[
 \| f_{\mathrm{SD}}^{(k)}-f\|_\infty
 \le\lambda^{-1}\|P-h\|_{\infty,[0,b]}
 <\epsilon.
\]
Since $f$ and $\epsilon$ were arbitrary, self-distillation spectral shrinkage functions with rational weights are uniformly dense in $\mathcal{F}_M$.

\subsection{Proof of Lemma~\ref{lem:extension}\ref{part:extension-bound}}
Using~\cite{Silverstein1995} and the strong law for the squared noise, on one probability-one event,
\[
 \limsup_n\|\widehat{\bSigma}\|\le L(1+\sqrt c)^2<M,
 \qquad \frac{\|\beps\|}{\sqrt n}\longrightarrow\sigma_\beps.
\]
On this event, for all sufficiently large $n$,
\begin{align*}
 \left\|\frac{\bX^\top \by}{n}\right\|
 &\le\|\widehat{\bSigma}\|\|\bbeta_0\|
       +\frac{\|\bX\|}{\sqrt n}\frac{\|\beps\|}{\sqrt n} \le MB+\sqrt M(\sigma_\beps+1).
\end{align*}
The reverse triangle inequality and the spectral calculus then give,
simultaneously for every $f,g\in\mathcal{C}([0, \infty); \mathbb{R})$,
\begin{align*}
 \big|\|\widehat{\bbeta}_f-\bbeta_0\|_\bSigma-\|\widehat{\bbeta}_g-\bbeta_0\|_\bSigma\big|
 &\le\|\widehat{\bbeta}_f-\widehat{\bbeta}_g\|_\bSigma\\
 &\le\sqrt L\|f(\widehat{\bSigma})-g(\widehat{\bSigma})\|_{\op}
                  \left\|\frac{\bX^\top \by}{n}\right\|\\
 &\le C\|f-g\|_{\infty, [0,M]},
\end{align*}
where the deterministic constant is $C=\sqrt{LM}(\sqrt M B+\sigma_\beps+1)$.

\subsection{Proof of Lemma~\ref{lem:extension}\ref{part:extension-limit}}

Let $\mathcal D$ be the countable dense subset in the lemma.
Intersect the probability-one event from the preceding part with
the convergence events for all $g\in\mathcal D$, and denote the
resulting probability-one event by $\Omega$. Write
\[
 r(g)=\lim_n\|\widehat{\bbeta}_g-\bbeta_0\|_\bSigma,
 \qquad g\in\mathcal D.
\]
These limits are finite, nonnegative and deterministic. Passing
to the limit in~\eqref{eq:continuity}
\[
 |r(g)-r(h)|\le C\|g-h\|_{\infty,[0,M]},
 \qquad g,h\in\mathcal D.
\]

We extend $r$ by the standard dense-set extension argument. For $f\in\mathcal{C}([0, \infty); \mathbb{R})$, choose $g_m\in\mathcal D$ such that
$\|g_m-f\|_{\infty,[0,M]}\to0$. The Lipschitz condition implies that $\{r(g_m)\}_m$ is a Cauchy sequence so completeness of $\mathbb{R}$ gives a finite nonnegative limit. Define
\[
 r(f):=\lim_m r(g_m),
 \qquad
 \mathcal{R}^{\mathrm{pred}}(f;c):=r(f)^2.
\]
It is well-known that this definition is independent of the approximating sequence. Thus $r(f)$ is deterministic. Approximating any two spectral shrinkage functions
by elements of $\mathcal D$ yields $|r(f)-r(g)|\le C\|f-g\|_{\infty,[0,M]}$, for all $f,g\in\mathcal{C}([0, \infty); \mathbb{R})$, which proves \eqref{eq:lim-continuity}.

It remains to establish risk convergence on $\Omega$.
Fix an arbitrary outcome in $\Omega$, a spectral shrinkage function $f\in\mathcal{C}([0, \infty); \mathbb{R})$,
and approximants $g_m\in\mathcal D$ as above.
For all sufficiently large $n$, \eqref{eq:continuity} gives
\begin{align*}
 \bigl|\|\widehat{\bbeta}_f-\bbeta_0\|_\bSigma-r(f)\bigr|
 &\le
 \bigl|\|\widehat{\bbeta}_f-\bbeta_0\|_\bSigma
       -\|\widehat{\bbeta}_{g_m}-\bbeta_0\|_\bSigma\bigr| +
 \bigl|\|\widehat{\bbeta}_{g_m}-\bbeta_0\|_\bSigma-r(g_m)\bigr|
 +|r(g_m)-r(f)|\\
 &\le
 2C\|f-g_m\|_{\infty,[0,M]}
 +\bigl|\|\widehat{\bbeta}_{g_m}-\bbeta_0\|_\bSigma-r(g_m)\bigr|.
\end{align*}
Taking the upper limit in $n$ gives
\[
 \limsup_{n\to\infty}
 \bigl|\|\widehat{\bbeta}_f-\bbeta_0\|_\bSigma-r(f)\bigr|
 \le 2C\|f-g_m\|_{\infty,[0,M]}.
\]
Letting $m\to\infty$ and then squaring yields
$\|\widehat{\bbeta}_f-\bbeta_0\|_\bSigma^2
\xrightarrow{\text{a.s.}} r(f)^2=\mathcal{R}^{\mathrm{pred}}(f;c).$

\subsection{Proof of Lemma~\ref{lem:cross_risk}}\label{app:cross_risk}

Fix $z,w\in\mathcal U$. Since $(\widehat{\bSigma}-z\bI_p)^{-1}\widehat{\bSigma}=\bI_p+z(\widehat{\bSigma}-z\bI_p)^{-1}$,
the estimation error in the definition of $\mathcal C_n$ satisfies
\[
 (\widehat{\bSigma}-z\bI_p)^{-1}\frac{\bX^\top\by}{n}-\bbeta_0
 =z(\widehat{\bSigma}-z\bI_p)^{-1}\bbeta_0+(\widehat{\bSigma}-z\bI_p)^{-1}\frac{\bX^\top\beps}{n}.
\]
Multiplying the errors at $z$ and $w$ gives
\begin{align}
 \mathcal C_n(z,w)
 ={}&zw\bbeta_0^\top(\widehat{\bSigma}-z\bI_p)^{-1}\bSigma(\widehat{\bSigma}-w\bI_p)^{-1}\bbeta_0 \notag\\
 &+\frac{z}{n}\bbeta_0^\top(\widehat{\bSigma}-z\bI_p)^{-1}\bSigma(\widehat{\bSigma}-w\bI_p)^{-1}\bX^\top\beps
  +\frac{w}{n}\beps^\top\bX(\widehat{\bSigma}-z\bI_p)^{-1}\bSigma(\widehat{\bSigma}-w\bI_p)^{-1}\bbeta_0 \notag\\
 &+\frac{1}{n^2}\beps^\top\bX(\widehat{\bSigma}-z\bI_p)^{-1}\bSigma(\widehat{\bSigma}-w\bI_p)^{-1}\bX^\top\beps.
 \label{eq:cross-expansion}
\end{align}

As in the proof of Proposition~\ref{prop: general formula}, the two
terms linear in $\beps$ vanish almost surely, and the quadratic noise
term can be replaced by its conditional expectation given $\bX$.
For complex $z,w$, the same argument applies to the real and imaginary
parts, since for each fixed $z\in\mathcal U$,
\[
 \|(\widehat{\bSigma}-z\bI_p)^{-1}\|_{\mathrm{op}}
 \le \frac{1}{\operatorname{dist}(z,[0,\infty))},
 \qquad
 \frac{\|(\widehat{\bSigma}-z\bI_p)^{-1}\bX^\top\|_{\mathrm{op}}}{\sqrt n}
 \le \sup_{x\ge0}\frac{\sqrt{x}}{|x-z|}<\infty.
\]
Thus
\[
 \mathcal C_n(z,w)
 -zw\bbeta_0^\top(\widehat{\bSigma}-z\bI_p)^{-1}\bSigma(\widehat{\bSigma}-w\bI_p)^{-1}\bbeta_0
 -\frac{\sigma_\varepsilon^2}{n}
       \tr\bigl(\bSigma\widehat{\bSigma}(\widehat{\bSigma}-z\bI_p)^{-1}(\widehat{\bSigma}-w\bI_p)^{-1}\bigr)
 \xrightarrow{\mathrm{a.s.}}0.
\]

Using~\eqref{eq:two-DE} with $\bTheta_n=\bbeta_0\bbeta_0^\top$,
the signal term differs almost surely by $o(1)$ from
\[
 \frac{1}{\Delta_n(z,w)}
 \int\frac{x\,dG_n(x)}{(1+xv_n(z))(1+xv_n(w))}.
\]

For the noise trace, we claim that
\[
 \frac1n\tr\bigl(\bSigma\widehat{\bSigma}(\widehat{\bSigma}-z\bI_p)^{-1}(\widehat{\bSigma}-w\bI_p)^{-1}\bigr)
 -\bigl[\Delta_n(z,w)^{-1}-1\bigr]
 \xrightarrow{\mathrm{a.s.}}0.
\]
It is well-known that
\begin{equation}\label{eq:one-DE}
 (\widehat{\bSigma}-z\bI_p)^{-1}\asymp-\frac1z(\bI_p+v_n(z)\bSigma)^{-1},
\end{equation}
see for example~\cite{RubioMestre2011SpectralConvergence, hastie2022surprises, wonder}.  Testing against $\bSigma/n$ and using the fixed-point equation for
$v_n(z)$ gives
\[
 \frac{z}{n}\tr\bigl(\bSigma(\widehat{\bSigma}-z\bI_p)^{-1}\bigr)
 -\left[-\frac1{v_n(z)}-z\right]
 \xrightarrow{\mathrm{a.s.}}0.
\]
For $z\ne w$, the resolvent identity gives
\[
 \widehat{\bSigma}(\widehat{\bSigma}-z\bI_p)^{-1}(\widehat{\bSigma}-w\bI_p)^{-1}
 =\frac{z(\widehat{\bSigma}-z\bI_p)^{-1}-w(\widehat{\bSigma}-w\bI_p)^{-1}}{z-w}.
\]
Taking the divided difference of the preceding trace equivalent yields
\[
 \frac1n\tr\bigl(\bSigma\widehat{\bSigma}(\widehat{\bSigma}-z\bI_p)^{-1}(\widehat{\bSigma}-w\bI_p)^{-1}\bigr)
 -\left[
    \frac{v_n(w)^{-1}-v_n(z)^{-1}}{z-w}-1
  \right]
 \xrightarrow{\mathrm{a.s.}}0.
\]
By~\eqref{eq:divided}, the bracket equals $\Delta_n(z,w)^{-1}-1$. For $w=z$, apply the differentiation rule
\cite[Theorem~11]{wonder} to the same tested equivalent. Since
\[
 \frac{d}{dz}\{z(\widehat{\bSigma}-z\bI_p)^{-1}\}
 =(\widehat{\bSigma}-z\bI_p)^{-1}+z(\widehat{\bSigma}-z\bI_p)^{-2}
 =\widehat{\bSigma}(\widehat{\bSigma}-z\bI_p)^{-2},
\]
we obtain
\[
 \frac1n\tr\bigl(\bSigma\widehat{\bSigma}(\widehat{\bSigma}-z\bI_p)^{-2}\bigr)
 -\left[\frac{v_n'(z)}{v_n(z)^2}-1\right]
 \xrightarrow{\mathrm{a.s.}}0.
\]
The diagonal identity in~\eqref{eq:divided} identifies the bracket as
$\Delta_n(z,z)^{-1}-1$, proving the claim also on the diagonal.

Combining the signal and noise terms gives
\[
 \mathcal C_n(z,w)
 -
 \left[
 \frac{\displaystyle \sigma_\varepsilon^2+
       \int\frac{x\,dG_n(x)}{(1+xv_n(z))(1+xv_n(w))}}
      {\Delta_n(z,w)}
 -\sigma_\varepsilon^2
 \right]
 \xrightarrow{\mathrm{a.s.}}0.
\]
The bracket is exactly $\mathcal K_n(z,w)$. This proves the claimed
convergence for each fixed pair $z,w\in\mathcal U$.

Next, we show that the convergence above holds on $\mathcal{U} \times \mathcal{U}$, on the same measure 1 event.  First, we intersect the convergence events for positive rational ridge penalties $(\lambda,\tau)$ with the event $\|\beps\|^2/n\to\sigma_{\varepsilon}^2$. Call the resulting probability-one event $\Omega$. This gives convergence
of $\mathcal C_n-\mathcal{K}_n$ on $\mathbb Q_{<0}^2$.
For each outcome in $\Omega$, the functions $\mathcal C_n$ are
holomorphic on $\mathcal{U} \times \mathcal{U}$  and also locally uniformly bounded. Indeed, spectral calculus gives, for every compact $A\subset\mathcal{U}$,
\[
 \sup_n\sup_{z\in A}
 \left(\|(\widehat{\bSigma} - z\mathbf{I}_p)^{-1}\|_{\op}
       +\frac{\|(\widehat{\bSigma} - z\mathbf{I}_p)^{-1}\bX^\top\|_{\op}}{\sqrt n}\right)<\infty.
\]
For the second term, put $d=\operatorname{dist}(A,[0,\infty))>0$;
then
\[
 \sup_{z\in A,\,x\ge0}\frac{x}{|x-z|^2}
 \le \frac1d+\frac{\sup_{z\in A}|z|}{d^2}<\infty.
\]
The exact identity
\[
 (\widehat{\bSigma} - z\mathbf{I}_p)^{-1}\frac{\bX^\top \by}{n}-\bbeta_0
 =z(\widehat{\bSigma} - z\mathbf{I}_p)^{-1}\bbeta_0+(\widehat{\bSigma} - z\mathbf{I}_p)^{-1}\frac{\bX^\top\beps}{n}
\]
and the bounded signal norm therefore imply
\[
 \sup_{z\in A}\left\|(\widehat{\bSigma} - z\mathbf{I}_p)^{-1}\frac{\bX^\top \by}{n}-\bbeta_0\right\|
 \le C_A\left(1+\frac{\|\beps\|}{\sqrt n}\right).
\]
Consequently $\mathcal C_n$, and hence $\mathcal C_n-\mathcal{K}_n$, are locally
bounded on this event. The same Vitali argument from
Lemma~\ref{lem:double-kernel_identity} extends convergence to local uniform convergence on $\mathcal{U} \times \mathcal{U}$. Finally, the locally uniform convergence $\mathcal{K}_n\to\mathcal{K}$ proves the second claim. 

Convergence of all fixed-order mixed derivatives follows from
Cauchy's integral formula. Fix nonnegative integers $r,s$ and
a compact $A\subset\mathcal{U}\times\mathcal{U}$. Choose $\rho>0$ and a compact
$A'\subset\mathcal{U}\times\mathcal{U}$ such that, for every $(z,w)\in A$, $
 \{(\zeta,\omega):
   |\zeta-z|\le\rho,\ |\omega-w|\le\rho\}
 \subset A'.$ Cauchy's formula on this set gives
\begin{align*}
 &\partial_z^r\partial_w^s(\mathcal C_n-\mathcal{K}_n)(z,w) =
 \frac{r!s!}{(2\pi i)^2}
 \oint_{|\zeta-z|=\rho}
 \oint_{|\omega-w|=\rho}
 \frac{(\mathcal C_n-\mathcal{K}_n)(\zeta,\omega)}
      {(\zeta-z)^{r+1}(\omega-w)^{s+1}}
 \,d\omega\,d\zeta,
\end{align*}
where both circles are oriented counterclockwise.
Since each circle has length $2\pi\rho$,
\[
 \sup_{(z,w)\in A}
 \left|
 \partial_z^r\partial_w^s(\mathcal C_n-\mathcal{K}_n)(z,w)
 \right|
 \le
 \frac{r!s!}{\rho^{r+s}}
 \sup_{(\zeta,\omega)\in A'}
 |(\mathcal C_n-\mathcal{K}_n)(\zeta,\omega)|
\xrightarrow{\text{a.s.}}0.
\]
The same argument applied to $\mathcal C_n-\mathcal{K}$ proves
the second claim.

\subsection{Proof of Lemma~\ref{lem:derivatives}}\label{app:derivatives}
Using Leibniz's rule we have\[
 D_j^\lambda h
   =\sum_{r=0}^j\binom jr\frac{\lambda^r}{r!}
                         \partial_\lambda^r h.
\]
By the chain rule,
$\partial_\lambda^r\partial_\tau^s\mathcal{K}_n(-\lambda,-\tau)
 =(-1)^{r+s}(\partial_z^r\partial_w^s\mathcal{K}_n)(-\lambda,-\tau),$
and the same identity holds for $\mathcal C_n$ and $\mathcal{K}$.
Hence the derivative convergence proved in Lemma~\ref{lem:cross_risk} also applies to
$D_i^\lambda D_j^\tau$ acting on these compositions.  Applying the scalar identity to the spectral
decomposition of $\widehat{\bSigma}$ gives
\[
 D_j^\lambda(\widehat{\bbeta}_\lambda-\bbeta_0)
 =\widehat{\bSigma}^j\bQ_{\lambda}^{j+1}\frac{\bX^\top \by}{n}-\bbeta_0.
\]
The penalty derivatives act on separate factors of
$\mathcal C_n(-\lambda,-\tau)$, so
\begin{align*}
 D_i^\lambda D_j^\tau\mathcal C_n(-\lambda,-\tau)
 ={}&\left(\widehat{\bSigma}^i\bQ_{\lambda}^{i+1}\frac{\bX^\top \by}{n}-\bbeta_0\right)^\top\bSigma
 \left(\widehat{\bSigma}^j\bQ_{\tau}^{j+1}\frac{\bX^\top \by}{n}-\bbeta_0\right).
\end{align*}
Setting $\tau=\lambda$ after differentiating proves
\eqref{eq:basis-cross}.

For the SD risk formula, recall the expansion
\[
 f_{\mathrm{SD}}^{(k)}(x)=\sum_{j=0}^k a_j\frac{x^j}{(x+\lambda)^{j+1}},
\]
with  $\sum_{j=0}^k a_j=1$. Consequently,
\[
 \widehat{\bbeta}_{f_{\mathrm{SD}}^{(k)}}-\bbeta_0
 =
 \sum_{j=0}^k a_j
 \left(\widehat{\bSigma}^jQ_\lambda^{j+1}
                \frac{\bX^\top \by}{n}-\bbeta_0\right).
\]
Expanding its squared $\bSigma$-norm and applying
\eqref{eq:basis-cross} gives
\[
 \|\widehat{\bbeta}_{f_{\mathrm{SD}}^{(k)}}-\bbeta_0\|_\bSigma^2
 -
 \sum_{i,j=0}^k a_i a_j
 \left.D_i^\lambda D_j^\tau\mathcal{K}_n(-\lambda,-\tau)
 \right|_{\tau=\lambda}
 \xrightarrow[]{\text{a.s.}}0.
\]
The convergence of the fixed-order derivatives of $\mathcal{K}_n$
to those of $\mathcal{K}$ now proves \eqref{eq:SD-risk}. 

\section{Self-Distillation GCV for General Covariance}\label{app: J}
\subsection{Proof of Theorem~\ref{thm:gcv}}\label{app:gcv-accuracy-proof}
To treat a growing number of steps in SD, we first bound a cross-product at two penalties. For the proof only, extend the penalties to
$\operatorname{Re}\lambda,\operatorname{Re}\tau>0$ and write
\begin{equation}\label{eq:gcvapp-cross}
\begin{aligned}
 \bg_\lambda&=\bW_\lambda\boldsymbol y/\widehat v_\lambda,\qquad
 \widehat C_n(\lambda,\tau)=\bg_\lambda^\top \bg_\tau/n,\\
 C_n(\lambda,\tau)&=\sigma_\varepsilon^2+
 (\bb_0(\lambda)-\boldsymbol\beta_0)^\top\boldsymbol\Sigma
 (\bb_0(\tau)-\boldsymbol\beta_0)
\end{aligned}
\end{equation}
Let $\bb_D(\lambda)$ denote ridge after deleting the observations in $D$,
always keeping the divisor $n$:
\[
 \bb_D(\lambda)
 =\left(\frac{\boldsymbol X_{-D}^{\top}\boldsymbol X_{-D}}{n} +\lambda\mathbf I_p\right)^{-1} \frac{\boldsymbol X_{-D}^{\top}\boldsymbol y_{-D}}{n}.
\]
where $\bb_{\varnothing}(\lambda)=\bb_0(\lambda)$.
\begin{lemma}\label{lem:gcvapp-delete}
For deterministic $D$ with $|D|\le2$ and $\operatorname{Re}\lambda\ge\alpha$:
\begin{enumerate}
\item[(a)] $\|\bb_D(\lambda)\|_{L^4}\le C_\alpha$, and, if $i\in D$,
$\|y_i-\boldsymbol x_i^\top \bb_D(\lambda)\|_{L^4}\le C_\alpha$.
\item[(b)] If $|D|\le1$ and $i\notin D$, then
\[
 \|\bb_D(\lambda)-\bb_{D\cup\{i\}}(\lambda)\|_{L^4}\le C_\alpha n^{-1/2}.
\]
In particular, for $i\ne j$,
\[
 \|\boldsymbol x_i^\top (\bb_{\{i\}}(\lambda)-\bb_{\{i,j\}}(\lambda))\|_{L^4} \le C_\alpha n^{-1/2}.
\]
\end{enumerate}
\end{lemma}

The next lemma centers each leave-one-out residual product using its own
training sample. Part~(b) of Lemma~\ref{lem:gcvapp-delete} then controls
the covariance between the centered products from two different rows.
\begin{lemma}\label{lem:gcvapp-loo}
Let $e_i(\lambda)=y_i-\boldsymbol x_i^\top \bb_{\{i\}}(\lambda)$.
Then
\[
\sup_{\operatorname{Re}\lambda,\operatorname{Re}\tau\ge\alpha} \mathbb E\left| \frac1n \sum_{i=1}^n e_i(\lambda)e_i(\tau) - C_n(\lambda,\tau)\right|
 \le C_\alpha n^{-1/2}.
\]
\end{lemma}

GCV uses the average of the diagonal entries of $\bW_\lambda$ in place
of the individual diagonal entries appearing in the leave-one-out
residuals. The following estimate controls that replacement.
\begin{lemma}\label{lem:gcvapp-trace}
Let $\bK=\boldsymbol X\boldsymbol X^\top/n$ and
define the event $\mathcal B_n=\{\operatorname{tr}(K)/n\le c_+L+1\}$. 

Then
$\mathbb P(\mathcal B_n^c)\le Cn^{-2}$ and
\begin{equation}\label{eq:gcvapp-base}
 \sup_{\operatorname{Re}\lambda,\operatorname{Re}\tau\ge\alpha}
 \mathbb{E}\left[\mathbf 1_{\mathcal B_n}
         |\widehat C_n(\lambda,\tau)-C_n(\lambda,\tau)|\right]
 \le C_\alpha n^{-1/2}.
\end{equation}
\end{lemma}

To transfer the ridge estimate to SD, set $E_{ij}(\lambda) = \widehat \bT_{ij}(\lambda) - \bT_{ij}(\lambda)$ and $M_k(\lambda)=\max_{0\le i,j\le k}|E_{ij}(\lambda)|.$
The following two bounds distinguish a fixed penalty from optimization over $I$.
\begin{lemma}\label{lem:gcvapp-transfer}
For every integer $k\ge0$, with a constant $C$ independent of $n,k$,
\begin{align}
 \sup_{\lambda\in I}\mathbb E[\mathbf 1_{\mathcal B_n}M_k(\lambda)]
 &\le C\log^2(k+2)/\sqrt n,\label{eq:gcvapp-fixed}\\
 \mathbb{E}\left[\mathbf 1_{\mathcal B_n}
                  \sup_{\lambda\in I}M_k(\lambda)\right]
 &\le C(k+1)\log^2(k+2)/\sqrt n.\label{eq:gcvapp-uniform}
\end{align}
\end{lemma}

\begin{proof}[Proof of Theorem~\ref{thm:gcv}]
By Lemmas~\ref{lem:gcvapp-trace} and~\ref{lem:gcvapp-transfer},
for every $t>0$ and $k\ge0$, Markov's inequality gives
\[
 \mathbb P\!\left\{
 \sup_{\lambda\in I}M_k(\lambda)>
 t\frac{(k+1)\log^2(k+2)}{\sqrt n}\right\}
 \le \mathbb P(\mathcal B_n^c)+\frac Ct
 \le \frac C{n^2}+\frac Ct .
\]
Taking $k=K_n$ proves the claimed bound for every deterministic depth sequence.
\end{proof}

\subsection{Proof of Theorem~\ref{th:growing}}\label{app:gcv-optimality-proof}
\begin{proof}
Fix a deterministic $\lambda_0$ in the search set for every $n$.
Write $\mathcal{R}_\ast=\inf_{f\in\mathcal F_M}\mathcal R^{\mathrm{pred}}(f;c)$.
Every continuous function on $[0,M]$ extends continuously to $[0,\infty)$,
for example by holding it constant beyond $M$. By
Theorem~\ref{thm:general-covariance}(a) the risk is independent of which
extension is used. Therefore, part~(b) of that theorem gives, for every $\epsilon>0$,
a fixed finite-step SD spectral shrinkage function $f_\epsilon$
at $\lambda_0$ such that
\[
 \mathcal R^{\mathrm{pred}}(f_\epsilon;c)
 \le \mathcal R_\ast+\epsilon.
\]
Write $k_\epsilon$ for its depth and $\ba_\epsilon$ for its
coefficient vector. Padding this vector with trailing zeros
leaves the function unchanged:
\[
 f_\epsilon(x) = \sum_{j=0}^{k_\epsilon}a_{\epsilon,j}\frac{x^j}{(x+\lambda_0)^{j+1}} = \sum_{j=0}^{K_n}a_{\epsilon,j} \frac{x^j}{(x+\lambda_0)^{j+1}}
\] where $a_{\epsilon,j} = 0$ for $j>k_\epsilon$.
Since $K_n,A_n\to\infty$, we have $K_n\ge k_\epsilon$ and
$A_n\ge\|\ba_\epsilon\|_1$ for all sufficiently large $n$.
Thus the padded vector is eventually feasible. For every feasible $(a,\lambda)$, the affine constraint
$\mathbf 1^\top a=1$ gives
\[
 |\widehat{\mathcal R}(\ba,\lambda)-\mathcal R(f_{\ba,\lambda})|
 =\left|\sum_{i,j}a_i a_j E_{ij}(\lambda)\right|
 \le A_n^2\sup_{\theta\in I}M_{K_n}(\theta).
\]
Denote the preceding upper bound by $e_n$.
Theorem~\ref{thm:gcv} and the stated growth condition give
$e_n\xrightarrow{\text{p}}0$. Since $f_\epsilon$ is feasible
for all sufficiently large $n$, we obtain
\begin{align*}
 \mathcal R(\widehat f_n) \le \widehat{\mathcal R}(\widehat{\ba}_n,\widehat\lambda_n)+e_n \le \widehat{\mathcal R}(\ba_\epsilon,\lambda_0)+ e_n + \delta_n \le \mathcal R(f_\epsilon)+2e_n+\delta_n,
\end{align*}
where the first and third inequalities use the uniform bound for the GCV approximation error, while the second follows from the optimization tolerance. Since $f_\epsilon$ is fixed, Theorem~\ref{thm:general-covariance}(a) gives
\[
 \mathcal R(f_\epsilon)
 \xrightarrow{\mathrm{a.s.}}
 \sigma_\varepsilon^2+\mathcal R^{\mathrm{pred}}(f_\epsilon;c).
\]
Consequently,
\begin{align*}
 \mathcal R(\widehat f_n) \le \sigma_\varepsilon^2
+\mathcal R^{\mathrm{pred}}(f_\epsilon;c)+o_{\text{p}}(1) \le \sigma_\varepsilon^2+\mathcal R_\ast +\epsilon+o_{\text{p}}(1).
\end{align*}
For any $t>0$, choosing $\epsilon<t/2$ therefore yields
\[
 \mathbb P\!\left\{
 \mathcal R(\widehat f_n)
 >\sigma_\varepsilon^2+\mathcal R_\ast+t
 \right\}\longrightarrow0,
\]
which proves~\eqref{eq:optimal}. 

The same argument applies to a data-dependent grid $\Lambda_n\subset I$ containing $\lambda_0$. Indeed, the fixed $f_\epsilon$ remains eventually feasible, and the uniform score bound over $I$ also holds over $\Lambda_n$.
\end{proof}

\subsection{Proof of Lemma~\ref{lem:gcvapp-delete}}
For deterministic $\bv\in\mathbb R^p$, applying
\cite[Lemma~B.26]{lifesaver} with $q=4$ and
$\ba=\boldsymbol\Sigma^{1/2}\bv$ gives
\[\mathbb E|\boldsymbol x_i^\top \bv|^4
 \le C\|\boldsymbol\Sigma^{1/2}\bv\|_2^4
 \le C L^2\|v\|_2^4,\]
It is immediate that the same bound, possibly with a different constant, applies for $\bv \in \mathbb{C}^p$. Thus for $\bv$ independent of observation $i$,
\begin{equation}\label{eq:gcvapp-linear}
 \mathbb E[|\boldsymbol x_i^\top 
 \bv|^4\mid v]\le C\|\bv\|_2^4,
 \qquad
 \mathbb E[|y_i-\boldsymbol x_i^\top \bv|^4\mid \bv]\le C(1+\|\bv\|_2^4).
\end{equation}
The second inequality uses the noise moment and the bounded signal norm.

For $s\ge0$, $\sqrt{s}/|s+\lambda|\le(2\sqrt{\alpha})^{-1}$.
Applying this singular-value bound to the deleted ridge estimator gives
\[
 \|\bb_D(\lambda)\|_2\le
 \frac{\|\boldsymbol y_{-D}\|_2}{2\sqrt{n\alpha}},\qquad
 \mathbb E\|\bb_D(\lambda)\|_2^4\le
 \frac{1}{16n\alpha^2}\sum_{i\notin D}\mathbb E|y_i|^4\le C_\alpha .
\]
When $i\in D$, the independence of $\bb_D$ and row $i$ gives the other
assertion of part~(a) by \eqref{eq:gcvapp-linear}.

For part~(b), let $R$ be the resolvent after deleting $D\cup\{i\}$,
and set $a_i=\boldsymbol x_i^\top R\boldsymbol x_i/n$.
Sherman-Morrison gives
\[
 \bb_D-\bb_{D\cup\{i\}}
 =\frac{R\boldsymbol x_i}{n(1+a_i)}
     (y_i-\boldsymbol x_i^\top \bb_{D\cup\{i\}}).
\]
The scalar inequality $\operatorname{Re}\{(s+\lambda)^{-1}\}
\ge\alpha/|s+\lambda|^2$ implies
$\operatorname{Re} a_i\ge\alpha\|R\boldsymbol x_i\|_2^2/n$.
Consequently,
\[
 \left\|\frac{R\boldsymbol x_i}{n(1+a_i)}\right\|_2^2
 \le \frac{\operatorname{Re} a_i}{n\alpha|1+a_i|^2}
 \le \frac{1}{4n\alpha}.
\]
The first deletion bound follows by taking fourth moments and applying
part~(a). For the second, both fits exclude $i$, so conditioning on
their difference in \eqref{eq:gcvapp-linear} bounds the held-out
linear form by the first deletion bound.

\subsection{Proof of Lemma~\ref{lem:gcvapp-loo}}

Fix $\operatorname{Re}\lambda,\operatorname{Re}\tau\ge\alpha$.
We center the product $e_i(\lambda)e_i(\tau)$ directly, so that the
noise and prediction terms can be handled together. Write
\[
 u_i(\theta)=\bbeta_0-b_{\{i\}}(\theta),
 \qquad
 u_{ij}(\theta)=\bbeta_0-b_{\{i,j\}}(\theta),
 \qquad
 e_i(\theta)=\varepsilon_i+x_i^\top u_i(\theta).
\]
Let $\mathcal D_{-i}$ be the sigma-field generated by all observations
except $i$. Since $u_i(\theta)$ is $\mathcal D_{-i}$-measurable and
$\mathbb E[x_i\varepsilon_i]=0$, we have
\[
 \mathbb E\!\left[e_i(\lambda)e_i(\tau)\mid\mathcal D_{-i}\right]
 =\sigma_\varepsilon^2+u_i(\lambda)^\top\bSigma u_i(\tau).
\]
Denote the product minus this conditional expectation by $Z_i$.
We first show that its average is of order $n^{-1/2}$ in $L^2$. For $i\ne j$, let $Z_i^{(-j)}$ be defined in the same way as $Z_i$,
with $u_i$ replaced throughout by $u_{ij}$. Thus the held-out
observation is still $i$, but the fit now omits both $i$ and $j$.
Write $e_i^{(-j)}(\theta)=\varepsilon_i+x_i^\top u_{ij}(\theta)$.
By Lemma~\ref{lem:gcvapp-delete}(a)-(b), for $\theta\in\{\lambda,\tau\}$,
\[
 \|e_i(\theta)\|_{L^4}+\|e_i^{(-j)}(\theta)\|_{L^4}\le C_\alpha,
 \qquad
 \|e_i(\theta)-e_i^{(-j)}(\theta)\|_{L^4}
 =\|x_i^\top(u_i(\theta)-u_{ij}(\theta))\|_{L^4}
 \le \frac{C_\alpha}{\sqrt n}.
\]
Since $Z_i$ is $e_i(\lambda)e_i(\tau)$ minus its conditional
expectation given $\mathcal D_{-i}$, centering and H\"older's
inequality give
\[
 \|Z_i\|_{L^2}
 \le \|e_i(\lambda)e_i(\tau)\|_{L^2}
 \le \|e_i(\lambda)\|_{L^4}\|e_i(\tau)\|_{L^4}
 \le C_\alpha,
\]
where the last bound follows from Lemma~\ref{lem:gcvapp-delete}(a).

The fit omitting both $i$ and $j$ is also known given $\mathcal D_{-i}$.
Thus $Z_i^{(-j)}$ is $e_i^{(-j)}(\lambda)e_i^{(-j)}(\tau)$ minus
its conditional expectation given the same sigma-field. By linearity,
$Z_i-Z_i^{(-j)}$ is therefore $e_i(\lambda)e_i(\tau)-e_i^{(-j)}(\lambda)e_i^{(-j)}(\tau)$
minus its conditional expectation given $\mathcal D_{-i}$.
Centering again and changing one factor at a time, we obtain
\begin{align*}
 \|Z_i-Z_i^{(-j)}\|_{L^2}
 &\le \|e_i(\lambda)e_i(\tau)
        -e_i^{(-j)}(\lambda)e_i^{(-j)}(\tau)\|_{L^2}\\
 &\le \|e_i(\lambda)-e_i^{(-j)}(\lambda)\|_{L^4}
       \|e_i(\tau)\|_{L^4} +\|e_i^{(-j)}(\lambda)\|_{L^4}
       \|e_i(\tau)-e_i^{(-j)}(\tau)\|_{L^4}
 \\ &\le \frac{C_\alpha}{\sqrt n},
\end{align*}
In the last step, each residual has bounded fourth moment by
Lemma~\ref{lem:gcvapp-delete}(a), and each deletion difference
has $L^4$ norm at most $C_\alpha/\sqrt n$ by part~(b).

Note that $Z_i^{(-j)}$ does not use observation $j$. Since $Z_j$ is centered given $\mathcal D_{-j}$,
$\mathbb E[Z_i^{(-j)}\overline{Z_j}]=0$ and interchanging $i$ and $j$ gives $\mathbb E[Z_i\overline{Z_j^{(-i)}}]=0$. Moreover, conditional
on all observations outside $\{i,j\}$, the variables $Z_i^{(-j)}$
and $Z_j^{(-i)}$ are centered functions of independent observations.
Their cross term therefore also vanishes. We deduce that
\[
 \big|\mathbb E[Z_i\overline{Z_j}]\big|
 =
 \left|\mathbb E\!\left[
   (Z_i-Z_i^{(-j)})\overline{(Z_j-Z_j^{(-i)})}
 \right]\right|
 \le \frac{C_\alpha}{n}.
\]
Combining the diagonal and off-diagonal bounds gives
\[
 \mathbb E\left|\frac1n\sum_{i=1}^n Z_i\right|^2
 \le \frac{1}{n^2}
 \left(nC_\alpha+n(n-1)\frac{C_\alpha}{n}\right)
 \le \frac{C_\alpha}{n}.
\]

It remains to replace the leave-one-out fits in the conditional
expectations by the full-sample fit. Write
$u_0(\theta)=\bbeta_0-b_0(\theta)$. By
Lemma~\ref{lem:gcvapp-delete}(a)-(b), applied with no observations
initially deleted, the fits have bounded $L^2$ norms and
$\|u_i(\theta)-u_0(\theta)\|_{L^2}\le C_\alpha/\sqrt n$.
Replacing one factor at a time and using $\|\bSigma\|_{\mathrm{op}}\le L$,
we obtain
\begin{align*}
 \mathbb E\left|
 u_i(\lambda)^\top\bSigma u_i(\tau)
 -u_0(\lambda)^\top\bSigma u_0(\tau)\right| &\le L\left(
 \|u_i(\lambda)-u_0(\lambda)\|_{L^2}\|u_i(\tau)\|_{L^2}
 +\|u_0(\lambda)\|_{L^2}\|u_i(\tau)-u_0(\tau)\|_{L^2}
 \right) \\
& \le \frac{C_\alpha}{\sqrt n}.
\end{align*}
Since $C_n(\lambda,\tau)=\sigma_\varepsilon^2+
u_0(\lambda)^\top\bSigma u_0(\tau)$, averaging this last bound
and using the $L^2$ bound for $n^{-1}\sum_i Z_i$ proves
\[
 \mathbb E\left|
 \frac1n\sum_{i=1}^n e_i(\lambda)e_i(\tau)-C_n(\lambda,\tau)
 \right|
 \le \frac{C_\alpha}{\sqrt n},
\]
uniformly over $\operatorname{Re}\lambda,\operatorname{Re}\tau\ge\alpha$.

\subsection{Proof of Lemma~\ref{lem:gcvapp-trace}}
\begin{proof}
Let $c_+=\sup_n p/n$. We first prove that
$\mathbb P(\mathcal B_n^c)\le Cn^{-2}$. Recall that
\[
 \mathcal B_n=\{s_n\le c_+L+1\},
 \qquad
 s_n=\frac{\tr\bK}{n}
     =\frac{1}{n^2}\sum_{i=1}^n\bz_i^\top\bSigma\bz_i.
\]
Since the coordinates of $\bz_i$ are standardized,
 $\mathbb E s_n={\tr\bSigma}/{n}\le c_+L.$
It is therefore enough to bound the variance of $s_n$. Expanding
the quadratic form and using independence of the coordinates, we get
\[
 \operatorname{Var}(\bz_1^\top\bSigma\bz_1)
 =
 (\mathbb E Z_{11}^4-1)\sum_j\Sigma_{jj}^2
 +4\sum_{j<l}\Sigma_{jl}^2
 \le C\tr(\bSigma^2)\le CL^2p.
\]
By independence across rows,
\[
 \operatorname{Var}(s_n)
 =\frac{1}{n^3}\operatorname{Var}(\bz_1^\top\bSigma\bz_1)
 \le \frac{CL^2p}{n^3}\le \frac{C}{n^2}.
\]
On $\mathcal B_n^c$ we have $s_n-\mathbb E s_n>1$, so the conclusion follows from Chebyshev's inequality.

We next show why $\mathcal B_n$ is useful: on this event, the denominator in the GCV residual is bounded away from zero after
multiplication by $\lambda$. Write $\lambda=u+iv$, with
$u\ge\alpha$. For every $d\ge0$,
\[
 \operatorname{Re}\left(\frac{\lambda}{d+\lambda}\right)
 =
 \frac{u}{d+u}
 +\frac{dv^2}{(d+u)\{(d+u)^2+v^2\}}
 \ge \frac{u}{d+u}.
\]
Let $\underline{d}_1,\ldots,\underline{d}_n$ be the eigenvalues of $\bK$. Since
$d\mapsto u/(d+u)$ is convex, Jensen's inequality gives
\[ |\lambda\widehat v_\lambda|
 \ge \operatorname{Re}\left(\lambda\widehat v_\lambda\right)
 \ge \frac1n\sum_{j=1}^n\frac{u}{\underline{d}_j+u}
 \ge \frac{u}{s_n+u}.
\]
Thus, uniformly over $\operatorname{Re}\lambda\ge\alpha$, we have on $\mathcal B_n$ \begin{equation}\label{eq:gcvapp-denom}
 |\lambda\widehat v_\lambda|
 \ge c_\alpha
 :=\frac{\alpha}{\alpha+c_+L+1}>0.
\end{equation}

We now compare each diagonal entry $w_i=(\bW_\lambda)_{ii}$ with
their average $\widehat v_\lambda$. We claim that, uniformly in $i$
and $\operatorname{Re}\lambda \ge \alpha$,
\begin{equation}\label{eq:gcvapp-diag}
 \mathbb E|\lambda w_i-\lambda\widehat v_\lambda|^4
 \le C_\alpha n^{-2}.
\end{equation}
To prove this, let
\[
 \bQ_{-i}=
 \left(\frac{\bX_{-i}^\top\bX_{-i}}{n}+\lambda\bI_p\right)^{-1},
 \qquad
 \bQ=(\widehat{\bSigma}+\lambda\bI_p)^{-1},
\]
and write
\[
 a_i=\frac{\bx_i^\top\bQ_{-i}\bx_i}{n},
 \qquad
 h_i=\frac{\tr(\bSigma\bQ_{-i})}{n},
 \qquad
 h=\frac{\tr(\bSigma\bQ)}{n}.
\]
Conditional on $\bX_{-i}$, the matrix
$\bA=\bSigma^{1/2}\bQ_{-i}\bSigma^{1/2}/n$ is fixed and
$\bz_i$ is independent of it. We can therefore apply the
quadratic-form bound of Bai-Silverstein
\cite[Lemma~B.26]{lifesaver} with exponent $4$. Since the
standardized coordinates have uniformly bounded eighth moments and
$\tr\{(\bA\bA^*)^2\}\le\{\tr(\bA\bA^*)\}^2$, this gives
\[
 \mathbb{E}\left[|a_i-h_i|^4\mid\bX_{-i}\right]
 =
 \mathbb{E}\left[
   |\bz_i^\top\bA\bz_i-\tr\bA|^4\mid\bX_{-i}\right]
 \le C\{\tr(\bA\bA^*)\}^2.
\]
Using $\|\bQ_{-i}\|_{\mathrm{op}}\le\alpha^{-1}$, we obtain
\[
 \tr(\bA\bA^*)\le\frac{pL^2}{n^2\alpha^2}
 \le\frac{C_\alpha}{n}.
\]
Taking expectations, we deduce that
$\mathbb E|a_i-h_i|^4\le C_\alpha n^{-2}$.

To replace $h_i$ by the common trace $h$, we use Sherman-Morrison:
\[
 \bQ_{-i}-\bQ
 =\frac{\bQ_{-i}\bx_i\bx_i^\top\bQ_{-i}}{n(1+a_i)}.
\]
By spectral calculus,
$\operatorname{Re}\left(a_i\right)\ge\alpha\|\bQ_{-i}\bx_i\|_2^2/n$. Taking the trace
against $\bSigma/n$ therefore gives
\[
 |h_i-h|
 \le\frac{L\|\bQ_{-i}\bx_i\|_2^2}{n^2|1+a_i|}
 \le\frac{L}{n\alpha}\frac{\operatorname{Re}(a_i)}{|1+a_i|}
 \le\frac{L}{n\alpha}.
\]
Combining the last two bounds, we get
$\mathbb E|a_i-h|^4\le C_\alpha n^{-2}$.

Next, note that the Schur complement gives the two leave-one-out identities $ (\bW_\lambda\by)_i=w_i e_i(\lambda)$ and $(\lambda w_i)^{-1}={1+a_i}$. Since $a_i$ and $h$ have nonnegative real part, the first identity gives
\[
 \left|\lambda w_i-\frac{1}{1+h}\right|
 =\frac{|a_i-h|}{|1+a_i|\,|1+h|}
 \le |a_i-h|.
\]
We first show that their average
is close to the same quantity. Since
$\lambda\widehat v_\lambda=n^{-1}\sum_{j=1}^n\lambda w_j$, we have
\[
 \lambda\widehat v_\lambda-\frac{1}{1+h}
 =
 \frac1n\sum_{j=1}^n
 \left(\lambda w_j-\frac{1}{1+h}\right).
\]
We apply Jensen's inequality to the convex function $z\mapsto |z|^4$ and then take expectations:
\[
 \mathbb E\left|\lambda\widehat v_\lambda-\frac{1}{1+h}\right|^4
 \le
 \frac1n\sum_{j=1}^n
 \mathbb E\left|\lambda w_j-\frac{1}{1+h}\right|^4
 \le \frac1n\sum_{j=1}^n\mathbb E|a_j-h|^4\le C_\alpha n^{-2}\le C_\alpha n^{-2}.
\]
It is now enough to subtract the same quantity from both
$\lambda w_i$ and $\lambda\widehat v_\lambda$. Using
$|a-b|^4\le 8(|a|^4+|b|^4)$, we obtain
\begin{align*}
 \mathbb E|\lambda w_i-\lambda\widehat v_\lambda|^4
 &\le
 8\mathbb E\left|\lambda w_i-\frac{1}{1+h}\right|^4
 +8\mathbb E\left|\lambda\widehat v_\lambda-\frac{1}{1+h}\right|^4 \le C_\alpha n^{-2},
\end{align*}
uniformly in $i$. This proves \eqref{eq:gcvapp-diag}.

We can now compare the GCV residual vector $\bg_\lambda$ with the leave-one-out residual vector $\be(\lambda)=(e_i(\lambda))_{i=1}^n$.
By $ (\bW_\lambda\by)_i=w_i e_i(\lambda)$,
\[
 (\bg_\lambda)_i-e_i(\lambda)
 =r_i e_i(\lambda),
 \qquad
 r_i=\frac{w_i}{\widehat v_\lambda}-1
     =\frac{\lambda w_i-\lambda\widehat v_\lambda}{\lambda\widehat v_\lambda}.
\]
From \eqref{eq:gcvapp-denom} and \eqref{eq:gcvapp-diag}, we obtain $$\mathbb E[\mathbf 1_{\mathcal B_n}|r_i|^4]
 \le c_\alpha^{-4} \mathbb E|\lambda w_i-\lambda\widehat v_\lambda|^4
 \le C_\alpha n^{-2}.$$
Lemma~\ref{lem:gcvapp-delete}(a), applied with $D=\{i\}$, gives
$\mathbb E|e_i(\lambda)|^4\le C_\alpha$. Thus Cauchy-Schwarz gives
\begin{align*}
 \mathbb{E}\left[\mathbf 1_{\mathcal B_n}
       \frac{\|\bg_\lambda-\be(\lambda)\|_2^2}{n}\right]
 &\le \frac1n\sum_{i=1}^n
 \bigl(\mathbb E[\mathbf 1_{\mathcal B_n}|r_i|^4]\bigr)^{1/2}
 \bigl(\mathbb E|e_i(\lambda)|^4\bigr)^{1/2}
 \le \frac{C_\alpha}{n}.
\end{align*}
We also have $|\lambda w_i|=|1+a_i|^{-1}\le1$. On $\mathcal B_n$,
it follows that $|(\bg_\lambda)_i|\le c_\alpha^{-1}|e_i(\lambda)|$.
Hence
\[
 \mathbb E\left[\frac{\|\be(\lambda)\|_2^2}{n}\right]\le C_\alpha,
 \qquad
 \mathbb{E}\left[\mathbf 1_{\mathcal B_n}
             \frac{\|\bg_\lambda\|_2^2}{n}\right]\le C_\alpha.
\]

It is now enough to compare the residual products at the two penalties.
Recall from \eqref{eq:gcvapp-cross} that
$\widehat C_n(\lambda,\tau)=\bg_\lambda^\top\bg_\tau/n$. We write
\[
 \bg_\lambda^\top\bg_\tau-\be(\lambda)^\top\be(\tau)
 =
 (\bg_\lambda-\be(\lambda))^\top\bg_\tau
 +\be(\lambda)^\top(\bg_\tau-\be(\tau)).
\]
Applying Cauchy-Schwarz to each term and using the preceding bounds,
we obtain
\begin{align*}
 \mathbb{E}\left[\mathbf 1_{\mathcal B_n}
 \left|\widehat C_n(\lambda,\tau)
       -\frac{\be(\lambda)^\top\be(\tau)}{n}\right|\right]
&\le
 \left(\mathbb{E}\left[\mathbf 1_{\mathcal B_n}
       \frac{\|\bg_\lambda-\be(\lambda)\|_2^2}{n}\right]\right)^{1/2}
 \left(\mathbb{E}\left[\mathbf 1_{\mathcal B_n}
       \frac{\|\bg_\tau\|_2^2}{n}\right]\right)^{1/2}\\& +  \left(\mathbb E\left[\frac{\|\be(\lambda)\|_2^2}{n}\right]\right)^{1/2}
 \left(\mathbb{E}\left[\mathbf 1_{\mathcal B_n}
       \frac{\|\bg_\tau-\be(\tau)\|_2^2}{n}\right]\right)^{1/2}\\
 &\le \frac{C_\alpha}{\sqrt n}.
\end{align*}
Finally, Lemma~\ref{lem:gcvapp-loo} gives
\[
 \mathbb E\left|
 \frac{\be(\lambda)^\top\be(\tau)}{n}-C_n(\lambda,\tau)
 \right|\le \frac{C_\alpha}{\sqrt n}.
\]
Combining the two estimates by the triangle inequality and taking the
supremum over $\operatorname{Re}\lambda,\operatorname{Re}\tau\ge\alpha$ proves
\eqref{eq:gcvapp-base}.
\end{proof}

\subsection{Proof of Lemma~\ref{lem:gcvapp-transfer}}
Write $Z_n=\widehat C_n-C_n$. The function $Z_n$ is holomorphic in both penalties on the product of
right half-planes: the resolvents are holomorphic, and
$\operatorname{Re}(\widehat v_\lambda) > 0$ prevents division by zero.

For a scalar function $h$ analytic near a positive $\lambda$, the
Leibniz rule and the identity
\[
 \sum_{j=r}^\infty\binom jr z^j
 =\frac{z^r}{(1-z)^{r+1}}
\]
give, for $z$ sufficiently small,
\begin{align*}
 \sum_{j\ge0}D_j^\lambda h(\lambda)z^j
 &=\sum_{r\ge0}\frac{\lambda^rh^{(r)}(\lambda)}{r!}
             \sum_{j\ge r}\binom jr z^j =\frac1{1-z}\sum_{r\ge0}\frac{h^{(r)}(\lambda)}{r!}
                      \left(\frac{\lambda z}{1-z}\right)^r.
\end{align*}
The sums converge absolutely near zero, so the exchange is justified.
Taylor's formula yields
\begin{equation}\label{eq:generating}
 \sum_{j\ge0}D_j^\lambda h(\lambda)z^j
 =\frac{h(\lambda/(1-z))}{1-z}.
\end{equation}
Apply this identity in both arguments of $Z_n$, we deduce that
\begin{equation}\label{eq:error-generating}
 \sum_{i,j\ge0}E_{ij}(\lambda)z^iw^j
 =\frac{Z_n\!\left(\lambda/(1-z),\lambda/(1-w)\right)}
        {(1-z)(1-w)}.
\end{equation}
Initially this identifies the Taylor coefficients near $(0,0)$. The
right side is holomorphic on the unit bidisc because
for $|z|<1$, multiplying the numerator and denominator by
$1-\overline z$ gives
\[
 \operatorname{Re}\!\left(\frac{1}{1-z}\right)-\frac12
 =\frac{1-\operatorname{Re}z}{|1-z|^2}-\frac12
 =\frac{2(1-\operatorname{Re}z)-(1-2\operatorname{Re}z+|z|^2)}
        {2|1-z|^2}
 =\frac{1-|z|^2}{2|1-z|^2}>0.
\]
Its Taylor expansion therefore converges on that bidisc and has the
same coefficients. In particular, both transformed penalties in
\eqref{eq:error-generating} have real part at least $\lambda_{\min}/2$, where we let the compact interval $I = [\lambda_{\min}, \lambda_{\max}]$.

For $0<r<1$, Cauchy's coefficient formula gives
\[
 E_{ij}(\lambda)
 =\frac1{(2\pi\mathrm i)^2}
   \oint_{|z|=r}\oint_{|w|=r}
   \frac{Z_n\left(\lambda/(1-z),\lambda/(1-w)\right)}
        {(1-z)(1-w)z^{i+1}w^{j+1}}\,dw\,dz.
\]
Using the same two contours for every $i,j\le k$, take the maximum
before bounding the integral:
\begin{align*}
 \max_{i,j\le k}|E_{ij}(\lambda)|
 &\le\frac{r^{-2k}}{(2\pi)^2}
 \int_{-\pi}^{\pi}\int_{-\pi}^{\pi}
 \frac{\left|Z_n\!\left(\lambda/(1-re^{\mathrm it}),
                         \lambda/(1-re^{\mathrm is})\right)\right|}
 {|1-re^{\mathrm it}|\,|1-re^{\mathrm is}|}\,ds\,dt.
\end{align*}
Tonelli's theorem and the preceding half-plane bound imply
\begin{equation}\label{eq:cauchy}
 \mathbb{E}\left[\mathbf{1}_{\mathcal B_n}\max_{i,j\le k}|E_{ij}(\lambda)|\right]
 \le\frac{Cr^{-2k}}{\sqrt n}
 \left(\frac1{2\pi}\int_{-\pi}^{\pi}
                 \frac{dt}{|1-re^{\mathrm it}|}\right)^2.
\end{equation}
For $r\ge1/2$ and $|t|\le\pi$,
\[
 |1-re^{\mathrm it}|^2
 =(1-r)^2+2r(1-\cos t)
 \ge(1-r)^2+\frac{2t^2}{\pi^2}.
\]
Integrating its inverse square root gives
\[
 \frac1{2\pi}\int_{-\pi}^{\pi}\frac{dt}{|1-re^{\mathrm it}|}
 \le C\log\frac{2}{1-r}.
\]
We choose $r=(k+1)/(k+2)$ so that 
\[
 r^{-2k}=\left(1+\frac1{k+1}\right)^{2k}\le e^2.
\]
Substitution into \eqref{eq:cauchy} proves \eqref{eq:gcvapp-fixed}

We now bound the supremum over $\lambda\in I$. Recall that $E_{ij}(\lambda)$ is obtained by evaluating
$D_i^uD_j^v Z_n(u,v)$ at $u=v=\lambda$. Since both arguments vary
with $\lambda$, the chain rule gives
\[
 E_{ij}'(\lambda)
 =\left.(\partial_u+\partial_v)D_i^uD_j^v Z_n(u,v)
   \right|_{u=v=\lambda}.
\]
Leibniz's rule gives
$u\,\partial_u D_i^u h=(i+1)(D_{i+1}^u h-D_i^u h)$.
Applying this identity in each argument yields
\[
 E_{ij}'(\lambda)
 =\frac{i+1}{\lambda}
    \bigl(E_{i+1,j}(\lambda)-E_{ij}(\lambda)\bigr)
  +\frac{j+1}{\lambda}
    \bigl(E_{i,j+1}(\lambda)-E_{ij}(\lambda)\bigr).
\]
For $i,j\le k$, each difference is bounded in absolute value by
$2M_{k+1}(\lambda)$. Integrating from $\lambda_{\min}$ to $\lambda$, where recall that $I = [\lambda_{\min}, \lambda_{\max}]$, and then taking the maximum over $i,j$ gives
\[
 \sup_{\lambda\in I} M_k(\lambda)
 \le M_k(\lambda_{\min})
   +\frac{4(k+1)}{\lambda_{\min}}
      \int_{\lambda_{\min}}^{\lambda_{\max}} M_{k+1}(\theta)\,d\theta.
\]
Multiplying by $\mathbf 1_{\mathcal B_n}$, taking expectations,
and applying \eqref{eq:gcvapp-fixed} at depths $k$ and $k+1$
proves \eqref{eq:gcvapp-uniform}, since
$\log^2(k+3)\le C\log^2(k+2)$.

\section{Decentralized Learning - Proof of Theorem~\ref{thm: product_of_shrinkage}}\label{app: K}

\subsection{Notation and The Bias-Variance Decomposition}\label{bias variance decomp section}

We begin by rewriting the aggregated estimator in a form that absorbs the aggregation weights into the shrinkage functions themselves, which makes the subsequent bias-variance decomposition more transparent. Throughout, the index \(\ell\) refers to a client and ranges over \(\{1,\dots,K\}\), the index \(j\) refers to a spike and ranges over \(\{1,\dots,\rk\}\), and the index \(i\) refers to an eigenvector and ranges over \(\{1,\dots,p\}\).
For each client \(\ell\), the corresponding  spectral shrinkage estimator is
\[
\widehat\bbeta_\ell
= f_\ell(\widehat{\boldsymbol{\Sigma}}_\ell)\frac{\bX_\ell^\top \by_\ell}{n_\ell}
= {\bW}_\ell f_\ell(\bLambda_\ell) \bW_\ell^\top \frac{\bX_\ell^\top \by_\ell}{n_\ell}, \qquad \text{where }\quad 
\widehat{\boldsymbol{\Sigma}}_\ell= \bW_\ell \bLambda_\ell \bW_\ell^\top.
\]

The next proposition rewrites each local estimator in a form that separates signal and noise, and then introduces a rescaled spectral shrinkage function that incorporates the aggregation weight.

\begin{lemma}\label{prop: normalized estimator federated}
For each client \(\ell\), define
\[
\widetilde f_\ell(x):=K \rho_\ell f_\ell(x),
\qquad
\widetilde\bbeta_\ell
:=\widetilde f_\ell(\widehat{\boldsymbol{\Sigma}}_\ell)\frac{\bX_\ell^\top \by_\ell}{n_\ell}.
\]
Moreover, if we let $g_\ell(x):=x\widetilde f_\ell(x)-1$, then
\[
\widetilde\bbeta_\ell-\bbeta_0
=
g_\ell(\widehat{\boldsymbol{\Sigma}}_\ell)\bbeta_0
+
\widetilde f_\ell(\widehat{\boldsymbol{\Sigma}}_\ell)\frac{\bX_\ell^\top \boldsymbol{\varepsilon}_\ell}{n_\ell},
\]
and hence
\[
\widehat\bbeta_{\mathrm{agg}}-\bbeta_0
=
\frac{1}{K}\sum_{\ell=1}^K(\widetilde\bbeta_\ell-\bbeta_0) = \frac{1}{K}\sum_{\ell=1}^K g_\ell(\widehat{\boldsymbol{\Sigma}}_\ell)\bbeta_0 + \frac{1}{K}\sum_{\ell=1}^K \widetilde f_\ell(\widehat{\boldsymbol{\Sigma}}_\ell)\frac{\bX_\ell^\top \boldsymbol{\varepsilon}_\ell}{n_\ell} \,.
\]
\end{lemma}
Lemma~\ref{prop: normalized estimator federated} is useful because it expresses the aggregated estimator as an average of \(K\) terms with the same normalization. This allows the conditional prediction risk to be decomposed into a bias term and \(K\) separate variance contributions. The proof is a straightforward algebraic manipulation and follows the first part of Section~\ref{proof: general formula}. We omit the details.
Using these notations, we can deduce a standard bias-variance decomposition of the prediction risk of the aggregate estimator. 

\begin{lemma}\label{prop: bias variance decomposition federated}
The following decomposition holds:
\[
\|\widehat\bbeta_{\mathrm{agg}} - \bbeta_0\|_{\boldsymbol{\Sigma}}^2
=
\|\mathcal{B}^{\mathrm{agg}}\|_{\boldsymbol{\Sigma}}^2 
+ \sum_{\ell=1}^K \mathcal{V}_{\ell}^{\mathrm{agg}}
+ \mathcal{E}^{\mathrm{agg}},
\]
where the bias term is
\[
\mathcal{B}^{\mathrm{agg}}
:= \frac{1}{K}\sum_{\ell=1}^K g_\ell(\widehat{\boldsymbol{\Sigma}}_\ell)\bbeta_0,
\]
the variance contribution of client $\ell$ is
\[
\mathcal{V}_{\ell}^{\mathrm{agg}}
:=
\frac{\sigma_{\boldsymbol{\varepsilon}}^2}{K^2 n_\ell}
\tr\bigl(
\boldsymbol{\Sigma}\, \widehat{\boldsymbol{\Sigma}}_\ell\, 
{\widetilde{f}_\ell}^2(\widehat{\boldsymbol{\Sigma}}_\ell)
\bigr),
\]
and $\mathcal{E}^{\mathrm{agg}} \xrightarrow{a.s.} 0$ as $p \to \infty$.
\end{lemma}
In the next lemma, we rewrite the conditional bias norm in a form that separates the bulk contribution from the spike contributions. 
\begin{lemma}
\label{prop: bias decomposition federated}
We can decompose $\|\mathcal{B}^{\mathrm{agg}}\|_{\boldsymbol{\Sigma}}^2$ as
\begin{align}\label{eq: B_0 avg}
\left\|\frac1K\sum_{\ell=1}^K g_\ell(\widehat{\boldsymbol{\Sigma}}_\ell)\bbeta_0\right\|^2_{{\boldsymbol{\Sigma}}}
&=
\frac{\sx^2}{K^2}\,\bbeta_0^\top \left(\sum_{\ell=1}^K g_\ell(\widehat{\boldsymbol{\Sigma}}_\ell)\right)^2 \bbeta_0
+
\frac{1}{K^2}
\sum_{j=1}^{\rk}\delta_j\,
\left(\bbeta_0^\top \sum_{\ell=1}^K g_\ell(\widehat{\boldsymbol{\Sigma}}_\ell) \bv_j\right)^2
\notag\\
& := \mathcal{B}_{0}^{\mathrm{agg}} + \sum_{j=1}^{\rk} \mathcal{B}_{j}^{\mathrm{agg}}.
\end{align}
\end{lemma}
Here \(\mathcal B^{\mathrm{agg}}_0\) corresponds to the contribution from the isotropic part of \({\boldsymbol{\Sigma}}\), while each \(\mathcal B^{\mathrm{agg}}_j\) captures the contribution of the spike direction \(\bv_j\).
It is immediate from Lemma~\ref{prop: bias variance decomposition federated} and Lemma~\ref{prop: bias decomposition federated} that to quantify the limiting prediction risk, we need to establish the limit of the bulk bias term \(\mathcal B^{\mathrm{agg}}_0\), the spike bias terms \(\mathcal B^{\mathrm{agg}}_j\), and the client-specific variance terms \(\mathcal V^{\mathrm{agg}}_\ell\). 

\subsection{Proof of Theorem~\ref{thm: product_of_shrinkage}}\label{sec: app h proof of agg}
From the definition of $\mathcal R^\mathrm{pred}_K(\rho_1, \cdots, \rho_k, f_1,\dots, f_K)$ (Equation \eqref{eq: opt agg}) and the definition of $(\tilde f_1, \dots, \tilde f_k)$ and $\widehat \bbeta_{\rm agg}$ (Lemma \ref{prop: general formula}), we have: 
$$
\mathcal R^\mathrm{pred}_K(\rho_1, \cdots, \rho_k, f_1,\dots, f_K) = \mathcal R^\mathrm{pred}_K(\widetilde f_1,\dots,\widetilde f_K):= \lim_{\substack{p \to \infty}} \|\widehat\bbeta_\mathrm{agg} - \bbeta_0\|_{\boldsymbol{\Sigma}}^2
$$
First, we argue that the limit exists almost surely.  Towards that end, we use the same tools as used in Section~\ref{app: B} and in the proof of Theorem~\ref{lemma:non-free limit}. In particular, using the Radon-Nikodym identities from Lemma~\ref{lemma: def mu_j mu_0} and the functions introduced in~\eqref{eq: define w ,g , h_j}, we obtain the following:  
\begin{lemma}\label{prop: general formula for agg}
For any functions $f_1, \dots, f_K \in \mathcal{F}$ and $\rho_1, \dots, \rho_k \in \reals$, the almost sure limiting prediction risk 
of $\widehat{\bbeta}_{\mathrm{agg}}$ is
\[
\begin{aligned}
& \mathcal R^\mathrm{pred}_K(\widetilde f_1,\dots,\widetilde f_K) \\
& = \frac{1}{K^2}\sum_{\ell=1}^K \|\widetilde f_\ell\|_w^2
-\frac{2}{K}\sum_{\ell=1}^K \langle g,\widetilde f_\ell\rangle_w
+\frac{2\sx^2\,r^2\omega_0}{K^2}\sum_{\ell < {\ell'}}\langle h_0,\widetilde f_\ell\rangle_w\langle h_0,\widetilde f_{\ell'}\rangle_w \\
&
+\frac{2}{K^2}\sum_{j=1}^{\rk}\sx^2\,\alpha_j^2\sum_{\ell < {\ell'}}\langle h_j,\widetilde f_\ell\rangle_w\langle h_j,\widetilde f_{\ell'}\rangle_w
+\frac{1}{K^2}\sum_{j=1}^{\rk}\delta_j\alpha_j^2
\left(\sum_{\ell=1}^K \langle h_j,\widetilde f_\ell\rangle_w\right)^2 + r^2+\sum_{j=1}^{\rk}\delta_j\alpha_j^2.
\end{aligned}
\]
\end{lemma}
To obtain the optimal shrinkage rule for each client, we minimize the limiting risk with respect to $(\tilde f_1, \dots, \tilde f_K)$. We first show that at the minimum value, all shrinkage rules $\widetilde f_\ell$ are equal to each other, which essentially proves the second part of Theorem~\ref{thm: product_of_shrinkage}.
To do this, we introduce the mean and the centered functions
\[
\bar \vartheta :=\frac1K\sum_{\ell=1}^K \widetilde f_\ell,
\qquad
\vartheta_\ell:=\widetilde f_\ell-\bar \vartheta \,.
\]
It is immediate that $\sum_{\ell=1}^K \vartheta_\ell=0$. By some simple algebraic manipulation, we can rewrite the limiting risk, obtained in Lemma \ref{prop: general formula for agg}, as follows: 
\begin{lemma}
\label{prop: decoupled formula}
Under the assumptions of Lemma~\ref{prop: general formula for agg},
\begin{align*}
& K^2 \, \mathcal R^\mathrm{pred}_K(\widetilde f_1,\dots,\widetilde f_K) \\
&\stackrel{\text{a.s.}}{=} \sum_{\ell=1}^K \|\vartheta_\ell\|_w^2
-\sx^2\,r^2\omega_0\sum_{\ell=1}^K \langle h_0,\vartheta_\ell\rangle_w^2
-\sum_{j=1}^{\rk}\sx^2\,\alpha_j^2\,\sum_{\ell=1}^K \langle h_j,\vartheta_\ell\rangle_w^2 \\
&\quad
+K\|\bar \vartheta\|_w^2
-2K^2\langle g,\bar \vartheta\rangle_w
+\sx^2\,r^2\omega_0\,K(K-1)\langle h_0,\bar \vartheta\rangle_w^2 \\
&\quad
+\sum_{j=1}^{\rk}\alpha_j^2\bigl(K(K-1)\sx^2\,+\delta_jK^2\bigr)\langle h_j,\bar \vartheta\rangle_w^2
+K^2\left(\sx^2\,r^2+\sum_{j=1}^{\rk}\delta_j\alpha_j^2\right).
\end{align*}
\end{lemma}
Crucially, note that in the form above the mean and the centered functions decouple so we can minimize them separately.  Indeed, the next Lemma shows that the term that depends on $\bar \vartheta$ is always nonnegative. This is the key step forcing all optimal functions $\widetilde f_\ell$ to coincide on \(\mathcal S_c^+ \setminus \{0\}\).

\begin{lemma}\label{all f' need to be the same}
For every \(\phi \in \mathcal{F}\),
\[
\|\phi\|_w^2
-\sx^2\,r^2\omega_0\langle h_0,\phi\rangle_w^2
-\sum_{j=1}^{\rk}\sx^2\,\alpha_j^2\, \langle h_j,\phi\rangle_w^2
\ge 0.
\]
Moreover, equality holds if and only if \(\phi\) vanishes on \(\mathcal S_c^+\setminus\{0\}\).
\end{lemma}
Now we have all the tools required to prove Theorem~\ref{thm: product_of_shrinkage}.
\begin{proof}[Proof of Theorem~\ref{thm: product_of_shrinkage}]
From Lemma~\ref{prop: decoupled formula}, we know optimizing $\cR^\mathrm{pred}_K(\widetilde f_1,\dots,\widetilde f_K)$ is equivalent to optimizing the prediction risk as a function of $(\vartheta_1, \dots, \vartheta_k ,\bar{v})$. We can conclude using Lemma~\ref{all f' need to be the same} that any optimizer must have $$\vartheta_{\ast, 1}(x) = \vartheta_{\ast, 2}(x) = \dots = \vartheta_{\ast, K}(x) = 0, \qquad \qquad \text{for } x \in \mathcal S_c^+\setminus\{0\}.$$ Thus, it is enough to minimize 
\begin{align*}
& K\|\bar \vartheta\|_w^2
-2K^2\langle g,\bar \vartheta\rangle_w
+\sx^2\,r^2\omega_0\,K(K-1)\langle h_0,\bar \vartheta\rangle_w^2 \\
& \qquad +\sum_{j=1}^{\rk}\alpha_j^2\,\bigl(K(K-1)\sx^2\,+\delta_jK^2\bigr)\langle h_j,\bar \vartheta\rangle_w^2
+K^2\left(\sx^2\,r^2+\sum_{j=1}^{\rk}\delta_j\alpha_j^2\right),
\end{align*}
or equivalently, after dropping terms that do not depend on $\bar \vartheta$,
\begin{align*}
\mathcal R^\mathrm{pred}_K(\bar \vartheta) & := \|\bar \vartheta\|_w^2
-2K\langle g,\bar \vartheta\rangle_w
+\sx^2\,r^2\omega_0\,(K-1)\langle h_0,\bar \vartheta\rangle_w^2 \\
& \qquad \qquad \qquad +\sum_{j=1}^{\rk}\alpha_j^2\,\bigl((K-1)\sx^2\,+\delta_jK\bigr)\langle h_j,\bar \vartheta\rangle_w^2.
\end{align*}

As in Lemma~\ref{lemma: define A and solve f_ast}, we can characterize the minimizer using a linear operator. Define $\mathcal{A}_K : \mathcal{F} \to \mathcal{F}$ as $$\mathcal{A}_K f =  f +\sx^2\,r^2\omega_0(K-1)\,  \langle f, h_0\rangle_w \, h_0  + \sum_{j=1}^\rk \left((K-1)\sx^2\, + K\delta_j\right)\,  \alpha_j^2 \, \langle f, h_j\rangle_w\, h_j.$$\
For one client $K=1$, the operator $\cA_K$ simplies to $\mathcal{A}_1 = \mathcal A$ where $\mathcal{A}$ is the operator introduced in Lemma~\ref{lemma: define A and solve f_ast}. 
It is straightforward to check that $\mathcal{A}_K$ satisfies the same properties as $\mathcal{A}$; in particular, $\mathcal{A}_K$ is positive-semidefinite, self-adjoint and 
$$
\mathcal{R}^\text{avg}(\bar \vartheta) = \langle \mathcal{A}_K\, \bar \vartheta, \bar \vartheta\rangle_w - 2K\langle g, \bar \vartheta\rangle_w.
$$
Next, we prove that there exists $f_K^\ast \in \mathcal{F}$ such that $[\mathcal{A}_Kf_K^\ast](x) = g(x)$ for all $x \in \mathcal{S}_c^+$.  As in Section~\ref{HD is positive-definite}, we consider a candidate $$f_K^\ast(x) = \sum_{j=0}^\rk b^{(K)}_j h_j(x)$$ and find real coefficients $\{b^{(K)}_0, b^{(K)}_1,  \dots, b^{(K)}_\rk\}$ for which 
\begin{align*}
\mathcal{A}_K f_K^\ast
& =
\sum_{i=0}^{\rk} b^{(K)}_i h_i
+
\sx^2\,r^2\omega_0(K-1)
\left(\sum_{i=0}^{\rk} b^{(K)}_i \langle h_i,h_0\rangle_w\right) h_0
\\
& \qquad \qquad \qquad +
\sum_{j=1}^{\rk}
\bigl((K-1)\sx^2\,+K\delta_j\bigr)\alpha_j^2
\left(\sum_{i=0}^{\rk} b^{(K)}_i \langle h_i,h_j\rangle_w\right) h_j
\end{align*}
is exactly equal to $g(x)$.  Equivalently, recall $\bH$ from Section~\ref{HD is positive-definite} and define
\[
\mathfrak D_K
:=
\mathrm{diag}\Bigl(\sx^2\,r^2\omega_0\left(K-1\right), \bigl((K-1)\sx^2\,+K\delta_1\bigr)\alpha_1^2,\dots,
\bigl((K-1)\sx^2\,+K\delta_\rk\bigr)\alpha_\rk^2\Bigr).
\]
Thus, exactly as in the $K=1$ case, using~\ref{eq: g is linear in h_j}, the equality $[\mathcal{A}_Kf_K^\ast](x) = g(x)$ on $x\in \mathcal{S}_c^+$ is equivalent to solving 
\begin{align}\label{eq: linear system with b0}
    \left(\mathbf I_{\rk+ 1} + \mathfrak D_K \bH\right) \begin{pmatrix}
    b^{(K)}_0\\ b^{(K)}_1 \\\vdots\\ b^{(K)}_{\rk} 
\end{pmatrix} = \begin{pmatrix}
    \sx^2\,r^2\omega_0\\ (\delta_1 + \sx^2\,)\alpha_1^2 \\\vdots\\ (\delta_\rk + \sx^2\,)\alpha_\rk^2 
\end{pmatrix}.
\end{align}

By the same reasoning as in~\eqref{eq: HD is invertible}, the matrix $\mathbf I_{\rk+ 1} + \mathfrak D_K \bH$ is invertible, and the claim is proven.
Next, we simplify the formula for $f_K^\ast$ to put it in self-distillation form. Since $f_K^\ast(x) = \sum_{j=0}^\rk b^{(K)}_j h_j(x)$ we deduce that \begin{equation}
f_K^\ast(x) = \frac{\sum_{j=0}^\rk b^{(K)}_j \mu_j(x)}{w(x)} = \frac{ b^{(K)}_0\nu(x) + \sum_{j=1}^\rk b^{(K)}_j \nu_{-j}(x)}{\sx^2\,r^2x\left(\omega_0 \nu(x) + \sum_{i=1}^\rk \omega_i \nu_{-i}(x)\right) + c\sx^2\, \sigma_{\boldsymbol{\varepsilon}}^2 \nu(x)}.\end{equation}

Next, we find the optimal $\bar \vartheta^\ast$ that minimizes $\mathcal{R}^\text{avg}(\bar{v})$. First, it is immediate that $\mathcal{A}_K (Kf_K^\ast) = Kg$. Since \(\mathcal A_K\) is self-adjoint, we deduce that
\begin{align*}
\mathcal{R}^\text{avg}(\bar \vartheta) & = \langle \mathcal A_K \bar \vartheta,\bar \vartheta\rangle_w - 2K\langle g,\bar \vartheta\rangle_w \\
& = \langle \mathcal A_K(\bar \vartheta-Kf_K^\ast),\bar \vartheta-Kf_K^\ast\rangle_w
-\langle \mathcal A_K Kf_K^\ast,Kf_K^\ast\rangle_w.
\end{align*}
Mutatis mutandis, as in the proof of Lemma~\ref{lemma: define A and solve f_ast}, we deduce that the optimal $\bar \vartheta^\ast = Kf_K^\ast$ and it is unique on $\mathcal{S}_c^+ \setminus \{0\}$. Combining this fact with Lemma~\ref{all f' need to be the same} we conclude that for the (unique) optimal choice $(\widetilde f_1^\ast, \widetilde f_2^\ast, \dots, \widetilde f_K^\ast)$ all of them have to be equal to $Kf_K^\ast$. 
Recall now the relation between $(\rho_\ell, f_\ell)$ and $\widetilde f_{\ell}$ from Lemma~\ref{prop: normalized estimator federated}: $$K\rho_\ell f_\ell(x) = \widetilde f_\ell(x).$$ We deduce that for each client $\ell$ the optimal choice of weight and self-distillation shrinkage procedure $(\rho_{\ast,\ell}, f^{\rm pred}_{*,\ell})$ must satisfy \begin{align}\label{eq: f_opt federated} \rho_{\ast,\ell}\, f^{\rm pred}_{*,1}(x) =\widetilde f_\ell^\ast(x) =  f^\ast_K(x) = \frac{ b^{(K)}_0\nu(x) + \sum_{j=1}^\rk b^{(K)}_j \nu_{-j}(x)}{\sx^2\,r^2 x\left(\omega_0 \nu(x) + \sum_{i=1}^\rk \omega_i \nu_{-i}(x)\right) + c\sx^2\, \sigma_{\boldsymbol{\varepsilon}}^2 \nu(x)}.\end{align}
From Lemma~\ref{lemma: sd is in mathcal F}, it is clear that any self-distillation procedure is the ratio of two monic polynomials. Since $b^{(K)}_0 \neq 0$ by Assumption~\ref{assm: b0^K neq 0}, we can thus uniquely identify $\rho_{\ast,\ell} = b^{(K)}_0 /(\sx^2\,r^2\omega_0)$. The proof that $f^{\rm pred}_{*,\ell}$ is a valid self-distillation procedure follows now from Lemma~\ref{lemma: polynomials P Q} (since the denominator is the same as that of $f_\ast^{\mathrm{pred}}$) and Lemma~\ref{lemma: lagrange interpolation}.
\end{proof}

\subsection{Proof of Lemma~\ref{prop: bias variance decomposition federated}}
\begin{proof}
By Lemma~\ref{prop: normalized estimator federated},
\[
\widehat\bbeta_{\mathrm{agg}}-\bbeta_0
=
\underbrace{\frac{1}{K}\sum_{\ell=1}^K
g_\ell(\widehat{\boldsymbol{\Sigma}}_\ell)\bbeta_0}_{\mathcal{B}^{\mathrm{agg}}}
+
\frac{1}{K}\sum_{\ell=1}^K
\widetilde f_\ell(\widehat{\boldsymbol{\Sigma}}_\ell)\frac{\bX_\ell^\top \boldsymbol{\varepsilon}_\ell}{n_\ell}.
\]
Expanding the squared $\boldsymbol{\Sigma}$-norm as a quadratic gives
\begin{align*}
\|\widehat\bbeta_{\mathrm{agg}}-\bbeta_0\|_{\boldsymbol{\Sigma}}^2
&=
\|\mathcal{B}^{\mathrm{agg}}\|_{\boldsymbol{\Sigma}}^2 \\
&\quad+
\frac{2}{K^2}\sum_{\ell=1}^K\sum_{\ell'=1}^K
\bbeta_0^\top g_\ell(\widehat{\boldsymbol{\Sigma}}_\ell)
\boldsymbol{\Sigma}\,
\widetilde f_{\ell'}(\widehat{\boldsymbol{\Sigma}}_{\ell'})
\frac{\bX_{\ell'}^\top\boldsymbol{\varepsilon}_{\ell'}}{n_{\ell'}} \\
&\quad+
\frac{1}{K^2}\sum_{\ell=1}^K
\frac{\boldsymbol{\varepsilon}_\ell^\top \bX_\ell}{n_\ell}
\widetilde f_\ell(\widehat{\boldsymbol{\Sigma}}_\ell)
\boldsymbol{\Sigma}\,
\widetilde f_\ell(\widehat{\boldsymbol{\Sigma}}_\ell)
\frac{\bX_\ell^\top\boldsymbol{\varepsilon}_\ell}{n_\ell} \\
&\quad+
\frac{2}{K^2}\sum_{\ell< \ell'}
\frac{\boldsymbol{\varepsilon}_\ell^\top \bX_\ell}{n_\ell}
\widetilde f_\ell(\widehat{\boldsymbol{\Sigma}}_\ell)
\boldsymbol{\Sigma}\,
\widetilde f_{\ell'}(\widehat{\boldsymbol{\Sigma}}_{\ell'})
\frac{\bX_{\ell'}^\top\boldsymbol{\varepsilon}_{\ell'}}{n_{\ell'}}.
\end{align*}
Adding and subtracting the conditional expectation of the diagonal quadratic
terms, we can write
\[
\|\widehat\bbeta_{\mathrm{agg}}-\bbeta_0\|_{\boldsymbol{\Sigma}}^2
=
\|\mathcal{B}^{\mathrm{agg}}\|_{\boldsymbol{\Sigma}}^2
+
\mathcal{V}^{\mathrm{agg}}
+
\mathcal{E}^{\mathrm{agg}},
\]
where the variance term is
\[
\mathcal{V}^{\mathrm{agg}}
=
\frac{\sigma_\beps^2}{K^2}
\sum_{\ell=1}^K
\frac{1}{n_\ell}
\tr\left(
\widehat{\boldsymbol{\Sigma}}_\ell\,
\widetilde f_\ell^2(\widehat{\boldsymbol{\Sigma}}_\ell)\,
\boldsymbol{\Sigma}
\right) := 
\sum_{\ell=1}^K\mathcal{V}^{\mathrm{agg}}_{\ell}
\]
and the error term $\mathcal{E}^{\mathrm{agg}}$ collects three negligible
contributions:
\begin{align*}
\mathcal{E}^{\mathrm{agg}}
&:=
\frac{2}{K^2}\sum_{\ell, \ell' = 1}^K
\bbeta_0^\top g_\ell(\widehat{\boldsymbol{\Sigma}}_\ell)
\boldsymbol{\Sigma}\,
\widetilde f_\ell(\widehat{\boldsymbol{\Sigma}}_{\ell'})
\frac{\bX_{\ell'}^\top\boldsymbol{\varepsilon}_{\ell'}}{n_{\ell}'} \\
&\quad+
\frac{1}{K^2}\sum_{\ell=1}^K
\left[
\frac{\boldsymbol{\varepsilon}_\ell^\top \bX_\ell}{n_\ell}
\widetilde f_\ell(\widehat{\boldsymbol{\Sigma}}_\ell)
\boldsymbol{\Sigma}\,
\widetilde f_\ell(\widehat{\boldsymbol{\Sigma}}_\ell)
\frac{\bX_\ell^\top\boldsymbol{\varepsilon}_\ell}{n_\ell}
-
\frac{\sigma_\beps^2}{n_\ell}
\tr\left(
\widehat{\boldsymbol{\Sigma}}_\ell\,
\widetilde f_\ell^2(\widehat{\boldsymbol{\Sigma}}_\ell)\,
\boldsymbol{\Sigma}
\right)
\right]\\
&\quad+
\frac{2}{K^2}\sum_{\ell < \ell'}
\frac{\boldsymbol{\varepsilon}_\ell^\top \bX_\ell}{n_\ell}
\widetilde f_\ell(\widehat{\boldsymbol{\Sigma}}_\ell)
\boldsymbol{\Sigma}\,
\widetilde f_{\ell'}(\widehat{\boldsymbol{\Sigma}}_{\ell'})
\frac{\bX_{\ell'}^\top\boldsymbol{\varepsilon}_{\ell'}}{n_{\ell'}}.
\end{align*}
The first two contributions to $\mathcal{E}^{\mathrm{agg}}$ converge to
zero almost surely by the same concentration arguments as in
Section~\ref{proof: general formula}: the first linear terms vanish by
Lemma~\ref{lem:linear_concentration}, the second quadratic terms vanish
by Lemma~\ref{lem:quadratic_concentration}.  Finally, we justify that the cross-client term is negligible too. For fixed $\ell \neq \ell'$, define
\[
\ba_p := \frac{1}{n_{\ell'}} \bX_\ell \widetilde{f}_\ell(\widehat{\boldsymbol{\Sigma}}_\ell) 
\boldsymbol{\Sigma}\, \widetilde{f}_{\ell'}(\widehat{\boldsymbol{\Sigma}}_{\ell'}) 
\bX_{\ell'}^\top \beps_{\ell'} \in \mathbb{R}^{n_\ell},
\]
so that the $(\ell,\ell')$ cross-client term equals $\ba_p^\top \beps_\ell / n_\ell$. By independence across machines, the vector $\ba_p$ is independent 
of $\beps_\ell$. Using $\bX_\ell^\top \bX_\ell = n_\ell \widehat{\boldsymbol{\Sigma}}_\ell$,
\begin{align*}
\frac{\|\ba_p\|_2^2}{n_\ell} 
&= \frac{1}{K^4 n_{\ell'}^2}\, 
\beps_{\ell'}^\top \bX_{\ell'} \widetilde{f}_{\ell'}(\widehat{\boldsymbol{\Sigma}}_{\ell'}) 
\boldsymbol{\Sigma}\, \widetilde{f}_\ell(\widehat{\boldsymbol{\Sigma}}_\ell)\, 
\widehat{\boldsymbol{\Sigma}}_\ell\, \widetilde{f}_\ell(\widehat{\boldsymbol{\Sigma}}_\ell) 
\boldsymbol{\Sigma}\, \widetilde{f}_{\ell'}(\widehat{\boldsymbol{\Sigma}}_{\ell'}) 
\bX_{\ell'}^\top \beps_{\ell'} \\
&\leq \frac{
\|\widetilde{f}_\ell(\widehat{\boldsymbol{\Sigma}}_\ell)\|_{\mathrm{op}}^2\, 
\|\boldsymbol{\Sigma}\|_{\mathrm{op}}^2\, 
\|\widetilde{f}_{\ell'}(\widehat{\boldsymbol{\Sigma}}_{\ell'})\|_{\mathrm{op}}^2\, 
\|\widehat{\boldsymbol{\Sigma}}_\ell\|_{\mathrm{op}}}{K^4} 
\cdot \frac{\|\bX_{\ell'}^\top \beps_{\ell'}\|_2^2}{n_{\ell'}^2}.
\end{align*}
The first factor is almost surely bounded for large $p$ by the same operator norm bounds of Section~\ref{proof: general formula}.  For the second factor, observe that $\|\bX_{\ell'}^\top \beps_{\ell'}\|^2 = \beps_{\ell'}^\top (\bX_{\ell'}\bX_{\ell'}^\top)\beps_{\ell'}$ is a quadratic form in $\beps_{\ell'}$, with $\bX_{\ell'}$ independent of $\beps_{\ell'}$. Applying Lemma~\ref{lem:quadratic_concentration} directly, we obtain
\[
\frac{\beps_{\ell'}^\top(\bX_{\ell'}\bX_{\ell'}^\top)\beps_{\ell'}}{n_{\ell'}^2}
-\sigma_\beps^2 \,\frac{\tr(\bX_{\ell'}\bX_{\ell'}^\top)}{n_{\ell'}^2}
\xrightarrow{a.s.} 0.
\]
Moreover,
\[
\sigma_\beps^2 \,\frac{\tr(\bX_{\ell'}\bX_{\ell'}^\top)}{n_{\ell'}^2}
=
\sigma_\beps^2 \,\frac{p}{n_{\ell'}}\,\frac{1}{p}\tr(\widehat\bSigma_{\ell'})
\xrightarrow[]{\text{a.s.}} \sigma_\beps^2 c \int x\,\mathrm{d}F_{\mathrm{MP}},
\]
by the standard Marchenko-Pastur limit from Lemma~\ref{lemma:f_delta}(a), applied to client~$\ell'$. We therefore deduce that
$$\lim_{p\to \infty} \frac{\|\bX_{\ell'}^\top \beps_{\ell'}\|_2^2}{n_{\ell'}^2} = \sigma_\beps^2 c \int x\,\mathrm{d}F_{\mathrm{MP}} < \infty,$$
hence $\limsup_p \|\ba_p\|^2/n_\ell < \infty$ almost surely, and Lemma~\ref{lem:linear_concentration} gives $\ba_p^\top\beps_\ell/n_\ell \xrightarrow{a.s.} 0$. This concludes the proof.
\end{proof}
\subsection{Proof of Lemma~\ref{prop: bias decomposition federated}}
\begin{proof}
By definition of the \({\boldsymbol{\Sigma}}\)-norm,
\[
\left\|\frac1K\sum_{\ell=1}^K g_\ell(\widehat{\boldsymbol{\Sigma}}_\ell)\bbeta_0\right\|^2_{{\boldsymbol{\Sigma}}}
=
\frac{1}{K^2}\bbeta_0^\top
\left(\sum_{\ell=1}^K g_\ell(\widehat{\boldsymbol{\Sigma}}_\ell)\right)
{\boldsymbol{\Sigma}}
\left(\sum_{\ell=1}^K g_\ell(\widehat{\boldsymbol{\Sigma}}_\ell)\right)\bbeta_0.
\]
Using the spiked decomposition of the covariance
\(
{\boldsymbol{\Sigma}} = \sx^2\mathbf{I}_p + \sum_{j=1}^{\rk}\delta_j \bv_j\bv_j^\top,
\)
we obtain
\begin{align*}
& \left\|\frac1K\sum_{\ell=1}^K g_\ell(\widehat{\boldsymbol{\Sigma}}_\ell)\bbeta_0\right\|^2_{{\boldsymbol{\Sigma}}} \\
&=
\frac{\sx^2}{K^2}\bbeta_0^\top
\left(\sum_{\ell=1}^K g_\ell(\widehat{\boldsymbol{\Sigma}}_\ell)\right)^2
\bbeta_0 +
\frac{1}{K^2}
\sum_{j=1}^{\rk}\delta_j\,
\bbeta_0^\top
\left(\sum_{\ell=1}^K g_\ell(\widehat{\boldsymbol{\Sigma}}_\ell)\right)
\bv_j\bv_j^\top
\left(\sum_{\ell=1}^K g_\ell(\widehat{\boldsymbol{\Sigma}}_\ell)\right)\bbeta_0.
\end{align*}
Now each matrix \(g_\ell(\widehat{\boldsymbol{\Sigma}}_\ell)\) is symmetric, because it is obtained by applying a spectral shrinkage function to the symmetric matrix \(\widehat{\boldsymbol{\Sigma}}_\ell\). Hence the sum \(\sum_{\ell=1}^K g_\ell(\widehat{\boldsymbol{\Sigma}}_\ell)\) is also symmetric, and so
\[
\bbeta_0^\top
\left(\sum_{\ell=1}^K g_\ell(\widehat{\boldsymbol{\Sigma}}_\ell)\right)
\bv_j\bv_j^\top
\left(\sum_{\ell=1}^K g_\ell(\widehat{\boldsymbol{\Sigma}}_\ell)\right)\bbeta_0
=
\left[
\bbeta_0^\top
\left(\sum_{\ell=1}^K g_\ell(\widehat{\boldsymbol{\Sigma}}_\ell)\right)
\bv_j
\right]^2.
\]
which gives exactly the claimed decomposition.
\end{proof}

\subsection{Proof of Lemma~\ref{prop: general formula for agg}}
\begin{proof}

  Using the decomposition from Lemma~\ref{prop: bias variance decomposition federated} and Lemma~\ref{prop: bias decomposition federated}, it is enough to compute the limit of each term $\mathcal{B}_{0}^{\mathrm{agg}}, \mathcal{B}_{j}^{\mathrm{agg}}, \mathcal{V}_{\ell}^\mathrm{agg}$ individually.   First, note that for each $\ell \in \{1, 2, \dots K\}$ $$\mathcal{V}_{\ell}^{\mathrm{agg}} \xrightarrow[]{\text{a.s.}} 
    \frac{c\sx^2{\sigma_{\beps}^2}}{K^2} \int x\widetilde{f}_\ell(x)^2 
    \, \mathrm{d}F_{\mathrm{MP}}(x)$$
follows mutatis mutandis from the result in~\eqref{limit of variance}. 
Similarly, for each $j \in \{1, 2,\dots, \rk\}$, $$
K^2\mathcal{B}_{j}^{\mathrm{agg}} \xrightarrow{\text{a.s.}} 
    \delta_j \alpha_j^2 \left( \sum_{\ell=1}^K \int (x\widetilde{f}_\ell(x) - 1) 
    \, \mathrm{d}F_{\delta_j} \right)^2,
$$
which can be rearranged to
\begin{align*}
K^2\mathcal{B}_{j}^{\mathrm{agg}} &\xrightarrow{\text{a.s.}} \delta_j \alpha_j^2 \left( \sum_{\ell=1}^K \int (x\widetilde  f_\ell(x) - 1) \, \mathrm dF_{\delta_j} \right)^2 \\
& = \delta_j \alpha_j^2 \left[ \left( \sum_{\ell=1}^K \int x \widetilde f_\ell \, \mathrm dF_{\delta_j} \right)^2 - 2K \sum_{\ell=1}^K \int x \widetilde f_\ell \, \mathrm dF _{\delta_j} + K^2 \right].
\end{align*}
The proof follows by the same method as in~\eqref{limit of bias term}, using the corresponding empirical measure $\widehat C_n^j$ (see Lemma~\ref{lemma: hat C nj}) of each client and adding the limits together. We skip the details for the sake of brevity. 
Finally, using Theorem~\ref{lemma:non-free limit}, we can immediately find the contribution of the bulk bias $\mathcal{B}_{0}^{\mathrm{agg}}$ as follows
\begin{align*}
\frac{K^2\mathcal{B}_{0}^{\mathrm{agg}}}{\sx^2\,\|\bbeta_0\|_2^2} 
    & \xrightarrow{\text{a.s.}} 
    2\omega_0 \sum_{\ell < {\ell'}}\int g_\ell \,\mathrm{d}F_{\mathrm{MP}}
    \cdot \int g_{\ell'} \,\mathrm{d}F_{\mathrm{MP}} \\
    & + 2\sum_{j=1}^\rk \omega_j \sum_{\ell < {\ell'}}  
    \left(\int g_\ell \,\mathrm{d}F_{\delta_j}\right) 
    \left(\int g_{\ell'} \,\mathrm{d}F_{\delta_j}\right) 
    + \sum_{\ell = 1}^K \int g_\ell^2 \, \mathrm{d}F_\alpha.
\end{align*}
Using the definition of $g_\ell$ in Lemma~\ref{prop: normalized estimator federated} we obtain
\begin{align*}
& K^2 \mathcal{B}_{0}^{\mathrm{agg}}  \\
&\xrightarrow{\text{a.s.}} 2\sx^2 r^2\omega_0 \left[ \sum_{\ell < {\ell'}} \left( \int x \widetilde{f}_\ell \, \mathrm dF _{\mathrm{MP}} \right) \left( \int x \widetilde{f}_{\ell'} \, \mathrm dF _{\mathrm{MP}} \right) \right. \\
& \hspace{13em} \left. - (K-1)\sum_{\ell=1}^K \int x \widetilde f_\ell \, \mathrm dF _{\mathrm{MP}} + \binom{K}{2} \right] \\
& + 2\sx^2\,r^2 \sum_{j=1}^{\rk} \omega_j \left[ \sum_{\ell < {\ell'}} \left( \int x \widetilde f_\ell \, \mathrm dF_{\delta_j} \right) \left( \int x \widetilde f_{\ell'} \, \mathrm dF _{\delta_j} \right) - (K-1)\sum_{\ell=1}^K \int x \widetilde f_\ell \, \mathrm dF _{\delta_j} + \binom{K}{2} \right] \\
&\quad + \sx^2\,r^2 \left[ \sum_{\ell=1}^K \int x^2 \widetilde f_\ell^2 \, \mathrm dF _\alpha - 2 \sum_{\ell=1}^K \int x \widetilde f_\ell \, \mathrm dF _\alpha + K \right].
\end{align*}
Therefore, using the Radon-Nikodym derivatives from Lemma~\ref{lemma: def mu_j mu_0}, and collecting all terms together, the full limit is
\begin{align*}
& K^2 \, \mathcal R^\mathrm{pred}_K(\widetilde f_1,\dots,\widetilde f_K) \\
&= \sum_{\ell=1}^K \int \left( \sx^2\,r^2x^2 + c\sx^2\,{\sigma_{\beps}^2} x \mu_0(x) \right) \widetilde f_\ell(x)^2 \, \mathrm{d}F_\alpha(x) \\
& \quad - 2 \sum_{\ell=1}^K \int x \widetilde f_\ell(x) \left[ \sx^2\,r^2 + \sx^2\,r^2\omega_0 (K-1) \mu_0(x) \right. \\
& \qquad \qquad \qquad \left.+ \sum_{j=1}^{\rk} \alpha_j^2 \left(K\delta_j + \sx^2 (K-1)\right) \mu_j(x) \right] \, \mathrm{d}F_\alpha(x) \\
&\quad + 2\sx^2\,r^2\omega_0 \sum_{\ell < {\ell'}} \left( \int x \widetilde f_\ell \mu_0 \, \mathrm{d}F_\alpha \right) \left( \int x \widetilde f_{\ell'} \mu_0 \, \mathrm{d}F_\alpha \right)  \\ 
&\quad+ 2 \sum_{j=1}^{\rk} \sx^2\,\alpha_j^2 \sum_{\ell < {\ell'}} \left( \int x \widetilde f_\ell \mu_j \, \mathrm{d}F_\alpha \right) \left( \int x \widetilde f_{\ell'} \mu_j \, \mathrm{d}F_\alpha \right) \\
&\quad + \sum_{j=1}^{\rk} \delta_j \alpha_j^2 \left( \sum_{\ell=1}^K \int x \widetilde f_\ell \mu_j \, \mathrm{d}F_\alpha \right) ^2 + K^2 \left( \sx^2\,r^2 + \sum_{j=1}^{\rk} \delta_j \alpha_j^2 \right).
\end{align*}
Recall now from Lemma~\ref{lemma: def mu_j mu_0} that
\[
\omega_0 \mu_0(x) + \sum_{j=1}^\rk \omega_j \mu_j(x) = 1
\qquad\text{for all }x\in\mathcal S_c^+,
\]
and \(\alpha_j^2=r^2\omega_j\), so we can simplify the linear term to obtain
\begin{align*}
K^2 \, \mathcal R^\mathrm{pred}_K(\widetilde f_1,\dots,\widetilde f_K) &= \sum_{\ell=1}^K \int \left( \sx^2\,r^2x^2 + c\sx^2\,{\sigma_{\beps}^2} x \mu_0(x) \right) \widetilde f_\ell(x)^2 \, \mathrm{d}F_\alpha(x)\\
&\quad- 2 \sum_{\ell=1}^K \int x \widetilde f_\ell(x) \left[ \sx^2\,r^2 K + K \sum_{j=1}^{\rk} \delta_j \alpha_j^2 \mu_j(x) \right] \, \mathrm{d}F_\alpha(x) \\
&\quad + 2\sx^2\,r^2\omega_0 \sum_{\ell < {\ell'}} \left( \int x \widetilde f_\ell \mu_0 \, \mathrm{d}F_\alpha \right) \left( \int x \widetilde f_{\ell'} \mu_0 \, \mathrm{d}F_\alpha \right)\\ 
&\quad+ 2 \sum_{j=1}^{\rk} \sx^2\,\alpha_j^2 \sum_{\ell < {\ell'}} \left( \int x \widetilde f_\ell \mu_j \, \mathrm{d}F_\alpha \right) \left( \int x \widetilde f_{\ell'} \mu_j \, \mathrm{d}F_\alpha \right) \\
&\quad + \sum_{j=1}^{\rk} \delta_j \alpha_j^2 \left( \sum_{\ell=1}^K \int x \widetilde f_\ell \mu_j \, \mathrm{d}F_\alpha \right) ^2  + K^2 \left( r^2 + \sum_{j=1}^{\rk} \delta_j \alpha_j^2 \right).
\end{align*}
Using the auxiliary functions defined in \eqref{eq: define w ,g , h_j}, and the definition of the $xw(x)$-weighted inner product~\eqref{eq: define inner} conclusion follows by straightforward algebraic manipulations.
\end{proof}

\subsection{Proof of Lemma~\ref{prop: decoupled formula}}
\begin{proof}
Substituting \(\widetilde f_\ell=\bar \vartheta+\vartheta_\ell\), we first compute
\begin{align}
\label{eq: f tilde_ell with bar v}
\sum_{\ell=1}^K \|\widetilde f_\ell\|_w^2
& =
\sum_{\ell=1}^K \|\bar \vartheta+\vartheta_\ell\|_w^2 \notag \\
& =
\sum_{\ell=1}^K \|\vartheta_\ell\|_w^2
+2\sum_{\ell=1}^K \langle \vartheta_\ell,\bar \vartheta\rangle_w
+K\|\bar \vartheta\|_w^2 \notag \\
& = \sum_{\ell=1}^K \|\vartheta_\ell\|_w^2
+K\|\bar \vartheta\|_w^2.
\end{align}
where the middle term vanishes since \(\sum_{\ell=1}^K \vartheta_\ell=0\).
Next,
\begin{align}\label{eq: g with bar v}
 -2K\sum_{\ell=1}^K \langle g,\widetilde f_\ell\rangle_w
=
-2K\sum_{\ell=1}^K \langle g,\bar \vartheta+\vartheta_\ell\rangle_w
=
-2K^2\langle g,\bar \vartheta\rangle_w,
\end{align}
again because \(\sum_{\ell=1}^K \vartheta_\ell=0\). Now fix \(j\in\{0,1,\dots,\rk\}\). Then
\[
\begin{aligned}
& \sum_{\ell < {\ell'}}\langle h_j,\widetilde f_\ell\rangle_w\langle h_j,\widetilde f_{\ell'}\rangle_w \\
&=
\sum_{\ell < {\ell'}}
\bigl(\langle h_j,\bar \vartheta\rangle_w+\langle h_j,\vartheta_\ell\rangle_w\bigr)
\bigl(\langle h_j,\bar \vartheta\rangle_w+\langle h_j,\vartheta_{\ell'}\rangle_w\bigr) \\
&=
\binom{K}{2}\langle h_j,\bar \vartheta\rangle_w^2
+\langle h_j,\bar \vartheta\rangle_w
\sum_{\ell < {\ell'}}\bigl(\langle h_j,\vartheta_\ell\rangle_w+\langle h_j,\vartheta_{\ell'}\rangle_w\bigr) +\sum_{\ell < {\ell'}}\langle h_j,\vartheta_\ell\rangle_w\langle h_j,\vartheta_{\ell'}\rangle_w.
\end{aligned}
\]
Because each \(\langle h_j,\vartheta_\ell\rangle_w\) appears exactly \(K-1\) times in the mixed sum,
\[
\sum_{\ell < {\ell'}}\bigl(\langle h_j,\vartheta_\ell\rangle_w+\langle h_j,\vartheta_{\ell'}\rangle_w\bigr)
=
(K-1)\sum_{\ell=1}^K \langle h_j,\vartheta_\ell\rangle_w
=
0,
\]
and so
\begin{align}\label{eq: cross h_j h_ell with bar v}
\sum_{\ell < {\ell'}}\langle h_j,\widetilde f_\ell\rangle_w\langle h_j,\widetilde f_{\ell'}\rangle_w
=
\binom{K}{2}\langle h_j,\bar \vartheta\rangle_w^2
+\sum_{\ell < {\ell'}}\langle h_j,\vartheta_\ell\rangle_w\langle h_j,\vartheta_{\ell'}\rangle_w.
\end{align}
Also, it is immediate by definition that
\begin{align}\label{eq: single h_j with bar v}
\left(\sum_{\ell=1}^K \langle h_j,\widetilde f_\ell\rangle_w\right)^2
=
K^2\langle h_j,\bar \vartheta\rangle_w^2.
\end{align}
Collecting the above identities ~\eqref{eq: f tilde_ell with bar v}, ~\eqref{eq: g with bar v}, \eqref{eq: cross h_j h_ell with bar v} and ~\eqref{eq: single h_j with bar v} gives
\[
\begin{aligned}
& K^2 \, \mathcal R^\mathrm{pred}_K(\widetilde f_1,\dots,\widetilde f_K) \\
&=
\sum_{\ell=1}^K \|\vartheta_\ell\|_w^2 + 
K\|\bar \vartheta\|_w^2
-2K^2\langle g,\bar \vartheta\rangle_w 
+2\sx^2\,r^2\omega_0\sum_{\ell < {\ell'}}\langle h_0,\vartheta_\ell\rangle_w\langle h_0,\vartheta_{\ell'}\rangle_w \\
&\quad
+2\sx^2\,r^2\omega_0\binom{K}{2}\langle h_0,\bar \vartheta\rangle_w^2
+2\sum_{j=1}^{\rk}\sx^2\,\alpha_j^2
\left[
\binom{K}{2}\langle h_j,\bar \vartheta\rangle_w^2
+\sum_{\ell < {\ell'}}\langle h_j,\vartheta_\ell\rangle_w\langle h_j,\vartheta_{\ell'}\rangle_w
\right] \\
&\quad +\sum_{j=1}^{\rk}\delta_j\alpha_j^2 K^2\langle h_j,\bar \vartheta\rangle_w^2
+K^2\left(r^2+\sum_{j=1}^{\rk}\delta_j\alpha_j^2\right).
\end{aligned}
\]
Finally, since \(\sum_{\ell=1}^K \langle h_j,\vartheta_\ell\rangle_w=0\), we use
\[
\left(\sum_{\ell=1}^K \langle h_j,\vartheta_\ell\rangle_w\right)^2
=
\sum_{\ell=1}^K \langle h_j,\vartheta_\ell\rangle_w^2
+2\sum_{\ell < {\ell'}}\langle h_j,\vartheta_\ell\rangle_w\langle h_j,\vartheta_{\ell'}\rangle_w
=0,
\]
which implies that
\[
\sum_{\ell < {\ell'}}\langle h_j,\vartheta_\ell\rangle_w\langle h_j,\vartheta_{\ell'}\rangle_w
=
-\frac12\sum_{\ell=1}^K \langle h_j,\vartheta_\ell\rangle_w^2.
\]
Substituting this identity for \(j=0,1,\dots,\rk\) yields
\[
\begin{aligned}
& K^2 \, \mathcal R^\mathrm{pred}_K(\widetilde f_1,\dots,\widetilde f_K) \\
&=\sum_{\ell=1}^K \|\vartheta_\ell\|_w^2
-\sx^2\,r^2\omega_0\sum_{\ell=1}^K \langle h_0,\vartheta_\ell\rangle_w^2
-\sum_{j=1}^{\rk}\sx^2\,\alpha_j^2\,\sum_{\ell=1}^K \langle h_j,\vartheta_\ell\rangle_w^2 \\
&\quad
+K\|\bar \vartheta\|_w^2
-2K^2\langle g,\bar \vartheta\rangle_w
+\sx^2\,r^2\omega_0\,K(K-1)\langle h_0,\bar \vartheta\rangle_w^2 \\
&\quad
+\sum_{j=1}^{\rk}\alpha_j^2\bigl(K(K-1)\sx^2\,+\delta_jK^2\bigr)\langle h_j,\bar \vartheta\rangle_w^2
+K^2\left(\sx^2\,r^2+\sum_{j=1}^{\rk}\delta_j\alpha_j^2\right),
\end{aligned}
\]
which is exactly the claimed formula.
\end{proof}

\subsection{Proof of Lemma~\ref{all f' need to be the same}}
\begin{proof}
We first rewrite the quadratic terms using the definitions from
Lemma~\ref{lemma: def mu_j mu_0} and \eqref{eq: define w ,g , h_j}. For the bulk term,
\begin{align*}
\langle h_0,\phi\rangle_w
&= \int x\,\phi(x)\,w(x)\,h_0(x)\,dF_\alpha(x) = \int x\,\phi(x)\,\mu_0(x)\,dF_\alpha(x)  = \int x\,\phi(x)\,dF_{\mathrm{MP}}(x),
\end{align*}
and therefore
\[
\sx^2\,r^2\omega_0\langle h_0,\phi\rangle_w^2
=
\sx^2\,r^2\omega_0\left(\int x\,\phi(x)\,dF_{\mathrm{MP}}(x)\right)^2.
\]
Similarly, for each spike \(j\),
\begin{align*}
\langle h_j,\phi\rangle_w
&= \int x\,\phi(x)\,w(x)\,h_j(x)\,dF_\alpha(x) = \int x\,\phi(x)\,\mu_j(x)\,dF_\alpha(x) = \int x\,\phi(x)\,dF_{\delta_j}(x),
\end{align*}
so that
\[
\alpha_j^2\langle h_j,\phi\rangle_w^2
=
\alpha_j^2\left(\int x\,\phi(x)\,dF_{\delta_j}(x)\right)^2.
\]
Next, since \(\mu_0(x)\ge 0\) on \(\mathcal S_c^+\), we have
\begin{align*}
\|\phi\|_w^2
&= \int x\,\phi(x)^2\,w(x)\,dF_\alpha(x) \\
&= \int x\,\phi(x)^2\bigl(\sx^2\,r^2x+c\sx^2\,\sigma_{\boldsymbol{\varepsilon}}^2\mu_0(x)\bigr)\,dF_\alpha(x) \\
&\ge \sx^2\,r^2\int x^2\phi(x)^2\,dF_\alpha(x).
\end{align*}
Recall from Lemma~\ref{lemma: f_alpha} that \(F_\alpha\) decomposes as a mixture, hence
\[
\sx^2\,r^2\int x^2\phi(x)^2\,dF_\alpha(x)
=
\sx^2\,r^2\omega_0\int x^2\phi(x)^2\,dF_{\mathrm{MP}}(x)
+\sum_{j=1}^{\rk}\sx^2\,\alpha_j^2\,\int x^2\phi(x)^2\,dF_{\delta_j}(x),
\]
where we used \(\alpha_j^2=r^2\omega_j\). Applying Cauchy-Schwarz with respect to the probability measures
\(F_{\mathrm{MP}}\) and \(F_{\delta_j}\) gives
\[
\left(\int x\,\phi(x)\,dF_{\mathrm{MP}}(x)\right)^2
\le
\int x^2\phi(x)^2\,dF_{\mathrm{MP}}(x),
\]
and, for each \(j\),
\[
\left(\int x\,\phi(x)\,dF_{\delta_j}(x)\right)^2
\le
\int x^2\phi(x)^2\,dF_{\delta_j}(x).
\]
Combining the previous displays, we obtain
\begin{align*}
& \sx^2\,r^2\omega_0\langle h_0,\phi\rangle_w^2
+\sum_{j=1}^{\rk}\sx^2\,\alpha_j^2\,\langle h_j,\phi\rangle_w^2 \\
&\le
\sx^2\,r^2\omega_0\int x^2\phi(x)^2\,dF_{\mathrm{MP}}(x)
+\sum_{j=1}^{\rk}\sx^2\,\alpha_j^2\,\int x^2\phi(x)^2\,dF_{\delta_j}(x) \\
&=
\sx^2\,r^2\int x^2\phi(x)^2\,dF_\alpha(x) \le \|\phi\|_w^2.
\end{align*}
This proves the inequality. For equality to hold, equality must occur in every step above. In particular,
\[
\int x\,\phi(x)^2\,\mu_0(x)\,dF_\alpha(x)=0.
\]
Since the integrand is nonnegative and \(\mu_0(x)>0\) on \(\mathcal S_c^+\), this implies
\(\phi(x)=0\) on \(\mathcal S_c^+\). Conversely, if \(\phi\) vanishes on \(\mathcal S_c^+\), then every term above is equal to zero, so equality holds.
\end{proof}

\subsection{Analysis of Assumption~\ref{assm: b0^K neq 0}}\label{sec: assm b0} 
Recall from equation~\eqref{eq: linear system with b0} that $b^{(K)}_0$ is found by solving the following linear system:
\begin{align}
    \left(\mathbf I_{\rk+ 1} + \mathfrak D_K \bH\right) \begin{pmatrix}
    b^{(K)}_0\\ b^{(K)}_1 \\\vdots\\ b^{(K)}_{\rk} 
\end{pmatrix} = \begin{pmatrix}
    \sx^2\,r^2\omega_0\\ (\delta_1 + \sx^2\,)\alpha_1^2 \\\vdots\\ (\delta_\rk + \sx^2\,)\alpha_\rk^2 
\end{pmatrix}.
\end{align}
In this section, we present certain conditions under which $b^{(K)}_0 \neq 0$, and further strengthen this property with simulations.

\subsubsection{Isotropic Model}
We begin with the isotropic case ${\boldsymbol{\Sigma}} = \sx^2\mathbf{I}_p$, corresponding to 
$\rk = 0$, in which the system~\eqref{eq: linear system with b0} can be solved  explicitly.

\begin{proposition}
Assume $\rk=0$. Then for every integer $K\ge 1$, the coefficient $b^{(K)}_0$ 
from~\eqref{eq: linear system with b0} is given by
\[
b^{(K)}_0=\frac{\sx^2\,r^2}{1+\sx^2\,r^2(K-1)\langle h_0,h_0\rangle_w},
\]
and in particular $b^{(K)}_0>0$.
\end{proposition}

\begin{proof}
If $\rk=0$, then $\omega_0=1$ since we have no spike directions and the system reduces to the scalar equation
\[
\left(1+\sx^2\,r^2(K-1)\|h_0\|_w^2\right)b^{(K)}_0=\sx^2\,r^2.
\]
Hence
\[
b^{(K)}_0=\frac{\sx^2\,r^2}{1+\sx^2r^2(K-1)\|h_0\|_w^2}.
\]
Since $\|h_0\|_w^2\ge 0$, the denominator is strictly positive, so $b^{(K)}_0>0$.
\end{proof}
\noindent
Using equation~\eqref{eq: f_opt federated}, we may choose the optimal aggregation weights as \[
\rho_{\ast,1} = \cdots = \rho_{\ast,K} = \frac{b^{(K)}_0}{\sx^2r^2} = \frac{1}{1 + \sx^2\,r^2(K-1)\|h_0\|_w^2}.
\]
{For each client, the weighted filter is}
\[
f_K^\ast(x)
= \frac{b^{(K)}_0}{\sx^2\,r^2x + c\sx^2\,\sigma_{\boldsymbol{\varepsilon}}^2}
= \frac{b^{(K)}_0}{\sx^2\,r^2}
  \cdot \frac{1}{x + c\sigma_{\boldsymbol{\varepsilon}}^2/r^2}.
\]
{Since $f_K^\ast(x)=\rho_{\ast,\ell}f^{\rm pred}_{\ast,\ell}(x)$,
the local shrinkage rule is the same tuned ridge filter for every client:
\[
 f^{\rm pred}_{\ast,\ell}(x)
 = \frac{f_K^\ast(x)}{\rho_{\ast,\ell}}
 = \frac{1}{x+c\sigma_{\boldsymbol{\varepsilon}}^2/r^2},
 \qquad 1\le\ell\le K.
\]
}

Thus, the optimal shrinkage rule for each client corresponds to tuned Ridge and for $\sx^2 = 1$ we recover the aggregated estimator from Theorem 4.5 in \cite{wonder}.
\subsubsection{One-Spike Model}
In this section, we analyze the one-spike case, ${\boldsymbol{\Sigma}} = \sx^2\mathbf{I}_p + 
\delta_1 \bv_1 \bv_1^\top$. To simplify the notation,
we write
\[
\delta:=\delta_1,\qquad \omega:=\omega_1,\qquad \omega_0=1-\omega,\qquad \alpha^2=r^2\omega.
\]
\begin{proposition}\label{prop:one_spike_b0_nonzero}
Assume $\rk=1$. Then for every integer $K\ge 1$, the coefficient $b^{(K)}_0$ 
from~\eqref{eq: linear system with b0} satisfies $b^{(K)}_0>0$. 
In particular, $b^{(K)}_0\neq0$.
\end{proposition}

\begin{proof}
When $\rk=1$, the linear system~\ref{eq: linear system with b0} becomes
$$
\mathfrak{R}\begin{pmatrix}
b^{(K)}_0\\ b^{(K)}_1
\end{pmatrix}
=
\begin{pmatrix}
\sx^2\,r^2\omega_0\\ (\sx^2\,+\delta)\alpha^2
\end{pmatrix} \,,
$$
where 
\[
\mathfrak{R} = \begin{pmatrix}
1+(K-1)\sx^2\,r^2\omega_0\langle h_0,h_0\rangle_w &
(K-1)\sx^2\,r^2\omega_0\langle h_0,h_1\rangle_w \\[1mm]
\bigl((K-1)\sx^2\,+K\delta\bigr)\alpha^2\langle h_1,h_0\rangle_w &
1+\bigl((K-1)\sx^2\,+K\delta\bigr)\alpha^2\langle h_1,h_1\rangle_w
\end{pmatrix}\,.
\]
Since $\langle h_1,h_0\rangle_w=\langle h_0,h_1\rangle_w$, Cramer's rule gives
\begin{align*}
b^{(K)}_0 &=
\frac{
\sx^2r^2\omega_0\Bigl(1+\bigl((K-1)\sx^2+K\delta\bigr)\alpha^2\langle h_1,h_1\rangle_w\Bigr)
-(K-1)\sx^2r^2\omega_0(\sx^2+\delta)\alpha^2\langle h_0,h_1\rangle_w
}{
\det\bigl(\mathbf I_2+\mathfrak D_K\bH\bigr)
} \\
&=
\frac{
\sx^2r^2\omega_0\Bigl[
1+\delta\alpha^2\langle h_1,h_1\rangle_w
+(K-1)(\sx^2+\delta)\alpha^2\bigl(\langle h_1,h_1\rangle_w-\langle h_0,h_1\rangle_w\bigr)
\Bigr]
}{
\det \bigl(\mathbf I_2+\mathfrak D_K\bH\bigr)
}.
\end{align*}
By the same argument as in~\ref{eq: HD is invertible}, the denominator is positive, thus it is enough to prove that
\[
\langle h_1,h_1\rangle_w-\langle h_0,h_1\rangle_w \geq 0,
\]
or equivalently, using equation~\ref{eq: define w ,g , h_j} 
\begin{align*}
\langle h_1,h_1\rangle_w-\langle h_0,h_1\rangle_w & = \int xw(x)h_1(x)(h_1(x) - h_0(x))\, \mathrm{d}F_{\alpha} \\
& = \int x(h_1(x) - h_0(x))\, \mathrm{d}F_{\delta} \geq 0.
\end{align*}
Specializing the formulas in Lemma~\ref{lemma: def mu_j mu_0} to $\rk = 1$ we have 
$$\mu_0(x) = \frac{\nu(x)}{(1-\omega)\nu(x) + \omega}, \qquad \mu_1(x) = \frac{1}{\omega_0\nu(x) + \omega_1}, \qquad \nu(x) = \frac{\delta(x_\delta^\star - x)}{c\sx^2(\delta + \sx^2)},$$ 
so
\begin{align*}
\mu_1(x) - \mu_0(x) 
&= \frac{1-\nu(x)}{(1-\omega)\nu(x)+\omega} \,, \quad \text{and}\\ 
r^2x + c\sigma_{\boldsymbol{\varepsilon}}^2\mu_0(x) & = \frac{r^2x\bigl[(1-\omega)\nu(x)+\omega\bigr]+c\sigma_{\boldsymbol{\varepsilon}}^2\nu(x)}
        {(1-\omega)\nu(x)+\omega}.
\end{align*}
Dividing cancels the common denominator:
\begin{align}\label{eq: xmu_1 - mu_0}
x(h_1 - h_0) = \frac{x(\mu_1-\mu_0)}{\sx^2\,r^2x+c\sx^2\,\sigma_{\boldsymbol{\varepsilon}}^2\mu_0} 
= \frac{x(1-\nu(x))}{\sx^2\,r^2x\bigl[(1-\omega)\nu(x)+\omega\bigr]+c\sx^2\,\sigma_{\boldsymbol{\varepsilon}}^2\nu(x)}.
\end{align}
Using the definition of $x_\delta^\star$ (see Lemma~\ref{lemma:f_delta}), we have
$x_\delta^\star = (\delta+\sx^2)(1+c\sx^2/\delta)$, so 
$$c\sx^2(\delta+\sx^2) - \delta x_\delta^\star 
= c\sx^2(\delta+\sx^2) - \delta(\delta+\sx^2)(1+c\sx^2/\delta) = -\delta(\delta+\sx^2),$$ 
and thus
\[
1-\nu(x) 
= 1 - \frac{\delta(x_\delta^\star-x)}{c\sx^2(\delta+\sx^2)} 
= \frac{c\sx^2(\delta+\sx^2)-\delta x_\delta^\star + \delta x}{c\sx^2(\delta+\sx^2)} 
= \frac{\delta\bigl(x-(\delta+\sx^2)\bigr)}{c\sx^2(\delta+\sx^2)}.
\]
Substituting the previous display back in equation~\eqref{eq: xmu_1 - mu_0}:
\begin{align}\label{eq: def phi corr ineq}
\frac{x(\mu_1-\mu_0)}{r^2x+c\sigma_{\boldsymbol{\varepsilon}}^2\mu_0} 
= \frac{\delta x}{c\sx^2(\delta+\sx^2)} 
  \cdot 
  \frac{x-(\delta + \sx^2)}{r^2x\bigl[(1-\omega)\nu(x)+\omega\bigr]+c\sigma_{\boldsymbol{\varepsilon}}^2\nu(x)}
= \bigl[x-(\delta+\sx^2)\bigr]\,\phi(x),
\end{align}
where
\[
\phi(x) = \frac{\delta x}
{\delta(x_\delta^\star-x)\bigl[(r^2-\alpha^2)x+c\sigma_{\boldsymbol{\varepsilon}}^2\bigr]
+\alpha^2 c\sx^2(\delta+\sx^2)x}.
\]
A quick computation gives
\[
\phi'(x) 
= \frac{\delta^2\bigl[(r^2-\alpha^2)x^2 + c\sigma_{\boldsymbol{\varepsilon}}^2 x_\delta^\star\bigr]}
       {\left[\delta(x_\delta^\star-x)\bigl[(r^2-\alpha^2)x+c\sigma_{\boldsymbol{\varepsilon}}^2\bigr]
+\alpha^2 c\sx^2(\delta+\sx^2)x\right]^2} > 0,
\]
hence $\phi$ is strictly increasing on the support of $F_\delta$.
Finally, note that 
\[
1 = \int 1 \, \mathrm{d}F_{\mathrm{MP}}= \int \nu(x)\,\mathrm{d}F_{\delta}(x)
= \frac{\delta}{c\sx^2(\delta+\sx^2)}\int \bigl(x_\delta^\star - x\bigr)\,\mathrm{d}F_{\delta}(x),
\]
where in the last step we substituted the definition of $\nu(x)$. Rearranging gives
\begin{align}\label{eq: mean of F_delta}
    \int x\,\mathrm{d}F_{\delta} 
= x_\delta^\star - \frac{c\sx^2(\delta+\sx^2)}{\delta} 
= (\delta+\sx^2)\Bigl(1+\frac{c\sx^2}{\delta}\Bigr) - \frac{c\sx^2(\delta+\sx^2)}{\delta} 
= \delta+\sx^2,
\end{align}
where we used $x_\delta^\star = (\delta+\sx^2)(1+c\sx^2/\delta)$.
Plugging \eqref{eq: def phi corr ineq} into \eqref{eq: xmu_1 - mu_0}, we deduce 
that 
\[
\langle h_1,h_1\rangle_w-\langle h_0,h_1\rangle_w 
= \int \bigl[x-(\delta+\sx^2)\bigr]\,\phi(x) \, \mathrm{d}F_{\delta}.
\]
Now let $T$ be a random variable with distribution $F_{\delta}$. 
By~\eqref{eq: mean of F_delta}, $\mathbb{E}[T] = \delta + \sx^2$, so we deduce that
\[
\langle h_1,h_1\rangle_w-\langle h_0,h_1\rangle_w = \cov(T,\, \phi(T)).
\]
Since $\phi$ is strictly increasing on the support of $F_\delta$, the conclusion 
follows by Chebyshev's correlation inequality.
\end{proof}
\subsubsection{General case}
Finally, for arbitrary number of spikes $\rk \geq 2$, we show that only finitely many values 
of $\sigma_{\boldsymbol{\varepsilon}}^2$ can make $b^{(K)}_0$ vanish. We prove this in the next Lemma.
\begin{lemma}\label{lemma:finitely_many_sigma_bad}
Fix $K\geq 2$, $c>0$, $r>0$, $(\delta_j)_{j=1}^{\rk}$, and 
$(\omega_j)_{j=0}^{\rk}$ satisfying Assumption~\ref{assm:main}. Then there 
exist at most finitely many values of $\sigma_{\boldsymbol{\varepsilon}}^2\geq 0$ such that
\[
b^{(K)}_0(\sigma_{\boldsymbol{\varepsilon}}^2)=0.
\]
\end{lemma}

\begin{proof}
Recall from~\eqref{eq: linear system with b0} that the vector
\[
b(\sigma_{\boldsymbol{\varepsilon}}^2):=
\begin{pmatrix}
b^{(K)}_0(\sigma_{\boldsymbol{\varepsilon}}^2),& \!\!\!
b^{(K)}_1(\sigma_{\boldsymbol{\varepsilon}}^2),&\!\!\!
\dots & \!\!\!b^{(K)}_\rk(\sigma_{\boldsymbol{\varepsilon}}^2)
\end{pmatrix}\!^\top
\]
is defined by
\[
\Bigl(\bI_{\rk+1}+\mathfrak D_K\bH(\sigma_{\boldsymbol{\varepsilon}}^2)\Bigr)
b(\sigma_{\boldsymbol{\varepsilon}}^2)
=
\begin{pmatrix}
\sx^2r^2\omega_0\\
(\sx^2+\delta_1)\alpha_1^2\\
\vdots\\
(\sx^2+\delta_\rk)\alpha_\rk^2
\end{pmatrix}
:= \gamma.
\]

We first show that each entry $\bH_{ij}(\sigma_{\boldsymbol{\varepsilon}}^2)$ is 
real-analytic on $[0,\infty)$. Since the integrand contains the prefactor $x$, 
the atom at $0$ does not contribute. Thus, using the definitions 
in~\eqref{eq: define w ,g , h_j},
\[
\bH_{ij}(\sigma_{\boldsymbol{\varepsilon}}^2)
=
\int_a^b \frac{x\,\mu_i(x)\mu_j(x)}{\sx^2r^2x+c\sx^2\,\sigma_{\boldsymbol{\varepsilon}}^2\mu_0(x)}
\,\mathrm{d}F_\alpha(x)
+
\sum_{\ell=1}^\rk
\frac{x_\ell^\star\,\mu_i(x_\ell^\star)\mu_j(x_\ell^\star)}
{\sx^2r^2x_\ell^\star+c\sx^2\,\sigma_{\boldsymbol{\varepsilon}}^2\mu_0(x_\ell^\star)}
\,F_\alpha(\{x_\ell^\star\}).
\]
At every nonzero outlier atom $x_\ell^\star$, one has
\[
\mu_0(x_\ell^\star)=0,\qquad
\mu_i(x_j^\star)=\omega_i^{-1} \mathbf 1[i = j].
\]
Hence
\begin{align}\label{eq: expansion of H_ij tilde}
\bH_{ij}(\sigma_{\boldsymbol{\varepsilon}}^2)
=
\int_a^b \frac{x\,\mu_i(x)\mu_j(x)}{\sx^2r^2x+c\sx^2\,\sigma_{\boldsymbol{\varepsilon}}^2\mu_0(x)}
\,\mathrm{d}F_\alpha(x)
+
\frac{\mathbf{1}[{i=j\geq 1}]}{\sx^2\alpha_i^2}\,F_\alpha(\{\bx_i^\star\}).
\end{align}
The outlier term does not depend on $\sigma_{\boldsymbol{\varepsilon}}^2$. Since $x>0$ on 
$[a,b]$,
\[
r^2x+c\sigma_{\boldsymbol{\varepsilon}}^2\mu_0(x)\geq r^2x>0,
\]
so the bulk integrand is real-analytic in $\sigma_{\boldsymbol{\varepsilon}}^2$ for every 
$x\in[a,b]$. Since $[a,b]$ is compact, differentiation under the integral sign 
holds on every compact subinterval of $[0,\infty)$. Hence each 
$\bH_{ij}(\sigma_{\boldsymbol{\varepsilon}}^2)$ is real-analytic on $[0,\infty)$, 
and therefore
\[
b(\sigma_{\boldsymbol{\varepsilon}}^2)
=
\Bigl(\mathbf{I}_{\rk+1}+\mathfrak D_K\bH(\sigma_{\boldsymbol{\varepsilon}}^2)\Bigr)^{-1}
\gamma
\]
is real-analytic on $[0,\infty)$, and in particular so is $b^{(K)}_0(\sigma_{\boldsymbol{\varepsilon}}^2)$.
We next prove that
\begin{align}\label{eq: limit of b0}
\lim_{\sigma_{\boldsymbol{\varepsilon}}^2 \to \infty}b^{(K)}_0(\sigma_{\boldsymbol{\varepsilon}}^2) = \sx^2\,r^2\omega_0.
\end{align}
It suffices to show that $\bH_{0j}(\sigma_{\boldsymbol{\varepsilon}}^2)\to 0$ for all 
$0\leq j\leq \rk$, since then $\bH_{j0}(\sigma_{\boldsymbol{\varepsilon}}^2)\to 0$ 
by symmetry, and the first row and first column of 
$\bI_{\rk+1}+\mathfrak D_K\bH(\sigma_{\boldsymbol{\varepsilon}}^2)$ converge to those 
of the identity matrix. This implies that the first coordinate of the solution 
converges to the first coordinate of $\gamma$, namely $\sx^2\,r^2\omega_0$.
Using~\eqref{eq: expansion of H_ij tilde}, for all $0\leq j\leq \rk$,
\[
\bH_{0j}(\sigma_{\boldsymbol{\varepsilon}}^2)
=
\int_a^b 
\frac{x\,\mu_0(x)\mu_j(x)}{\sx^2\,r^2x+c\sx^2\,\sigma_{\boldsymbol{\varepsilon}}^2\mu_0(x)}
\,\mathrm{d}F_\alpha(x).
\]
For each $x\in[a,b]$, one has $\mu_0(x)>0$, so
\[
\frac{x\,\mu_0(x)\mu_j(x)}{\sx^2\,r^2x+c\sx^2\,\sigma_{\boldsymbol{\varepsilon}}^2\mu_0(x)}
\longrightarrow 0
\qquad\text{as }\sigma_{\boldsymbol{\varepsilon}}^2\to\infty.
\]
Moreover,
\[
0\leq
\frac{x\,\mu_0(x)\mu_j(x)}{\sx^2\,r^2x+c\sx^2\,\sigma_{\boldsymbol{\varepsilon}}^2\mu_0(x)}
\leq
\frac{x\,\mu_0(x)\mu_j(x)}{\sx^2\,r^2x}
=
\frac{\mu_0(x)\mu_j(x)}{\sx^2\,r^2},
\]
and the right-hand side is integrable on $[a,b]$. By dominated convergence, 
$\bH_{0j}(\sigma_{\boldsymbol{\varepsilon}}^2)\to 0$, and~\eqref{eq: limit of b0} 
follows. In particular, there exists $M>0$ such that $b^{(K)}_0(\sigma_{\boldsymbol{\varepsilon}}^2)
\neq 0$ for all $\sigma_{\boldsymbol{\varepsilon}}^2\geq M$.

It remains to consider the compact interval $[0,M]$. Since $b^{(K)}_0$ is 
real-analytic on $[0,M]$, if it had infinitely many zeros in $[0,M]$, then by 
compactness these zeros would have an accumulation point in $[0,M]$. By the 
identity theorem for real-analytic functions, $b^{(K)}_0$ would then vanish 
identically on $[0,\infty)$, contradicting $b^{(K)}_0(\sigma_{\boldsymbol{\varepsilon}}^2)\to 
\sx^2\,r^2\omega_0>0$ as $\sigma_{\boldsymbol{\varepsilon}}^2\to\infty$.

Therefore $b^{(K)}_0$ has only finitely many zeros in $[0,M]$ and none in 
$[M,\infty)$, which proves that there are at most finitely many values of 
$\sigma_{\boldsymbol{\varepsilon}}^2\geq 0$ such that $b^{(K)}_0(\sigma_{\boldsymbol{\varepsilon}}^2)=0$.

\end{proof}

\subsubsection{Simulations}

To numerically verify that $b^{(K)}_0 \neq 0$, we evaluate $b^{(K)}_0$ across a wide 
range of parameter configurations. Without loss of generality, we scale the covariance $\bSigma$ so that $\sx^2 = 1$. We fix base parameters $K=2$, $r=8$, 
$\sigma=5$, $c=4$, and base spike vectors 
$$\boldsymbol{\delta} = (1.0, 3.0, 0.1, 0.5, 4.5, 7.0, 7.5, 8.0)$$ and 
$$\boldsymbol{\alpha} = (3.0, 2.5, 2.0, 1.5, 1.2, 1.0, 0.8, 0.5).$$
For panels (a)-(d), we use the first two entries 
$\boldsymbol{\delta}_{1,2
} = (1.0, 3.0)$, $\boldsymbol{\alpha}_{1,2} = (3.0, 2.5)$, 
scaling the spike strengths by $t \in (1, 3)$ while holding the angles $\boldsymbol{\alpha}$ fixed, and varying one parameter at a 
time while holding the rest fixed at their base values. The results are 
displayed in Figure~\ref{fig:b0_heatmaps}.

\begin{figure}[tbp]
    \centering
    \includegraphics[width=1\linewidth]{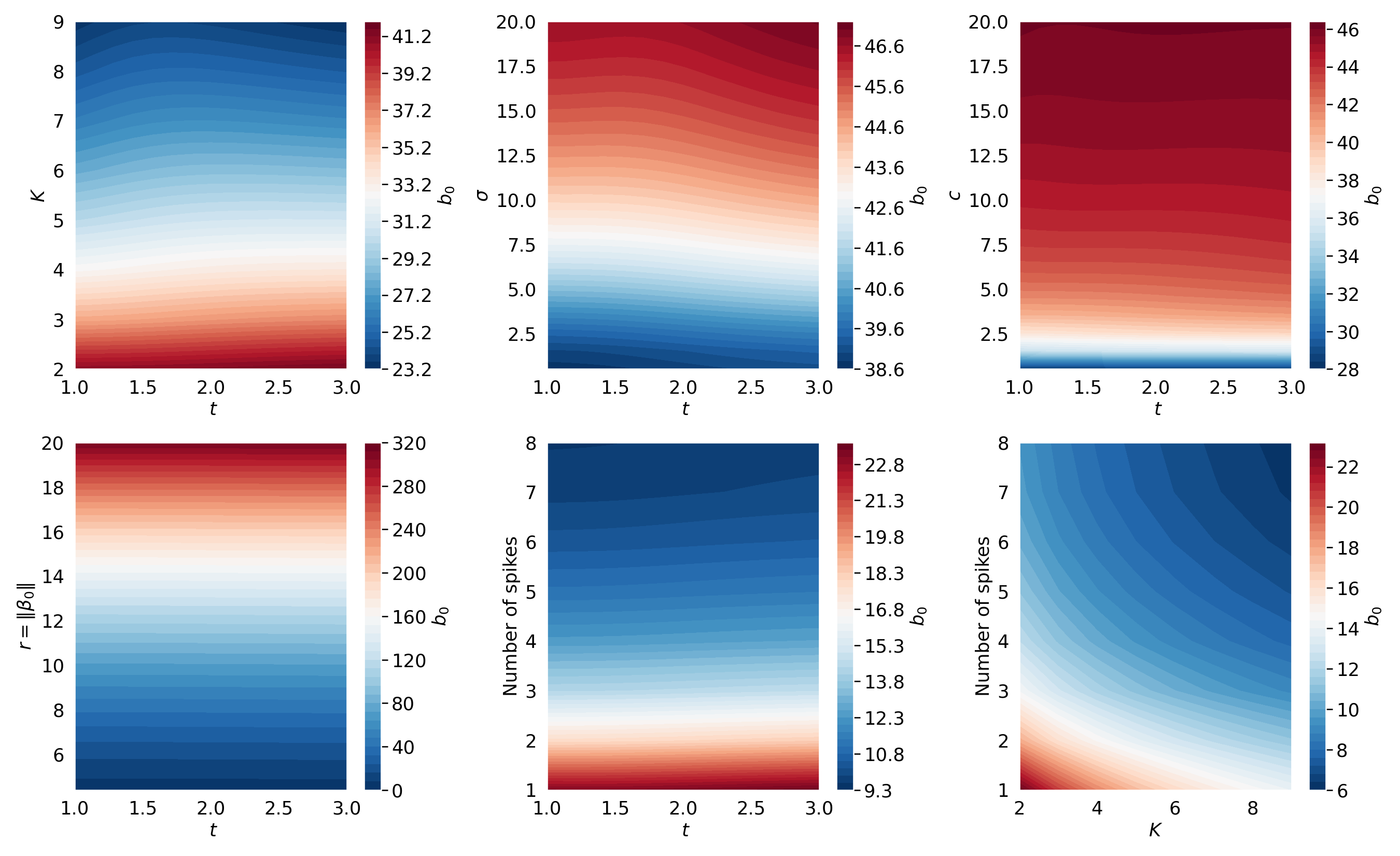}
    \caption{Numerical check that $b^{(K)}_0 \neq 0$ across a wide range 
    of parameter configurations.
    \textbf{(a)} $b^{(K)}_0$ vs.\ spike scale $t$ and number of clients $K \in \{2, \ldots, 9\}$.
    \textbf{(b)} $b^{(K)}_0$ vs.\ $t$ and noise level $\sigma \in (0.5, 20)$.
    \textbf{(c)} $b^{(K)}_0$ vs.\ $t$ and aspect ratio $c \in (0.5, 20)$, 
    covering both the underparametrized ($c < 1$) and overparametrized 
    ($c > 1$) regimes.
    \textbf{(d)} $b^{(K)}_0$ vs.\ $t$ and signal norm $r \in 
    (\sqrt{\sum_j \alpha_j^2},\ 20)$.
    \textbf{(e)} $b^{(K)}_0$ vs.\ $t$ and number of spikes 
    $\rk \in \{1, \ldots, 8\}$.
    \textbf{(f)} $b^{(K)}_0$ vs.\ $K$ and $\rk$ at fixed $t=1$, 
    jointly varying the number of clients and the number of 
    spikes. In all panels $b^{(K)}_0$ remains strictly positive.}
    \label{fig:b0_heatmaps}
\end{figure}

\section{Quadratic Forms of Products - Proof of Theorem~\ref{lemma:non-free limit}}\label{app: L}
In this section, we introduce several technical results needed to prove 
Theorem~\ref{lemma:non-free limit}. Recall $m(z)$ and $m_{\delta}(z)$ from 
Section~\ref{section: stieltjes} and Lemma~\ref{lemma: F_mp << F_delta}. 
Since Theorem~\ref{lemma:non-free limit} allows the two matrices to have 
different aspect ratios $c_\ell$ and $c_k$, we use the subscripts 
$m_{c_\ell}(z)$, $m_{c_k}(z)$ and $m_{\delta, \,c_\ell}(z)$, 
$m_{\delta, \,c_k}(z)$ to denote the corresponding Stieltjes transforms 
throughout this section.

We start by computing limits of products of resolvents of independent covariance matrices in the following preliminary Lemma. 

\begin{lemma}\label{lemma: product of non-free Ridge}
Let 
\[
\widehat{\boldsymbol{\Sigma}}_\ell = \frac{\bX_{\ell}^\top \bX_{\ell}}{n_\ell}, 
\qquad 
\widehat{\boldsymbol{\Sigma}}_k = \frac{\bX_{k}^\top \bX_{k}}{n_{k}}
\]
be two independent sample covariance matrices where the rows of  
$\bX_{\ell} \in \mathbb{R}^{n_\ell \times p}$ and $\bX_{k}\in \mathbb{R}^{n_k \times p}$ satisfy Assumption~\ref{assm:dgp}. Assume  that 
\(p/n_\ell \to c_\ell\) and \(p/n_{k} \to c_{k}\). Then, almost surely, for all $z_\ell,  z_k \in \mathbb{C}\setminus \mathbb{R}^+$ we have
\begin{align*}
& \frac{\bbeta_0^\top (\widehat{\boldsymbol{\Sigma}}_\ell - z_\ell\mathbf{I}_p)^{-1}
(\widehat{\boldsymbol{\Sigma}}_k -z_k\mathbf{I}_p)^{-1}\bbeta_0}{\|\bbeta_0\|_2^2} \\
& \xrightarrow{\mathrm{a.s.}}
\omega_0\,m_{c_\ell}(z_\ell)\,m_{c_k}(z_k)
+
\sum_{j=1}^{\rk}\omega_j\,m_{\delta_j,\,c_\ell}(z_\ell)\,m_{\delta_j,\,c_k}(z_k).
\end{align*}
\end{lemma}
For each fixed $z_\ell, z_{k}$, almost sure convergence follows from the 
deterministic equivalents for products of independent random matrices developed 
in~\cite{wonder, patil2023baggingoverparameterizedlearningrisk}. We then extend 
this to simultaneous convergence over all $z$ using the following theorem, due  to Vitali and Porter; for a proof, see Section~2.4 in~\cite{schiff1993normal}.
\begin{theorem}\label{thm: vitali-porter}
Let \(D\subseteq \mathbb C\) be a connected open set and $\{f_n\}$ be a locally bounded sequence of analytic functions on $D$. If $\lim_{n\to\infty} f_n(z)$ exists for each $z$ 
belonging to a set $E \subseteq D$ which has an accumulation point in $D$, then 
$\{f_n\}$ converges uniformly on compact subsets of $D$ to an analytic function.
\end{theorem}

The key argument to prove Theorem~\ref{lemma:non-free limit} is that we can extend Lemma~\ref{lemma: product of non-free Ridge} to general functions in $\mathcal{F}$. The proof of the theorem proceeds in two steps. First, we extend Lemma~\ref{lemma: product of non-free Ridge} from first-order resolvents to arbitrary powers of resolvents. For this first step, we use the following standard Lemma from complex analysis, 
which is commonly employed in the study of high-dimensional ridge asymptotics; 
see, for instance, \cite{hastie2022surprises, wonder, 
patil2023baggingoverparameterizedlearningrisk}. For a proof, see Lemma~2.14 
in \cite{lifesaver}.
\begin{lemma}\label{lemma: vitali}
Let \(D\subseteq \mathbb C\) be a connected open set, and let \(f_n:D\to\mathbb C\) be analytic for each \(n\). Assume that the sequence \((f_n)\) is uniformly bounded on \(D\), that is, there exists \(M>0\) such that
\[
|f_n(z)|\le M
\qquad\text{for all } n \text{ and all } z\in D.
\]
If \(f_n(z)\to f(z)\) pointwise on \(D\) for some analytic function \(f\), then
\[
f_n'(z)\to f'(z)
\qquad\text{for all } z\in D.
\]
\end{lemma}
For any analytic function \(f\), we write \(f^{(t)}(z)\) for its \(t\)-th derivative evaluated at \(z\) and for notational convenience, for any analytic function \(f\) and any integer \(t\ge 1\), define
\[
f^{[t]}(z):=\frac{1}{(t-1)!}\,f^{(t-1)}(z).
\]
Using this notation, we can formalize the following Lemma.
\begin{lemma}\label{lemma: higher-order non-free Ridge}
Under the Assumptions of Lemma~\ref{lemma: product of non-free Ridge}, with probability one, for any integers $t_\ell, t_k \ge 1$ and any 
$z_\ell,z_k\in\mathbb{C}\setminus\mathbb{R}_+$, we have
\begin{align*}
& \frac{\bbeta_0^\top (\widehat{\boldsymbol{\Sigma}}_\ell-z_\ell\mathbf{I}_p)^{-t_\ell}
(\widehat\Sigma_{k}-z_k\mathbf{I}_p)^{-t_k}\bbeta_0}{\|\bbeta_0\|_2^2} \\
& \xrightarrow{\mathrm{a.s.}}
\omega_0\,m_{c_\ell}^{[t_\ell]}(z_\ell)\,m_{c_k}^{[t_k]}(z_k)
+
\sum_{j=1}^{\rk}\omega_j\,m_{\delta_j,\,c_\ell}^{[t_\ell]}(z_\ell)\,
m_{\delta_j,\,c_k}^{[t_k]}(z_k).
\end{align*}
\end{lemma}
Fix now any real $\lambda > 0$ and let \[
\phi(x)=a_0+\sum_{q=1}^{Q} a_q (x+\lambda)^{-q}.
\]
By the integral representation of the Stieltjes transform, for every integer \(q\ge 1\),
\[
m^{[q]}(-\lambda)=\int \frac{1}{(x+\lambda)^q}\,\mathrm dF_{\mathrm{MP}}(x),
\qquad
m_{\delta}^{[q]}(-\lambda)=\int \frac{1}{(x+\lambda)^q}\,\mathrm dF_{\delta}(x).
\]
Thus
\[
a_0+\sum_{q=1}^{Q} a_q\,m^{[q]}(-\lambda)
=
\int \phi\,\mathrm dF_{\mathrm{MP}},
\qquad
a_0+\sum_{q=1}^{Q} a_q\,m_{\delta}^{[q]}(-\lambda)
=
\int \phi\,\mathrm dF_{\delta}.
\]
Therefore, since $-\lambda \in \mathbb{C}\setminus\mathbb{R}^+$, by linearity of the quadratic form and Lemma~\ref{lemma: higher-order non-free Ridge}, we obtain the following extension. 
\begin{corollary}\label{cor: linearity_extension_rational}
Let 
\[
\phi(x)=a_0+\sum_{q=1}^{Q} a_q (x+\lambda_\ell)^{-q},
\qquad
\psi(x)=b_0+\sum_{t=1}^{T} b_t (x+\lambda_k)^{-t_k},
\]
with $\lambda_\ell,\lambda_k>0$ and $a_0,a_q,b_0,b_t\in\mathbb{R}$. 
Then, with probability one,
\begin{align*}
& \frac{\bbeta_0^\top \phi(\widehat{\boldsymbol{\Sigma}}_\ell)\psi(\widehat{\boldsymbol{\Sigma}}_k)\bbeta_0}
{\|\bbeta_0\|_2^2} \\
& \xrightarrow{\mathrm{a.s.}}
\omega_0 
\left(\int \phi \,\mathrm{d}F_{\mathrm{MP},\,c_\ell}\right)
\left(\int \psi \,\mathrm{d}F_{\mathrm{MP},\,c_k}\right)
+
\sum_{j=1}^{\rk}\omega_j
\left(\int \phi \,\mathrm{d}F_{\delta_j,\,c_\ell}\right)
\left(\int \psi \,\mathrm{d}F_{\delta_j,\,c_k}\right).
\end{align*}
\end{corollary}
Finally, we can combine Corollary~\ref{cor: linearity_extension_rational} with the Weierstrass approximation theorem to obtain Theorem~\ref{lemma:non-free limit}.

\begin{proof}[Proof of Theorem~\ref{lemma:non-free limit}]
Let \(\mathrm{spec}(\widehat{\boldsymbol{\Sigma}}_\ell)\) and \(
\mathrm{spec}(\widehat{\boldsymbol{\Sigma}}_k)
\)
be the spectra of \(\widehat{\boldsymbol{\Sigma}}_\ell\) and \(\widehat{\boldsymbol{\Sigma}}_k\), respectively. By \cite{baik2004}, there exists a deterministic constant \(M<\infty\) such that, almost surely, there exists a random $N: = N(\omega)$ such that for \(p \geq N\) 
\begin{align}\label{eq: spectrum in [0,M]}
    \mathrm{spec}(\widehat{\boldsymbol{\Sigma}}_\ell)\cup \mathrm{spec}(\widehat{\boldsymbol{\Sigma}}_k)\subset [0,M].
\end{align}

Next, fix $\phi \in \mathcal{F}_{c_\ell}$ and $\psi\in\mathcal{F}_{c_k}$. 
Since the test functions are only required to be continuous on 
$\mathcal{S}_{c_\ell}$ (resp.\ $\mathcal{S}_{c_k}$) and to take finite 
values at the outliers, we may argue exactly as in Lemma~\ref{lem:int-conv}. 
We first replace $\phi$ and $\psi$ by continuous functions on $[0,M]$ that 
agree with the originals on $\mathcal{S}_{c_\ell}^+$ (resp.\ 
$\mathcal{S}_{c_k}^+$) and on sufficiently small neighborhoods of the 
outlier locations. For simplicity, we keep writing $\phi$ and $\psi$ for 
these extensions. This modification does not affect the argument, since for 
all large $p$ the eigenvalues of $\widehat{\boldsymbol{\Sigma}}_\ell$ and $\widehat{\boldsymbol{\Sigma}}_k$ 
lie in the region where the original and extended functions coincide, and the 
limiting integrals are unchanged for the same reason; see 
Lemma~\ref{lem:int-conv} for a fully worked-out proof of a similar argument.
Now define
\[
\widetilde \phi(y):=\phi\left(\frac1y-1\right),
\qquad
\widetilde \psi(y):=\psi\left(\frac1y-1\right),
\qquad y\in \Bigl[\frac1{M+1},1\Bigr].
\]
Since \(y\mapsto \frac1y-1\) is continuous on \(\bigl[\frac1{M+1},1\bigr]\), both \(\widetilde\phi\) and \(\widetilde\psi\) are continuous on that interval. By the Weierstrass approximation theorem, there exist polynomials \(P_m\) and \(Q_m\) such that
\[
\sup_{y\in[\frac1{M+1},1]}|P_m(y)-\widetilde\phi(y)|\xrightarrow[]{m\to\infty} 0,
\qquad
\sup_{y\in[\frac1{M+1},1]}|Q_m(y)-\widetilde\psi(y)|\xrightarrow[]{m\to\infty} 0.
\]
Set
\[
\phi_m(x):=P_m\left(\frac1{x+1}\right),
\qquad
\psi_m(x):=Q_m\left(\frac1{x+1}\right).
\]Then \(\phi_m,\psi_m\) are rational functions as in Corollary~\ref{cor: linearity_extension_rational}, and
\[
\sup_{x\in[0,M]}|\phi_m(x)-\phi(x)|\xrightarrow[]{m\to\infty} 0,
\qquad
\sup_{x\in[0,M]}|\psi_m(x)-\psi(x)|\xrightarrow[]{m\to\infty} 0.
\]
For \(m\ge 1\), define
\[
L_p:=
\frac{\bbeta_0^\top \phi(\widehat{\boldsymbol{\Sigma}}_\ell)\psi(\widehat{\boldsymbol{\Sigma}}_k)\bbeta_0}{\|\bbeta_0\|_2^2},
\qquad
L_{p,m}:=
\frac{\bbeta_0^\top \phi_m(\widehat{\boldsymbol{\Sigma}}_\ell)\psi_m(\widehat{\boldsymbol{\Sigma}}_k)\bbeta_0}{\|\bbeta_0\|_2^2},
\]
and
\begin{align*}
R & :=
\omega_0 
\left(\int \phi \,\mathrm{d}F_{\mathrm{MP},\,c_\ell}\right)
\left(\int \psi \,\mathrm{d}F_{\mathrm{MP},\,c_k}\right)
+
\sum_{j=1}^{\rk}\omega_j
\left(\int \phi \,\mathrm{d}F_{\delta_j,\,c_\ell}\right)
\left(\int \psi \,\mathrm{d}F_{\delta_j,\,c_k}\right), \\
R_m & := \omega_0 
\left(\int \phi_m \,\mathrm{d}F_{\mathrm{MP},\,c_\ell}\right)
\left(\int \psi_m \,\mathrm{d}F_{\mathrm{MP},\,c_k}\right) \\
& \qquad \qquad +
\sum_{j=1}^{\rk}\omega_j
\left(\int \phi_m \,\mathrm{d}F_{\delta_j,\,c_\ell}\right)
\left(\int \psi_m \,\mathrm{d}F_{\delta_j,\,c_k}\right).
\end{align*}
For each fixed \(m\), we already know from Corollary~\ref{cor: linearity_extension_rational} that as $p\to \infty$
\begin{align}\label{eq: L-p,m to R_m}
L_{p,m}\xrightarrow{\mathrm{a.s.}}R_m.
\end{align}
Moreover, since \(\phi_m\xrightarrow[]{m\to\infty}\phi\) and \(\psi_m\xrightarrow{m\to\infty}\psi\) uniformly on \([0,M]\), we deduce that for all probability measures \(F \in \{F_{\mathrm{MP}}, F_{\delta_1}, \dots, F_{\delta_\rk}\}\) whose supports are subsets of \([0,M]\),
\begin{align*}
\left|\int \phi_m\,\mathrm dF-\int \phi\,\mathrm dF\right| & \le \sup_{x\in[0,M]}|\phi_m(x)-\phi(x)|   \,, \\
\left|\int \psi_m\,\mathrm dF-\int \psi\,\mathrm dF\right| & \le \sup_{x\in[0,M]}|\psi_m(x)-\psi(x)| \,,
\end{align*}
and therefore \(R_m\xrightarrow{m\to\infty} R\).
Next, by the triangle inequality and the bound
$$
\frac{|\bbeta_0^\top A\bbeta_0|}{\|\bbeta_0\|_2^2} \le \|A\|_{\op},
$$
we have: 
\[
\begin{aligned}
|L_p-L_{p,m}|
&\le
\left\|
\phi(\widehat{\boldsymbol{\Sigma}}_\ell)\psi(\widehat{\boldsymbol{\Sigma}}_k)
-
\phi_m(\widehat{\boldsymbol{\Sigma}}_\ell)\psi_m(\widehat{\boldsymbol{\Sigma}}_k)
\right\|_{\op} \\
&\le
\|\phi(\widehat{\boldsymbol{\Sigma}}_\ell)-\phi_m(\widehat{\boldsymbol{\Sigma}}_\ell)\|_{\op}\,
\|\psi(\widehat{\boldsymbol{\Sigma}}_k)\|_{\op}
+
\|\phi_m(\widehat{\boldsymbol{\Sigma}}_\ell)\|_{\op}\,
\|\psi(\widehat{\boldsymbol{\Sigma}}_k)-\psi_m(\widehat{\boldsymbol{\Sigma}}_k)\|_{\op}.
\end{aligned}
\]
By linearity,  for all \(p \geq N\) so that~\eqref{eq: spectrum in [0,M]} holds,
\begin{align*}
    \|\phi(\widehat{\boldsymbol{\Sigma}}_\ell)-\phi_m(\widehat{\boldsymbol{\Sigma}}_\ell)\|_{\op}
& \le \sup_{x\in[0,M]}|\phi_m(x)-\phi(x)|, \\
\|\psi(\widehat{\boldsymbol{\Sigma}}_k)-\psi_m(\widehat{\boldsymbol{\Sigma}}_k)\|_{\op}
& \le \sup_{x\in[0,M]}|\psi_m(x)-\psi(x)|,
\end{align*}
and also
\[
\|\psi(\widehat{\boldsymbol{\Sigma}}_k)\|_{\op}\le \sup_{x\in[0,M]}|\psi(x)|,
\qquad
\|\phi_m(\widehat{\boldsymbol{\Sigma}}_\ell)\|_{\op}\le \sup_{x\in[0,M]}|\phi_m(x)|.
\]
Since \(\phi_m\to\phi\) uniformly on \([0,M]\), the quantity \(\sup_{x\in[0,M]}|\phi_m(x)|\) is bounded by a finite constant. Hence
\[
\sup_{p\geq N} |L_p-L_{p,m}| \xrightarrow[]{m\to\infty} 0.
\]
Finally, for all $p \geq N$
\begin{align*}
|L_p-R|
&\le |L_p-L_{p,m}|+|L_{p,m}-R_m|+|R_m-R| \\
&\le  \sup_{p\geq N} |L_p-L_{p,m}|+|L_{p,m}-R_m|+|R_m-R|
\end{align*}
To conclude, fix $\eta > 0$, and choose $m$ large enough such that both the first term and the third term are less than $\eta/3$. Finally, we can choose $p \geq  N$ large enough in such that the middle term is also less than $\eta/3$ by~\eqref{eq: L-p,m to R_m}. The conclusion follows.
\end{proof}

\subsection{Proof of Lemma~\ref{lemma: product of non-free Ridge}}
\begin{proof}
First, we prove that for any $z_{\ell}, z_k\in \mathbb{C}\setminus \mathbb{R}^+$, the conclusion holds almost surely. Recall from \eqref{eq:determinsitic_equivalent formula} that, for each  \(\ell\), there exists an analytic function \(a_{n_\ell}(z)\) on \(\mathbb C\setminus \mathbb R_+\) such that
\[
(\widehat{{\boldsymbol{\Sigma}}}_\ell-z\mathbf{I}_p)^{-1}
\asymp
(a_{n_\ell}(z){\boldsymbol{\Sigma}}-z\mathbf{I}_p)^{-1},
\qquad z\in \mathbb C\setminus \mathbb R_+.
\]
Similarly for \(k\), there exists an analytic function \(a_{n_k}(z)\) such that
\[
(\widehat{{\boldsymbol{\Sigma}}}_k-z\mathbf{I}_p)^{-1}
\asymp
(a_{n_k}(z){\boldsymbol{\Sigma}}-z\mathbf{I}_p)^{-1},
\qquad z\in \mathbb C\setminus \mathbb R_+.
\]
For any \(\delta>0\), define
\begin{align}\label{eq: D_delta}
D_\delta
:=
\{z\in\mathbb C\setminus \mathbb R_+:\ |\Im z|>\delta\} \cup \{z\in\mathbb C\setminus \mathbb R_+:\ \mathbb{R}e z<-\delta\}.
\end{align}
Since \(z_\ell\) and \(z_k\) are fixed and lie outside \(\mathbb R_+\), there exists \(\delta>0\) such that \(z_\ell,z_k\in D_\delta\). By the proof of Theorem~3.1(b) in \cite{wonder} (where \(x_p(z)\) corresponds to \(a_{n_\ell}(z)\) or \(a_{n_k}(z)\))
\begin{align}\label{eq: bound on D_delta}
\|(\widehat{\boldsymbol{\Sigma}}_\ell-z_\ell \mathbf{I}_p)^{-1}\|_2 \le \delta^{-1},
\qquad
\|(a_{n_\ell}(z_\ell){\boldsymbol{\Sigma}}-z_\ell \mathbf{I}_p)^{-1}\|_2 \le \delta^{-1},
\end{align}
and similarly for $k$. Hence the hypotheses of Lemma D.8.4 in \cite{patil2023baggingoverparameterizedlearningrisk} are satisfied, and therefore
\[
(\widehat{\boldsymbol{\Sigma}}_\ell-z_\ell \mathbf{I}_p)^{-1}
(\widehat{\boldsymbol{\Sigma}}_k-z_k \mathbf{I}_p)^{-1}
\asymp
(a_{n_\ell}(z_\ell){\boldsymbol{\Sigma}}-z_\ell\mathbf{I}_p)^{-1}
(a_{n_k}(z_k){\boldsymbol{\Sigma}}-z_k\mathbf{I}_p)^{-1}.
\]
In particular, if we let $\boldsymbol{\bTheta}_n = \bbeta_0\bbeta_0^\top$ in equation~\eqref{eq:determinsitic_equivalent formula} as in a previous section~\ref{proof of quadratic f_alpha} then
\begin{align}\label{eq: two matrices quad form limit}
&\bbeta_0^\top (\widehat{\boldsymbol{\Sigma}}_\ell-z_\ell \mathbf{I}_p)^{-1}(\widehat{\boldsymbol{\Sigma}}_k-z_k\mathbf{I}_p)^{-1}\bbeta_0 -
\bbeta_0^\top (a_{n_\ell}{\boldsymbol{\Sigma}}-z_\ell \mathbf{I}_p)^{-1}(a_{n_k}{\boldsymbol{\Sigma}}-z_k\mathbf{I}_p)^{-1}\bbeta_0
\xrightarrow{\mathrm{a.s.}}0.
\end{align}
Therefore, it is enough to compute the limit of the deterministic quadratic form. 
Recall from Proposition~\ref{lem:resolvent-spiked} that
\[
(a_{n_\ell}{\boldsymbol{\Sigma}}-z_\ell \mathbf{I}_p)^{-1}
=
\frac{1}{\sx^2\,a_{n_\ell}-z_\ell}
\left[
\mathbf{I}_p-\sum_{j=1}^{\rk}
\frac{\delta_j\,a_{n_\ell}}{\sx^2\,a_{n_\ell}-z_\ell+\delta_j\,a_{n_\ell}}\,\bv_j\bv_j^\top
\right],
\]
with a similar formula for $k$. Expanding the product, and using the orthonormality of the spike vectors \(v_1,\dots,v_{\rk}\),
we obtain
\begin{align*}
&(a_{n_\ell}{\boldsymbol{\Sigma}}-z_\ell \mathbf{I}_p)^{-1}(a_{n_k}{\boldsymbol{\Sigma}}-z_k \mathbf{I}_p)^{-1} \\
&=
\frac{1}{(\sx^2\,a_{n_\ell}-z_\ell)(\sx^2\,a_{n_k}-z_k)}
\left[
\mathbf{I}_p
+\sum_{j=1}^{\rk}
\left(
-\frac{\delta_j\,a_{n_\ell}}{\sx^2\,a_{n_\ell}-z_\ell+\delta_j\,a_{n_\ell}}
-\frac{\delta_j\,a_{n_k}}{\sx^2\,a_{n_k}-z_k+\delta_j\,a_{n_k}}
\right.\right.\\
&\hspace{4.8cm}\left.\left.
+\frac{\delta_j\,a_{n_\ell}}{\sx^2\,a_{n_\ell}-z_\ell+\delta_j\,a_{n_\ell}}\cdot
\frac{\delta_j\,a_{n_k}}{\sx^2\,a_{n_k}-z_k+\delta_j\,a_{n_k}}
\right)\bv_j\bv_j^\top
\right] \\
&=
\frac{1}{(\sx^2\,a_{n_\ell}-z_\ell)(\sx^2\,a_{n_k}-z_k)}\mathbf{I}_p \\
& +
\sum_{j=1}^{\rk}
\left[
\frac{1}{(a_{n_\ell}(\sx^2\,+\delta_j)-z_\ell)(a_{n_k}(\sx^2\,+\delta_j)-z_k)}
-\frac{1}{(\sx^2\,a_{n_\ell}-z_\ell)(\sx^2\,a_{n_k}-z_k)}
\right]\bv_j\bv_j^\top.
\end{align*}
Taking the quadratic form against \(\bbeta_0\), we get
\begin{align*}
&\bbeta_0^\top (a_{n_\ell}{\boldsymbol{\Sigma}}-z_\ell \mathbf{I}_p)^{-1}(a_{n_k}{\boldsymbol{\Sigma}}-z_k \mathbf{I}_p)^{-1}\bbeta_0 \\
&=
\frac{\|\bbeta_0\|_2^2}{(\sx^2\,a_{n_\ell}-z_\ell)(\sx^2\,a_{n_k}-z_k)} \\
& +
\sum_{j=1}^{\rk}(\bbeta_0^\top \bv_j)^2
\left[
\frac{1}{(a_{n_\ell}(\sx^2\,+\delta_j)-z_\ell)(a_{n_k}(\sx^2\,+\delta_j)-z_k)}
-\frac{1}{(\sx^2\,a_{n_\ell}-z_\ell)(\sx^2\,a_{n_k}-z_k)}
\right].
\end{align*}
Rearranging terms,
\begin{align*}
&\bbeta_0^\top (a_{n_\ell}{\boldsymbol{\Sigma}}-z_\ell \mathbf{I}_p)^{-1}(a_{n_k}{\boldsymbol{\Sigma}}-z_k \mathbf{I}_p)^{-1}\bbeta_0 \\
&=
\left(\|\bbeta_0\|_2^2-\sum_{j=1}^{\rk}(\bbeta_0^\top \bv_j)^2\right)
\frac{1}{(\sx^2\,a_{n_\ell}-z_\ell)(\sx^2\,a_{n_k}-z_k)} \\
& \qquad \qquad + \sum_{j=1}^{\rk}(\bbeta_0^\top \bv_j)^2
\frac{1}{(a_{n_\ell}(\sx^2\,+\delta_j)-z_\ell)(a_{n_k}(\sx^2\,+\delta_j)-z_k)}.
\end{align*}
We now let \(n\to\infty\). By Lemma~\ref{lemma: a(z) = z + 1/m(z)} and equation~\ref{eq: m_delta in terms of a(z)}, we know that
\[
\frac{1}{\sx^2\,a_{n_\ell}-z_\ell}\to m_{c_\ell}(z_\ell),
\qquad
\frac{1}{a_{n_\ell}(\sx^2\,+\delta_j)-z_\ell}\to m_{\delta_j, \,c_\ell}(z_\ell),
\]
and similar limits hold for $k$. Hence
\begin{align*}
\bbeta_0^\top (a_{n_\ell}{\boldsymbol{\Sigma}}-z_\ell \mathbf{I}_p)^{-1}(a_{n_k}{\boldsymbol{\Sigma}}-z_k \mathbf{I}_p)^{-1}\bbeta_0 & \xrightarrow{\mathrm{a.s.}}
\left(r^2-\sum_{j=1}^{\rk}\alpha_j^2\right)m_{c_\ell}(z_\ell)m_{c_k}(z_k) \\
& \qquad \qquad + \sum_{j=1}^{\rk}\alpha_j^2\,m_{\delta_j,\,c_\ell}(z_\ell)m_{\delta_j, \,c_k}(z_k).
\end{align*}
The conclusion now follows after dividing by \(\|\bbeta_0\|_2^2\) and using equation~\eqref{eq: two matrices quad form limit}.

It remains to promote this to almost sure simultaneous convergence over all $z_\ell, z_k \in D_\delta$ on a single probability-one event via a 
standard density argument. Fix $\delta > 0$, and for all $z,w\in D_{\delta}$, define
\begin{align*}
F_n(z, w)  & := \frac{\beta_0^\top (\widehat\Sigma_\ell - z\mathbf{I}_p)^{-1}
(\widehat\Sigma_{\ell'} - w\mathbf{I}_p)^{-1}\beta_0}{\|\beta_0\|_2^2}, \\
F(z,w)  & := \omega_0\,m_{c_\ell}(z)\,m_{c_k}(w) 
+ \sum_{j=1}^{\rk}\omega_j\,m_{c_\ell, \, \delta_j}(z)\,m_{c_k, \, \delta_j}(w).
\end{align*}
Both $F_n$ and $F$ are analytic in each variable on $D_\delta$, and 
by~\eqref{eq: bound on D_delta}, $|F_n(z,w)| \leq \delta^{-2}$ uniformly 
in $n$ and $z,w\in D_\delta$. Fix $w \in D_\delta$. By our previous work, 
$F_n(z, w) \xrightarrow{\mathrm{a.s.}} F(z,w)$ for each fixed $z \in D_\delta$. 
Let $E = \{z_k\}_{k \geq 1} \subset D_\delta$ be a countable set with an 
accumulation point in $D_\delta$, and define
\[
\Omega_{0,w} 
:= \bigcap_{k=1}^\infty \bigl\{\omega : F_n(z_k, w) \to F(z_k, w)\bigr\},
\]
which satisfies $\mathbb{P}(\Omega_{0,w}) = 1$ as a countable intersection of 
almost sure events. On $\Omega_{0,w}$, the sequence $z \mapsto F_n(z, w)$ 
consists of analytic functions on $D_\delta$, is uniformly bounded by $\delta^{-2}$, 
and converges pointwise on $E$. By Theorem~\ref{thm: vitali-porter}, $F_n(\cdot, w)$ 
converges uniformly on compact subsets of $D_\delta$ to an analytic limit.
Since this limit agrees with $F(\cdot, w)$ on $E$, the identity theorem
implies that the limit equals $F(\cdot, w)$ on all of $D_\delta$.
Hence, on $\Omega_{0,w}$, $F_n(z,w) \to F(z,w)$ for all $z \in D_\delta$.
Similarly, we can extend the convergence to all $w \in D_{\delta}$ so that 
on the same probability-one common set $\Omega_0$, 
we have $F_n(z,w) \to F(z,w)$ for all $z, w \in D_\delta$.

Finally, set $\delta_n = 1/n$, so that $D_{\delta_n}\subseteq D_{\delta_{n+1}}$ and 
$\bigcup_{n=1}^\infty D_{\delta_n} = \mathbb{C}\setminus\mathbb{R}^+$. 
By the argument above, for each $n$, there exists $\Omega_{0,n}$ with 
$\mathbb{P}(\Omega_{0,n})=1$ on which convergence holds for all 
$z,w\in D_{\delta_n}$. Setting
\[
\Omega_0 := \bigcap_{n=1}^\infty \Omega_{0,n},
\]
we have $\mathbb{P}(\Omega^*)=1$. For any $z_\ell, z_{\ell'}\in\mathbb{C}
\setminus\mathbb{R}^+$, choose $n$ large enough so that 
$z_\ell, z_{\ell'}\in D_{\delta_n}$ to complete the proof.
\end{proof}

\subsection{Proof of Lemma~\ref{lemma: higher-order non-free Ridge}}
\begin{proof}
Fix \(\delta>0\) and work on \(D_\delta\) (see the definition in \eqref{eq: D_delta}), where the resolvents are uniformly bounded in operator norm by~\eqref{eq: bound on D_delta}. For \(t_\ell,t_k\ge 1\) and \(z,w\in D_\delta\), define
\[
F_n^{(t_\ell,t_k)}(z,w)
:=
\frac{\bbeta_0^\top (\widehat{\boldsymbol{\Sigma}}_\ell-z\mathbf{I}_p)^{-t_\ell}
(\widehat{\boldsymbol{\Sigma}}_k-w\mathbf{I}_p)^{-t_k}\bbeta_0}{\|\bbeta_0\|_2^2},
\]
and
\[
F^{(t_\ell,t_k)}(z,w)
:=
\omega_0\,m_{c_\ell}^{[t_\ell]}(z)\,m_{c_k}^{[t_k]}(w)
+
\sum_{j=1}^{\rk}\omega_j\,m_{\delta_j, \, c_\ell}^{[t_\ell]}(z)\,m_{\delta_j, \, c_k}^{[t_k]}(w).
\]
We prove by induction on \(t_\ell + t_k\) that almost surely, for all $z,w\in D_\delta$
\[
F_n^{(t_\ell,t_k)}(z,w)\xrightarrow{\mathrm{a.s.}}F^{(t_\ell,t_k)}(z,w).
\]
The base case $t_\ell = t_k = 1$ follows from Lemma~\ref{lemma: product of non-free Ridge}, 
which gives a probability-one event $\Omega_0$ on which 
$F_n^{(1,1)}(z,w)\to F^{(1,1)}(z,w)$ for all $z,w\in D_\delta$ simultaneously.  From now on, we work on $\Omega_0$.
Now assume the claim holds for some pair \((t_\ell,t_k)\). We prove it for \((t_\ell + 1,t_k)\).  By~\eqref{eq: bound on D_delta}
\[
\|(\widehat{\boldsymbol{\Sigma}}_\ell-z\mathbf{I}_p)^{-t_\ell}\|_2\le \left(\|(\widehat{\boldsymbol{\Sigma}}_\ell-z\mathbf{I}_p)^{-1}\|_2\right)^{t_\ell} \le \delta^{-t_\ell},
\]
and similarly for $(k, w, t_k)$. We deduce the uniform bound
\[
|F_n^{(t_\ell,t_k)}(z,w)|
\le
\delta^{-(t_\ell + t_k)}
\qquad\text{for all } z,w\in D_\delta.
\]
Thus, for each fixed \(w\in D_\delta\), the functions \(z\mapsto F_n^{(t_\ell,t_k)}(z,w)\) are analytic and uniformly bounded on \(D_\delta\), and by the induction hypothesis they converge pointwise to an analytic function \(z\mapsto F^{(t_\ell,t_k)}(z,w)\). Hence, by Lemma~\ref{lemma: vitali},
\[
\partial_z F_n^{(t_\ell,t_k)}(z,w)\xrightarrow{\mathrm{a.s.}} \partial_z F^{(t_\ell,t_k)}(z,w).
\]
Now
\[
\partial_z(\widehat{\boldsymbol{\Sigma}}_\ell-z\mathbf{I}_p)^{-t_\ell}
=
t_\ell(\widehat{\boldsymbol{\Sigma}}_\ell-z\mathbf{I}_p)^{-(t_\ell+1)},
\]
so
\[
\partial_z F_n^{(t_\ell,t_k)}(z,w)=t_\ell\,F_n^{(t_\ell+1,t_k)}(z,w).
\]
On the other hand, since \(\frac{d}{dz}m^{[t_\ell]}(z)=t_\ell\,m^{[t_\ell+1]}(z)\) and similarly for \(m_{\delta}^{[t_\ell]}\),
\[
\partial_z F^{(t_\ell,t_k)}(z,w)=t_\ell\,F^{(t_\ell+1,t_k)}(z,w).
\]
Therefore
\[
F_n^{(t_\ell+1,t_k)}(z,w)\xrightarrow{\mathrm{a.s.}}F^{(t_\ell+1,t_k)}(z,w)
\qquad\text{for all } z,w\in D_\delta.
\]
The same argument, now differentiating with respect to \(w\) holds, so the induction step is complete. Thus, the claim holds for all integers \(t_\ell,t_k\ge 1\).  Finally, since \(z_\ell,z_k\in\mathbb C\setminus\mathbb R_+\) are fixed, we may choose \(\delta>0\) such that \(z_\ell,z_k\in D_\delta\), and then evaluating the preceding limit at \(z=z_\ell\) and \(w=z_k\) yields the claim.
\end{proof}

\end{document}